\RequirePackage{fix-cm}
\documentclass[11pt,a4paper,twoside]{book}

\usepackage{ERPhDStyle}
\usepackage{ERIndex}
\usepackage{ERGlossary}

\usepackage{ERPhDCommands}

\newglossaryentry{not:Xtil}{
		type=notation,
		name={\ensuremath{(X,\sim)}},
		symbol={\ensuremath{(X,\sim)}},
		description={a set $X$ with equivalence relation $\sim$},
		sort={Xtil}}

\newglossaryentry{not:xbar}{
		type=notation,
		name={\ensuremath{\overline{x}}},
		symbol={\ensuremath{\overline{x}}},
		description={the equivalence class of $x$ in $(X,\sim)$},
		sort={xbar}}
		
\newglossaryentry{not:Xbar}{
		type=notation,
		name={\ensuremath{\overline{X}}},
		symbol={\ensuremath{\overline{X}}},
		description={the set of equivalence classes of $(X,\sim)$},
		sort={Xbar}}

\newglossaryentry{not:SymXtil}{
		type=notation,
		name={\ensuremath{\Sym(X,\sim)}},
		symbol={\ensuremath{\Sym(X,\sim)}},
		description={the group of bijections of $X$ to itself, preserving $\sim$},
		sort={SymXtil}}

\newglossaryentry{not:gbar}{
		type=notation,
		name={\ensuremath{\overline{g}}},
		symbol={\ensuremath{\overline{g}}},
		description={for $g\in\Sym(X,\sim)$, the induced action of $g$ on $\overline{X}$},
		sort={gbar}}

\newglossaryentry{not:M}{
		type=notation,
		name={\ensuremath{\M}},
		symbol={\ensuremath{\M}},
		description={a local Moufang set, usually acting on $(X,\sim)$ with root groups $\{U_x\}_{x\in X}$},
		sort={Maa}}

\newglossaryentry{not:Ux}{
		type=notation,
		name={\ensuremath{U_x}},
		symbol={\ensuremath{U_x}},
		description={the root group corresponding to $x$ in a local Moufang set},  
		sort={Ux}}
		
\newglossaryentry{not:G}{
		type=notation,
		name={\ensuremath{G}},
		symbol={\ensuremath{G}},
		description={the little projective group of a local Moufang set $\M$},  
		sort={Ga}}

\newglossaryentry{not:Uxbar}{
		type=notation,
		name={\ensuremath{U_{\overline{x}}}},
		symbol={\ensuremath{U_{\overline{x}}}},
		description={the induced action of $U_x$ on $\overline{X}$},
		sort={Uxbar}}
		
\newglossaryentry{not:Mbar}{
		type=notation,
		name={\ensuremath{\overline{\M}}},
		symbol={\ensuremath{\overline{\M}}},
		description={the quotient Moufang set of $\M$},
		sort={Maa(bar)}}

\newglossaryentry{not:0infty}{
		type=notation,
		name={\ensuremath{0,\infty}},
		symbol={\ensuremath{0,\infty}},
		description={a pair of fixed chosen elements of $X$ with $0\nsim\infty$},
		sort={0}}

\newglossaryentry{not:alphax}{
		type=notation,
		name={\ensuremath{\alpha_x}},
		symbol={\ensuremath{\alpha_x}},
		description={unique element of $U_\infty$ mapping $0$ to $x$ for $x\nsim\infty$},
		sort={a0}}

\newglossaryentry{not:gammax}{
		type=notation,
		name={\ensuremath{\gamma_x}},
		symbol={\ensuremath{\gamma_x}},
		description={unique element of $U_0$ mapping $\infty$ to $x\tau$, defined as $\alpha_x^\tau$},
		sort={g0}}

\newglossaryentry{not:Uinfcirc}{
		type=notation,
		name={\ensuremath{U_\infty^\circ}},
		symbol={\ensuremath{U_\infty^\circ}},
		description={the group $\{\alpha_x\mid x\sim0\}$},
		sort={Uinfcirc}}

\newglossaryentry{not:Uinftimes}{
		type=notation,
		name={\ensuremath{U_\infty^\times}},
		symbol={\ensuremath{U_\infty^\circ}},
		description={the set $\{\alpha_x\mid x\nsim0\}$},
		sort={Uinftimes}}
		
\newglossaryentry{not:Uxcirc}{
		type=notation,
		name={\ensuremath{U_x^\circ}},
		symbol={\ensuremath{U_x^\circ}},
		description={the group $\{ u\in U_x\mid \overline{u}=\id\}$},
		sort={Uxcirc}}

\newglossaryentry{not:Uxtimes}{
		type=notation,
		name={\ensuremath{U_x^\times}},
		symbol={\ensuremath{U_x^\circ}},
		description={the set $\{ u\in U_x\mid \overline{u}\neq\id\}$},
		sort={Uxtimes}}

\newglossaryentry{not:mux}{
		type=notation,
		name={\ensuremath{\mu_x}},
		symbol={\ensuremath{\mu_x}},
		description={$\mu$-map corresponding to $x$},
		sort={m0x}}

\newglossaryentry{not:tau}{
		type=notation,
		name={\ensuremath{\tau}},
		symbol={\ensuremath{\tau}},
		description={fixed element swapping $0$ and $\infty$},
		sort={t0au}}

\newglossaryentry{not:tilx}{
		type=notation,
		name={\ensuremath{\til x}},
		symbol={\ensuremath{\til x}},
		description={defined as $(-(x\tau^{-1}))\tau$},
		sort={xtil}}

\newglossaryentry{not:hx}{
		type=notation,
		name={\ensuremath{h_x}},
		symbol={\ensuremath{h_x}},
		description={Hua map corresponding to $x$ if $\tau$ is clear from context},
		sort={hx}}

\newglossaryentry{not:hxtau}{
		type=notation,
		name={\ensuremath{h_{x,\tau}}},
		symbol={\ensuremath{h_{x,\tau}}},
		description={Hua map corresponding to $x$, defined a $\tau\mu_x$},
		sort={hxtau}}

\newglossaryentry{not:H}{
		type=notation,
		name={\ensuremath{H}},
		symbol={\ensuremath{H}},
		description={Hua subgroup of $\M$},
		sort={H}}

\newglossaryentry{not:Gx}{
		type=notation,
		name={\ensuremath{G_x}},
		symbol={\ensuremath{G_x}},
		description={point stabilizer of $x$ in $G$},
		sort={Gx}}

\newglossaryentry{not:Gxy}{
		type=notation,
		name={\ensuremath{G_{x,y}}},
		symbol={\ensuremath{G_{x,y}}},
		description={two-point stabilizer of $x$ and $y$ in $G$},
		sort={Gxy}}

\newglossaryentry{not:xy}{
		type=notation,
		name={\ensuremath{{}^xy}},
		symbol={\ensuremath{{}^xy}},
		description={left quasi-inverse of a quasi-invertible pair $(x,y)$},
		sort={xy}}

\newglossaryentry{not:xy2}{
		type=notation,
		name={\ensuremath{x^y}},
		symbol={\ensuremath{x^y}},
		description={right quasi-inverse of a quasi-invertible pair $(x,y)$},
		sort={xy2}}

\newglossaryentry{not:M(U,tau)}{
		type=notation,
		name={\ensuremath{\M(U,\tau)}},
		symbol={\ensuremath{\M(U,\tau)}},
		description={local Moufang set acting on $(X,\sim)$ constructed with $U$ and $\tau$},
		sort={Maa(U,tau)}}

\newglossaryentry{not:phi}{
		type=notation,
		name={\ensuremath{\phi}},
		symbol={\ensuremath{\phi}},
		description={homomorphism between local Moufang sets},
		sort={p0hi}}
		
\newglossaryentry{not:phibar}{
		type=notation,
		name={\ensuremath{\overline{\phi}}},
		symbol={\ensuremath{\overline{\phi}}},
		description={induced homomorphism between the quotient Moufang sets},
		sort={p0hibar}}

\newglossaryentry{not:thetax}{
		type=notation,
		name={\ensuremath{\theta_x}},
		symbol={\ensuremath{\theta_x}},
		description={group homomorphism on $U_x$ induced by a homomorphism between local Moufang sets},
		sort={thetax}}

\newglossaryentry{not:imphi}{
		type=notation,
		name={\ensuremath{\Im(\phi)}},
		symbol={\ensuremath{\Im(\phi)}},
		description={image of a homomorphism between the quotient Moufang sets},
		sort={imphi}}

\newglossaryentry{not:Igeq}{
		type=notation,
		name={\ensuremath{(I,\succcurlyeq)}},
		symbol={\ensuremath{(I,\succcurlyeq)}},
		description={directed set on index set $I$ with transitive, reflexive relation $\succcurlyeq$},
		sort={Igeq}}

\newglossaryentry{not:Miphiij}{
		type=notation,
		name={\ensuremath{(\M_i,\phi_{ij})}},
		symbol={\ensuremath{(\M_i,\phi_{ij})}},
		description={inverse system of local Moufang sets indexed by the directed set $(I,\succcurlyeq)$},
		sort={Miphiij}}

\newglossaryentry{not:Xiphiij}{
		type=notation,
		name={\ensuremath{(X_i,\phi_{ij})}},
		symbol={\ensuremath{(X_i,\phi_{ij})}},
		description={inverse system of sets indexed by the directed set $(I,\succcurlyeq)$},
		sort={Xiphiij}}

\newglossaryentry{not:limMi}{
		type=notation,
		name={\ensuremath{\protect\varprojlim\M_i}},
		symbol={\ensuremath{\protect\varprojlim\M_i}},
		description={inverse limit of the inverse system $(\M_i,\phi_{ij})$},
		sort={limMi}}

\newglossaryentry{not:limXi}{
		type=notation,
		name={\ensuremath{\protect\varprojlim X_i}},
		symbol={\ensuremath{\protect\varprojlim X_i}},
		description={inverse limit of the inverse system $(X_i,\phi_{ij})$},
		sort={limXi}}

\newglossaryentry{not:R}{
		type=notation,
		name={\ensuremath{R}},
		symbol={\ensuremath{R}},
		description={local ring with maximal ideal $\m$},
		sort={R}}

\newglossaryentry{not:Rhat}{
		type=notation,
		name={\ensuremath{\complete{R}}},
		symbol={\ensuremath{\complete{R}}},
		description={completion of a local ring $R$},
		sort={Raaa}}

\newglossaryentry{not:Rtimes}{
		type=notation,
		name={\ensuremath{R^\times}},
		symbol={\ensuremath{R^\times}},
		description={invertible elements of $R$},
		sort={Raacross}}

\newglossaryentry{not:m}{
		type=notation,
		name={\ensuremath{\m}},
		symbol={\ensuremath{\m}},
		description={maximal ideal of a local ring $R$},
		sort={ma}}

\newglossaryentry{not:P1R}{
		type=notation,
		name={\ensuremath{\P^1(R)}},
		symbol={\ensuremath{\P^1(R)}},
		description={projective line over $R$},
		sort={Pa}}

\newglossaryentry{not:[ab]}{
		type=notation,
		name={\ensuremath{[a,b]}},
		symbol={\ensuremath{[a,b]}},
		description={point of $\P^1(R)$},
		sort={ab]}}

\newglossaryentry{not:M(R)}{
		type=notation,
		name={\ensuremath{\M(R)}},
		symbol={\ensuremath{\M(R)}},
		description={projective local Moufang set over $R$},
		sort={Maa(R)}}

\newglossaryentry{not:W}{
		type=notation,
		name={\ensuremath{W}},
		symbol={\ensuremath{W}},
		description={a right $R$-modules},
		sort={W}}

\newglossaryentry{not:P(W)}{
		type=notation,
		name={\ensuremath{\P(W)}},
		symbol={\ensuremath{\P(W)}},
		description={projective space corresponding to $R\times W\times R$},
		sort={Pa(W}}

\newglossaryentry{not:[rxs]}{
		type=notation,
		name={\ensuremath{[r,x,s]}},
		symbol={\ensuremath{[r,x,s]}},
		description={point of $\P(W)$},
		sort={rws}}

\newglossaryentry{not:q}{
		type=notation,
		name={\ensuremath{q}},
		symbol={\ensuremath{q}},
		description={a quadratic form on $W$},
		sort={q}}

\newglossaryentry{not:qtil}{
		type=notation,
		name={\ensuremath{\tilde{q}}},
		symbol={\ensuremath{\tilde{q}}},
		description={extended quadratic form of $q$ on $R\times W\times R$},
		sort={qtil}}

\newglossaryentry{not:f}{
		type=notation,
		name={\ensuremath{f}},
		symbol={\ensuremath{f}},
		description={bilinear form corresponding to the quadratic form $q$},
		sort={f}}

\newglossaryentry{not:QWq}{
		type=notation,
		name={\ensuremath{\rQ(W,q)}},
		symbol={\ensuremath{\rQ(W,q)}},
		description={isotropic points of $\tilde{q}$ in $\P(W)$},
		sort={QWq}}

\newglossaryentry{not:M(W,q)}{
		type=notation,
		name={\ensuremath{\M(W,q)}},
		symbol={\ensuremath{\M(W,q)}},
		description={orthogonal local Moufang set corresponding to $q$},
		sort={Maa(W,q)}}

\newglossaryentry{not:ast}{
		type=notation,
		name={\ensuremath{{}^\ast}},
		symbol={\ensuremath{{}^\ast}},
		description={an involution on $R$},
		sort={0ast}}

\newglossaryentry{not:Z(R)}{
		type=notation,
		name={\ensuremath{Z(R)}},
		symbol={\ensuremath{Z(R)}},
		description={the center of $R$},
		sort={Z(R)}}

\newglossaryentry{not:e0}{
		type=notation,
		name={\ensuremath{\epsilon}},
		symbol={\ensuremath{\epsilon}},
		description={an element of $Z(R)$ such that $\epsilon\epsilon^\ast=1$},
		sort={e0}}

\newglossaryentry{not:Lambda}{
		type=notation,
		name={\ensuremath{\Lambda}},
		symbol={\ensuremath{\Lambda}},
		description={a form parameter},
		sort={L0}}

\newglossaryentry{not:Lambdamin}{
		type=notation,
		name={\ensuremath{\Lambda_{\mathrm{min}}}},
		symbol={\ensuremath{\Lambda_{\mathrm{min}}}},
		description={the set of generalized traces},
		sort={L01}}

\newglossaryentry{not:Lambdamax}{
		type=notation,
		name={\ensuremath{\Lambda_{\mathrm{max}}}},
		symbol={\ensuremath{\Lambda_{\mathrm{max}}}},
		description={the set of skew-symmetric elements},
		sort={L02}}

\newglossaryentry{not:RLambda}{
		type=notation,
		name={\ensuremath{(R,\Lambda)}},
		symbol={\ensuremath{(R,\Lambda)}},
		description={a form ring},
		sort={RLambda}}

\newglossaryentry{not:fq}{
		type=notation,
		name={\ensuremath{(f,q)}},
		symbol={\ensuremath{(f,q)}},
		description={a $\Lambda$-quadratic form},
		sort={fq}}
		
\newglossaryentry{not:fqtil}{
		type=notation,
		name={\ensuremath{(\tilde{f},\tilde{q})}},
		symbol={\ensuremath{(\tilde{f},\tilde{q})}},
		description={extended $\Lambda$-quadratic form of $(f,q)$ on $R\times W\times R$},
		sort={fqtil}}

\newglossaryentry{not:HWq}{
		type=notation,
		name={\ensuremath{\rH(W,q)}},
		symbol={\ensuremath{\rH(W,q)}},
		description={isotropic points of $(\tilde{f},\tilde{q})$ in $\P(W)$},
		sort={HWq}}

\newglossaryentry{not:HW}{
		type=notation,
		name={\ensuremath{\rH}},
		symbol={\ensuremath{\rH}},
		description={$\rH(W,q)$ when $W$ and $q$ are clear from context},
		sort={HW}}

\newglossaryentry{not:M(W,f,q)}{
		type=notation,
		name={\ensuremath{\M(W,f,q)}},
		symbol={\ensuremath{\M(W,f,q)}},
		description={hermitian local Moufang set corresponding to $(f,q)$},
		sort={Maa(W,f,q)}}

\newglossaryentry{not:xyJP}{
		type=notation,
		name={\ensuremath{x^y}},
		symbol={\ensuremath{x^y}},
		description={quasi-inverse of a quasi-invertible pair $(x,y)$ in a Jordan pair},
		sort={xyJP}}

\newglossaryentry{not:V}{
		type=notation,
		name={\ensuremath{V}},
		symbol={\ensuremath{V}},
		description={a Jordan pair $V=(V^+,V^-)$},
		sort={V}}

\newglossaryentry{not:Vp}{
		type=notation,
		name={\ensuremath{(V^+,V^-)}},
		symbol={\ensuremath{(V^+,V^-)}},
		description={a Jordan pair $V=(V^+,V^-)$},
		sort={Vp}}

\newglossaryentry{not:Qx}{
		type=notation,
		name={\ensuremath{Q_x}},
		symbol={\ensuremath{Q_x}},
		description={quadratic maps in a Jordan pair},
		sort={Qx}}

\newglossaryentry{not:Qxz}{
		type=notation,
		name={\ensuremath{Q_{x,z}}},
		symbol={\ensuremath{Q_{x,z}}},
		description={bilinearization of the quadratic maps in a Jordan pair},
		sort={Qxz}}

\newglossaryentry{not:Dxy}{
		type=notation,
		name={\ensuremath{D_{x,y}}},
		symbol={\ensuremath{D_{x,y}}},
		description={bilinear maps in a Jordan pair},
		sort={Dxy}}

\newglossaryentry{not:xyz}{
		type=notation,
		name={\ensuremath{\{x\;y\;z\}}},
		symbol={\ensuremath{\{x\;y\;z\}}},
		description={triple product in a Jordan pair},
		sort={xyz}}

\newglossaryentry{not:Bxy}{
		type=notation,
		name={\ensuremath{B_{x,y}}},
		symbol={\ensuremath{B_{x,y}}},
		description={Bergman operators in a Jordan pair},
		sort={Bxy}}

\newglossaryentry{not:RadV}{
		type=notation,
		name={\ensuremath{\Rad V}},
		symbol={\ensuremath{\Rad V}},
		description={radical of a Jordan pair $V$},
		sort={RadV}}

\newglossaryentry{not:P(V)}{
		type=notation,
		name={\ensuremath{\P(V)}},
		symbol={\ensuremath{\P(V)}},
		description={projective space over $V$},
		sort={Pa(V)}}
		
\newglossaryentry{not:[xy]}{
		type=notation,
		name={\ensuremath{[x,y]}},
		symbol={\ensuremath{[x,y]}},
		description={point of $\P(V)$},
		sort={xy]}}

\newglossaryentry{not:M(V)}{
		type=notation,
		name={\ensuremath{\M(V)}},
		symbol={\ensuremath{\M(V)}},
		description={local Moufang set corresponding to $V$},
		sort={Maa(V)}}

\newglossaryentry{not:xn}{
		type=notation,
		name={\ensuremath{x\cdot n}},
		symbol={\ensuremath{x\cdot n}},
		description={$n$ times $x$ in $U_\infty$, defined as $0\alpha_x^n$},  
		sort={xn}}

\newglossaryentry{not:xntil}{
		type=notation,
		name={\ensuremath{x\protect\cdottil n}},
		symbol={\ensuremath{x\protect\cdottil n}},
		description={$n$ times $x$ in $U_0$, defined as $\infty\gamma_{x\tau^{-1}}^n$},  
		sort={xntil}}

\newglossaryentry{not:xyplus}{
		type=notation,
		name={\ensuremath{x+y}},
		symbol={\ensuremath{x+y}},
		description={defined as $0\alpha_x\alpha_y$},
		sort={xyplus}}

\newglossaryentry{not:xyplustil}{
		type=notation,
		name={\ensuremath{x\protect\plustil y}},
		symbol={\ensuremath{x\protect\plustil y}},
		description={defined as $\infty\gamma_{x\tau^{-1}}\gamma_{y\tau^{-1}}$},
		sort={xyplustil}}

\newglossaryentry{not:xymin}{
		type=notation,
		name={\ensuremath{x-y}},
		symbol={\ensuremath{x-y}},
		description={defined as $0\alpha_x\alpha_y^{-1}$},
		sort={xymin}}

\newglossaryentry{not:xymintil}{
		type=notation,
		name={\ensuremath{x\protect\mintil y}},
		symbol={\ensuremath{x\protect\mintil y}},
		description={defined as $\infty\gamma_{x\tau^{-1}}\gamma_{y\tau^{-1}}^{-1}$},
		sort={xymintil}}

\newglossaryentry{not:muxy}{
		type=notation,
		name={\ensuremath{\mu_{x,y}}},
		symbol={\ensuremath{\mu_{x,y}}},
		description={bilinearization of $\mu$ with respect to $+$},
		sort={m0xy}}

\newglossaryentry{not:muxytil}{
		type=notation,
		name={\ensuremath{\protect\mutil_{x,y}}},
		symbol={\ensuremath{\protect\mutil_{x,y}}},
		description={bilinearization of $\mu$ with respect to $\protect\plustil$},
		sort={m0xytil}}

\newglossaryentry{not:gdivk}{
		type=notation,
		name={\ensuremath{g/k}},
		symbol={\ensuremath{g/k}},
		description={the unique $k$-th of $g$ in a unique $k$-divisible group},
		sort={gdivk}}

\newglossaryentry{not:gdivk2}{
		type=notation,
		name={\ensuremath{g\cdot\tfrac{1}{k}}},
		symbol={\ensuremath{g\cdot\tfrac{1}{k}}},
		description={the unique $k$-th of $g$ in a unique $k$-divisible group},
		sort={gdivk2}}

\newglossaryentry{not:K}{
		type=notation,
		name={\ensuremath{K}},
		symbol={\ensuremath{K}},
		description={a field},
		sort={K}}

\newglossaryentry{not:va}{
		type=notation,
		name={\ensuremath{v(a)}},
		symbol={\ensuremath{v(a)}},
		description={the valuation of $a$},
		sort={va}}

\newglossaryentry{not:pi}{
		type=notation,
		name={\ensuremath{\pi}},
		symbol={\ensuremath{\pi}},
		description={the uniformizer of the field $K$ with discrete valuation $v$},
		sort={p0i}}

\newglossaryentry{not:O}{
		type=notation,
		name={\ensuremath{\rO}},
		symbol={\ensuremath{\rO}},
		description={the valuation ring of the field $K$ with discrete valuation $v$},
		sort={O}}

\newglossaryentry{not:L}{
		type=notation,
		name={\ensuremath{L}},
		symbol={\ensuremath{L}},
		description={an $\rO$-lattice in $K^2$},
		sort={La}}

\newglossaryentry{not:[L]}{
		type=notation,
		name={\ensuremath{[L]}},
		symbol={\ensuremath{[L]}},
		description={the homothety class of the $\rO$-lattice $L$},
		sort={La]}}

\newglossaryentry{not:dLL}{
		type=notation,
		name={\ensuremath{d([L_0],[L])}},
		symbol={\ensuremath{d([L_0],[L])}},
		description={the distance between two lattice classes},
		sort={dLL}}

\newglossaryentry{not:e}{
		type=notation,
		name={\ensuremath{e}},
		symbol={\ensuremath{e}},
		description={a unit in a local Moufang set, a local ring, a Jordan pair...},
		sort={ea}}

\newglossaryentry{not:J}{
		type=notation,
		name={\ensuremath{J}},
		symbol={\ensuremath{J}},
		description={a Jordan algebra (usually with quadratic map $W$)},
		sort={J}}

\newglossaryentry{not:Wx}{
		type=notation,
		name={\ensuremath{W_x}},
		symbol={\ensuremath{W_x}},
		description={the map corresponding to $x$ in a Jordan algebra},
		sort={Wx}}

\newglossaryentry{not:RadJ}{
		type=notation,
		name={\ensuremath{\Rad J}},
		symbol={\ensuremath{\Rad J}},
		description={the radical of a Jordan algebra $J$},
		sort={RadJ}}

\newglossaryentry{not:C}{
		type=notation,
		name={\ensuremath{\rC}},
		symbol={\ensuremath{\rC}},
		description={a category},
		sort={C}}

\newglossaryentry{not:I}{
		type=notation,
		name={\ensuremath{\rI}},
		symbol={\ensuremath{\rI}},
		description={the category corresponding to a partially ordered set $(I,\succcurlyeq)$},
		sort={I}}

\newglossaryentry{not:ObC}{
		type=notation,
		name={\ensuremath{\Ob\rC}},
		symbol={\ensuremath{\Ob\rC}},
		description={the class of objects of a category $\rC$},
		sort={ObC}}

\newglossaryentry{not:HomCD}{
		type=notation,
		name={\ensuremath{\Hom(C,D)}},
		symbol={\ensuremath{\Hom(C,D)}},
		description={the class of morphisms between $C$ and $D$},
		sort={HomCD}}

\newglossaryentry{not:1C}{
		type=notation,
		name={\ensuremath{\id_C}},
		symbol={\ensuremath{\id_C}},
		description={the identity morphism on $C$},
		sort={1C}}

\newglossaryentry{not:Set}{
		type=notation,
		name={\ensuremath{\Set}},
		symbol={\ensuremath{\Set}},
		description={the category of sets with maps},
		sort={Set}}

\newglossaryentry{not:Gr}{
		type=notation,
		name={\ensuremath{\Gr}},
		symbol={\ensuremath{\Gr}},
		description={the category of groups with group homomorphisms},
		sort={Gr}}

\newglossaryentry{not:Ab}{
		type=notation,
		name={\ensuremath{\Ab}},
		symbol={\ensuremath{\Ab}},
		description={the category of abelian groups with group homomorphisms},
		sort={Ab}}

\newglossaryentry{not:Ring}{
		type=notation,
		name={\ensuremath{\Ring}},
		symbol={\ensuremath{\Ring}},
		description={the category of unital commutative rings with ring homomorphisms},
		sort={Ring}}

\newglossaryentry{not:ModR}{
		type=notation,
		name={\ensuremath{\Mod_R}},
		symbol={\ensuremath{\Mod_R}},
		description={the category of right $R$-modules with $R$-linear maps},
		sort={ModR}}

\newglossaryentry{not:LMou}{
		type=notation,
		name={\ensuremath{\LMou}},
		symbol={\ensuremath{\LMou}},
		description={the category of local Moufang sets},
		sort={LMou}}

\newglossaryentry{not:Mou}{
		type=notation,
		name={\ensuremath{\Mou}},
		symbol={\ensuremath{\Mou}},
		description={the category of Moufang sets},
		sort={Mou}}

\newglossaryentry{not:GLn}{
		type=notation,
		name={\ensuremath{\GL_n(R)}},
		symbol={\ensuremath{\GL_n(R)}},
		description={the general linear group of degree $n$ over $R$},
		sort={GLn}}

\newglossaryentry{not:Scn}{
		type=notation,
		name={\ensuremath{\Sc_n(R)}},
		symbol={\ensuremath{\Sc_n(R)}},
		description={the scalar matrices of degree $n$ over $R$},
		sort={Scn}}

\newglossaryentry{not:SLn}{
		type=notation,
		name={\ensuremath{\SL_n(R)}},
		symbol={\ensuremath{\SL_n(R)}},
		description={the special linear group of degree $n$ over $R$},
		sort={SLn}}

\newglossaryentry{not:PSLn}{
		type=notation,
		name={\ensuremath{\PSL_n(R)}},
		symbol={\ensuremath{\PSL_n(R)}},
		description={the projective special linear group of degree $n$ over $R$},
		sort={PSLn}}

\newglossaryentry{not:PGLn}{
		type=notation,
		name={\ensuremath{\PGL_n(R)}},
		symbol={\ensuremath{\PGL_n(R)}},
		description={the projective general linear group of degree $n$ over $R$},
		sort={PGLn}}

\newglossaryentry{not:dT}{
		type=notation,
		name={\ensuremath{\partial T}},
		symbol={\ensuremath{\partial T}},
		description={the boundary of a tree $T$},
		sort={d0T}}

\newglossaryentry{not:T}{
		type=notation,
		name={\ensuremath{T}},
		symbol={\ensuremath{T}},
		description={Serre's tree for $\PSL_2$},
		sort={T}}
		
\title{Local Moufang sets}

\begin{document}
\frontmatter

	\begin{titlepage}
	\centering
	
	\vspace*{\fill}
	
	\begin{tikzpicture}[inner sep=0pt,outer sep=0pt]
		\node (title) {\resizebox{11cm}{!}{\Huge\bfseries\scshape\color{steelblue} Local Moufang sets}};
		\node[below, steelblue,yshift=-0.4cm] (bottomline) at (title.south) {\rule{\textwidth}{1.6mm}};
		\node[above, steelblue,yshift=0.4cm] (bottomline) at (title.north) {\rule{\textwidth}{1.6mm}};
		\node[left, yshift=-1.6cm] (name) at (title.east) {\huge\scshape Erik Rijcken};
	\end{tikzpicture}
	
	\vspace*{\fill}
	
	\begin{tikzpicture}[inner sep=0pt,outer sep=0pt]
		\node (top) {\rule{10.9cm}{0mm}};
		\node[below right] (ugentlogo) at (top.west) {\includegraphics[scale=1]{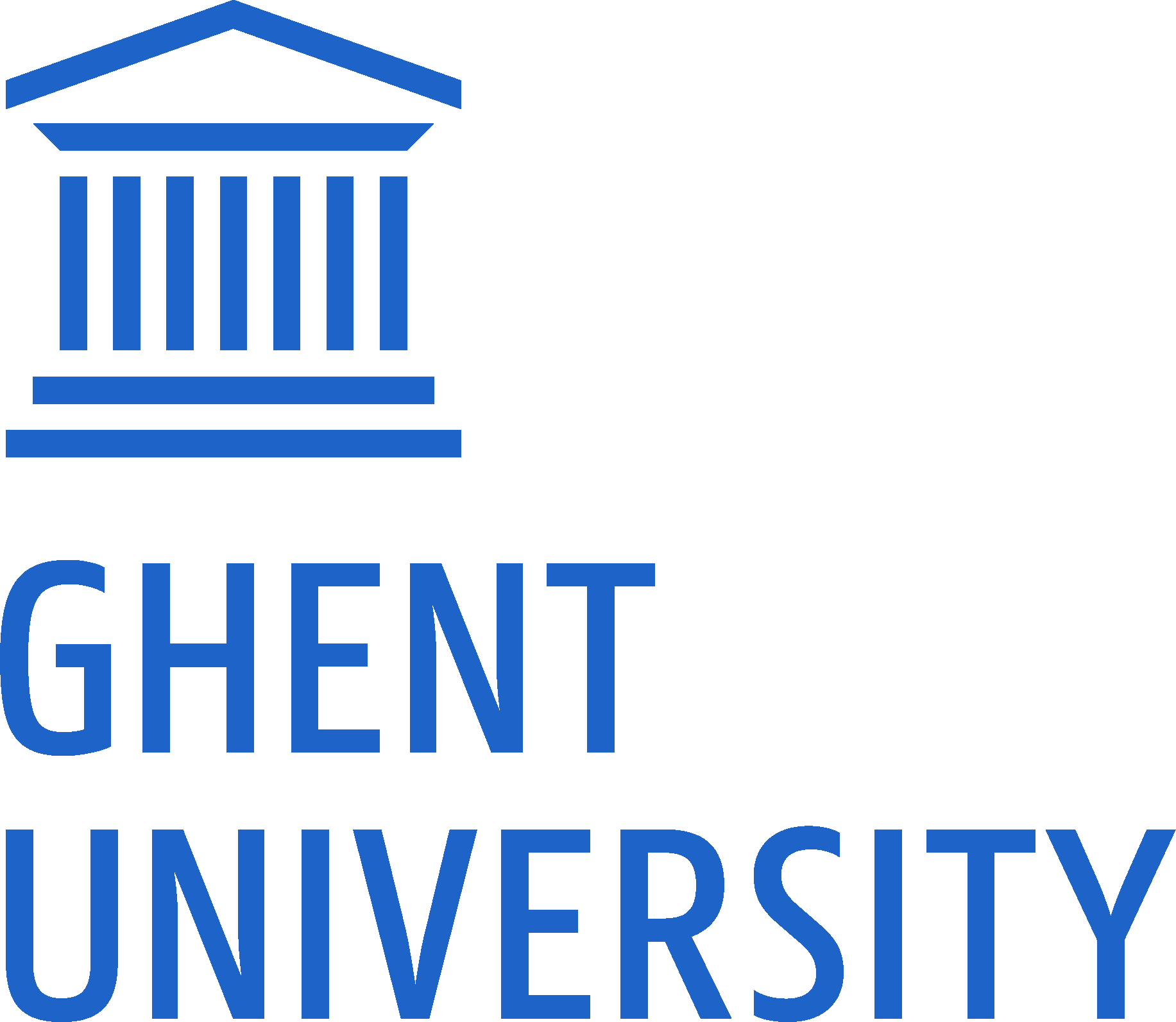}};
		\node[above right] (bottom) at (ugentlogo.south west) {\rule{10.9cm}{0mm}};
		
		\node[above left] (plaats) at (bottom.south east) {\normalsize\setlength\tabcolsep{0pt}
									\begin{tabular}{r}
										Department of Mathematics \\
										Faculty of Sciences \\
										Ghent University \\
										June 2017\vspace*{-0.14cm}
									\end{tabular}};
		\node[above left, yshift=0.6cm] (tom) at (plaats.north east) {\large Supervisor: Prof.\ Dr.\ Tom De Medts};
		\node[below, yshift=-0.8cm] (proefschr) at (bottom.south) {\small\begin{tabular}{c}
										Dissertation submitted in fulfillment of the requirements for \\
										the degree of Doctor in science: mathematics
									\end{tabular}};
	\end{tikzpicture}
\end{titlepage}
	
	\cleardoublepage
	\thispagestyle{empty}
	\vspace*{6em}
	\fquote[Kelsier]{The trick is to never stop looking.\\ There's \emph{always} another secret.}
	
	\cleardoublepage
	\notocchapterspace
	\phantomsection
	\addcontentsline{toc}{chapter}{Table of contents}\markboth{Table of contents}{}%
	\tableofcontents

\mainmatter

	\chapter*{Introduction}\markboth{Introduction}{}%
	\addcontentsline{toc}{chapter}{Introduction}
	\section*{Historical context}\markright{Historical context}{}%
\addcontentsline{toc}{section}{Historical context}%

Groups of matrices, or linear groups, have been studied since the 19th century, with applications in many areas of mathematics and physics.
Originally, these groups were mostly considered over fields, but starting in the middle of the 20th century, people began to study linear groups over arbitrary rings.
While linear groups over general rings are hard to study, there are certainly some interesting classes of rings over which interesting results can be obtained.
In the 1960s, Klingenberg studied some linear groups over a \emph{local ring}, a ring $R$ with a unique maximal ideal $\m$.
In \cite{KlingenbergGL} he investigated the normal subgroups of $\GL_n(R)$.
He also studied orthogonal and symplectic groups in \cite{KlingenbergO,KlingenbergS}.

In the 1990s, J.~Tits introduced the notion of \emph{Moufang sets} in \cite{TitsTwinBldgs} as an axiomatic approach to simple algebraic groups of relative rank one.
A Moufang set is a set, along with a class of groups acting on this set, one for each point of the set, satisfying some axioms.
For a simple algebraic group of relative rank one, the set consists of the parabolic subgroups, the class of groups are the corresponding root subgroups.
There are quite a few equivalent ways of looking at Moufang sets: they correspond to \emph{split BN-pairs of rank one} (another notion introduced by Tits), are closely related to \emph{abstract rank one groups} (introduced by Timmesfeld in \cite{TimmesfeldAR1G}), and are equivalent to \emph{division pairs} (more recently defined by Loos in \cite{LoosRogdiv}).

T.~De Medts and R.~Weiss initiated the study of arbitrary Moufang sets in 2006.
Since then, the theory of Moufang sets has been developed more deeply, and many examples of Moufang sets have been described.
All Moufang sets we currently know have some algebraic origin, but a classification of Moufang sets is still an open question.

The known Moufang sets are not only of algebraic origin, but they all have some underlying `\emph{division} structure', where all nonzero elements are invertible.
For example: projective Moufang sets can be defined over \emph{alternative division algebras}, \emph{Jordan division algebras} gives rise to Moufang sets, and more recently L.~Boelaert proved that every \emph{structurable division algebra} gives rise to a Moufang set (see \cite{BoelaertL}).
In fact, every Moufang set arising from a simple linear algebraic group of relative rank one (over a field of characteristic different from $2$ and $3$), arises from such a structurable division algebra (this was shown in \cite{BoelaertDMStavrova}).
We can see this division requirement in the constructions of known Moufang sets: inverses pop up everywhere!
There are also some Moufang sets which do not arise from algebraic groups directly, but these are still defined over fields (also known as \emph{division rings}).

A related consequence of the definition of Moufang sets is the fact that morphisms of Moufang sets are automatically injective. This means there are relatively few morphisms, and there is no meaningful notion of quotients in the theory of Moufang sets.

One can wonder what happens if we try to define Moufang sets using more general algebraic structures, and this is precisely where this dissertation enters the story. Instead of division structures, I looked at \emph{local} structures. This means that there can be non-invertible elements in the structure, but they can still be controlled easily. In trying to define the known constructions over \emph{local rings} and \emph{local Jordan algebras}, I found a set of axioms that generalize Moufang sets, to \emph{local Moufang sets}. (From this point of view, we could start using the name \emph{division Moufang sets} for Moufang sets, but to avoid confusion we will refrain from doing so. Loos does include the adjective division in his equivalent notion of \emph{division pairs}.)

\section*{Outline}\markright{Outline}{}%
\addcontentsline{toc}{section}{Outline}%

There are two parts in my dissertation.
In \autoref{part:theory}, we define the notion of local Moufang sets, and develop the general theory.
\autoref{part:examples} then contains some examples of local Moufang sets, characterizes some of them, and explores some connections.

The first part starts with the definition of local Moufang sets, along with some basic properties, in \autoref{chap:chap2_definitions}.
The difference with Moufang sets is the assumption of extra structure on the set, given by an equivalence relation on the set.
We still have a conjugacy class of groups, called \emph{root groups}, acting on the set, but the axioms need to be adapted in order to be compatible with the extra structure.
As in the theory of Moufang sets, it is useful to fix two (non-equivalent) points of our set (this could be called a \emph{basis}).
An important new notion for local Moufang sets, is the notion of \emph{units}, which have a connection to the invertible elements of the underlying algebraic structures in the examples.
Using these units, we can extend much of the theory from Moufang sets to local Moufang sets. One example is the existence of $\mu$-maps, which swap the points of the basis, and many of the identities of these $\mu$-maps.

A first goal in the development of the theory of local Moufang sets is determining the stabilizer of the basis.
A candidate for this stabilizer is the \emph{Hua subgroup}, the group generated by products of an even number of $\mu$-maps.
The Hua subgroup turns out to coincide with the stabilizer of the basis, which we show in \autoref{thm:hua2pt}.
In order to prove this, we introduced \emph{quasi-invertibility}, a notion originating from Jordan theory.

In \autoref{chap:chap3_constr}, we give a general way to construct local Moufang sets similar to a construction for Moufang sets. This construction uses a lot less information than is required in the definition: we only need one root group and one permutation of the set swapping the basis.
Using only these two things, we can construct the other root groups and complete the data of the local Moufang set.
However, this data will not always satisfy the axioms.
We determine a few necessary and sufficient conditions in \autoref{cor:construction_equivalentconditions}.
The construction and conditions will be used frequently to define local Moufang sets, and to determine when some structure is a local Moufang set.

The notion of homomorphisms of local Moufang sets is introduced in \autoref{chap:chap4_category}.
This is combined with the definition of local Moufang subsets, quotients, and eventually, the category of local Moufang sets.
It is interesting to observe how tightly a homomorphism of local Moufang sets is connected to a map on the underlying set and group homomorphisms between the root groups.
We mainly use these notions to define inverse limits of local Moufang sets, and to determine when the inverse limit exists.

In the final theoretical chapter, \autoref{chap:chap5_special}, we look at local Moufang sets which are \emph{special}, i.e.\ which satisfy a specific extra condition.
This extra condition is expected to be closely related to the root groups being abelian.
Making some necessary changes, we are able to generalize the theory of special Moufang sets to local Moufang sets.
The main results we can get here assume that a local Moufang set is both special and has abelian root groups.
In this case, we can show that $\mu$-maps are involutions (\autoref{prop:mu-involution}), and give a condition which ensures that the root groups are uniquely $k$-divisible (\autoref{prop:ndivglobal}).

Next, we get to the examples of local Moufang sets in \autoref{part:examples}. The first examples are the \emph{projective local Moufang sets} in \autoref{chap:chap6_projective}.
These are defined over a local ring $R$, where the underlying set is the projective line $\P^1(R)$, and the root groups are subgroups of $\PSL_2(R)$.
We succeeded in characterizing the projective local Moufang sets using some conditions (\autoref{thm:PSL2equiv}).
This includes the construction of a local ring from local Moufang sets satisfying some of these conditions.
The third section of this chapter makes a connection to the \emph{Bruhat-Tits tree} of $\PSL_2$ over a field $K$ with \emph{discrete valuation}. Such a field contains a local ring, and the action of $\PSL_2(K)$ on the tree induces many projective local Moufang sets. Using the inverse limit construction, we even get a projective local Moufang set action on the boundary.

In \autoref{chap:chap7_jordan}, we explore the connection of local Moufang sets with Jordan theory.
The most natural approach uses \emph{Jordan pairs}, and we are able to define a local Moufang set from any local Jordan pair.
Conversely, if we have a local Moufang set satisfying some conditions, we are able to construct a local Jordan pair.
Using this construction, we characterize the local Moufang sets originating from local Jordan pairs in \autoref{thm:extra}.

Both the projective local Moufang sets and those originating from Jordan pairs are special, and have abelian root groups.
In \autoref{chap:chap8_hermitian}, our main goal is to define Hermitian local Moufang sets.
We first define orthogonal local Moufang sets, which are a special case of Hermitian local Moufang sets (and can also be constructed using Jordan pairs).
This example shows the approach we will use to generalize the Hermitian Moufang sets to the local setting.
The definition of Hermitian local Moufang sets is a bit more technical, and we need a great amount of algebraic manipulations to prove that the axioms of local Moufang sets are satisfied.

A large part of this dissertation is contained in two articles \cite{DMRijckenPSL,DMRijckenLMS_JP}.
\autoref{chap:chap4_category}, \autoref{sec:Serre} and \autoref{chap:chap8_hermitian} are the main parts of this dissertation that are not contained in those articles.
	\tocchapterspace
	\chapter{Preliminaries}
	In this chapter, we give a short overview of some mathematics we will use in this dissertation. We start with some basic category theory, which will mainly be used in \autoref{chap:chap4_category}. Next, we choose our notation for everything involving group actions, we give some basic theory on local rings, and define some linear groups over rings. The third section gives a short overview of the theory of Jordan algebras and Jordan pairs, which we will use in \autoref{chap:chap7_jordan}. Finally, we discuss the basic theory and examples of Moufang sets, which we will later try to generalize to a local setting.

\section{Category theory}	
\subsection{Categories and functors}

Much of mathematics is a study of objects, being sets with some structure (like group operations, a topology...), and their connections. Usually, when we consider objects with some structure, we also consider maps between those objects that preserve the structure. Category theory tries to study properties that hold in great generality for such classes of objects and maps. A common reference work for category theory is \cite{McLaneCategories}, though we will usually refer to \cite{BorceuxCatI}. Another good initial work on category theory is \cite{LeinsterCat}.

\begin{definition}
	A \define{category} $\rC$\notatlink{C} consists of the following:
	\begin{nrenumerate}
		\item a class of \define[category!object]{objects} $\Ob\rC$\notatlink{ObC};
		\item for each two objects $C,D\in\Ob\rC$ a class of \define[category!morphism]{morphisms} denoted by $\Hom(C,D)$\notatlink{HomCD};
		\item for each $C\in\Ob\rC$ a morphism $\id_C\in\Hom(C,C)$ we call the \define[category!identity]{identity} on $C$\notatlink{1C};
		\item for all three objects $C,D,E\in\Ob\rC$, a composition map 
		\[\Hom(C,D)\times\Hom(D,E)\to\Hom(C,D)\;.\]
		We write the composition of $(f,g)$ as $fg$ or $g\circ f$.
	\end{nrenumerate}
	This data should satisfy the following axioms:
	\begin{romenumerate}
		\item For $f\in\Hom(C,D)$, $g\in\Hom(D,E)$ and $h\in\Hom(E,F)$, we have 
		\[h\circ(g\circ f) = (h\circ g)\circ f\;.\]
		\item For all $f\in\Hom(C,D)$ and $g\in\Hom(D,E)$, we have 
		\[\id_D\circ f = f\text{ and }g\circ\id_D = g\;.\]
	\end{romenumerate}
	We call a category \define[category!small category]{small} if $\Ob\rC$ is a set.
\end{definition}

Remark that $\Ob\rC$ need not be a set (so we can have the category of sets, for example), but it is often required that $\Hom(C,D)$ is. Such categories are also called \define[category!locally finite category]{locally finite} categories.

\begin{notation}
	For a morphism $f\in\Hom(C,D)$, we also write $f\colon C\to D$ or $C\xlongrightarrow{f}D$.
\end{notation}

\begin{definition}
	In a category $\rC$, a \define{diagram} is a directed graph where the vertices correspond to objects and the arrows correspond to morphisms between the corresponding objects. We say a diagram \define[diagram!commutative diagram]{commutes} if for any two paths between objects in the diagram, the corresponding composition of morphisms coincides.
\end{definition}

\begin{example}
	Let $C,D,E,F\in\Ob\rC$, then
	\begin{center}\begin{tikzpicture}[regulararrow,node distance=2cm]
		\node (C) {$C$};
		\node[right of=C] (D) {$D$};
		\node[below of=C,yshift=0.5cm] (E) {$E$};
		\node[right of=E] (F) {$F$};
		\draw[->]	(C) edge node[above] {\scriptsize$f$} (D)
				(C) edge node[left] {\scriptsize$g$} (E)
				(E) edge node[above] {\scriptsize$h$} (D)
				(D) edge node[right] {\scriptsize$i$} (F)
				(E) edge node[below] {\scriptsize$j$} (F);
	\end{tikzpicture}\end{center}
	is a diagram, which commutes when $i\circ f=j\circ g$, $f=h\circ g$ and $i\circ h=j$ (and hence also $i\circ f=j\circ g=i\circ h\circ g$).
\end{example}

\begin{definition}
	A \define[functor]{(covariant) functor} $F$ from a category $\rC$ to a category $\rD$ consists of the following:
	\begin{nrenumerate}
		\item a mapping from $\Ob\rC$ to $\Ob\rD$ that associates an object $F(C)\in\Ob\rD$ for each object $C\in\Ob\rC$;
		\item a mapping from $\Hom(C,D)$ to $\Hom(F(C),F(D))$ for all $C,D\in\Ob\rC$.
	\end{nrenumerate}
	This data should satisfy the following axioms:
	\begin{romenumerate}
		\item For $f\in\Hom(C,D)$, $g\in\Hom(D,E)$, we have 
		\[F(g\circ f) = F(g)\circ F(f)\;.\]
		\item For all $C\in\Ob\rC$, $F(\id_C) = \id_{F(C)}$.
	\end{romenumerate}
\end{definition}

Many types of structures give rise to a category. A few examples:

\begin{examples}\preenum
	\begin{nrenumerate}
		\item Sets with maps between sets is a category denoted by $\Set$\notatlink{Set}.
		\item Groups with group homomorphisms is a category denoted by $\Gr$\notatlink{Gr}. 
		\item Commutative unital rings with ring homomorphisms is a category denoted by $\Ring$\notatlink{Ring}.
		\item For a ring $R$ we write $\Mod_R$\notatlink{ModR} for the category of right $R$-modules with $R$-linear mappings. 
		\item Let $(I,\succcurlyeq)$ be a partially ordered set, we can define the category $\rI$ by
		\[\Ob\rI = I\text{ and }\Hom(i,j) = 
		\begin{cases}
			\{f_{ij}\} &\text{ if $i\succcurlyeq j$}\\
			\emptyset &\text{ otherwise}
		\end{cases}\;.\]
	\end{nrenumerate}
\end{examples}

Some functors can be used to embed a category in another in some way. Properties of these functors can then give more information on this embedding.

\begin{definition}
	A functor $F\colon\rC\to\rD$ is \define[functor!faithful functor]{faithful} if the induced maps 
		\[F\colon\Hom(C,D)\to\Hom(F(C),F(D))\]
	are all injective. We call $F$ \define[functor!full functor]{full} if the induced maps 
		\[F\colon\Hom(C,D)\to\Hom(F(C),F(D))\]
	are all surjective.
\end{definition}

\begin{example}
	As every group has an underlying set of group elements, we can define the \define{forgetful functor} $F\colon\Gr\to\Set$ that sends a group to its underlying set and a morphism to the map between the group elements. While the forgetful functor is far from injective on the objects, it is faithful.
\end{example}

We could look at all the functors between two fixed categories $\rC$ and $\rD$. It is possible to construct maps between these functors as well:

\begin{definition}
	Let $F,G\colon\rC\to\rD$ be two functors. A \define{natural transformation} $\alpha$ is a family of maps $\alpha_C\colon F(C)\to G(C)$ in $\rD$ such that for every morphism $f\colon C\to C'$ in $\rC$, the following diagram commutes in $\rD$:
	\begin{center}\begin{tikzpicture}[regulararrow,node distance=2cm]
		\node (FC) {$F(C)$};
		\node[right of=C] (FCb) {$F(C')$};
		\node[below of=C,yshift=0.5cm] (GC) {$G(C)$};
		\node[right of=E] (GCb) {$G(C')$};
		\draw[->]	(FC) edge node[above] {\scriptsize$F(f)$} (FCb)
				(FC) edge node[left] {\scriptsize$\alpha_C$} (GC)
				(FCb) edge node[right] {\scriptsize$\alpha_{C'}$} (GCb)
				(GC) edge node[below] {\scriptsize$G(f)$} (GCb);
	\end{tikzpicture}\end{center}
	A natural transformation is an \define{natural isomorphism} if every $\alpha_C$ is an isomorphism.
\end{definition}


\subsection{Mono-, epi- and isomorphisms}

We would like to find notions that are similar to injective and surjective morphisms to the context of general categories. As the objects are not necessarily sets, we need to base this purely on properties of the morphisms.

\begin{definition}
	Let $f\colon C\to D$ be a morphism in a category $\rC$. We call $f$ a \define{monomorphism} if for all $g,h\in\Hom(E,C)$, 
	\[f\circ g = f\circ h\implies g=h\;.\]
	We call $f$ an \define{epimorphism} if for all $g,h\in\Hom(D,E)$, 
	\[g\circ f = h\circ f\implies g=h\;.\]
	We call $f$ an \define{isomorphism} if there exists a $g\in\Hom(D,C)$ such that
	\[g\circ f = \id_C\text{ and }f\circ g=\id_D\;.\]
\end{definition}

\begin{notation}
	We write $C\xrightarrowtail{f} D$ to indicate $f$ is a monomorphism and $C\xtwoheadrightarrow{f} D$ to indicate $f$ is an epimorphism. If there is an isomorphism between $C$ and $D$, we write $C\cong D$.
\end{notation}

While monomorphisms and epimorphisms have many similar properties to injective and surjective mappings, they are not always the same when it makes sense to talk about injective and surjective mappings. It is useful to see some basic properties of these notions:

\begin{proposition}
	In a category $\rC$ the following hold.
	\begin{romenumerate}
		\item The composition of monomorphisms (epimorphisms, isomorphisms) is a monomorphism (epimorphism, isomorphism).
		\item If $f\circ g$ is a monomorphism, then so is $g$.
		\item If $f\circ g$ is an epimorphism, then so is $f$.
		\item An isomorphism is both a monomorphism and a epimorphism.
	\end{romenumerate}
\end{proposition}
\begin{proof}
	These are in Propositions~1.7.2,~1.8.2 and~1.9.2 of \citenobackref{BorceuxCatI}.
\end{proof}

\begin{examples}\preenum
	\begin{nrenumerate}
		\item In $\Set$, monomorphisms, epimorphisms and isomorphisms are precisely injections, surjections and bijections.
		\item In $\Gr$, monomorphisms, epimorphisms and isomorphisms are precisely injective, surjective and bijective group homomorphisms.
		\item In $\Ring$, the inclusion map $\Z\to\Q$ is an epimorphism, but not a surjection.
	\end{nrenumerate}
\end{examples}

%
%
%

\subsection{Limits}

Limits are objects that generalize many different concepts, like products, kernels or pullbacks. We will mostly be working with the specific case of inverse limits.

\begin{definition}
	Let $F\colon\rD\to\rC$ be a functor. A \define{cone} on $F$ consists of
	\begin{nrenumerate}
		\item an object $C\in\Ob\rC$;
		\item for each $D\in\Ob\rD$, a morphism $p_D\colon C\to F(D)$
	\end{nrenumerate}
	such that for each morphism $f\colon D\to D'$ in $\rD$, we have 
	\[p_{D'} = F(f)\circ p_D\;.\]
\end{definition}

This final condition can also be written as a commutative diagram: for all $f\colon D\to D'$ in $\rD$, the following diagram must commute:
\begin{center}\begin{tikzpicture}[regulararrow,node distance=2.5cm]
		\node (C) {$C$};
		\node[below of=C,yshift=1cm] (D) {$F(D)$};
		\node[right of=D] (Db) {$F(D')$};
		\draw[->]	(C) edge node[left] {\scriptsize$p_D$} (D)
				(C) edge node[above] {\scriptsize$p_{D'}$} (Db)
				(D) edge node[above] {\scriptsize$F(f)$} (Db);
	\end{tikzpicture}\end{center}

\begin{definition}
	Let $F\colon\rD\to\rC$ be a functor. A \define{limit} of $F$ is a cone $(L,(p_D)_D)$ on $F$ that satisfies the \define{universal property}:
	for every cone $(C,(q_D)_D)$ on $F$ there is a unique morphism $f\colon C\to L$ such that for all $D\in\rD$, $q_D = p_D\circ f$.
\end{definition}

While cones are not unique at all, the universal property implies that if a limit exists, it is automatically unique.

\begin{proposition}\label{prop:limitproperties}
	Let $F\colon\rD\to\rC$ be a functor.
	\begin{romenumerate}
		\item If a limit of $F$ exists, then it is unique up to isomorphism.
		\item Assume $(L,(p_D)_D)$ is a limit of $F$. If $G\colon\rD\to\rC$ is a functor and $\alpha\colon F\to G$ is a natural isomorphism, then $(L,(p_D\alpha_D)_D)$ is a limit of $G$.
	\end{romenumerate}
\end{proposition}
\begin{proof}\preenum
	\begin{romenumerate}
		\item This is Proposition~2.6.3 of \citenobackref{BorceuxCatI}.
		\item The definition of natural transformations and cone show that $(L,(p_D\alpha_D)_D)$ is a cone on $G$. If $(C,(q_D)_D)$ is a cone on $G$, then $(C,(q_D\alpha_D^{-1})_D)$ is a cone on $F$, so there is a unique map $f\colon C\to L$ such that $\alpha_D^{-1}\circ q_D = p_D\circ f$, which means $f$ is also the unique map such that $q_D = (\alpha_D\circ p_D)\circ f$, hence $(L,(p_D\alpha_D)_D)$ satisfies the universal property.
		\qedhere
	\end{romenumerate}
\end{proof}


The category $\Set$ is special in the sense that all functors from small categories have a limit (see \autoref{thm:set_inverse_limit}).

\begin{definition}
	A category $\rC$ is \define[category!complete category]{complete} if $F$ has a limit for all functors $F\colon\rD\to\rC$ with $\rD$ a small category.
\end{definition}


\begin{example}
	Some complete categories are $\Set$, $\Gr$ and $\Ring$.
\end{example}

\subsection{Inverse limits}

We now define inverse limits, and see what inverse limits are in the categories $\Set$, $\Gr$ and $\Ring$.

\begin{definition}
	A \define{directed set} is a (nonempty) partially ordered set $(I,\succcurlyeq)$\notatlink{Igeq} such that for every $i,j\in I$, there is a $k\in I$ for which $k\succcurlyeq i$ and $k\succcurlyeq j$, i.e.\ every two elements of $I$ have a common upper bound.
	The \define[category!associated to a directed set]{category associated to $(I,\succcurlyeq)$} is the category $\rI$\notatlink{I} defined by 
		\[\Ob\rI = I\text{ and }\Hom(i,j) = 
		\begin{cases}
			\{f_{ij}\} &\text{ if $i\succcurlyeq j$}\\
			\emptyset &\text{ otherwise}
		\end{cases}\;.\]
\end{definition}

With all categorical definitions, we can now easily define inverse limits:

\begin{definition}
	An \define{inverse limit} is a limit of a functor $F\colon\rI\to\rC$ with $\rI$ the category associated to a directed set.
\end{definition}

As this may be a bit abstract, we look at the definitions in a different way.

\begin{definition}
	Let $(I,\succcurlyeq)$ be a directed set. An \define{inverse system} in $\rC$ over $I$ is a collection of objects $(C_i)_{i\in I}$ and for each $i\succcurlyeq j$ a map $\phi_{ij}\colon C_i\to C_j$ such that
	\begin{romenumerate}
		\item $\phi_{ii} = \id_{C_i}$ for all $i\in I$;
		\item for all $i\succcurlyeq j\succcurlyeq k$, $\phi_{jk}\circ\phi_{ij} = \phi_{ik}$.
	\end{romenumerate}
\end{definition}

\begin{definition}
	Two inverse systems $(C_i,\phi_{ij})$ and $(C'_i,\phi'_{ij})$ are \define[inverse system!isomorphic inverse systems]{isomorphic} if there are isomorphisms $\alpha_i\colon C_i\to C'_i$ such that 
	\[\phi'_{ij}\circ\alpha_i = \alpha_j\circ\phi_{ij}\]
	for all $i\succcurlyeq j$.
\end{definition}

An inverse system corresponds to a functor $F\colon\rI\to\rC$ by setting $F(i) = C_i$ and $F(f_{ij}) = \phi_{ij}$. Isomorphic inverse systems correspond to naturally isomorphic functors.

\begin{definition}
	Let $(C_i,\phi_{ij})$ be an inverse system in $\rC$. An \define{inverse limit} consists of an object $L\in\Ob\rC$ and for each $i\in I$ a \define{projection map} $p_i\colon L\to C_i$ such that
	\begin{manualenumerate}[label=\textnormal{(IL\arabic*)},labelwidth=\widthof{(IL1)}]
		\item $\phi_{ij}\circ p_i = p_j$ for all $i\succcurlyeq j$;\label{axiom:IL1}
		\item $L$ is the universal object with this property, i.e.\ if $(C,(q_i))$ is another such object, there is unique morphism $f\colon C\to L$ such that $q_i = p_i\circ f$ for all $i\in I$.\label{axiom:IL2}
	\end{manualenumerate}
	When an inverse limit exists, we denote it by $\varprojlim C_i$.
\end{definition}

\begin{proposition}
	Let $(C_i,\phi_{ij})$ be an inverse system with inverse limit $\varprojlim C_i$.
	\begin{romenumerate}
		\item If $C$ is another inverse limit of $(C_i,\phi_{ij})$, then $C\cong\varprojlim C_i$.
		\item If $(C'_i,\phi'_{ij})$ is an inverse system isomorphic to $(C_i,\phi_{ij})$, then $\varprojlim C_i$ is an inverse limit to $(C'_i,\phi'_{ij})$.
	\end{romenumerate}
\end{proposition}
\begin{proof}
	These properties are the translations of \autoref{prop:limitproperties}.
\end{proof}

\subsection{Inverse limits of sets and algebraic structures}

We have already stated that $\Set$ is a complete category, meaning that limits for functors from small categories always exist. In particular, inverse limits always exist.

\begin{theorem}\label{thm:set_inverse_limit}
	Let $(X_i,\phi_{ij})$\notatlink{Xiphiij} be an inverse system in $\Set$ over $I$. Then
	\[\varprojlim X_i\notatlink{limXi} = \left\{(x_i)_{i\in I}\in\prod_{i\in I}X_i\;\middle|\; x_i\phi_{ij} = x_j\text{ for all $i\succcurlyeq j$}\right\}\;,\]
	and the projection maps are $p_j\colon \varprojlim X_i\to X_j\colon (x_i)_i\mapsto x_j$.
\end{theorem}
\begin{proof}
	A construction of the limit in $\Set$ can be found in \cite[p.~110]{McLaneCategories}. Making this more explicit gives the desired expression.
\end{proof}

Despite there always being an inverse limit, some problems can still occur. The biggest being that it is not necessarily the case that $\varprojlim X_i$ is not empty. G.~Bergman has some notes which give a good overview on this problem \cite{EmptyLimits}. One possible reason for having empty inverse limits is when the maps of the inverse system are not surjective. As can be seen from the expression in \autoref{thm:set_inverse_limit}, only the elements of $X_j$ which are in the image of all $\phi_{ij}$ matter. Hence it makes sense to assume surjectivity.

\begin{definition}\label{def:surj_inv_system}
	An inverse system $(X_i,\phi_{ij})$ in $\Set$ is \define[inverse system!surjective inverse system]{surjective} when $\phi_{ij}$ is surjective for all $i\succcurlyeq j$.
\end{definition}

Another possible obstruction to ensure nonempty limits is the size of the index set $I$. To be more precise: the \emph{cofinality} of $(I,\succcurlyeq)$:

\begin{definition}
	We say a partially ordered set $(I,\succcurlyeq)$ has a \define{cofinal sequence} (or also: has countable \define{cofinality}) if there is a countable sequence $i_1\preccurlyeq i_2\preccurlyeq i_3\preccurlyeq \cdots$ in $I$ such that for all $i\in I$ there is an $\ell\in\N$ such that $i_\ell\succcurlyeq i$.
\end{definition}

These two properties are sufficient to ensure that the inverse limit is nonempty.

\begin{theorem}\label{thm:inverse_limit_surjections}
	Let $(I,\succcurlyeq)$ be a directed set. Then $(I,\succcurlyeq)$ has a cofinal sequence if and only if for every surjective inverse system $(X_i,\phi_{ij})$ over $I$ with all $X_i$ nonempty, $\varprojlim X_i$ is not empty. In this case, the projection maps $p_j\colon \varprojlim X_i\to X_j$ are surjective.
\end{theorem}
\begin{proof}
	The equivalence is the main theorem of \cite{HenkinInverseLimits}. For the final remark, take any $x_j\in X_j$ and use the surjectivity of all $\phi_{ij}$ and induction to construct a sequence $(x_i)_{i\in I}$ such that $(x_i)p_j = x_j$.
\end{proof}

In the categories of $\Gr$ and $\Ring$, inverse limits also always exist. The underlying set of an inverse limit is the inverse limit of the underlying sets, while the group operations are done pointwise. An explicit description of all these constructions can be found in \S~7.1 of Chapter~III in~\cite{BourbakiTopologyI}. In the case of groups and rings we get the following:

\begin{theorem}\label{thm:group_inverse_limit}
	Let $(G_i,\phi_{ij})$ be an inverse system in $\Gr$. Then
	\[\varprojlim G_i = \left\{(g_i)_{i\in I}\in\prod_{i\in I}G_i\;\middle|\; g_i\phi_{ij} = g_j\text{ for all $i\succcurlyeq j$}\right\}\;,\]
	with $(g_i)\cdot(h_i) = (g_ih_i)$, $1 = (1_i)$ and projection maps 
	\[p_j\colon\varprojlim G_i\to G_j\colon(g_i)\mapsto g_j\;.\]
\end{theorem}

\begin{theorem}\label{thm:ring_inverse_limit}
	Let $(R_i,\phi_{ij})$ be an inverse system in $\Ring$. Then
	\[\varprojlim R_i = \left\{(r_i)_{i\in I}\in\prod_{i\in I}G_i\;\middle|\; r_i\phi_{ij} = r_j\text{ for all $i\succcurlyeq j$}\right\}\;,\]
	with $(r_i)+(s_i) = (r_i+s_i)$, $(r_i)\cdot(s_i) = (r_is_i)$, $0 = (0_i)$, $1 = (1_i)$ and projection maps $p_j\colon\varprojlim R_i\to R_j\colon(r_i)\mapsto r_j$.
\end{theorem}


\section{Groups and rings}	
\subsection{Group actions}

A great deal of this thesis is about group actions. Some conventions and notations:

\begin{notation}\preenum
	\begin{itemize}
		\item Group actions will always be \emph{right actions}.
		\item We denote the action of $g$ on $x$ by $x\cdot g$ or $xg$.
		\item We write $G_x:=\{g\in G\mid xg=x\}\notatlink{Gx}$ and $G_{x,y}:= (G_x)_y\notatlink{Gxy}$ for a point stabilizer and a two-point stabilizer.
		\item We write $x^G$ for the orbit of $x$.
		\item The group $\Sym(X)$ is the group of all bijections from $X$ to itself.
		\item An action of $G$ on $X$ is not necessarily faithful, we write 
		\[\Im(G\to\Sym(X))\]
			for the induced faithful action of $G$ on $X$.
		\item We define group conjugation as $g^h:= h^{-1}gh$.
	\end{itemize}
\end{notation}

What we will consider, are groups acting on sets with equivalence relations.

\begin{notation}\preenum
	\begin{itemize}
		\item If $(X,\sim)\notatlink{Xtil}$ is a set with equivalence relation, we denote the equivalence class of $x\in X$ by $\class{x}\notatlink{xbar}$, and the set of equivalence classes by $\class{X}\notatlink{Xbar}$.
		\item The group $\Sym(X,\sim)\notatlink{SymXtil}$ is the group of equivalence-preserving bijections of $X$.
		\item If $g\in\Sym(X,\sim)$, we denote the induced faithful action of $g$ on $\class{X}$ by $\induced{g}\notatlink{gbar}$. Similarly, if $G\leq\Sym(X,\sim)$, the induced faithful action of $G$ on $\class{X}$ is denoted by $\induced{G}$.
	\end{itemize}
\end{notation}

\subsection{Local rings}	

Local rings can be defined in various ways. %

\begin{definition}
	A (unital) ring $R\notatlink{R}$ is a \define{local ring} if the set $\m\notatlink{m}$ of non-invertible elements of $R$ is a (two-sided) ideal. We write $R^\times$\notatlink{Rtimes} for the invertible elements of $R$, and $Z(R)$\notatlink{Z(R)} for the center of $R$.
\end{definition}

If the maximal ideal is not clear from context, we also write `$(R,\m)$ is a local ring', which implies $\m$ is the maximal ideal in $R$. Older works often also require $R$ to be Noetherian, but we will not require it in the definition. A few equivalent definitions:

\begin{proposition}
	Let $R$ be a unital ring and write $\Rad R$ for the intersection of all maximal left ideals. Then the following are equivalent:
	\begin{romenumerate}
		\item $R$ is local;
		\item $R$ has a unique maximal (left/right) ideal;
		\item $R/\Rad R$ is a division ring.
	\end{romenumerate}
\end{proposition}
\begin{proof}
	This is part of Theorem~19.1 in \cite{LamRings}.
\end{proof}

\begin{definition}
	A \define{local ring homomorphism} between local rings $(R,\m)$ and $(S,\n)$ is a ring homomorphism $f\colon R\to S$ such that 
	\[f(\m)\subset \n\;.\]
\end{definition}

Before we delve deeper into the theory of local rings, some examples:

\begin{examples}\preenum
	\begin{nrenumerate}
		\item A field is a local ring with maximal ideal $(0)$.
		\item If $(R,\m)$ is a local ring and $I$ is an ideal in $R$, then $R/I$ is a local ring with maximal ideal $\m/I$.
		\item If $(R,\m)$ is a local ring, then the ring of formal power series $R[[x_1,\ldots,x_n]]$ is a local ring with maximal ideal $(\m,x_1,\ldots,x_n)$ (see \citenobackrefoptional{p.~70}{CohanCompleteLocalRings}).
		\item The ring $\Z_{(p)}$ of rational numbers without multiples of $p$ as denominators is a local ring with $p\Z_{(p)}$ as maximal ideal. The inclusion $\Z_{(p)}\to\Q$ is a ring homomorphism, but not a local ring homomorphism.
		\item The ring $\Z_p$ of $p$-adic integers is local with maximal ideal $p\Z_p$.
	\end{nrenumerate}
\end{examples}

We now restrict commutative local rings. A local ring naturally becomes a topological ring:

\begin{definition}
	Let $(R,\m)$ be a local ring. The \define[m-adic topology@$\m$-adic topology]{$\m$-adic topology} is the topology induced by taking $\{\m^i\mid i\in\N\}$ to be a neighborhood base of $0$. We call this topology the \define{natural topology} on $R$.
\end{definition}

If $R$ is a Noetherian local ring, we get $\cap_i\m^i=(0)$. Hence the natural topology on $R$ is Hausdorff, and in particular Cauchy sequences can have at most one limit.

\begin{definition}
	A sequence $\{r_n\}_n$ in $R$ is a \define{Cauchy sequence} if for all $N\in\N$, $r_n-r_m\in\m^N$ for all $n,m$ larger than some constant only depending on $N$. We call a local ring $R$ \define[local ring!complete local ring]{complete} if every Cauchy sequence has a limit.
\end{definition}

Every Noetherian local ring can be extended to a complete local ring. Even if $R$ is not Noetherian, one can still take a completion, but $R$ will no longer be a subring.

\begin{definition}
	Let $(R,\m)$ be a local ring, then we define the \define{completion} of $R$ as
	\[\complete{R}\notatlink{Rhat} := \varprojlim R/\m^i\;.\]
\end{definition}

A few basic properties of completions of local rings:

\begin{proposition}
	Let $(R,\m)$ be a local ring. Then the following hold.
	\begin{romenumerate}
		\item the map 
			\[\iota\colon R\to\complete{R}\colon r\to (r+\m^i)_i\]
			is a local ring homomorphism and $\Ker\iota = \cap_i\m^i$;
		\item $\doublehat{R} \cong \complete{R}$;
		\item if $R$ is Noetherian, $\complete{R}$ is a complete ring with maximal ideal $\complete{R}\bigotimes_R\m$.
	\end{romenumerate}
\end{proposition}
\begin{proof}
	These properties can all be found in \cite{AtiyahMacDonCommAlg}. The first two are in the first section of Chapter~10. The third property combines Proposition~10.15 and Proposition~10.16.
\end{proof}

In \cite{CohanCompleteLocalRings}, I.~Cohen determined the structure of the complete Noetherian local rings. There are two cases, depending on the characteristics of the ring and the residue field.

\begin{definition}
	The \define{characteristic} of a ring $R$ is 
	\[\min\{n>0\mid 1\cdot n = 0\}\]
	if such $n$ exist, and $0$ otherwise. We denote it by $\characteristic(R)$. \\
	If $(R,\m)$ is a local ring, we say we are in the \define[local ring!equal-characteristic]{equal-characteristic} case if $\characteristic(R) = \characteristic(R/\m)$ and in the \define[local ring!unequal-characteristic]{unequal-characteristic} case otherwise.
\end{definition}

As $R/\m$ is a field, and $R$ is local, we can describe the possibilities for the characteristic more precisely:
\begin{nrenumerate}
	\item $\characteristic(R) = \characteristic(R/\m) = 0$;
	\item $\characteristic(R) = \characteristic(R/\m) = p$ for some prime $p$;
	\item $\characteristic(R/\m) = p$ for some prime $p$, while $\characteristic(R)=0$;
	\item $\characteristic(R/\m) = p$ for some prime $p$, while $\characteristic(R)=p^n$ for $n>1$.
\end{nrenumerate}

In the equal-characteristic case, we can describe $R$:

\begin{theorem}
	Let $R$ be a complete Noetherian local ring in the equal-characteristic case with $\m$ generated by $n$ elements. Then $R$ is a quotient of the ring $k[[x_1,\ldots,x_n]]$ with $k$ the residue field of $R$.
\end{theorem}
\begin{proof}
	This is Theorem~9 on p.~72 of \citenobackref{CohanCompleteLocalRings}.
\end{proof}

In the unequal-characteristic case, a similar result holds. This is more technical, and we will not explain all the terminology:

\begin{theorem}
	Let $R$ be a complete Noetherian local ring in the unequal-characteristic case with $\m$ generated by $n$ elements. Then $R$ is a quotient of the ring $S[[x_1,\ldots,x_n]]$, with $S$ a complete, unramified discrete valuation ring of characteristic $0$ with residue field $R/\m$.
\end{theorem}
\begin{proof}
	This is Theorem~12 on p.~84 of \citenobackref{CohanCompleteLocalRings}.
\end{proof}

For similar results in the non-commutative case, see \cite{BathoLocalRings}.

\subsection{Linear groups}

The theory of matrix groups is most often considered over fields or division rings. Over general rings, we can still define these groups, though the relation between them is not always what we are used to.

\begin{definition}
	Let $R$ be a unital ring. We write $\Mat_n(R)$ for the set of $n\times n$-matrices with entries in $R$, and we define addition and multiplication as usual.\\
	The \define{general linear group} of degree $n$ over $R$ is
	\[\GL_n(R)\notatlink{GLn} := \left\{ M \in \Mat_n(R)\mid \exists N\in\Mat_n(R)\colon MN=NM=\I_n\right\}\;.\]
	The \define{group of scalar matrices} is
	\[\Sc_n(R)\notatlink{Scn} := \left\{ r\I_n\mid r\in R\right\}\;.\]
	If $R$ is commutative, we define the determinant of a matrix as usual. The \define{special linear group} is
	\[\SL_n(R)\notatlink{SLn} := \left\{ M \in \GL_n(R)\mid \det(M)=1\right\}\;.\]
\end{definition}

%
In linear algebra, row and column operations are often important. These translate to some very specific matrices:

\begin{definition}
	Let $R$ be a unital ring. An \define{elementary matrix} is a matrix $e_{ij}(r)$ which has $1$ along the diagonal, $r$ on the $(i,j)$ position and $0$ everywhere else. The \define{elementary group} of degree $n$ over $R$ is
	\[\E_n(R) := \big\langle e_{ij}(r)\in\GL_n(R) \;\big|\; 1\leq i,j\leq n, i\neq j, r\in R\big\rangle\;.\]
\end{definition}


Clearly, if $R$ is commutative, $\det(e_{ij}(r)) = 1$, so $\E_n(R)\subset \SL_n(R)$.

\begin{theorem}\label{thm:SL_is_E}
	If $R$ is a commutative, local ring, $\E_n(R) = \SL_n(R)$.
\end{theorem}
\begin{proof}
	This is Theorem~4.3.9 of \cite{HahnOMearaClassicalGroups}.
\end{proof}

As for fields and division rings, there is a natural notion of projective linear groups.

\begin{definition}
	Let $R$ be a commutative ring. We define the \define{projective general linear group} of degree $n$ over $R$ as
	\[\PGL_n(R)\notatlink{PGLn} := \GL_n(R)/\Sc_n(R)\;.\]
	The \define{projective special linear group} of degree $n$ over $R$ is
	\[\PSL_n(R)\notatlink{PSLn} := \SL_n(R)/(\SL_n(R)\cap\Sc_n(R))\;.\]
	If $M\in\GL_n(R)$, we write $[M]$ for the image of $M$ in the quotient $\PGL_n(R)$.
\end{definition}

An immediate consequence of \autoref{thm:SL_is_E} is now

\begin{theorem}\label{thm:PSL_is_E}
	If $R$ is a commutative, local ring, $\PSL_n(R)$ is generated by the matrices $[e_{ij}(r)]$ for all $1\leq i,j\leq n$, $i\neq j$ and $r\in R$.
\end{theorem}


\section{Jordan algebras and Jordan pairs}	
\subsection{Jordan algebras}

While our use of Jordan algebras is minimal, they are relevant through their connection to Jordan pairs. A good first introduction to Jordan algebras are Jacobson's lecture notes \cite{JacobsonJordanAlgebrasNotes}. A more detailed and complete reference for the theory of Jordan algebras is \cite{McCrimmonTaste}. While Jordan algebras are often defined over fields, they can be defined over commutative, unital rings $k$. In this case, Jordan algebras are also called Jordan rings. Before we can give the definition, we need the notion of quadratic maps.

\begin{definition}
	Let $A$ and $B$ be $k$-modules. A map $Q\colon A\to B$ is \define[quadratic map]{quadratic} if
	\begin{romenumerate}
		\item $Q(\lambda x) = \lambda^2Q(x)$ for all $\lambda\in k$ and all $x\in A$;
		\item $Q(\cdot,\cdot)\colon A\times A\to B\colon (x,y)\mapsto Q(x+y)-Q(x)-Q(y)$ is $k$-bilinear.
	\end{romenumerate}
\end{definition}

We will also need the following lemma on quadratic maps:
\begin{lemma}\label{lem:quadextend}
	Let $Q\colon A\to B$ be a quadratic map between $k$-modules, and let $\ell$ be a commutative ring extension of $k$. Then there exists a unique $\tilde{Q}\colon \ell\otimes_k A\to \ell\otimes_k B$ such that $\tilde{Q}(1\otimes x) = 1\otimes Q(x)$ for all $x\in A$.
\end{lemma}
\begin{proof}This is the lemma on p.~16 of \citenobackref{JacobsonJordanAlgebrasNotes}.\end{proof}

Using this notion of extending quadratic maps, we can now define Jordan algebras and a few relevant notions.

\begin{definition}
	Let $k$ be a commutative, unital ring. A \define[Jordan algebra]{(quadratic) Jordan algebra} is a triple $(J,W,1)$, where $J\notatlink{J}$ is a $k$-module, $1\in J$ and $W\colon J\to\End(J)\notatlink{Wx}$ is a quadratic map such that (writing $W_x:= W(x)$)
	\begin{manualenumerate}[label=\textnormal{(JA\arabic*)},labelwidth=\widthof{(JA1)}]
		\item $W_1 = \id$;\label{axiom:JA1}
		\item $W_xW_yW_x = W_{yW_x}$ for all $x,y\in J$;\label{axiom:JA2}
		\item if we define $W_{x,y}:=W_{x+y}-W_x-W_y$ and $V_{x,y}$ by $zV_{x,y} := xW_{z,y}$ for all $z\in J$, then $W_yV_{x,y} = V_{y,x}W_y$;\label{axiom:JA3}
		\item if $\ell$ is a commutative ring extension of $k$, then \hyperref[axiom:JA2]{\normalfont{(JA2-3)}} are satisfied for $\tilde{W}$ as in \autoref{lem:quadextend}.\label{axiom:JA4}
	\end{manualenumerate}
	A submodule $I\subset J$ is called an \define[Jordan algebra!ideal]{ideal} if $xW_y,yW_x\in I$ for all $x\in I$ and $y\in J$. If $I$ is an ideal, $J/I$ is a Jordan algebra\\
	We call an element $x\in J$ \define[Jordan algebra!invertible]{invertible} if $W_x$ is invertible. A Jordan algebra is \define[Jordan algebra!Jordan division algebra]{division} if all $x\in J\setminus\{0\}$ are invertible.\\
	We call an element $x\in J$ \define[Jordan algebra!quasi-invertible element]{quasi-invertible} if $1-x$ is invertible. An ideal $I$ is \define[Jordan algebra!quasi-invertible ideal]{quasi-invertible} if every $x\in I$ is quasi-invertible.
\end{definition}

In literature, the quadratic map is most often denoted $U$, but to avoid confusion with the root groups in a (local) Moufang set, we will use $W$. In the following proposition, we gather some properties on the radical of a Jordan algebra.

\begin{proposition}
	Let $(J,W,1)$ be a Jordan algebra.
	\begin{romenumerate}
		\item There is a maximal quasi-invertible ideal in $J$ we call the \define[Jordan algebra!radical]{radical} and denote $\Rad J\notatlink{RadJ}$.
		\item The quotient $J/\Rad J$ is a Jordan division algebra if and only if $\Rad J$ consists of the non-invertible elements.
	\end{romenumerate}
\end{proposition}
\begin{proof}\preenum
	\begin{romenumerate}
		\item This is Theorem~1 on p.~86 of \citenobackref{JacobsonJordanAlgebrasNotes}.
		\item This is equivalence (2.2) of \cite{CamburnLocalJordanAlgebras}.\qedhere
	\end{romenumerate}
\end{proof}

We now arrive at the definition of a local Jordan algebra.

\begin{definition}
	A Jordan algebra $(J,W,1)$ is \define[Jordan algebra!local Jordan algebra]{local} if $\Rad J$ consists of the non-invertible elements.
\end{definition}

\subsection{Jordan pairs}

Jordan pairs were first introduced by K.~Meyberg in \cite{MeybergJordanTriple} and studied extensively by O.~Loos in \cite{LoosJordanPairs}. We recall some notations and definitions from \citenobackref{LoosJordanPairs}. Remark that we will change the left action of \emph{loc.\ cit.}\ to a right action, in order to be consistent with the action of our local Moufang sets. The index $\sigma$ will always have $+$ and $-$ as possible values.

\begin{definition}
	Let $k$ be a commutative unital ring and $V\notatlink{V} = (V^+,V^-)\notatlink{Vp}$ a pair of $k$-modules with quadratic maps 
	\[Q\colon V^\sigma\to \Hom(V^{-\sigma},V^\sigma)\notatlink{Qx}\;.\]
	We write $Q_x:=Q(x)$, $Q_{x,z}\notatlink{Qxz} := Q_{x+z}-Q_x-Q_z$, $zD_{x,y}\notatlink{Dxy} := yQ_{x,z}$ and $\{xyz\}\notatlink{xyz}:=yQ_{x,z}$.
	Then $V$ is a \define{Jordan pair} if the following axioms are satisfied in all scalar extensions of the base ring:
			\begin{manualenumerate}[label=\textnormal{(JP\arabic*)},labelwidth=\widthof{(JP1)}]
				\item $\{x\,y\,zQ_x\} = \{yxz\}Q_x$;\label{axiom:JP1}
				\item $\{yQ_x\,y\,z\} = \{x\,xQ_y\,z\}$;\label{axiom:JP2}
				\item $Q_{yQ_x} = Q_xQ_yQ_x$.\label{axiom:JP3}
			\end{manualenumerate}
	A pair of submodules $I = (I^+,I^-)$ is an \define[Jordan pair!ideal]{ideal} if $xQ_y\in I^\sigma$, $yQ_x\in I^{-\sigma}$ and $\{xyz\}\in I^\sigma$ for all $x\in I^\sigma, y\in V^{-\sigma}, z\in V^\sigma$. If $(I^+,I^-)$ is an ideal, the quotient $V/I = (V^+/I^+, V^-/I^-)$ is a Jordan pair. An ideal $I$ is \define[Jordan pair!proper ideal]{proper} if $I\neq V$.
	A \define[Jordan pair!homomorphism]{homomorphism} of Jordan pairs is a pair of $k$-linear maps $h_\sigma\colon V^\sigma\to W^\sigma$ such that
	\[h_\sigma(yQ_x) = h_{-\sigma}(y)Q_{h_\sigma(x)}\quad\text{for all $x\in V^\sigma,y\in V^{-\sigma}$.}\]
\end{definition}

The following proposition gives some useful criteria to check whether a given structure is a Jordan pair.
\begin{proposition}Let $V$ be a pair of $k$-modules with quadratic maps $Q$ as before.\label{prop:JPsufficientaxioms}
	\begin{romenumerate}
		\item If $V$ has no $2$-torsion (i.e.\ neither $V^+$ nor $V^-$ have $2$-torsion), then \ref{axiom:JP3} follows from \hyperref[axiom:JP1]{\normalfont{(JP1-2)}}. Hence in this case $V$ is a Jordan pair if \hyperref[axiom:JP1]{\normalfont{(JP1-2)}} are satisfied in all scalar extensions of the base ring.
		\item If $V$ has no $2$-torsion and \ref{axiom:JP1}, all its linearizations and \ref{axiom:JP2} hold, then $V$ is a Jordan pair.\label{itm:sufficient_JP}
	\end{romenumerate}
\end{proposition}
\begin{proof}\preenum
	\begin{romenumerate}
		\item This is \citenobackrefoptional{Proposition 2.2(a)}{LoosJordanPairs}.
		\item This is remarked just after \citenobackrefoptional{Definition 1.2}{LoosJordanPairs}.\qedhere
	\end{romenumerate}
\end{proof}

We will also need the notions of invertibility, quasi-invertibility, and of course of a local Jordan pair, again from \citenobackref{LoosJordanPairs}.

\begin{definition}
	An element $x\in V^\sigma$ is \define[Jordan pair!invertible]{invertible} if and only if $Q_x$ is invertible. In this case, we define $x^{-1}:=xQ_x^{-1}$. A Jordan pair is \define[Jordan pair!division Jordan pair]{division} if all nonzero elements are invertible. A Jordan pair is \define[Jordan pair!local Jordan pair]{local} if the non-invertible elements form a proper ideal.\\
	For $(x,y)\in V$ (this means $x\in V^+, y\in V^-$), we define the \define[Jordan pair!Bergman operator]{Bergman operator}
		\[ B_{x,y}\notatlink{Bxy} := \id - D_{x,y} + Q_yQ_x \;, \]
	and $(x,y)$ is \define[Jordan pair!quasi-invertible]{quasi-invertible} if and only if $B_{x,y}$ is invertible. In this case, we define the \define[Jordan pair!quasi-inverse]{quasi-inverse} $x^y\notatlink{xyJP} := (x-yQ_x)B^{-1}_{x,y}$. An element $x\in V^+$ (or $y\in V^-$) is \define[Jordan pair!propertly quasi-invertible]{properly quasi-invertible} if and only if $(x,y)$ is quasi-invertible for all $y\in V^-$ (or all $x\in V^+$, respectively). The \define[Jordan pair!Jacobson radical]{Jacobson radical} $\Rad V = (\Rad V^+, \Rad V^-)\notatlink{RadV}$ is the pair of sets of all properly quasi-invertible elements.
\end{definition}

To do computations in local Jordan pairs, we will need some further properties and identities:

\begin{proposition}\label{prop:JPbasic} Let $V = (V^+,V^-)$ be a Jordan pair.
	\begin{romenumerate}[labelindent=0.2ex, leftmargin=*]
		\item For any $x\in V^\sigma$ and $y\in V^{-\sigma}$, we have, \label{itm:JPbasicidDQ}
			\[Q_{x,yQ_x} = Q_xD_{x,y} = D_{y,x}Q_x\;.\]
		\item For invertible $x\in V^\sigma$ and any $y\in V^\sigma$, we have \label{itm:JPbasicid}
			\[Q_{x,y}Q_x^{-1} = D_{x^{-1},y}\;.\]
		\item For $(x,y)\in V$ with $x$ invertible, $B_{x,y} = Q_{x^{-1}-y}Q_x$. If $y$ is invertible, we have $B_{x,y} = Q_yQ_{x-y^{-1}}$.\label{itm:JPbasicA}
		\item Assume $(x,y)\in V$ is quasi-invertible and $z\in V^-$. Then $(x^y,z)$ is quasi-invertible if and only if $(x,y+z)$ is quasi-invertible in $(V^-,V^+)$. In this case, we have $x^{y+z} = (x^y)^z$.\label{itm:JPbasicB}
		\item $(x,y)\in V$ is quasi-invertible if and only $(y,x)$ is quasi-invertible in $(V^-,V^+)$. In this case, $x^y = x+y^xQ_x$.\label{itm:JPbasicswitch}
		\item $(x,zQ_y)\in V$ is quasi-invertible if and only if $(xQ_y,z)$ is quasi-invertible. In this case, $(xQ_y)^z = x^{zQ_y}Q_y$.\label{itm:JPbasicQy}
		\item The pair of sets $\Rad V$ is an ideal.
		\item If $V$ is a local Jordan pair, then $\Rad V$ is the set of non-invertible elements of $V$.\label{itm:JPbasicD}
		\item If $V/\Rad V$ is a nontrivial Jordan division pair, then $V$ is a local Jordan pair.
		\item If $(x,y)\modulo{\Rad V}$ is quasi-invertible in $V/\Rad V$, then $(x,y)\in V$ is also quasi-invertible.\label{itm:JPbasicQIlift}
		\item If $x\in \Rad V^+$ and $y\in V^-$, then $x^y\in\Rad V^+$.
		\item If $x,y\in V^+$ are invertible, then\label{itm:JPbasicInvertRadical}
			\[x-y\in\Rad V^+ \implies x^{-1}-y^{-1}\in\Rad V^-\;.\]
	\end{romenumerate}
\end{proposition}
\begin{proof}\preenum
	\begin{romenumerate}[labelindent=0.2ex, leftmargin=*]
		\item This is JP4 in \citenobackrefoptional{2.1}{LoosJordanPairs}.
		\item From the definition of $Q_{y,z}$, it is clear that
		\[Q_xQ_{y,z}Q_x = Q_{yQ_x,zQ_x}\;,\]
		so
		\begin{align*}
			Q_{x,y}Q_x^{-1} &= Q_xQ_x^{-1}Q_{x,y}Q_x^{-1} = Q_xQ_{x^{-1},yQ_x^{-1}} \\
					&= Q_xQ_x^{-1}D_{x^{-1},y} = D_{x^{-1},y}\;.
		\end{align*}
		\item This is \citenobackrefoptional{2.12}{LoosJordanPairs}.
		\item This is \citenobackrefoptional{3.7(1)}{LoosJordanPairs}.
		\item This is \citenobackrefoptional{3.3}{LoosJordanPairs}.
		\item This is \citenobackrefoptional{3.5(1)}{LoosJordanPairs}.
		\item This is part of \citenobackrefoptional{4.2}{LoosJordanPairs}.
		\item This is \citenobackrefoptional{4.4(a)}{LoosJordanPairs}.
		\item This is \citenobackrefoptional{4.4(b)}{LoosJordanPairs}.
		\item This is \citenobackrefoptional{4.3}{LoosJordanPairs}.
		\item If $x\in\Rad V^+$, then $(x,z)$ is quasi-invertible for all $z\in V^-$. Hence $(x,y+z)$ is quasi-invertible for all $z\in V^-$, so $(x^y,z)$ is quasi-invertible for all $z\in V^-$. Hence $x^y$ is properly quasi-invertible, and $x^y\in\Rad V^+$.
		\item Since $\Rad V$ is an ideal and $Q_x$ is invertible, it is sufficient to prove that $(x^{-1}-y^{-1})Q_x$ is in $\Rad V^+$. This is indeed true, as
		\begin{align*}
			(x^{-1}-y^{-1})Q_x &= x - y^{-1}Q_{x-y+y} \\
				&= x - y^{-1}Q_{x-y,y} - y^{-1}Q_{x-y} - y^{-1}Q_{y} \\
				&= (x-y) - y^{-1}Q_{x-y,y} - y^{-1}Q_{x-y} \in \Rad V^+\;,
		\end{align*}
		since $x-y\in\Rad V^+$.\qedhere
	\end{romenumerate}
\end{proof}

The connections between Jordan algebras and Jordan pairs are as follows:

\begin{proposition}
    Let $(J,W,1)$ be a quadratic Jordan algebra. Then:
	\begin{romenumerate}
		\item $(J,J)$ is a Jordan pair with $Q := W$.\label{itm:alg-pair}
		\item $J$ is a Jordan division algebra if and only if $(J,J)$ is a Jordan division pair.\label{itm:pair-alg-div}
		\item $J$ is a local Jordan algebra if and only if $(J,J)$ is a local Jordan pair.\label{itm:pair-alg-loc}
		\item The radical of the Jordan pair $(J,J)$ is $(\Rad J,\Rad J)$, where $\Rad J$ is the radical of the Jordan algebra $J$.\label{itm:pair-alg-rad}
		\item The map $J\mapsto(J,J)$ induces a bijection from isotopy classes of local Jordan algebras to isomorphism classes of local Jordan pairs.\label{itm:pair-alg-classes}
	\end{romenumerate}
\end{proposition}
\begin{proof}
	The first statement~\ref{itm:alg-pair} is \citenobackrefoptional{1.6}{LoosJordanPairs}. Statements \ref{itm:pair-alg-div} and \ref{itm:pair-alg-loc} are in \citenobackrefoptional{1.10}{LoosJordanPairs}, and \ref{itm:pair-alg-classes} is a consequence of \ref{itm:pair-alg-loc} and \citenobackrefoptional{1.12}{LoosJordanPairs}. Finally, \ref{itm:pair-alg-rad} is one of the statements of \citenobackrefoptional{4.17}{LoosJordanPairs}.
\end{proof}

Local Jordan pairs will pop up in the theory of local Moufang sets in \autoref{chap:chap7_jordan}, which means it is interesting to know some examples.

\begin{examples}\preenum \label{ex:local}
	\begin{nrenumerate}
		\item Let $A$ be a local associative (not necessarily commutative) ring. Then $V = (A,A)$ with $Q_a\colon A\to A:x\mapsto axa$ is a local Jordan pair.
		\item Let $V=(V^+,V^-)$ be a Jordan division pair over a field $k$ and let $R$ be a commutative $k$-algebra which is a local ring. Then we can define $V\otimes_k R = (V^+\otimes_k R,V^-\otimes_k R)$ with
		\[(x\otimes r)Q_{y\otimes s}:= xQ_y\otimes rs^2\;.\]
		This is a local Jordan pair.
		\item Let $J$ be a finite dimensional Jordan division algebra over a field $K$ which is complete with respect to a discrete valuation $v$. Then we can extend the valuation on $K$ to a valuation $v_J$ on $J$. The subalgebra $J_0 = \{x\in J\mid v_J(x)\geq0\}$ is now a local Jordan algebra with $\Rad J_0 = \{x\in J\mid v_J(x)>0\}$ (and hence $(J_0,J_0)$ is a local Jordan pair). We refer to \cite{KnebuschJordanValuation,PeterssonJordanValuations,PeterssonExceptionalJordanValuation} for more details.
	\end{nrenumerate}
\end{examples}

\section{Moufang sets}	
\subsection{Definition and construction}

Moufang sets were introduced by J.~Tits in the context of twin buildings in \cite{TitsTwinBldgs}. From a geometric point of view, Moufang sets are Moufang buildings of rank one. From an algebraic point of view, they are connected to abstract rank one groups introduced by Timmesfeld \cite{TimmesfeldAR1G}. A good reference to learn about Moufang sets are the notes by T.~De Medts and Y.~Segev \cite{DMSegevCourseMS}.

\begin{definition}
	A \define{Moufang set} $\M$ is a set $X$ (with $\abs{X}>2$) with a collection of groups $(U_x)_{x\in X}$ such that
	\begin{manualenumerate}[label=\textnormal{(M\arabic*)},labelwidth=\widthof{(M1)}]
		\item For $x\in X$, $U_x\leq\Sym(X)$ fixes $x$ and acts sharply transitively on $X\setminus\{x\}$\label{axiom:M1}.
		\item For $x\in X$ and $g\in \langle U_x\mid x\in X\rangle$, we have $U_x^g = U_{xg}$\label{axiom:M2}.
	\end{manualenumerate}
	We write $\M = \big(X,(U_x)_{x\in X}\big)$. The groups $U_x$ are called the \define{root groups} of $\M$ and the group $G=\langle U_x\mid x\in X\rangle$ generated by these root groups is called the \define{little projective group} of $\M$.
\end{definition}

One immediate consequence of the definition of a Moufang set is that the little projective group acts doubly transitively on $X$ and furthermore, for any two different $x,y\in X$, we have $G = \langle U_x,U_y\rangle$. This allows us to fix any two points $0$ and $\infty$ in $X$ and only work with the groups $U_0$ and $U_\infty$. By the $2$-transitivity, there must also be elements $\tau\in G$ swapping $0$ and $\infty$,and hence also $U_\infty^\tau = U_0$ (and $U_0^\tau=U_\infty$). This means that $G = \langle U_\infty,\tau\rangle$. This led T.~De Medts and R.~Weiss to construct Moufang sets from a group $U$ and one permutation $\tau$ in \cite{DMWeissMSJordanAlgebras}. It is common to write this group additively even though it need not be commutative, hence we will write $0$ for the identity.

\begin{construction}\label{constr:MS}
	Let $U$ be a group and denote the identity element of $U$ by $0$ and assume $\tau$ is a permutation of $U\setminus\{0\}$. We now set $X = U\cup\{\infty\}$ with $\infty$ a new symbol. For each $x\in U$, let $\alpha_x$ be the permutation of $X$ given by $\infty\alpha_x = \infty$ and the right permutation action of $U$ elsewhere. Furthermore, extend $\tau$ to $X$ by defining $0\tau = \infty$ and $\infty\tau = 0$. We can now define a group $U_x$ for each $x\in X$ as follows:
	\begin{align*}
		U_\infty &:= \{\alpha_x\mid x\in U\} 	& 	U_0 &:= U_\infty^\tau 	&	U_x := U_0^{\alpha_x}
	\end{align*}
	for all $x\in X\setminus\{0,\infty\}$. We write
	\[\M(U,\tau) := \big(X,(U_x)_{x\in X}\big)\;.\]
\end{construction}

If you start from a Moufang set $\M$ and take $U_\infty$ and any $\tau\in G$ swapping $0$ and $\infty$, then by \ref{axiom:M2} we get $\M = \M(U_\infty,\tau)$. In general however, $\M(U,\tau)$ need not be a Moufang set. What we do get is a so-called \emph{pre-Moufang set} (a notion introduced by O.~Loos in \cite{LoosDivpairs}):

\begin{definition}
	A \define{pre-Moufang set} is a set $X$ with a collection of groups $(U_x)_{x\in X}$ satisfying \ref{axiom:M1}.
\end{definition}

Clearly a Moufang set is a pre-Moufang set. The little projective group of a pre-Moufang set still acts doubly transitively on $X$, so we can still always choose two points $0$ and $\infty$.

\subsection{\textit{\textmugreek}-maps and Hua maps}

Assume we have a (pre-)Moufang set $\M$ with fixed points $0$ and $\infty$. For any $x\neq\infty$, we write $\alpha_x$ for the unique element of $U_\infty$ such that $0\alpha_x = x$. Even though we know $G$ acts doubly transitively, it is interesting to have specific maps swapping $0$ and $\infty$. The following proposition describes such maps:

\begin{proposition}
	Let $\M$ be a pre-Moufang set with fixed points $0$ and $\infty$ and assume $x\in X\setminus\{0,\infty\}$. Then there is a unique element $\mu_x\in U_0\alpha_xU_0$ such that $0\mu_x = \infty$ and $\infty\mu_x=0$. We call this map the \define[m-map@$\mu$-map]{$\mu$-map} corresponding to $x$.
\end{proposition}
\begin{proof}
	The proof of Proposition~4.1.1 from \citenobackref{DMSegevCourseMS} only requires \ref{axiom:M1}.
\end{proof}

If $\M$ is a Moufang set, we usually choose $\tau$ to be a $\mu$-map. In $\M(U,\tau)$, $\tau$ is in general not a $\mu$-map, but we can use $\mu$-maps to determine which $\M(U,\tau)$ are Moufang sets:

\begin{theorem}
	The construction $\M(U,\tau)$ is a Moufang set if and only if $U^{\mu_x} = U^\tau$ for all $x\in X\setminus\{0,\infty\}$. In this case $\M(U,\tau) = \M(U,\mu_x)$ for any $x\in X\setminus\{0,\infty\}$.
\end{theorem}
\begin{proof}
	This is Theorem~3.1(i) in \citenobackref{DMWeissMSJordanAlgebras}, also using the fact that $\mu_x^{-1}$ is also a $\mu$-map. The second statement is Lemma~4.1.2 in \citenobackref{DMSegevCourseMS}.
\end{proof}

If we have a Moufang set and a fixed $\tau$ swapping $0$ and $\infty$, we also get a specific element fixing $0$ and $\infty$ for each $x\in X\setminus\{0,\infty\}$.

\begin{definition}
	Let $\M(U,\tau)$ be a Moufang set and $x\in X\setminus\{0,\infty\}$. The \define{Hua map} corresponding to $x$ is $h_x := \tau\mu_x$. The \define{Hua subgroup} of $\M$ is the group $H := \langle\mu_x\mu_y\mid x,y\in X\setminus\{0,\infty\}\rangle$.
\end{definition}

Observe that if $\tau$ is a $\mu$-map, the Hua subgroup is the group generated by Hua maps. The Hua subgroup fixes $0$ and $\infty$, and it is in fact equal to the two-point stabilizer of the little projective group.

\begin{theorem}
	If $\M$ is a Moufang set with little projective group $G$, then $H = G_{0,\infty}$ and $G_\infty = U_\infty H$.
\end{theorem}
\begin{proof}
	This is Theorem~3.1(ii) in \citenobackref{DMWeissMSJordanAlgebras}.
\end{proof}

The Hua maps can also be used to determine when the construction $\M(U,\tau)$ is a Moufang set, but it requires a different definition of the Hua maps (which coincides when we know we have a Moufang set). When using this different definition of $h_x$, one gets

\begin{theorem}
	The construction $\M(U,\tau)$ is a Moufang set if and only if $\restr{h_x}{U}$ is an automorphism of $U$ for all $x\in X\setminus\{0,\infty\}$.
\end{theorem}
\begin{proof}
	This is Theorem~3.2 in \citenobackref{DMWeissMSJordanAlgebras}.
\end{proof}

\subsection{Special Moufang sets}

Suppose we have a Moufang set $\M$. We generally write $\alpha_x$ for the unique element of $U_\infty$ such that $0\alpha_x = x$. We write $-x:=0\alpha_x^{-1}$, so $\alpha_x^{-1} = \alpha_{-x}$.

\begin{definition}
	A Moufang set $\M(U,\tau)$ is \define[Moufang set!special Moufang set]{special} if $(-x)\tau = -(x\tau)$ for all $x\in X\setminus\{0,\infty\}$.
\end{definition}

One reason to study special Moufang sets is because it puts extra restrictions on the structure of the Moufang set. One main motivation is given by the following conjecture:

\begin{conjecture}
	A special Moufang set $\M$ has abelian root groups.
\end{conjecture}

There is also a converse to this, which was proven by Y.~Segev in \cite{SegevAbelianSpecial}:

\begin{theorem}
	If $\M(U,\tau)$ is a Moufang set with $U$ abelian, then either $H=1$ or $\M(U,\tau)$ is special.
\end{theorem}

As it is likely that having a special Moufang set implies having abelian root groups, it is interesting to see what we can prove if we assume both. The following proposition gives a short list of identities one can prove in this case.

\begin{proposition}
	Let $\M$ be a special Moufang set with abelian root groups. Then
	\begin{romenumerate}
		\item $\mu_x^2 = \id$ for all $x\in X\setminus\{0,\infty\}$;
		\item $h_x = h_{-x}$ for all $x\in X\setminus\{0,\infty\}$;
		\item $h_yh_xh_y = h_{xh_y}$ for all $x\in X\setminus\{0,\infty\}$;
		\item $\mu_x\mu_{x\alpha_y}\mu_y = \mu_y\mu_{x\alpha_y}\mu_x$ for all $x,y\in X\setminus\{0,\infty\}$ with $y\neq -x$.
	\end{romenumerate}
\end{proposition}
\begin{proof}\preenum
	\begin{romenumerate}
		\item This is Lemma~5.1 in \cite{DMSegevIdentitiesMS}.
		\item This is Proposition~5.2(1) in \cite{DMSegevIdentitiesMS}.
		\item This is Proposition~5.2(4) in \cite{DMSegevIdentitiesMS}.
		\item This is Proposition~5.8(1) in \cite{DMSegevIdentitiesMS}.\qedhere
	\end{romenumerate}
\end{proof}

\subsection{Homomorphisms of Moufang sets}

Homomorphisms of Moufang sets have not gotten much attention so far. A short attempt was made by T.~De Medts in an unpublished note \cite{TDMcategoryMoufang}, but he was not satisfied with the definition and abandoned the idea. O.~Loos picked up his idea, and created a categorical framework for Moufang sets in \cite{LoosDivpairs}.

\begin{definition}
	Let $\M=(X,(U_x))$ and $\M'=(X',(U'_x))$ be two Moufang sets. A \define[Moufang set!homomorphism]{homomorphism} is an injective map $\phi\colon X\to X'$ such that
	\[U_x\phi\subset\phi U'_{x\phi}\text{ for all $x\in X$.}\]
	We also write $\phi\colon\M\to\M'$. The category $\Mou$\notatlink{Mou} is the category with Moufang sets as objects and homomorphisms as morphisms.
	
	A homomorphism $\phi$ is an \define[Moufang set!isomorphism]{isomorphism} if there exists a homomorphism of Moufang sets $\phi^{-1}\colon\M'\to\M$ such that $\phi\phi^{-1} = \phi^{-1}\phi=\id$.
\end{definition}

T.~De Medts was dissatisfied with the definition because of the injectivity requirement. If injectivity is not required, one can also get a constant map, but a non-injective map will always be constant. In short, there are no interesting quotients in this category.

When $\phi\colon\M\to\M'$ is a homomorphism of Moufang sets, O.~Loos proved for each $x\in X$ the existence of a unique map $\theta_x\colon U_x\to U'_{x\phi}$ such that
\[u\phi = \phi \theta_x(u)\text{ for all $u\in U_x$,}\]
and using these homomorphisms he also proved the following:
\begin{proposition}
	Let $\phi\colon X\to X'$ be a map. Then the following are equivalent:
	\begin{romenumerate}
		\item $\phi$ is an isomorphism of Moufang sets;
		\item $\phi$ is a bijective homomorphism of Moufang sets;
		\item $\phi$ is a bijective map and $U_x = \phi U'_{x\phi}\phi^{-1}$ for all $x\in X$.
	\end{romenumerate}
\end{proposition}
\begin{proof}
	See Lemma~2.3 in \citenobackref{LoosDivpairs}.
\end{proof}

\subsection{Examples of Moufang sets}

We now give some examples of Moufang sets. There are many different kinds of Moufang sets, and often also many different descriptions of these Moufang sets. The following examples are ones that we have extended to the setting of local Moufang sets in \autoref{part:examples}.

\subsubsection{Projective Moufang sets}

The easiest class of Moufang sets corresponds to projective lines over fields or skew fields. Let $K$ be a (skew) field and set
\[U = (K,+)\text{ and }\tau\colon U\setminus\{0\}\to U\setminus\{0\}\colon x\to -x^{-1}\;.\]
Then $\M(U,\tau)$ is a Moufang set we denote by $\M(K)$ with Hua maps given by $h_x\colon y\mapsto xyx$. It is possible to consider the underlying set $U\cup\{\infty\}$ as $\P^1(K)$, the projective line over $K$ by the identification 
\[x\mapsto [1,x]\text{ and }\infty\mapsto[0,1]\;.\]
With this identification the little projective group becomes $\PSL_2(K)$. We call this Moufang set a \define[Moufang set!projective Moufang set]{projective Moufang set}.

In \cite{DMWeissMSJordanAlgebras}, the projective Moufang sets over fields of characteristic different from $2$ are characterized up to isomorphism:

\begin{theorem}
	Suppose $\M$ is a special Moufang set with abelian, uniquely $2$-divisible root groups and an abelian Hua subgroup. Then there is a field $K$ with $\characteristic(K)\neq2$ such that $\M\cong\M(K)$.
\end{theorem}
\begin{proof}
	This is Theorem~6.1 in \citenobackref{DMWeissMSJordanAlgebras}.
\end{proof}

The case where $U$ is not uniquely $2$-divisible has been studied by M.~Gr\"uninger in \cite{GruningerPSL2}. In this case, $\M$ arises from a specific Jordan division algebra.

\subsubsection{Moufang sets from Jordan division algebras}

A more general class of Moufang sets arises from Jordan division algebras. Let $(J,W,1)$ be a Jordan division algebra, then set
\[U = (J,+)\text{ and }\tau\colon U\setminus\{0\}\to U\setminus\{0\}\colon x\to -x^{-1}\;.\]
Again $\M(U,\tau)$ is a Moufang set, and we denote it by $\M(J)$. One interesting property of these Moufang sets is that we can retrieve $W$ from the Moufang set, as
\[h_x = W_x\text{ for all $x\in U\setminus\{0\}$.}\]
This would suggest that, under some conditions, we should be able to get a Jordan division algebra from a Moufang set. This will be the case: assume $\M(U,\tau)$ is a Moufang set with $\tau$ a $\mu$-map, and suppose $U$ is an abelian, uniquely $2$- and $3$-divisible group. For all $x,y\in U$ we set
\[h_{x,y} := h_{x+y}-h_x-h_y\;,\]
where we set $h_0:=0$. Then the following theorem characterizes Moufang sets coming from most Jordan division algebras originally proved in \citenobackref{DMWeissMSJordanAlgebras}:

\begin{theorem}
	If $\M(U,\tau)$ is a Moufang set with $\tau=\mu_e$ and $U$ an abelian, uniquely $2$- and $3$-divisible group. If the map $(x,y)\mapsto h_{x,y}$ is biadditive, then $(U,h,e)$ is a Jordan division algebra.
\end{theorem}
\begin{proof}
	This is Corollary~5.12 in \citenobackref{DMSegevIdentitiesMS}.
\end{proof}

It is interesting to remark that this correspondence can also be made with Jordan pairs. In \cite{LoosRogdiv}, O.~Loos introduced the notion of \define{division pairs}, and in \cite{LoosDivpairs} he showed that every division pair gives rise to a Moufang set. As Jordan division pairs are also division pairs, we also get a Moufang set from Jordan division pairs.

\subsubsection{Orthogonal Moufang sets}

One specific type of Moufang sets arising from Jordan division algebras can also be described by quadratic forms. Let $K$ be a field and $W$ a $K$-vector space. A \emph{quadratic form} is a quadratic map $q\colon W\to K$ with corresponding bilinear form $f$ defined by
\[f(x,y) := q(x+y)-q(x)-q(y)\;.\]
We assume $q$ is \emph{anisotropic}, i.e.\ $q(x)=0\iff x=0$. We set
\begin{align*}
	U &= \{(x,t)\in W\times K\mid q(x)=t\}	\\
	\tau&\colon U\setminus\{(0,0)\}\to U\setminus\{(0,0)\}\colon (x,t)\mapsto(xt^{-1},t^{-1})\;.
\end{align*}
Now $U$ is an abelian group with group operation 
\[(x,s)+(y,t) = (x+y,s+t+f(x,y))\;.\]
Furthermore, $\M(U,\tau)$ is a Moufang set we denote by $\M(W,q)$.

\subsubsection{Hermitian Moufang sets}

Hermitian Moufang sets arise from Hermitian forms. Let $K$ be a (skew) field with involution $\ast$ and $W$ a right $K$-vector space. We write 
\[\Lambda = \{t-t^\ast\mid t\in K\}\;.\]
A map $h\colon W\times W\to K$ is \emph{Hermitian} if for all $x,y\in W$ and $s,t\in K$
\[h(xt,ys) = t^\ast h(x,y)s\text{ and }h(x,y)^\ast=h(y,x).\]
A \emph{$\Lambda$-quadratic form} is a map $q\colon W\to K/\Lambda$ such that there is a hermitian map $h$ with
\[q(x+y) = q(x)+q(y)+h(x,y)\text{ and }q(xt)=t^\ast q(x)t\]
for all $x,y\in W$ and $t\in K$. We call $q$ \define[L-quadratic form@$\Lambda$-quadratic form!anisotropic ---]{anisotropic} if $q(x)=0\iff x=0$. We now set
\begin{align*}
	U &= \{(x,t)\in W\times K\mid q(x)-t=\Lambda\}	\\
	\tau&\colon U\setminus\{(0,0)\}\to U\setminus\{(0,0)\}\colon (x,t)\mapsto(xt^{-1},t^{-1})\;.
\end{align*}
Again $U$ is a group with
\[(x,s)\cdot(y,t) = (x+y,s+t+h(y,x))\;.\]
If $h$ is not symmetric, $U$ is not abelian. We get a Moufang set $\M(U,\tau)$ which we denote by $\M(W,q)$.

	\part{The theory of local Moufang sets}\label{part:theory}

	\chapter[Definitions and first properties]{Definitions and\\ first properties}\label{chap:chap2_definitions}
	We define local Moufang sets and try to develop some of the basic theory analogous to the theory of Moufang sets. We soon need the notion of units in local Moufang sets, and we use these to define $\mu$-maps an Hua maps. Finally, we prove that the Hua subgroup equals the two-point stabilizer in local Moufang sets, a result that was known for Moufang sets.

\section{Local Moufang sets}	
\subsection{Defining a local Moufang set}

Local Moufang sets are generalizations of Moufang sets which encompass a larger class of groups with very nice actions. While Moufang sets act on sets, local Moufang sets need a richer structure to act on: sets with an equivalence relation. Informally, this equivalence relations describes when two points are `close'. This idea already pops up in \cite{BartoloneSperaPGL2L}, even though I did not know of this article when I discovered the approach.

\begin{definition}
	A \define{local Moufang set} $\M$\notatlink{M} consists of the following data:
	\begin{nrenumerate}
		\item a set with equivalence relation $(X,\sim)$ with $\abs{\class{X}}>2$;
		\item for each $x\in X$ a \define{root group} $U_x\notatlink{Ux}\leq\Sym(X,\sim)$.
	\end{nrenumerate}
	The \define{little projective group} is the group generated by the root groups, and is usually denoted by $G\notatlink{G}:=\langle U_x\mid x\in X\rangle$. This data must satisfy the following axioms:
	\begin{manualenumerate}[label=\textnormal{(LM\arabic*)},start=0,labelwidth=\widthof{(LM0')}]
		\setcounter{enumi}{-1}
		\item 	If $x\sim y$ for $x,y\in X$, then $\induced{U_x}=\induced{U_y}$\label{axiom:LM0}.
		\item 	For $x\in X$, $U_x$ fixes $x$ and acts sharply transitively on $X\setminus\class{x}$\label{axiom:LM1}.
		\item[\mylabel{axiom:LM1'}{\textnormal{(LM1')}}]%
			For $\class{x}\in \class{X}$, $\induced{U_x}$ fixes $\class{x}$ and acts sharply transitively on $\class{X}\setminus\{\class{x}\}$.
		\item	 For $x\in X$ and $g\in G$, we have $U_x^g = U_{xg}$\label{axiom:LM2}.
	\end{manualenumerate}
	We use the notation $U_{\class{x}}\notatlink{Uxbar}:=\induced{U_x}$, which is justified by \ref{axiom:LM0}.
\end{definition}

It is worth noting that from these axioms, we also get
\begin{manualenumerate}[label=\textnormal{(LM2')},labelwidth=\widthof{(LM0')}]
	\item For $\class{x}\in \class{X}$ and $g\in G$, we have $U_{\class{x}}^g = U_{\class{xg}}$;\label{axiom:LM2'}
\end{manualenumerate}
this follows from \ref{axiom:LM2} and the fact that we are working with the induced action.
By \ref{axiom:LM0} the group $U_{\class{x}}$ only depends on $\class{x}$, and \ref{axiom:LM1'} and \ref{axiom:LM2'} precisely state that $(\class{X},(U_{\class{x}})_{\class{x}\in \class{X}})$ is a Moufang set.

\begin{definition}
	Let $\M=(X,(U_x)_{x\in X})$ be a local Moufang set. The \define{quotient Moufang set} is the Moufang set $(\class{X},(U_{\class{x}})_{\class{x}\in \class{X}})$ we denote by $\induced{M}$\notatlink{Mbar}.
\end{definition}

The fact that this Moufang set is indeed a quotient of $\M$, for some sensible notion of quotient, will be shown in \autoref{prop:quotientMS}.

As in the case of Moufang sets, the axiom \ref{axiom:LM2} is the most restrictive. Hence, we also name the structures which do not necessarily satisfy \ref{axiom:LM2}.

\begin{definition}
	A \define{local pre-Moufang set} consists of the same data as a local Moufang set, but need not satisfy \ref{axiom:LM2}.
\end{definition}

As we intended to generalize Moufang sets, it is worthwhile to notice when a local Moufang set is in fact a Moufang set. The structure we added is the equivalence relation, so as one would expect, we get a Moufang set when this equivalence relation is trivial.

\begin{proposition}
	Let $\M$ be a local (pre-)Moufang set acting on a set $(X,\sim)$. Then $\M$ is a (pre-)Moufang set if and only if $\sim$ is the identity relation.
\end{proposition}%
\begin{proof}%
	If $\sim$ is the identity relation, \ref{axiom:LM1} reduces to \ref{axiom:M1}. Hence $\M$ is a pre-Moufang set. Conversely, assume $\M$ is a pre-Moufang set. Take any $x\in X$ and any $y\in X\setminus\class{x}$. By \ref{axiom:M1}, we find $y^{U_x} = X\setminus\{x\}$, but by \ref{axiom:LM1}, we get $y^{U_x} = X\setminus\class{x}$. Hence $\class{x}=\{x\}$ for all $x\in X$, so $\sim$ is the identity relation. Finally, \ref{axiom:LM2} actually coincides with \ref{axiom:M2}, so the same statement holds for Moufang sets and local Moufang sets.
\end{proof}

\subsection{Generation and transitivity}

The little projective group $G$ may seem hard to grasp, but it turns out that it is sufficient to take two root groups in `general position' to generate all of $G$.

\begin{proposition}\label{prop:UxUy}
	Let $\M$ be a local Moufang set, and $x,y\in X$ with $x\nsim y$. Then $\langle U_x,U_y \rangle = G$.
\end{proposition}
\begin{proof}
	It is sufficient to show that $U_z\subset\langle U_x,U_y\rangle$ for any $z\in X$. Assume first that $z\sim x$. Then $z\nsim y$, so by \ref{axiom:LM1} there is an element $g\in U_y$ such that $xg=z$. By \ref{axiom:LM2}, $U_z = U_x^g \subset \langle U_x,U_y\rangle$. Similarly, if $z\nsim x$, then there is an element $g\in U_x$ such that $yg = z$, so $U_z = U_y^g \subset \langle U_x,U_y\rangle$.
\end{proof}

This means that it is sufficient to give the set with its equivalence relation and two such root groups to generate all the data of a local Moufang set. A next natural question is whether or not it makes a difference which two such root groups we take. It turns out that every two non-equivalent points play the same role:

\begin{proposition}\label{prop:twotrans}
	Let $\M$ be a local Moufang set. The little projective group $G$ acts transitively on $\{(x,y)\in X^2\mid x\nsim y\}$.
\end{proposition}
\begin{proof}
	Let $(x,y)$ and $(x',y')$ be such pairs. We will first map $x$ to $x'$. Since $\abs{\class{X}}>2$, there is a point $z\in X$ such that $x\nsim z$ and $x'\nsim z$. By \ref{axiom:LM1}, there is an element $g\in U_z$ s.t.\ $x\cdot g=x'$, and hence $(x,y)\cdot g = (x',y'')$ for some $y''\nsim x'$.
	
	So, after renaming, we have reduced the question to finding an element mapping $(x,y)$ to $(x,y')$. Since $y\nsim x$ and $y'\nsim x$, there is an element $g\in U_x$ mapping $y$ to $y'$, hence $(x,y)\cdot g = (x,y')$.
\end{proof}

These two propositions indicate that we can choose two points and do computations involving the two corresponding root groups.

\begin{definition}
	Let $(X,\sim)$ be a set with equivalence relation. We say a tuple of points $(x,y)$ in $X$ is a \define{basis} if $x\nsim y$.
\end{definition}

In the theory of Moufang sets, there are big differences between Moufang sets that act sharply transitive on the bases, and those whose transitive action is not sharp.

\begin{definition}
	When the little projective group of a local Moufang set acts sharply transitively on $\{(x,y)\in X^2\mid x\nsim y\}$, we call it \define[local Moufang set!improper local Moufang set]{improper}. Otherwise, we say we have a \define[local Moufang set!proper local Moufang set]{proper} local Moufang set.
\end{definition}

To simplify matter, we will henceforth fix a basis.

\begin{notation}\preenum
	\begin{itemize}
		\item We fix a basis $(0,\infty)$ in $X$\notatlink{0infty}.
		\item For any $x\nsim\infty$, we write $\alpha_x$\notatlink{alphax} for the unique element of $U_\infty$ mapping $0$ to $x$ (which exists by \ref{axiom:LM1}).
		\item For $x\nsim\infty$, we set $-x:=0\cdot\alpha_x^{-1}$.
	\end{itemize}
\end{notation}

As a consequence of this notation, we have $\alpha_0 = \id$ and $\alpha_x^{-1} = \alpha_{-x}$ for $x\nsim\infty$.

\subsection{Units}

In this paragraph, we will assume we have a local Moufang set with fixed basis $(0,\infty)$.

\begin{definition}
	We call $x\in X$ a \defin{unit} if $x\nsim0$ and $x\nsim\infty$. We denote $U_\infty^\times=\{\alpha_x\mid x\text{ a unit}\}$\notatlink{Uinftimes} and $U_\infty^\circ=U_\infty\backslash U_\infty^\times$\notatlink{Uinfcirc}.
\end{definition}

While the name `unit' may seem odd at this time, it originates from the projective local Moufang sets, where these `units' correspond to invertible elements (see \autoref{prop:inverse}). These elements will pop up everywhere in the theory of local Moufang sets. Note that the definition of a unit depends on the choice of basis.

There are a few other ways of characterizing the units, based on their corresponding elements of $U_\infty$.

\begin{proposition}\label{prop:unit}
	Let $\M$ be a local Moufang set, and $x\in X\setminus\class{\infty}$. Then the following are equivalent:
	\begin{romenumerate}
		\item $x$ is a unit\label{itm:unit1};
		\item $\induced{\alpha_x}$ does not fix $\class{0}$\label{itm:unit2};
		\item $\induced{\alpha_x}$ does not fix any element of $\class{X} \setminus \class{\infty}$\label{itm:unit3}.
	\end{romenumerate}
\end{proposition}
\begin{proof}\preenum
	\begin{itemize}[labelindent=5em,leftmargin=*]
	\item[{\ref{itm:unit1} $\Leftrightarrow$ \ref{itm:unit2}}.] We have $x = 0\alpha_x\in \class{0}\induced{\alpha_x}$, so
	\[x\text{ is a unit}\iff x\nsim0 \iff x\not\in \class{0}\iff \class{0}\neq\class{0} \cdot \induced{\alpha_x}\,.\]	
	\item[{\ref{itm:unit2} $\Leftrightarrow$ \ref{itm:unit3}}.] The induced permutation $\induced{\alpha_x}$ is contained in $U_{\class{\infty}}$. By \ref{axiom:LM1'}, this element fixes either all elements or no elements of $\class{X}\setminus\{\class{\infty}\}$.
    \qedhere
    \end{itemize}
\end{proof}

\begin{corollary}
	Let $\M$ be a local Moufang set, and $x\in X$ with $x\nsim\infty$. Then $x$ is a unit if and only if $-x$ is a unit.
\end{corollary}
\begin{proof}
	This follows from \autoref{prop:unit}\ref{itm:unit2}, as $\induced{\alpha_{-x}} = \induced{\alpha_x}^{-1}$.
\end{proof}

Using the equivalent descriptions of units, we can now give another description of $U_\infty^\times$ and $U_\infty^\circ$:

\begin{corollary}
	The following equalities hold:
	\begin{align*}
		U_\infty^\times &= \{g\in U_\infty \mid \induced{g}\text{\textnormal{ does not fix any element of $\class{X} \setminus \class{\infty}$}}\} \\
			&= \{g\in U_\infty \mid \induced{g}\neq\id\}\;, \\
		U_\infty^\circ &= \{g\in U_\infty \mid \induced{g}\text{\textnormal{ fixes $\class{X}$}}\}
			= \{g\in U_\infty \mid \induced{g}=\id\}\;.
	\end{align*}
\end{corollary}
\begin{proof}
	The first equality is immediate using \autoref{prop:unit}\ref{itm:unit3}. By \ref{axiom:LM1'}, $\induced{g}$ fixes one element of $\class{X} \setminus \class{\infty}$ if and only if it fixes all of $\class{X}$, which proves the second equality.
\end{proof}

As some of these descriptions do not depend on the choice of basis, we can use them in all root groups:

\begin{definition}
	Let $x\in X$, then we define 
	\begin{align*}
		U_x^\times &= \{g\in U_x \mid \induced{g}\neq\id\}\notatlink{Uxtimes} \\
		U_x^\circ &= U_x\backslash U_x^\times = \Ker(U_x\to U_{\class{x}})\notatlink{Uxcirc}\;.
	\end{align*}
\end{definition}

As for $\infty$, the elements of $U_x^\times$ are precisely those that do not fix $\class{y}$ for any $y\nsim x$. By these characterizations, we can see that $(U_x^\times)^g = U_{x\cdot g}^\times$ and $(U_x^\circ)^g = U_{x\cdot g}^\circ$.
	

\section{\texorpdfstring{$\mu$}{\textit{\textmugreek}}-maps and Hua maps}	
\subsection{\texorpdfstring{$\mu$}{\textit{\textmugreek}}-maps}

In a local Moufang set with a fixed basis $(0,\infty)$, we know by \autoref{prop:twotrans} that there must be an element of $G$ swapping these two points. We look at the double cosets $U_0\alpha_x U_0$, and find that there often is an element switching our two points.

\begin{proposition}\label{prop:mu_defin}
	For each unit $x\in X$, there is a unique element $\mu_x\in U_0\alpha_x U_0$ such that $0\mu_x=\infty$ and $\infty\mu_x=0$;
	it is called the \define[m-map@$\mu$-map]{$\mu$-map} corresponding to $x$.
	
	Moreover, $\mu_x\notatlink{mux} = g\alpha_x h\in U_0^\times\alpha_x U_0^\times$, where $g$ is the unique element of $U_0$ mapping $\infty$ to $-x$ and $h$ is the unique element of $U_0$ mapping $x$ to $\infty$.
\end{proposition}
\begin{proof}
	Let $g\alpha_xh$ be an element of $U_0\alpha_x U_0$, then the conditions translate to
	\[\infty=0g\alpha_xh=xh\quad\text{and}\quad \infty=0h^{-1}\alpha^{-1}_x g^{-1}=(-x) g^{-1}\,,\]
	so $g$ is the unique element of $U_0$ mapping $\infty$ to $-x$, and $h$ is the unique element of $U_0$ mapping $x$ to $\infty$. Since $x\nsim\infty$, both $g$ and $h$ are in $U_0^\times$, and since both $g$ and $h$ are unique, so is $\mu_x$.
\end{proof}

Some properties of the $\mu$-maps:

\begin{lemma}Let $e$\notatlink{e} be a unit and write $\tau = \mu_e$. Then
	\begin{romenumerate}
		\item $U_0^\tau = U_\infty$ and $U_\infty^\tau = U_0$.\label{itm:Utau}
		\item $G = \langle U_0,U_\infty\rangle = \langle U_\infty, \tau\rangle$.
		\item Let $x \in X$. Then $x$ is a unit if and only if $x\tau$ is a unit.
	\end{romenumerate}
\end{lemma}
\begin{proof}\preenum
	\begin{romenumerate}
		\item This follows immediately from \ref{axiom:LM2} and \autoref{prop:mu_defin}.
		\item The fact that $G = \langle U_0,U_\infty \rangle$ already follows from \autoref{prop:UxUy}, and the second equality $\langle U_0,U_\infty\rangle = \langle U_\infty, \tau\rangle$ then follows from~\ref{itm:Utau}.
		\item Since $\tau$ preserves the equivalence and switches $0$ and $\infty$, we have $x\sim0\iff x\tau\sim\infty$ and $x\sim\infty\iff x\tau\sim0$.
        \qedhere
	\end{romenumerate}
\end{proof}

Before, we noticed that two root groups were sufficient to generate all root groups. Now, we have seen that one root group and a $\mu$-map suffices.

\begin{notation}\preenum
	\begin{itemize}
		\item We henceforth fix a $\mu$-map and call it $\tau$\notatlink{tau} (one always exists as $\abs{\class{X}}>2$).
		\item For each $x\nsim\infty$, we define $\gamma_x := \alpha_x^\tau\in U_0$\notatlink{gammax}, which is the unique element of $U_0$ mapping $\infty$ to $x\tau$.
	\end{itemize}
\end{notation}

Many of the following identities will be crucial in later calculations.

\begin{lemma}\label{lem:mu}Let $x$ be a unit, and set $\til x := (-(x\tau^{-1}))\tau$\notatlink{tilx}.
Then
	\begin{romenumerate}
		\item $\mu_x$ does not depend on the choice of $\tau$; \label{itm:muindep}
		\item $\mu_x = \alpha_{(-x)\tau^{-1}}^\tau\,\alpha_x\,\alpha_{-(x\tau^{-1})}^\tau$; \label{itm:muform}
		\item $\mu_{-x} = \mu_x^{-1}$; \label{itm:muinv}
		\item $\mu_{x\tau} = \mu_{-x}^\tau$; \label{itm:mutau}
		\item $\mu_{x} = \alpha_x\alpha_{-(x\tau^{-1})}^\tau\,\alpha_{-\til x}$; \label{itm:muform2}
		\item $\til x = -((-x)\mu_x)$;\label{itm:tilmu}
		\item $\til x$ does not depend on the choice of $\tau$.
		\item $\mu_{-x} = \alpha_{-\til x}\mu_{-x}\alpha_x\mu_{-x}\alpha_{\til -x}$. \label{itm:muform3}
	\end{romenumerate}
\end{lemma}
\begin{proof}\preenum
	\begin{romenumerate}
		\item This follows from the definition of $\mu_x$.
		\item By \autoref{prop:mu_defin}, $\mu_x = g\alpha_x h$, where $g$ is the unique element of $U_0$ mapping $\infty$ to $-x$ and $h$ is the unique element of $U_0$ mapping $x$ to $\infty$. Hence $g = \gamma_{(-x)\tau^{-1}} = \alpha_{(-x)\tau^{-1}}^\tau$ and $h = \gamma_{x\tau^{-1}}^{-1} = \alpha_{x\tau^{-1}}^{-\tau} = \alpha_{-(x\tau^{-1})}^\tau$.
		\item As $\mu_x$ swaps $0$ and $\infty$, so does $\mu_x^{-1}\in U_0\alpha_x^{-1} U_0 = U_0\alpha_{-x} U_0$. Since $\mu_{-x}$ is the unique such element in $U_0\alpha_{-x} U_0$, we must have $\mu_x^{-1}=\mu_{-x}$.
		\item By \ref{itm:muindep}, we can use \ref{itm:muform} with $\tau^{-1}$ in place of $\tau$ for the right-hand side, so we get
		\begin{align*}
			&\mu_{x\tau} = \mu_{-x}^\tau \\
			&\iff
			\alpha_{(-(x\tau))\tau^{-1}}^\tau\,\alpha_{x\tau}\,\alpha_{-x}^\tau = 
			\bigl(\alpha_{x\tau}^{\tau^{-1}}\,\alpha_{-x}\,\alpha_{-((-x)\tau)}^{\tau^{-1}}\bigr)^\tau \\
			&\iff \alpha_{(-(x\tau))\tau^{-1}}^\tau\,\alpha_{x\tau}\,\alpha_{-x}^\tau = 
			\alpha_{x\tau}\,\alpha_{-x}^\tau\,\alpha_{-((-x)\tau)}\\
			&\iff \alpha_{-x}^{-\tau}\alpha_{x\tau}^{-1}\alpha_{(-(x\tau))\tau^{-1}}^\tau\,\alpha_{x\tau}\,\alpha_{-x}^\tau = \alpha_{-((-x)\tau)} \\
			&\iff \alpha_{-x}^{-\tau}\alpha_{x\tau}^{-1}\alpha_{-((-(x\tau))\tau^{-1})}^\tau\,\alpha_{x\tau}\,\alpha_{-x}^\tau = 
			\alpha_{(-x)\tau}\,.
		\end{align*}
		By \ref{axiom:LM2}, the left-hand side belongs to 
		\[U_0^{\alpha_{x\tau}\,\alpha_{-x}^\tau} = U_{0\cdot \alpha_{x\tau}\,\alpha_{-x}^\tau} = U_\infty\;,\]
		so the left-hand side is equal to $\alpha_y$ for
		\begin{align*}
			y &= 0\cdot\alpha_{-x}^{-\tau}\alpha_{x\tau}^{-1}\alpha_{-((-(x\tau))\tau^{-1})}^\tau\,\alpha_{x\tau}\,\alpha_{-x}^\tau \\
				&= -(x\tau)\cdot\tau^{-1}\alpha_{-((-(x\tau))\tau^{-1})}\tau\,\alpha_{x\tau}\,\alpha_{-x}^\tau \\
				&= \infty\cdot\alpha_{-x}^\tau = (-x)\tau .
		\end{align*}
		Hence the last of the equivalent equalities holds, and indeed $\mu_{x\tau} = \mu_{-x}^\tau$.
		\item By \ref{itm:muform}, we have
		\[\mu_{x\tau} = \alpha_{(-(x\tau))\tau^{-1}}^\tau\,\alpha_{x\tau}\,\alpha_{-x}^\tau\,,\]
		and hence by \ref{itm:mutau}, 
		\[\mu_x = \mu_{x\tau}^{-\tau^{-1}} = \alpha_{x}\,\alpha_{-x\tau}^{\tau^{-1}}\,\alpha_{-(-(x\tau))\tau^{-1}}\,.\]
		Since $\mu_x$ does not depend on the choice of $\tau$, we can replace $\tau$ by $\tau^{-1}$ in the right-hand side, which gives the required identity.
		\item When we apply both sides of the identity \ref{itm:muform2} to $-x$, we get
		\[(-x)\cdot\mu_x = (-x)\cdot\alpha_x\alpha_{-(x\tau^{-1})}^\tau\,\alpha_{-\til x} = 0\cdot\alpha_{-\til x} = -\til x\,\]
		which gives $\til x = - ((-x)\mu_x)$.
		\item Since $\mu_x$ does not depend on the choice of $\tau$, and we have just shown $\til x = - ((-x)\mu_x)$, we conclude that $\til x$ does not depend on the choice of $\tau$ either.
		\item The equation \ref{itm:muform2} does not depend on the choice of $\tau$. If we substitute $-x$ for $x$ and $\mu_{-x}$ for $\tau$, then we get, using \ref{itm:tilmu},
		\[\mu_{-x} = \alpha_{-x}\mu_x\alpha_{-((-x)\mu_x)}\mu_{-x}\alpha_{-\til -x} = \alpha_{-x}\mu_x\alpha_{\til x}\mu_{-x}\alpha_{-\til -x}\,.\]
		Moving all terms except $\mu_{-x}$ from the right hand side to the left hand side gives the result.\qedhere
	\end{romenumerate}
\end{proof}

\begin{proposition}\label{prop:sumform}
	Let $x,y\in X$ be units such that $x\nsim y$. Then 
	\[z := x\tau^{-1}\alpha_{-(y\tau^{-1})}\tau\]
	is independent on the choice of $\tau$. Furthermore, 
	\[ z = x\alpha_{-y}\mu_y\alpha_{\til y} \text{ and } \til z = y\alpha_{-x}\mu_x\alpha_{\til x}\;. \]
	Finally, $\mu_y\mu_z\mu_{-x} = \mu_{y\alpha_{-x}}$.
\end{proposition}
\begin{proof}
	By \autoref{lem:mu}\ref{itm:muform2} we get $z = x\alpha_{-(y\tau^{-1})}^\tau = x\alpha_{-y}\mu_{y}\alpha_{\til y}$, so it does not depend on the choice of $\tau$. Now we have $z = 0\cdot\alpha_{x\tau^{-1}}\alpha_{-(y\tau^{-1})}\tau$, so
    \[ \til z = 0\cdot(\alpha_{x\tau^{-1}}\alpha_{-(y\tau^{-1})})^{-1}\tau = 0\cdot\alpha_{y\tau^{-1}}\alpha_{-(x\tau^{-1})}\tau . \]
    This coincides with our definition of $z$ with $x$ and $y$ interchanged, so $\til z = y\alpha_{-x}\mu_{x}\alpha_{\til x}$.
    
     For the final equality, we repeatedly use \autoref{lem:mu}\ref{itm:muform2}.
	\begin{align*}
		\mu_z 	&= \alpha_z\alpha_{-(z\tau^{-1})}^{\tau}\alpha_{-\til z} \\
			&= \alpha_z\alpha_{x\tau^{-1}\alpha_{-(y\tau^{-1})}}^{-\tau}\alpha_{-\til z} \\
			&= \alpha_z(\alpha_{x\tau^{-1}}\alpha_{-(y\tau^{-1})})^{-\tau}\alpha_{-\til z} \\
			&= \alpha_z(\alpha_{-(y\tau^{-1})}^{-1}\alpha_{-(x\tau^{-1})})^{\tau}\alpha_{-\til z} \\
			&= \alpha_z\alpha_{-(y\tau^{-1})}^{-\tau}\alpha_{-(x\tau^{-1})}^{\tau}\alpha_{-\til z} \\
			&= \alpha_z\alpha_{-\til y}\mu_{-y}\alpha_y\alpha_{-x}\mu_x\alpha_{\til x}\alpha_{-\til z} \\
			&= \alpha_{x\alpha_{-y}\mu_y\alpha_{\til y}}\alpha_{-\til y}\mu_{-y}\alpha_y\alpha_{-x}\mu_x(\alpha_{y\alpha_{-x}\mu_x\alpha_{\til x}}\alpha_{-\til x})^{-1} \\
			&= \alpha_{x\alpha_{-y}\mu_y}\mu_{-y}\alpha_y\alpha_{-x}\mu_x\alpha_{-(y\alpha_{-x}\mu_x)}
	\intertext{Hence, again using \autoref{lem:mu}\ref{itm:muform2} but replacing $\tau$ by $\mu_{-x}$ and $\mu_{-y}$, we get}
		\mu_y\mu_z\mu_{-x} &= \mu_y\alpha_{x\alpha_{-y}\mu_y}\mu_{-y}\alpha_y\alpha_{-x}\mu_x\alpha_{-(y\alpha_{-x}\mu_x)}\mu_{-x} \\
			&= \alpha_{x\alpha_{-y}\mu_y}^{\mu_{-y}}\alpha_y\alpha_{-x}\alpha_{-(y\alpha_{-x}\mu_x)}^{\mu_{-x}} \\
			&= \alpha_{-\til(x\alpha_{-y})}\mu_{-(x\alpha_{-y})}\alpha_{x\alpha_{-y}}\alpha_y\alpha_{-x}\alpha_{-(y\alpha_{-x})}\mu_{y\alpha_{-x}}\alpha_{\til(y\alpha_{-x})} \\
			&= \alpha_{-\til(x\alpha_{-y})}\mu_{-(x\alpha_{-y})}\alpha_{x\alpha_{-y}}\mu_{-(x\alpha_{-y})}\alpha_{\til-(x\alpha_{-y})} \\
			&= \mu_{-(x\alpha_{-y})} = \mu_{y\alpha_{-x}}\;,
	\end{align*}
	using \autoref{lem:mu}\ref{itm:muform3} for the second to last equality.
\end{proof}

\subsection{Hua maps}

As $\mu$-maps swap $0$ and $\infty$, products of an even number of $\mu$-maps fix $0$ and $\infty$. The Hua maps are a specific case of these maps.

\begin{definition}
	The \define{Hua map} $h_{x,\tau}$ corresponding to a unit $x$ is the element
	\[h_{x,\tau}\notatlink{hxtau} = \tau\mu_x = \tau\alpha_x\tau^{-1}\alpha_{-(x\tau^{-1})}\tau\,\alpha_{-\til x}\in G\,.\]
\end{definition}

\begin{remark}
	Since the $\mu$-maps did not depend on the choice of $\tau$, the Hua maps do. This is made clear by the inclusion of $\tau$ in the notation $h_{x,\tau}$. When it is clear (or irrelevant) which $\tau$ is used, we will omit this addition and simply write $h_x$\notatlink{hx}.
\end{remark}

Some basic properties of the Hua maps are the following:
\begin{lemma}\label{lem:hua} Let $x,y \in X$ be units. Then
	\begin{multicols}{2}
	\begin{romenumerate}
		\item $h_{x,\tau^{-1}} = h_{x\tau,\tau}^{-1}$;
		\item $\mu_{xh_y} = \mu_x^{h_y}$;
		\item $h_{x\tau} = h_{-x}^\tau$;\label{itm:huatau}
		\item $h_{x h_y} = h_{-y}h_{x\tau}^{-1}h_y$\label{itm:huahua}.
	\end{romenumerate}
	\end{multicols}
\end{lemma}
\begin{proof}\preenum
	\begin{romenumerate}
		\item We have $h_{x,\tau^{-1}} = h_{x\tau,\tau}^{-1}$ if and only if $\tau^{-1}\mu_x = (\tau\mu_{x\tau})^{-1}$, which holds by \autoref{lem:mu}\ref{itm:mutau}.
		\item Applying \autoref{lem:mu}\ref{itm:mutau} twice (with $\tau$ and $\mu_y$), we get
			\[\mu_{xh_y} = \mu_{x\tau\mu_y} = \mu_{x\tau}^{-\mu_y} = (\mu_x^{-\tau})^{-\mu_y} = \mu_x^{\tau\mu_y} = \mu_x^{h_y}\,.\]
		\item Using \autoref{lem:mu}\ref{itm:mutau}, we get
			\[h_{x\tau} = \tau\mu_{x\tau} = \tau\mu_{-x}^\tau = (\tau\mu_{-x})^\tau = h_{-x}^\tau\,.\]
		\item Using \autoref{lem:mu}\ref{itm:mutau} twice, and inserting $\tau^{-1}\tau$, we get
			\begin{align*}
				h_{x h_y} &= \tau\mu_{x\tau\mu_y} = \tau\mu_{-y}\mu^{-1}_{x\tau}\mu_y \\
					&= \tau\mu_{-y}\,\mu^{-1}_{x\tau}\tau^{-1}\,\tau\mu_y = h_{-y}h_{x\tau}^{-1}h_y\,.\qedhere
			\end{align*}
	\end{romenumerate}
\end{proof}

The action of Hua maps on $U_\infty$ by conjugation behaves well with respect to the action on $X$:
\begin{lemma}\label{lem:huaAut}
	Let $x$ be a unit. Then for any $y\in X\setminus\overline{\infty}$, we have $\alpha_y^{h_x} = \alpha_{yh_x}$.
	In particular, $(\alpha_y\alpha_z)^{h_x} = \alpha_{yh_x}\alpha_{zh_x}$ for all $y,z\in X\setminus\overline{\infty}$.
\end{lemma}
\begin{proof}
	Since Hua maps normalize $U_\infty$, we have $\alpha_y^{h_x}\in U_\infty$. Since 
	\[0\alpha_y^{h_x} = 0h_x^{-1}\alpha_yh_x = yh_x\;,\]
	the first equality holds. The second identity now follows immediately.
\end{proof}

In particular, the Hua maps induce automorphisms of $U_\infty$:
\begin{corollary}\label{cor:huaaut}
	For any unit $x$, permutation of $U_\infty$ defined by $\alpha_a\mapsto\alpha_{ah_x}$ is an automorphism.
\end{corollary}

Finally, we consider the group generated by Hua maps. To be consistent with a later definition, we define it as the group of products of an even amount of $\mu$-maps:

\begin{definition}
	The \define{Hua subgroup} is
	\[H := \langle \mu_x\mu_y\mid x,y\text{ units}\rangle\,.\notatlink{H}\]
\end{definition}
As $\tau$ is a $\mu$-map, we also have $H = \langle h_x\mid x\text{ a unit}\rangle$. Note that $H\leq G_{0,\infty}$, where $G_{0,\infty}$ is the two-point stabilizer of $0$ and $\infty$. In fact, $H=G_{0,\infty}$, as we will show in \autoref{thm:hua2pt}.


\section{The Hua subgroup and the two-point stabilizer}	
\subsection{Quasi-invertibility}

We define quasi-invertibility and the left and right quasi-inverse similar to \cite[\textsection 4]{LoosRogdiv}. These notions, and the identity that follows from them, will be important in the next two subsections.

\begin{definition}
	A couple $(x,y)\in X^2$ s.t.\ $x\nsim\infty$ and $y\nsim\infty$ is \define{quasi-invertible} if one of the following is satisfied: $x\tau \nsim -y$, $x\sim 0$ or $y\sim 0$.
\end{definition}

If $(x,y)$ is quasi-invertible, we can associate two other elements to the pair:

\begin{definition}
	Let $(x,y)$ be quasi-invertible, then we define the \define{left quasi-inverse} and \define{right quasi-inverse} as
	\[{}^xy\notatlink{xy} = (-y)\cdot\alpha_{-x}^\tau\qquad\text{ and }\qquad x^y\notatlink{xy2} = -(x\cdot\alpha_y^{\tau^{-1}})\;.\]
\end{definition}

Note that the condition for quasi-invertibility ensures that the left and right quasi-inverse are not in $\class{\infty}$. Furthermore, $x\sim 0\Leftrightarrow x^y\sim0$ and $y\sim 0\Leftrightarrow {}^xy\sim0$.

\begin{proposition}\label{prop:quasi_inv}
	Let $(x,y)$ be quasi-invertible (so $x\nsim\infty\nsim y$) with $y\nsim0$, then
	\[\alpha_{{}^x\!y}\,\alpha_x^\tau\,\alpha_y\,\alpha_{x^y}^\tau = \mu_{{}^x\!y}\,\mu_y\;.\]
\end{proposition}
\begin{proof}
	By the observation above, ${}^xy\nsim0$, so the right-hand side is defined. By \autoref{lem:mu}\ref{itm:muform2}, we have $\mu_{{}^x\!y} = \alpha_{{}^x\!y}\,\alpha_{-({}^x\!y\tau^{-1})}^\tau\,\alpha_{-\til{}^x\!y}$. Furthermore, by the definition of ${}^xy$, we have 
	\[-({}^x\!y\tau^{-1})=-(0\cdot\alpha_{(-y)\tau^{-1}}\alpha_{-x}) = 0\cdot\alpha_x\alpha_{-(-y)\tau^{-1}}\;,\]
	so $\alpha_{-({}^x\!y\tau^{-1})} = \alpha_x\alpha_{-(-y)\tau^{-1}}$. Plugging these into our equality gives
	\begin{align*}
		&\alpha_{{}^x\!y}\,\alpha_x^\tau\,\alpha_y\,\alpha_{x^y}^\tau = \mu_{{}^x\!y}\,\mu_y \\
		&\iff \alpha_y\,\alpha_{x^y}^\tau = \alpha_{-(-y)\tau^{-1}}^\tau\,\alpha_{-\til{}^x\!y}\,\mu_y \\
		&\iff \alpha_{\til{}^x\!y}\alpha_{(-y)\tau^{-1}}^\tau\alpha_y\,\alpha_{x^y}^\tau= \mu_y\;.
	\end{align*}
	Now we use \autoref{lem:mu}\ref{itm:muform} to find $\mu_y = \alpha_{(-y)\tau^{-1}}^\tau\alpha_y\alpha_{-(y\tau^{-1})}^\tau$, changing the identity to prove to
	\begin{align*}
		&\alpha_{{}^x\!y}\,\alpha_x^\tau\,\alpha_y\,\alpha_{x^y}^\tau = \mu_{{}^x\!y}\,\mu_y \\
		&\iff \alpha_{\til{}^x\!y}\alpha_{(-y)\tau^{-1}}^\tau\alpha_y\,\alpha_{x^y}^\tau= \alpha_{(-y)\tau^{-1}}^\tau\alpha_y\alpha_{-(y\tau^{-1})}^\tau \\ 
		&\iff \alpha_{\til{}^x\!y}\alpha_{(-y)\tau^{-1}}^\tau\alpha_y\,\alpha_{x^y}^\tau\alpha_{y\tau^{-1}}^\tau\alpha_y^{-1}\alpha_{(-y)\tau^{-1}}^{-\tau} = \id_{X} \;.
	\end{align*}
	Now we have 
	\[\alpha_{(-y)\tau^{-1}}^\tau\alpha_y\,\alpha_{x^y}^\tau\alpha_{y\tau^{-1}}^\tau\alpha_y^{-1}\alpha_{(-y)\tau^{-1}}^{-\tau}\in U_0^{\alpha_y^{-1}\alpha_{(-y)\tau^{-1}}^{-\tau}} = U_{0\cdot\alpha_y^{-1}\alpha_{(-y)\tau^{-1}}^{-\tau}}\]
	by \ref{axiom:LM2}, and
	\begin{align*}
		0\cdot\alpha_y^{-1}\alpha_{(-y)\tau^{-1}}^{-\tau} &= (-y)\cdot\tau^{-1}\alpha_{-(-y)\tau^{-1}}\tau \\
			&= (-y)\tau^{-1}\cdot\alpha_{-(-y)\tau^{-1}}\tau = 0\tau = \infty\;,
	\end{align*}
	so the left-hand side of our last identity to prove is an element of $U_\infty$. To prove it is the identity, it is now sufficient to prove that it maps $0$ to $0$, by \ref{axiom:LM1}. Note first that $\til{}^x\!y = x\alpha_{(-y)\tau^{-1}}^{-1}\tau$. We have
	\begin{align*}
		& 0\cdot\alpha_{\til{}^x\!y}\alpha_{(-y)\tau^{-1}}^\tau\alpha_y\,\alpha_{x^y}^\tau\alpha_{y\tau^{-1}}^\tau\alpha_y^{-1}\alpha_{(-y)\tau^{-1}}^{-\tau} \\
		&= x\alpha_{(-y)\tau^{-1}}^{-1}\tau\tau^{-1}\alpha_{(-y)\tau^{-1}}\tau\alpha_y\,\alpha_{x^y}^\tau\alpha_{y\tau^{-1}}^\tau\alpha_y^{-1}\alpha_{(-y)\tau^{-1}}^{-\tau} \\
		&= x\cdot\alpha_y^{\tau^{-1}}\alpha_{x^y}\tau\alpha_{y\tau^{-1}}^\tau\alpha_y^{-1}\alpha_{(-y)\tau^{-1}}^{-\tau} \\
		&= (-x^y)\cdot\alpha_{x^y}\tau\alpha_{y\tau^{-1}}^\tau\alpha_y^{-1}\alpha_{(-y)\tau^{-1}}^{-\tau} \\
		&= 0\cdot\alpha_{y\tau^{-1}}\tau\alpha_{-y}\alpha_{-(-y)\tau^{-1}}^\tau = y\cdot\alpha_{-y}\alpha_{-(-y)\tau^{-1}}^\tau \\
		&= 0\cdot\alpha_{-(-y)\tau^{-1}}^\tau = 0\;,
	\end{align*}
	so the identity holds!
\end{proof}

\begin{remark}
	We could prove a similar identity to the one in the above lemma, in the case where $x\nsim 0$. In this case, we have
	\[\alpha_{{}^x\!y}\,\alpha_x^\tau\,\alpha_y\,\alpha_{x^y}^\tau = \mu_{x}\,\mu_{x^y}\;.\]
\end{remark}

\subsection{A Bruhat decomposition of \texorpdfstring{$G$}{G}}

By refining the argument in \autoref{prop:twotrans}, we will be able to obtain decompositions of the little projective group $G$ which resemble the Bruhat decomposition. This is based on a case distinction depending on where the basis $(0,\infty)$ is mapped to by a given element.

\begin{proposition}
	The little projective group $G$ can be split into a disjoint union 
	\[G = U_0 G_{0,\infty} U_\infty \,\cup\, U_0 G_{0,\infty} \tau U_0^\circ\;.\]
	Furthermore, the decomposition of an element of $G$ is unique in each of the cases.
\end{proposition}
\begin{proof}
	Let $g$ be in $G$ and denote $(x,y) = (0,\infty)\cdot g$. Now there are two mutually exclusive cases:
	\begin{itemize}[labelindent=3em, leftmargin=*]
		\item[$x\sim \infty$:] In this case there is a unique $u_0\in U_0$ such that $x\cdot u_0 = \infty$. Since $\induced{u_0}$ fixes $\class{\infty}$, we have $u_0\in U_0^\circ$. We get 
		\[(x,y)\cdot u_0\tau^{-1} = (0,y')\]
		for some $y'\nsim 0$. Hence there is a unique $u'_0\in U_0$ such that $y'\cdot u'_0 = \infty$, so we get
		\begin{align*}
			(0,\infty)\cdot gu_0\tau^{-1} u'_0 &= (x,y)\cdot u_0\tau^{-1} u'_0 \\
				&= (0,y')\cdot u'_0 = (0,\infty)\;,
		\end{align*}
		so $gu_0\tau^{-1} u'_0=h \in G_{0,\infty}$. Hence 
		\[g =  u'^{-h^{-1}}_0h\tau u_0^{-1} \in U_0 G_{0,\infty}\tau U_0^\circ\;.\]
		Note that since $u_0$ and $ u'_0$ are unique, so is $h$, and hence the entire decomposition.
		\item[$x\nsim \infty$:] By \ref{axiom:LM1}, we find a unique $u_\infty\in U_\infty$ such that $x\cdot u_\infty = 0$, so $(x,y)\cdot u_\infty = (0,y')$. Next, we find a unique $u_0\in U_0$ such that $y'\cdot u_0 = \infty$, so
		\[(0,\infty)\cdot g u_\infty u_0 = (x,y)\cdot u_\infty u_0 = (0,y')\cdot u_0 = (0,\infty)\;.\]
		This means $g u_\infty u_0 = h\in G_{0,\infty}$, so 
		\[g = u_0^{-h^{-1}} h u^{-1}_\infty\in U_0 G_{0,\infty} U_\infty\;.\]
		Again, since $u_0$ and $u_\infty$ are unique, so is $h$ and the entire decomposition.
		\qedhere
	\end{itemize}
\end{proof}

By making similar case distinctions, one can get different decompositions, for example
\[ G = U_\infty G_{0,\infty} U_0^\circ \,\cup\, U_\infty G_{0,\infty} \tau U_\infty\,.\]
For this decomposition, we would separate two cases: $\infty\cdot g \sim\infty$ or $\infty\cdot g \nsim\infty$. In particular, we can check that if we take $g\in U_0^\circ U_\infty$, we always end up in the first component $U_\infty G_{0,\infty} U_0^\circ$, i.e.\@
\begin{equation}\label{eq:UoU}
    U_0^\circ U_\infty \subset U_\infty G_{0,\infty} U_0^\circ\,.
\end{equation}
In \autoref{ss:hua} below, we will show that $G_{0,\infty} = H$.
A first step towards this consists of showing that the inclusion~\eqref{eq:UoU} holds with $G_{0,\infty}$ replaced by $H$.
To obtain this, we will need the notion of quasi-invertibility we introduced.

\begin{proposition}\label{prop:huaincl}
	In a local Moufang set, $U_0^\circ U_\infty \subset U_\infty H U_0^\circ$.
\end{proposition}
\begin{proof}
	We will show this in two steps. First, we will prove that
	\begin{equation}\label{eq:half}
		U_0^\circ U_\infty^\times \subset U_\infty^\times H U_0^\circ\,,
	\end{equation}
	and from that we will deduce the general inclusion.
	
	So take an arbitrary element of $U_0^\circ U_\infty^\times$, and denote it by $\alpha_x^\tau\alpha_y$. Then $x\sim0$ and $\infty\nsim y\nsim0$, so $(x,y)$ is quasi-invertible. By \autoref{prop:quasi_inv}, we have
	\[\alpha_x^\tau\alpha_y = \alpha_{-\prescript{x}{}y}\,\mu_{\prescript{x}{}y}\,\mu_y\,\alpha_{-x^y}^\tau \in U^\times_\infty H U_0^\circ\,,\]
	since $0\nsim{-}\prescript{x}{}y\nsim\infty$ and $-x^y\sim 0$. This proves~\eqref{eq:half}.
	
	For the general inclusion, we use the fact that if we have $\alpha_y\in U_\infty^\circ$, we can split it up as $\alpha_y = \alpha_{y'}\alpha_{e}$, for units $y'$ and $e$;
	indeed, units exist, and if $e$ is a unit, then $\alpha_y\alpha_e^{-1}$ does not fix $\class{0}$, so $y'$ is also a unit. So we get
	\begin{align*}
		U_0^\circ U_\infty^\circ &\subset U_0^\circ U_\infty^\times U_\infty^\times \subset U_\infty^\times H U_0^\circ U_\infty^\times \\
					&\subset  U_\infty^\times H  U_\infty^\times H U_0^\circ = U_\infty^\times U_\infty^\times H U_0^\circ\subset U_\infty H U_0^\circ\,,
	\end{align*}
	where we have used~\eqref{eq:half} twice, as well as $HU_\infty^\times = U_\infty^\times H$.
	
	Putting these two inclusions together, we get
	\[U_0^\circ U_\infty = U_0^\circ U_\infty^\times\cup U_0^\circ U_\infty^\circ \subset U_\infty^\times H U_0^\circ\cup U_\infty H U_0^\circ = U_\infty H U_0^\circ\,.\qedhere\]
\end{proof}

\subsection{The Hua subgroup is the two-point stabilizer}\label{ss:hua}

In the case of Moufang sets, one can use the Bruhat decomposition to prove that $G_{0,\infty} = H$, and as a consequence that the point stabilizer $G_0 = U_0 H$.
In the case of local Moufang sets, the additional $U_0^\circ$ in the decomposition seems to cause further difficulties in the proof.
However, using \autoref{prop:huaincl}, we will be able to resolve these difficulties,
and we will again be able to prove that the Hua subgroup coincides with the full two-point stabilizer of $0$ and $\infty$ in $G$.

\begin{theorem}\label{thm:hua2pt}
	For a local Moufang set $\M$, we have the decomposition 
	\[G = U_0 H U_\infty \,\cup\, U_0 H \tau U_0^\circ\]
	and hence $G_0 = U_0 H$ and $H = G_{0,\infty}$\notatlink{H}.
\end{theorem}
\begin{proof}
    Let $K = U_0 H = HU_0$. We will examine the set 
    \[Q = K U_\infty\cup K\tau U_0^\circ\;;\]
    our aim is to prove that it equals $G$.
    More precisely, we will show that $Q\langle U_\infty,\tau\rangle = QG\subset Q$, from which $Q=G$ will follow immediately.
    We will do this for each of the two pieces of $Q$ separately.\vspace{1ex}
    \begin{nrenumerate}
        \item
            We will first show that $K U_\infty G\subset Q$.
            It is immediate that 
            \[K U_\infty U_\infty = K U_\infty\subset Q\;,\]
            so all we need to prove is that $K U_\infty\tau\subset Q$, or equivalently, that $K \alpha_a\tau\subset Q$ for all $\alpha_a\in U_\infty$. Assume first that $\alpha_a\in U_\infty^\circ$; then $\alpha_a^\tau\in U_0^\circ$, so 
            \[K \alpha_a\tau = K \tau\alpha_a^\tau\subset K\tau U_0^\circ\subset Q\,.\]
            Finally, when $\alpha_a\in U_\infty\backslash U_\infty^\circ = U_\infty^\times$, we have 
            \[\mu_{a\tau} \alpha_{a}^\tau = \alpha_{(-a\tau)\tau^{-1}}^\tau\,\alpha_{a\tau}\]
            by \autoref{lem:mu}\ref{itm:muform}. Since $\mu_{a\tau}\tau^{-1}\in H\subset K$, this implies
            \[K \alpha_a\tau = K \tau\alpha_a^\tau = K\mu_{a\tau}\tau^{-1}\tau\alpha_a^\tau = K\alpha_{(-a\tau)\tau^{-1}}^\tau\,\alpha_{a\tau}\subset K U_\infty\,.\]
        \item
            We will now show that $K \tau U_0^\circ G\subset Q$. We have 
            \[K \tau U_0^\circ \tau = K \tau^2 U_\infty^\circ\subset KHU_\infty = K U_\infty\subset Q\;,\]
            so we need to prove that $K \tau U_0^\circ U_\infty\subset Q$.
            We now invoke \autoref{prop:huaincl}, and we get
            \[K\tau U_0^\circ U_\infty\subset K\tau U_\infty H U_0^\circ = KU_0H\tau U_0^\circ = K\tau U_0^\circ\subset Q.\]
    \end{nrenumerate}
	We conclude that $QG\subset Q$, and hence $Q=G$ as claimed.
	
	We now know that $G = U_0 H U_\infty\cup U_0 H\tau U_0^\circ$. If we take an element $g$ in the point stabilizer $G_0$ and look at the two possibilities of decomposing $g$, we get $g\in U_0 H$, so $G_0 = U_0 H$. If we assume in addition that $g$~fixes $\infty$, we also see that the factor in $U_0$ must be trivial, hence $G_{0,\infty} = H$.
\end{proof}

We can use this to give a different definition of improper local Moufang sets:
\begin{corollary}
	For a local Moufang set $\M$, the following are equivalent:
	\begin{romenumerate}
		\item $\M$ is improper;
		\item $H$ is trivial;
		\item $\mu_x=\mu_y$ for all units $x$ and $y$.
	\end{romenumerate}
\end{corollary}

	\chapter{Main construction}\label{chap:chap3_constr}
	In the theory of Moufang sets, T.~De Medts and R.~Weiss found a way to construct Moufang sets using a minimal amount of information. We generalize this construction to local Moufang sets, requiring a set with equivalence relation, a group and one permutation, satisfying some basic properties. We generate the data for a local pre-Moufang set, but without extra assumptions, this is not necessarily a local Moufang set. In \autoref{cor:construction_equivalentconditions}, we prove some useful necessary and sufficient conditions to ensure we get a local Moufang set.

\section{How to construct a local Moufang set}

We already know that, if we have a local Moufang set, then $G = \langle U_\infty, \tau\rangle$. We now consider the converse: given a group $U$ and a permutation $\tau$, both acting faithfully on a set with an equivalence relation, we will try to construct a local Moufang set.
Of course, we will need additional conditions on $U$ and $\tau$.

\begin{importantconstruction}\label{constr:MUtau}
	The construction requires some data to start with. We need
	\begin{nrenumerate}
		\item a set with an equivalence relation $(X,\sim)$, such that $\abs{\class{X}}>2$;
		\item a group $U\leq\Sym(X,\sim)$, and an element $\tau\in\Sym(X,\sim)$.
	\end{nrenumerate}
	The action of $U$ and $\tau$ will have to be sufficiently nice in order to do the construction.
	\begin{manualenumerate}[label=\textnormal{(C\arabic*)},labelwidth=\widthof{(C1')}]
		\item $U$ has a fixed point we call $\infty$, and acts sharply transitively on $X\setminus\class{\infty}$\label{axiom:C1}.
		\item[\mylabel{axiom:C1'}{\textnormal{(C1')}}] The induced action of $U$ on $\class{X}$ is sharply transitive on $\class{X}\setminus\{\class{\infty}\}$.
		\item $\infty\tau\nsim\infty$ and $\infty\tau^2=\infty$. We write $0:=\infty\tau$\label{axiom:C2}.
	\end{manualenumerate}
	In this construction, we now define the following objects:
	\begin{itemize}
		\item For $x\nsim\infty$, we let $\alpha_x$\notatlink{alphax} be the unique element of $U$ mapping $0$ to $x$ (this exists by \ref{axiom:C1}).
		\item For $x\nsim\infty$, we write $\gamma_x:=\alpha_x^\tau$\notatlink{gammax}, which then maps $\infty$ to $x\tau$.
		\item We set $U_\infty := U$ and $U_0:=U_\infty^\tau$. The other root groups are defined as
		\[U_x:=U_0^{\alpha_x} \text{ for $x\nsim\infty$,}\qquad U_x:=U_\infty^{\gamma_{x\tau^{-1}}} \ \text{for $x\sim\infty$}.\notatlink{Ux}\]
	\end{itemize}
	This gives us all the data that is needed for a local Moufang set; we denote the result of this construction by $\M(U,\tau)$\notatlink{M(U,tau)}.
\end{importantconstruction}

This construction does not usually give a local Moufang set, but it always is a local pre-Moufang set.

\begin{proposition}\label{prop:constr}
	The construction $\M(U,\tau)$ is a local pre-Moufang set.
\end{proposition}
\begin{proof}
	We first show \ref{axiom:LM0}. Let $x\sim y$ and suppose they are not equivalent to $\infty$. Then $U_x^{\alpha_x^{-1}\alpha_y} = U_y$ by definition. Now $\alpha_x^{-1}\alpha_y$ is in $U$, and fixes $\class{x}$. By \ref{axiom:C1'}, this implies that $\induced{\alpha_x^{-1}\alpha_y}=\id$, so the induced action of $U_x$ is the same as that of $U_y$.
	If $x\sim \infty$, we have $U_x = U_\infty^{\gamma_{x\tau^{-1}}}$. We now want to see what the induced action of $\gamma_{x\tau^{-1}}$ is. We have
	\[\induced{\gamma_{x\tau^{-1}}} = \induced{\tau^{-1}}\,\induced{\alpha_{x\tau^{-1}}}\,\induced{\tau}\;.\]
	Remark that $x\tau^{-1}\sim 0$, so 
	\[0\alpha_{x\tau^{-1}} = x\tau^{-1} \sim 0\;,\]
	hence $\alpha_{x\tau^{-1}}$ fixes $\class{0}$. By \ref{axiom:C1'}, $\induced{\alpha_{x\tau^{-1}}}=\id$, so
	\[\induced{\gamma_{x\tau^{-1}}} = \induced{\tau^{-1}}\induced{\tau} = \id\;,\]
	and hence $\induced{U_x} = \induced{U_\infty}$. Now for all $x\sim y\sim \infty$, we have $\induced{U_x} = \induced{U_\infty} = \induced{U_y}$.
	
	By \ref{axiom:C1}, \ref{axiom:LM1} holds for $U_\infty$. Now, by definition, any other $U_x$ is equal to $U_\infty^g=g^{-1}U_\infty g$ for some $g$ with $\infty g = x$.
	It follows that each $U_x$ fixes $x$ and acts sharply transitively on $(X\setminus\class{\infty})g = X\setminus\class{x}$,
	so \ref{axiom:LM1} holds for all root groups.
	
	Similarly, \ref{axiom:LM1'} holds for $U_{\class{\infty}}$ because of \ref{axiom:C1'} and since $U_{\class{\infty}}$ fixes $\class{\infty}$ because $U_\infty$ fixes $\infty$.
	As before, any $U_x$ is equal to $U_\infty^g=g^{-1}U_\infty g$ for some $g$ with $\infty g = x$, so $U_{\class{x}}$ is the induced action of $U_\infty^g$ on $\class{X}$.
	This implies that $U_{\class{x}}$ fixes $\class{\infty}g = \class{x}$ and acts sharply transitively on $(\class{X}\setminus\{\class{\infty}\})g = \class{X}\setminus\{\class{x}\}$.
\end{proof}

In a local pre-Moufang set given by the construction, we now introduce some notation that we already know from the theory developed in \autoref{chap:chap2_definitions}.

\begin{notation}\preenum
	\begin{itemize}
		\item We call $x\in X$ a unit if $x\nsim0$ and $x\nsim\infty$.
		\item For $x\nsim\infty$, we set $-x:=0\alpha_x^{-1}$.
		\item For a unit $x$, we define the $\mu$-map to be 
			\[\mu_x:=\gamma_{(-x)\tau^{-1}}\alpha_x\gamma_{-(x\tau^{-1})}\;.\notatlink{mux}\]
		\item For a unit $x$, we define the Hua map 
			\[h_x:=\tau\alpha_x\tau^{-1}\alpha_{-(x\tau^{-1})}\tau\alpha_{-(-(x\tau^{-1}))\tau}\;.\notatlink{hx}\]
		\item We set $H := \langle \mu_x\mu_y\mid x,y\text{ units}\rangle$\notatlink{H}.
		\item We set $G := \langle U_x\mid x\in X\rangle$\notatlink{G}.
		\item We set $U_0^\circ = \{\gamma_x\in U_0\mid x\sim0\}$.
	\end{itemize}
\end{notation}

\begin{remark}\label{remark:mu_hua_construction}
	The $\mu$-maps in this context are identical to the $\mu$-maps as in \autoref{prop:mu_defin}. As a consequence, \autoref{lem:mu}\ref{itm:muform} and \ref{itm:muinv} still hold for these $\mu$-maps. 
	
	On the other hand, we cannot ensure $h_x = \tau\mu_x$, as this requires \autoref{lem:mu}\ref{itm:muform2}. The proof of this identity requires \ref{axiom:LM2} in the proof of \autoref{lem:mu}\ref{itm:mutau}, though it can be shown using weaker assumptions. This is precisely what we will do in the next section to characterize when $\M(U,\tau)$ gives rise to a local Moufang set.
\end{remark}

\section{Conditions to satisfy \texorpdfstring{\ref{axiom:LM2}}{(LM2)}}

To ensure that \autoref{constr:MUtau} gives a local Moufang set, we will need more information about the action of the Hua maps.
We will first prove a few lemmas.
Throughout this section, we let $\M(U,\tau)$ be as in \autoref{constr:MUtau}.

\begin{lemma}
	Let $x\in X$ be a unit. Then the following are equivalent:
	\begin{romenumerate}
		\item $U_\infty^{h_x} = U_\infty$; \label{lem:loceq2}
		\item $U_\infty^{\gamma_{x\tau^{-1}}} = U_x$. \label{lem:loceq3}
		\item $U_0^{\mu_x} = U_\infty$; \label{lem:loceq4}
	\end{romenumerate}
\end{lemma}
\begin{proof}\preenum
	\begin{itemize}[labelindent=4em, leftmargin=*]
		\item[\ref{lem:loceq2}$\Leftrightarrow$\ref{lem:loceq3}.] We have
			\begin{align*}
				&U_\infty^{\tau\alpha_x\tau^{-1}\alpha_{-(x\tau^{-1})}\tau\alpha_{-(-(x\tau^{-1}))\tau}} = U_\infty \\
				&\iff U_0^{\alpha_x\tau^{-1}\alpha_{-(x\tau^{-1})}\tau} = U_\infty^{\alpha_{(-(x\tau^{-1}))\tau}} \\
				&\iff U_x^{\tau^{-1}\alpha^{-1}_{x\tau^{-1}}\tau} = U_\infty \\
				&\iff U_x = U_\infty^{\gamma_{x\tau^{-1}}}\,,
			\end{align*}
		where we only use the definitions of the root groups in $\M(U,\tau)$.
		\item[\ref{lem:loceq3}$\Leftrightarrow$\ref{lem:loceq4}.] We have $U_0^{\mu_x} = U_0^{\gamma_{(-x)\tau^{-1}}\alpha_x\gamma_{-(x\tau^{-1})}} = 	U_x^{\gamma_{-(x\tau^{-1})}}$, so the equivalence follows.\qedhere
	\end{itemize}
\end{proof}

\begin{lemma}\label{lem:gleq}
	The following are equivalent:
	\begin{romenumerate}
		\item $U_\infty^{h_x} = U_\infty$ for all units $x\in X$; \label{lem:gleq2}
		\item $U_\infty^{\gamma_{x\tau^{-1}}} = U_x$ for all units $x\in X$; \label{lem:gleq3}
		\item $U_0^{\mu_x} = U_\infty$ for all units $x\in X$; \label{lem:gleq4}
		\item $U_0 = U_\infty^{\mu_x}$ for all units $x\in X$. \label{lem:gleq5}
	\end{romenumerate}
\end{lemma}
\begin{proof}
	The equivalence of \ref{lem:gleq2}-\ref{lem:gleq4} is immediate from the previous lemma. The equivalence between \ref{lem:gleq4} and \ref{lem:gleq5} follows by replacing $x$ with $-x$ and noting that $\mu_{-x} = \mu_x^{-1}$.
\end{proof}

\begin{lemma}\label{lem:huafix}
	Assume that $h_x$ normalizes $U$ for all units $x\in X$. Then
	\begin{romenumerate}
        \item $U^h = U$ for all $h\in H$; \label{itm:huafix}
        \item $U_0^\circ U_\infty \subset U_\infty H U_0^\circ$.
	\end{romenumerate}
\end{lemma}
\begin{proof}\preenum
	\begin{romenumerate}
        \item
            By \autoref{lem:gleq}, $U_0^{\mu_x} = U_\infty$ and $U_0 = U_\infty^{\mu_x}$ for all units $x$.
            Now $H$ is generated by all products of two $\mu$-maps, which all normalize $U_\infty$, so any element of $H$ normalizes $U_\infty$.
        \item
            We can follow the proof of \autoref{prop:huaincl} mutatis mutandis.
            Note that we used \ref{axiom:LM2} only twice: once in the proof of \autoref{lem:mu}\ref{itm:mutau}, and once in the proof of \autoref{prop:quasi_inv}.
            
            In the proof of \autoref{lem:mu}\ref{itm:mutau}, we need that
            \[U_0^{\alpha_{x\tau}\,\alpha_{-x}^\tau} = U_\infty\]
            for any unit $x$. Now, by the definitions of the root groups in the construction, we get
            \begin{align*}
            	U_0^{\alpha_{x\tau}\,\alpha_{-x}^\tau} = U_\infty &\iff U_{x\tau}^{\alpha_x^{-\tau}} = U_\infty \\
			&\iff U_{x\tau} = U_\infty^{\gamma_x} \\
			&\iff U_{x\tau} = U_\infty^{\gamma_{(x\tau)\tau^{-1}}}\,.
            \end{align*}
            This final identity is true by the assumption and \autoref{lem:gleq}.
            
            In the proof of \autoref{prop:quasi_inv}, we need that
            \[U_0^{\alpha_y^{-1}\alpha_{(-y)\tau^{-1}}^{-\tau}} = U_\infty\]
            for any unit $y$. Again, by the definitions of the root groups in the construction, we get
            \begin{align*}
            	U_0^{\alpha_y^{-1}\alpha_{(-y)\tau^{-1}}^{-\tau}} = U_\infty &\iff U_{-y} = U_\infty^{\alpha_{(-y)\tau^{-1}}^\tau} \\
			&\iff U_{-y} = U_\infty^{\gamma_{(-y)\tau^{-1}}}\,,
            \end{align*}
            and the final equality again holds by \autoref{lem:gleq}.
        \qedhere
	\end{romenumerate}
\end{proof}

This additional assumption will also be sufficient to ensure that the construction is a local Moufang set.

\begin{theorem}\label{thm:constrMouf}
	Let $\M(U,\tau)$ be as in \autoref{constr:MUtau}. Then $\M(U,\tau)$ is a local Moufang set if and only if $h_x$ normalizes $U$ for all units $x$.
\end{theorem}
\begin{proof}
	Assume first that $\M(U,\tau)$ is a local Moufang set. By \ref{axiom:LM2}, all $\mu$-maps send $U_0$ to $U_\infty$. By \autoref{lem:gleq}, this implies that all Hua maps normalize $U = U_\infty$.

	For the converse, we have already shown that $\M(U,\tau)$ is a local pre-Moufang set, so what remains to show is \ref{axiom:LM2}. Fix some unit $e\in X$ and write $\mu = \mu_e$. Then, by our assumptions and by \autoref{lem:gleq}, $U_x\subset\langle U_\infty, U_0\rangle = \langle U_\infty, \mu\rangle$ for any $x\in X$, so $G = \langle U_\infty,\mu\rangle$. In order to show that $U_x^g = U_{xg}$ for all $g\in G$ and all $x\in X$, it is sufficient to show that $U_x^\mu = U_{x\mu}$ and $U_x^{\alpha_y} = U_{x\alpha_y}$ for any $y\nsim\infty$.
	
	We start by showing $U_x^{\alpha_y} = U_{x\alpha_y}$ for all $x\in X$ and $y\in X$ s.t.\ $y\nsim\infty$.
	We distinguish two cases.
	\begin{itemize}[labelindent=3em, leftmargin=*]
		\item[{$x\nsim\infty$}:] In this case, we have $U_x = U_0^{\alpha_x}$ by definition. Since $U_\infty$ is a group, we have $\alpha_x\alpha_y = \alpha_z$ for some $z$, and by looking at the image of $0$, we find $z = x\alpha_y$, so indeed $U_x^{\alpha_y}= U_0^{\alpha_x\alpha_y} = U_0^{\alpha_{x\alpha_y}} = U_{x\alpha_y}$.
	
		\item[{$x\sim\infty$}:] By definition, $U_x^{\alpha_y} = U_\infty^{\gamma_{x\tau^{-1}}\alpha_y}$. Since $\gamma_{x\tau^{-1}}\alpha_y\in U_0^\circ U_\infty$, we have $\gamma_{x\tau^{-1}}\alpha_y = u h\gamma_z$ for some $u\in U_\infty$, $z\sim0$ and $h\in H$. By calculating the image of $\infty$, we see $z\tau = x\alpha_y$. So we get 
		\begin{align*}
			U_x^{\alpha_y} &= U_\infty^{\gamma_{x\tau^{-1}}\alpha_y} = U_\infty^{u h\gamma_{x\alpha_y\tau^{-1}}} = U_\infty^{h\gamma_{x\alpha_y\tau^{-1}}} \\
				&= U_\infty^{\gamma_{x\alpha_y\tau^{-1}}} = U_{x\alpha_y}\,,
		\end{align*}
		where we used \autoref{lem:huafix}\ref{itm:huafix}.
	\end{itemize}
	Secondly, we need to show that $U_x^\mu = U_{x\mu}$ for all $x\in X$; we make the same case distinction.
	\begin{itemize}[labelindent=3em, leftmargin=*]
		\item[{$x\nsim\infty$}:] We have $U_x^\mu = U_0^{\alpha_x\mu} = U_0^{\mu\alpha_x^\mu} = U_\infty^{\alpha_x^\mu}$. Now $\alpha_x^\mu\in U_0$, so $\alpha_x^\mu = \gamma_z$ for some $z$, and by checking the image of $\infty$ we find $z\tau = x\mu$. If $x\mu\sim\infty$, we can use the definition of a root group to find $U_\infty^{\alpha_x^\mu} = U_\infty^{\gamma_{x\mu\tau^{-1}}} = U_{x\mu}$. If $x\mu\nsim\infty$, $x\mu$ is a unit (as $x\nsim\infty$, we also have $x\mu\nsim 0$), so we have $U_\infty^{\alpha_x^\mu}=U_\infty^{\gamma_{x\mu\tau^{-1}}} = U_{x\mu}$ by \autoref{lem:gleq}\ref{lem:gleq3}.
	
		\item[{$x\sim\infty$}:] We have $U_x^\mu = U_\infty^{\gamma_{x\tau^{-1}}\mu} = U_\infty^{\mu\gamma_{x\tau^{-1}}^\mu} = U_0^{\gamma_{x\tau^{-1}}^\mu}$. Now $\gamma_{x\tau^{-1}}^\mu\in U_\infty$, so $\gamma_{x\tau^{-1}}^\mu = \alpha_z$ for some $z$, and by checking the image of $0$ we find $x\mu = z$. Notice that since $x\sim \infty$, we have $x\mu\sim0$, so we can use the definition of the root group of $x\mu$ to get $U_0^{\gamma_{x\tau^{-1}}^\mu} = U_0^{\alpha_{x\mu}} = U_{x\mu}$.\qedhere
	\end{itemize}
\end{proof}

Combining the previous theorem with \autoref{lem:gleq} gives more equivalent conditions:

\begin{corollary}\label{cor:construction_equivalentconditions}
	Let $\M(U,\tau)$\notatlink{M(U,tau)} be as in \autoref{constr:MUtau}. Then the following are equivalent:
	\begin{romenumerate}
		\item $\M(U,\tau)$ is a local Moufang set;
		\item $U^{h_x} = U$ for all units $x\in X$;
		\item $U^\tau = U^{\mu_x}$ for all units $x\in X$.
	\end{romenumerate}
\end{corollary}

\begin{remark}\label{rem:constr_UU'}
	In \autoref{constr:MUtau}, we started with a root group and a specific element $\tau$. We could also construct a local Moufang set using two root groups corresponding to elements that are not equivalent. I.e.\ assume we have the following data:
	\begin{nrenumerate}
		\item a set with an equivalence relation $(X,\sim)$, such that $\abs{\class{X}}>2$;
		\item two groups $U,U'\leq\Sym(X,\sim)$;
	\end{nrenumerate}
	and assume these satisfy
	\begin{romenumerate}
		\item $U$ fixes an element we call $\infty$ and acts sharply transitively on $X\setminus\class{\infty}$;
		\item[\textnormal{(\hspace{-0.2ex}\textit{i}'\hspace{-0.1ex})}] $\induced{U}$ acts sharply transitively on $\class{X}\setminus\{\class{\infty}\}$;
		\item $U'$ fixes an element we call $0$ and acts sharply transitively on $X\setminus\class{0}$;
		\item[\textnormal{(\hspace{-0.2ex}\textit{ii}'\hspace{-0.1ex})}] $\induced{U'}$ acts sharply transitively on $\class{X}\setminus\{\class{0}\}$;
		\item $0\nsim\infty$.
	\end{romenumerate}
	With this data, we set $U_\infty := U$ and $U_0:= U'$. We now define
	\[U_x:=U_0^{\alpha_x}\text{ for $x\nsim\infty$}\qquad U_x:=U_\infty^{\zeta_x}\text{ for $x\sim\infty$}\;,\]
	where $\alpha_x\in U_\infty$ is the unique element mapping $0$ to $x$, and $\zeta_x\in U_0$ is the unique element mapping $\infty$ to $x$.
	
	This again gives rise to a local pre-Moufang set, and we can still use the same argument as \autoref{prop:mu_defin} to get $\mu$-maps. We could then prove that this construction gives a local Moufang set if and only if $U = U'^\mu$ for all $\mu$-maps. In the examples, we will sometimes use this approach by defining two root groups, as in \autoref{chap:chap7_jordan}, but we will always return to \autoref{constr:MUtau} and \autoref{thm:constrMouf} to prove we have a local Moufang set.
\end{remark}
	\chapter[Categorical notions]{Categorical\\ notions}\label{chap:chap4_category}
	In this chapter we develop the categorical theory of local Moufang sets. We introduce the natural substructures we call local Moufang subsets, homomorphisms of local Moufang sets, and this gives rise to quotients and finally the category of local Moufang sets. We use these notions to determine when inverse limits of local Moufang sets exist, and to construct these inverse limits. We will use these inverse limits in \autoref{sec:Serre}.

\section{The category of local Moufang sets}	
\subsection{Local Moufang subsets}

It is natural to look for substructures in a local Moufang set $\M=(X,(U_x)_{x\in X})$ that are also local Moufang sets. Let $Y\subset X$ and assume $\abs{\class{Y}}>2$. By changing the basis in $X$, we can assume $0,\infty\in Y$. Without loss of generality, we add this as an assumption. We want to use \autoref{constr:MUtau} to ensure our subset gives rise to a local Moufang set, so we need a root group and a $\tau$.

\begin{definition}
	Let $\M=(X,(U_x)_{x\in X})$ be a local Moufang set and $Y\subset X$ such that $0,\infty\in Y$ and $\abs{\class{Y}}>2$. If
	\begin{manualenumerate}[label=\textrm{(S\arabic*)}]
		\item $V := \{u\in U_\infty\mid 0u\in Y\}$ is a group;\label{axiom:S1}
		\item $Yv \subset Y$ for all $v\in V$;\label{axiom:S2}
		\item there is a unit $y\in Y$ such that $Y\mu_y = Y$,\label{axiom:S3}
	\end{manualenumerate}
	then $\M(V,\mu_y)$ is a \define{local Moufang subset} of $\M$. We call $\M(V,\mu_y)$ the local Moufang subset induced by $Y$.
\end{definition}

\begin{proposition}
	If $\M(V,\mu_y)$ is a local Moufang subset of $\M$, then it is a local Moufang set. For all units $y'\in Y$, $Y\mu_{y'}=Y$.
\end{proposition}

\begin{proof}
	We first remark that the conditions \ref{axiom:C1}, \ref{axiom:C1'} and \ref{axiom:C2} are satisfied: the first and the second follow from the corresponding properties of $\M$, while the third is immediate as $\mu_y$ is a $\mu$-map of $\M$. We also observe that for any $v\in V$ we also get $Yv^{-1}\subset Y$, so $Yv=Y$. 
	
	To finish the proof that we get a local Moufang set, we first observe that
	\begin{align*}
		V_0	&= V^{\mu_y} = \{u^{\mu_y}\mid u\in U_\infty, 0u\in Y\} \\
			&= \{u\in U_0\mid \infty u\in Y\mu_y\} = \{u\in U_0\mid \infty u\in Y\}\;,
	\end{align*}
	and $Yv^{\mu_y} = Yv\mu_y = Y\mu_y=Y$. Now for any unit $y'\in Y$, the map $\mu_{y'}$ preserves $Y$ (as it is a composition of elements of $V$ and $V_0$ which both preserve $Y$), and it is the restriction of $\mu_{y'}$ in $\M$. Hence by the same argument as for $\mu_y$, we get 
	\[V^{\mu_{y'}} = \{u\in U_0\mid \infty u\in Y\} = V_0\;,\]
	and this shows $\M(V,\mu_y)$ is a local Moufang set by \autoref{cor:construction_equivalentconditions}.
\end{proof}


\subsection{Morphisms of local Moufang sets}

Next, we would like to define homomorphisms between local Moufang sets. Let $\M_1 = (X,(U_x)_{x\in X})$ and $\M_2 = (Y,(V_y)_{y\in Y})$ be two local Moufang sets (where we denote both equivalences on $X$ and $Y$ by $\sim$). A morphism from $\M_1$ to $\M_2$ will be induced by an equivalence-preserving map from $X$ to $Y$, i.e.\ $\phi\colon X\to Y$ such that $x\sim x'\implies x\phi\sim x'\phi$. Furthermore, we want to ensure that $\phi$ interacts nicely with the root groups. We demand $U_x\phi\subset \phi V_{x\phi}$ for all $x\in X$. This corresponds to the natural definition for Moufang sets introduced by De Medts in \cite{TDMcategoryMoufang}, although he thought it to be a `bad attempt' at the time, since homomorphisms of Moufang sets are always injective. O.\ Loos observed that this definition did make sense in \cite{LoosRogdiv}. The same `problem' arises here:

\begin{lemma}
	$\M_1 = (X,(U_x)_{x\in X})$ and $\M_2 = (Y,(V_y)_{y\in Y})$ be local Moufang sets and $\phi\colon X\to Y$ an equivalence-preserving map such that $U_x\phi\subset \phi V_{x\phi}$ for all $x\in X$. Then one of the two following cases occurs:
	\begin{romenumerate}
		\item $x\phi\sim x'\phi\implies x\sim x'$ for all $x,x'\in X$;
		\item $x\phi\sim x'\phi$ for all $x,x'\in X$.
	\end{romenumerate}
\end{lemma}
\begin{proof}
	Suppose there are $x,x'\in X$ such that $x\phi\sim x'\phi$ and $x\nsim x'$. Take any $\tilde{x}\in X$, then without loss of generality, we can assume $\tilde{x}\nsim x$. Hence there is an element $u\in U_x$ such that $x'u = \tilde{x}$. By the assumption, there is an element $v\in V_{x\phi}$ such that $u\phi = \phi v$, so
	\[\tilde{x}\phi = x'u\phi = x'\phi v\sim x\phi v = x\phi\;,\]
	hence $\tilde{x}\phi\sim x\phi$, so indeed $x\phi\sim x'\phi$ for all $x,x'\in X$.
\end{proof}

We want to avoid the second case, as this would mean that the equivalence on $\Im(\phi)$ has only one equivalence class (and we would like the image to be a local Moufang set). Hence the following definition for homomorphisms of local Moufang sets:

\begin{definition}
	A \define[homomorphism]{homomorphism} between two local Moufang sets $\M_1 = (X,(U_x)_{x\in X})$ and $\M_2 = (Y,(V_y)_{y\in Y})$ is a map $\phi\colon X\to Y$\notatlink{phi} such that $x\sim x'$ if and only if $x\phi\sim x'\phi$ for all $x,x'\in X$ and
	\begin{align}
		U_x\phi&\subset\phi V_{x\phi}\text{ for all }x\in X\;.\label{eq:morphism}\tag{H}
	\end{align}
\end{definition}

Note that the identity map on $X$ is an identity homomorphism of $\M$ and that the composition of two homomorphisms is again a homomorphism.

\begin{remark}\label{rem:induced_inj_map}
	The assumption on $\phi$ saying that $x\sim x'$ if and only if $x\phi\sim x'\phi$ for all $x,x'\in X$ implies that there is a well-defined injective map $\induced{\phi}\colon \class{X}\to \class{Y}\colon \class{x}\to \class{x\phi}$\notatlink{phibar}. If we have a composition of homomorphisms of local Moufang sets, we immediately get 
	\[\induced{\phi\circ\psi} = \induced{\phi}\circ\induced{\psi}\;.\]
\end{remark}

If the local Moufang sets are in fact Moufang sets, both types of homomorphisms coincide.

\begin{proposition}\label{prop:hom_mouf_set}
	Let $\M$ and $\M'$ be Moufang sets. Then $\phi\colon\M\to\M'$ is a homomorphism of local Moufang sets if and only if it is a homomorphism of Moufang sets.
\end{proposition}
\begin{proof}
	The condition \eqref{eq:morphism} coincides for both homomorphisms of local Moufang sets and of Moufang sets. Now, as $\M$ and $\M'$ are Moufang sets, the corresponding equivalence relation on the sets is the identity relation. This means
	\begin{align*}
		(x\sim x' \iff x\phi\sim x'\phi) &\Longleftrightarrow (x=x' \iff x\phi = x'\phi) \\
						 &\Longleftrightarrow \text{$\phi$ is an injective map}\;.
	\end{align*}
	Hence the second condition for being a homomorphism of local Moufang sets is equivalent to being a homomorphism of Moufang sets.
\end{proof}

From the map $\phi$, we immediately get homomorphisms between the root groups:

\begin{lemma}\label{lem:theta_maps}
	Let $\phi$ be a homomorphism from $\M_1 = (X,(U_x)_{x\in X})$ to $\M_2 = (Y,(V_y)_{y\in Y})$.
	\begin{romenumerate}
		\item For all $x\in X$ we get a homomorphism $\theta_x\colon U_x\to V_{x\phi}$\notatlink{thetax} such that $\theta_x(u)$ is the unique element of $V_{x\phi}$ for which
		\begin{align}
			u\phi &= \phi\theta_x(u)\;.\label{eq:morphism_theta}\tag{H'}
		\end{align}
		\item Let $x\in X$, $u\in U_x$, $v\in V_{x\phi}$ and $x'\in X\setminus\class{x}$. If $x'u\phi = x'\phi v$, then $v = \theta_x(u)$, so $u\phi = \phi v$.\label{itm:unique_theta}
	\end{romenumerate}
	Now choose the bases in $X$ and $Y$ such that $0_Y = 0_X\phi$ and $\infty_Y = \infty_X\phi$ (we often just write $0$ and $\infty$, so $0\phi=0$ and $\infty\phi = \infty$). Let $x\in X$.
	\begin{romenumerate}[resume]
		\item If $x\nsim\infty$, we have $\theta_\infty(\alpha_x) = \alpha_{x\phi}$ and $(-x)\phi = -(x\phi)$.
		\item If $x$ is a unit, then so is $x\phi$ and $\mu_x\phi = \phi\mu_{x\phi}$. If we have the unique expression $\mu_x = g\alpha_x h$, for $g,h\in U_0$, then \label{itm:morphism_mu}
		\[\mu_{x\phi} = \theta_0(g)\alpha_{x\phi}\theta_0(h)\;.\]
		\item If $x$ is a unit, then $(\til x)\phi = \til(x\phi)$.
		\item Denoting $H_1$ and $H_2$ for the Hua subgroups of $\M_1$ and $\M_2$, we have $H_1\phi\subset \phi H_2$.
	\end{romenumerate}
\end{lemma}
\begin{proof}\preenum
	\begin{romenumerate}
		\item Take $u\in U_x$. By \eqref{eq:morphism}, there is an element $v\in V_{x\phi}$ such that $u\phi = \phi v$. Now take any $x'\nsim x$, then $x'u\phi = x'\phi v$, so $v$ is the unique element of $V_{x\phi}$ mapping $x'\phi$ to $x'u\phi$ (observe that $x'\phi\nsim x\phi$). Hence $v$ is uniquely determined by $u$, so we can set $\theta_x(u) = v$. Now $\theta_x(\id_X)$ must map $x'\phi$ to $x'\phi$, hence $\theta_x(\id_X) = \id_Y$. Finally
		\[\phi\theta_x(uu') = uu'\phi = u\phi\theta_x(u') = \phi\theta_x(u)\theta_x(u')\;,\]
		so by unicity, $\theta_x(uu') = \theta_x(u)\theta_x(u')$. Hence $\theta_x$ is a group homomorphism.
		\item This follows from the first statement.
		\item By the definition of a homomorphism, $x\nsim \infty\implies x\phi\nsim \infty$. By \eqref{eq:morphism_theta}, $\theta_\infty(\alpha_x)$ maps $0\phi$ to $0\alpha_x\phi = x\phi$, hence $\theta_\infty(\alpha_x) = \alpha_{x\phi}$. As $\theta_x$ is a group homomorphism, 
		\[(-x)\phi = 0\alpha_x^{-1}\phi = 0\phi\alpha_{x\phi}^{-1} = -(x\phi)\;.\]
		\item By the assumptions on $\phi$, $x\nsim 0\implies x\phi\nsim 0$ (and as before $x\phi\nsim \infty$), so $x\phi$ is a unit. Now, by \autoref{prop:mu_defin}, $\mu_x = g\alpha_x h$, with $g, h\in U_0$ the unique elements such that $\infty g = -x$ and $xh=\infty$. Hence
		\[\mu_x\phi = g\alpha_x h\phi = \phi\theta_0(g)\alpha_{x\phi}\theta_0(h)\;,\]
		and $\theta_0(g), \theta_0(h)\in V_0$ are the unique elements such that 
		\[\infty \theta_0(g) = -(x\phi)\text{ and }x\phi\theta_0(h)=\infty\;.\]
		Hence, again by \autoref{prop:mu_defin}, $\theta_0(g)\alpha_{x\phi}\theta_0(h) = \mu_{x\phi}$, so indeed $\mu_x\phi = \phi\mu_{x\phi}$.
		\item By \autoref{lem:mu}\ref{itm:tilmu}, $\til x = -((-x)\mu_x)$, so 
		\begin{align*}
			(\til x)\phi 	&= (-((-x)\mu_x))\phi = -((-x)\mu_x\phi) \\
					&= -((-x\phi)\mu_{x\phi}) = \til(x\phi)\;.
		\end{align*}
		\item As the Hua subgroups are generated by $\mu$-maps, this follows from \ref{itm:morphism_mu}.\qedhere
	\end{romenumerate}
\end{proof}

If we have a homomorphism from $\M_1$ to $\M_2$, we will henceforth assume that we chose $0$ and $\infty$ in $\M_2$ in such a way that $0\phi = 0$ and $\infty\phi = \infty$. It turns out that there is a tight connection between the different $\theta$-maps: they are all `conjugate' in the following way:

\begin{lemma}\label{lem:diagram_thetas}
	Let $x\in X$ and $y,z\in X\setminus\class{x}$. Let $g\in U_x$ be the unique element such that $yg=z$. Then the following diagram commutes:
	\begin{center}
		\begin{tikzpicture}[regulararrow]
			\node(LT) {$U_y$};
			\node[right of=LT, node distance=3cm](RT) {$U_z$};
			\node[below of=LT, node distance=1.5cm](LB) {$V_{y\phi}$};
			\node[right of=LB, node distance=3cm](RB) {$V_{z\phi}$};
			\draw[->] (LT) -- node[above]() {$u\mapsto u^g$} (RT);
			\draw[->] (LB) -- node[above]() {$v\mapsto v^{\theta_x(g)}$} (RB);
			\draw[->] (LT) -- node[left]() {$\theta_y$} (LB);
			\draw[->] (RT) -- node[right]() {$\theta_z$} (RB);
		\end{tikzpicture}
	\end{center}
	i.e.\ $\theta_z(u^g) = \theta_y(u)^{\theta_x(g)}$.
\end{lemma}
\begin{proof}
	Remark first that both $\theta_z(u^g)$ and $\theta_y(u)^{\theta_x(g)}$ are in $V_{z\phi}$. By \eqref{eq:morphism_theta}, $\theta_z(u^g)$ is the unique element of $V_{z\phi}$ such that $\phi \theta_z(u^g) = u^g\phi$, but
	\[u^g\phi = gug^{-1}\phi = \phi\theta_x(g)\theta_y(u)\theta_x(g)^{-1} = \phi\theta_y(u)^{\theta_x(g)}\;,\]
	so the desired equality indeed holds.
\end{proof}

One corollary to this is that a similar conjugation property holds for the $\mu$-maps:

\begin{corollary}
	Let $x\in X$ be a unit, then $\theta_0(u^{\mu_x}) = \theta_\infty(u)^{\mu_{x\phi}}$, or equivalently, to following diagram commutes:
	\begin{center}
		\begin{tikzpicture}[regulararrow]
			\node(LT) {$U_0$};
			\node[right of=LT, node distance=3cm](RT) {$U_\infty$};
			\node[below of=LT, node distance=1.5cm](LB) {$V_0$};
			\node[right of=LB, node distance=3cm](RB) {$V_\infty$};
			\draw[->] (LT) -- node[above]() {$u\mapsto u^{\mu_x}$} (RT);
			\draw[->] (LB) -- node[above]() {$v\mapsto v^{\mu_{x\phi}}$} (RB);
			\draw[->] (LT) -- node[left]() {$\theta_0$} (LB);
			\draw[->] (RT) -- node[right]() {$\theta_\infty$} (RB);
		\end{tikzpicture}
	\end{center}
\end{corollary}
\begin{proof}
	We look at the following diagram, with $\mu_x = g\alpha_x h$, for $g,h\in U_0$:
	\begin{center}
		\begin{tikzpicture}[regulararrow]
			\node(T1) {$U_0$};
			\node[right of=T1, node distance=3cm](T2) {$U_0$};
			\node[right of=T2, node distance=3cm](T3) {$U_x$};
			\node[right of=T3, node distance=3cm](T4) {$U_\infty$};
			\node[below of=T1, node distance=1.5cm](B1) {$V_0$};
			\node[right of=B1, node distance=3cm](B2) {$V_0$};
			\node[right of=B2, node distance=3cm](B3) {$V_{x\phi}$};
			\node[right of=B3, node distance=3cm](B4) {$V_\infty$};
			\draw[->] (T1) -- node[above]() {$u\mapsto u^{g}$} (T2);
			\draw[->] (T2) -- node[above]() {$u\mapsto u^{\alpha_x}$} (T3);
			\draw[->] (T3) -- node[above]() {$u\mapsto u^{h}$} (T4);
			\draw[->] (B1) -- node[above]() {$v\mapsto v^{\theta_0(g)}$} (B2);
			\draw[->] (B2) -- node[above]() {$v\mapsto v^{\alpha_{x\phi}}$} (B3);
			\draw[->] (B3) -- node[above]() {$v\mapsto v^{\theta_0(h)}$} (B4);
			\draw[->] (T1) -- node[left]() {$\theta_0$} (B1);
			\draw[->] (T2) -- node[left]() {$\theta_0$} (B2);
			\draw[->] (T3) -- node[left]() {$\theta_x$} (B3);
			\draw[->] (T4) -- node[right]() {$\theta_\infty$} (B4);
		\end{tikzpicture}
	\end{center}
	The first square commutes as $\theta_0$ is a group homomorphism, and the other two squares commute by the previous lemma. Hence the entire diagram commutes. As $\mu_{x\phi} = \theta_0(g)\alpha_{x\phi}\theta_0(h)$, this means the diagram from the statement commutes.
\end{proof}

Using \autoref{lem:diagram_thetas}, we also get some nice connections between properties of the mapping $\phi$ and corresponding properties of the $\theta$-maps.

\begin{corollary}\label{cor:morphism_inj}
	Let $\phi$ be a homomorphism from $\M_1$ to $\M_2$. Then the following are equivalent:
	\begin{romenumerate}
		\item $\phi$ is injective;\label{itm:morphism_inj1}
		\item $\theta_x$ is injective for all $x\in X$;\label{itm:morphism_inj2}
		\item $\theta_x$ is injective for an $x\in X$.\label{itm:morphism_inj3}
	\end{romenumerate}
\end{corollary}
\begin{proof}\preenum
	\begin{itemize}[labelindent=4.5em, leftmargin=*]
		\item[{\ref{itm:morphism_inj1} $\Rightarrow$ \ref{itm:morphism_inj2}}.] 
		Take any $x\in X$ and assume $\theta_x(u)=\id$ for $u\in U_x$. Then
		\[u\phi = \phi\theta_x(u) = \phi\;,\]
		so as we assume $\phi$ is injective, $u=\id$. Hence $\theta_x$ is injective.
		\item[{\ref{itm:morphism_inj2} $\Rightarrow$ \ref{itm:morphism_inj1}}.]
		Assume we have $x,x'\in X$ such that $x\phi=x'\phi$. Now take $z\in X$ such that $x,x'\not\in \class{z}$. Then there is an element $g\in U_z$ such that $xg=x'$. Hence
		\[x\phi = x'\phi = xg\phi = x\phi\theta_z(g)\;.\]
		As $V_{z\phi}$ must have a unique element mapping $x\phi$ to $x\phi$, we get $\theta_z(g)=\id$. Since $\theta_z$ is injective, this means $g=\id$ and hence $x' = xg = x$, so $\phi$ is injective.
		\item[{\ref{itm:morphism_inj2} $\Rightarrow$ \ref{itm:morphism_inj3}}.]
		This is immediate.
		\item[{\ref{itm:morphism_inj3} $\Rightarrow$ \ref{itm:morphism_inj2}}.]
		Let $x\in X$ be such that $\theta_x$ is injective, and take any $x'\in X$. Now take $z\in X$ such that $x,x'\not\in \class{z}$. Then there is an element $g\in U_z$ such that $xg=x'$. By \autoref{lem:diagram_thetas}, 
		\[\theta_{x'}(u^g)^{\theta_z(g)^{-1}} = \theta_x(u)\;.\]
		Hence if $\theta_{x'}(u)=\id$, then $\theta_x(u^{g^{-1}})=\id$, so $u^{g^{-1}}=\id$ and $u=\id$. This means $\theta_{x'}$ is injective.
		\qedhere
	\end{itemize}
\end{proof}
\begin{corollary}\label{cor:morphism_surj}
	Let $\phi$ be a homomorphism from $\M_1$ to $\M_2$. Then the following are equivalent:
	\begin{romenumerate}
		\item $\phi$ is surjective;\label{itm:morphism_surj1}
		\item $\theta_x$ is surjective for all $x\in X$;\label{itm:morphism_surj2}
		\item $\theta_x$ is surjective for an $x\in X$.\label{itm:morphism_surj3}
	\end{romenumerate}
\end{corollary}
\begin{proof}\preenum
	\begin{itemize}[labelindent=4.5em, leftmargin=*]
		\item[{\ref{itm:morphism_inj1} $\Rightarrow$ \ref{itm:morphism_inj2}}.] 
		Take any $x\in X$ and any $v\in V_{x\phi}$. Let $y\in Y\setminus\class{x\phi}$, then $v$ is the unique element of $V_{x\phi}$ mapping $y$ to $yv$. As $\phi$ is surjective, there are $z,z'\in X$ such that $z\phi = y$ and $z'\phi = yv$. As $y$ and $yv$ are not equivalent to $x\phi$, we know $z$ and $z'$ are not equivalent to $x$. Hence there is an element $u\in U_x$ mapping $z$ to $z'$. We get
		\[yv = z'\phi = zu\phi = z\phi\theta_x(u) = y\theta_x(u)\;,\]
		so $v = \theta_x(u)$ and $\theta_x$ is surjective.
		\item[{\ref{itm:morphism_inj2} $\Rightarrow$ \ref{itm:morphism_inj1}}.]
		Take $y\in Y$. Now there is a point $x$ such that $x\phi\nsim y$. Take any $x'\nsim x$, then there is an element $v\in V_{x\phi}$ such that $x'\phi v = y$. As $\theta_x$ is assumes surjective, there is an element $u\in U_x$ such that $v = \theta_x(u)$. Hence
		\[y = x'\phi v = x'\phi\theta_x(u) = x'u\phi\;,\]
		so $y$ is in the image of $\phi$ and $\phi$ is surjective.
		\item[{\ref{itm:morphism_inj2} $\Rightarrow$ \ref{itm:morphism_inj3}}.]
		This is immediate.
		\item[{\ref{itm:morphism_inj3} $\Rightarrow$ \ref{itm:morphism_inj2}}.]
		Let $x\in X$ be such that $\theta_x$ is surjective, and take any $x'\in X$. Now take $z\in X$ such that $x,x'\not\in \class{z}$. Then there is an element $g\in U_z$ such that $xg=x'$. By \autoref{lem:diagram_thetas}, $\theta_{x'}(u^g) = \theta_x(u)^{\theta_z(g)}$. Now if $v'\in V_{x'\phi}$, there is an element $v\in V_{x\phi}$ such that $v' = v^{\theta_z(g)}$, and there is an element $u\in U_x$ such that $\theta_x(u) = v$. Hence $\theta_{x'}(u^g) = v'$, so $\theta_{x'}$ is surjective.
		\qedhere
	\end{itemize}
\end{proof}

Combining these shows that the inverse map to a bijective homomorphisms is again a homomorphism.

\begin{corollary}\label{cor:isomorphism}
	If $\phi$ is a bijective homomorphism from $\M_1$ to $\M_2$, then $\phi^{-1}$ is a homomorphism from $\M_2$ to $\M_1$. In this case the $\theta$-maps are group isomorphisms.
\end{corollary}
\begin{proof}
	Clearly, as $x\sim x'\iff x\phi\sim x'\phi$ and $\phi$ is bijective, we get $x\phi^{-1}\sim x'\phi^{-1}\iff x\sim x'$.
	
	As $\phi$ is surjective, \eqref{eq:morphism_theta} implies that $U_x\phi = \phi V_{x\phi}$ for all $x\in X$. Hence $V_{x\phi}\phi^{-1} = \phi^{-1} U_x$ for all $x\in X$, and as $\phi$ is bijective, 
	\[V_{y}\phi^{-1} = \phi^{-1} U_{y\phi^{-1}}\]
	for all $y\in Y$. In particular, $V_{y}\phi^{-1} \subset \phi^{-1} U_{y\phi^{-1}}$.
	
	By the previous corollaries, the $\theta$-maps are injective and surjective group morphisms, hence they are group isomorphims.
\end{proof}

Using this corollary, we can define isomorphic local Moufang sets

\begin{definition}
	Two local Moufang sets $\M_1$ and $\M_2$ are \define{isomorphic} if and only if there is a bijective homomorphisms between them.
\end{definition}

An equivalent way of defining isomorphisms is the following:

\begin{proposition}\label{prop:isomorphism1}
	Two local Moufang sets $\M_1 = (X,(U_x)_{x\in X})$ and $\M_2 = (Y,(V_y)_{y\in Y})$ are isomorphic if and only if there is a bijective map $\phi\colon X\to Y$ and for each $x\in X$ an isomorphism $\theta_x\colon U_x\to V_{x\phi}$ such that
	\begin{romenumerate}
		\item $x\sim x'\iff x\phi\sim x'\phi$ for all $x,x'\in X$;
		\item $u\phi = \phi\theta_x(u)$ for all $u\in U_x$ and all $x\in X$.
	\end{romenumerate}
\end{proposition}
\begin{proof}
	Assume such $\phi$ and $\theta_x$ exist. By the assumptions, $\phi$ is a bijective map. Now take any $x\in X$ and any $u\in U_x$, then 
	\[u\phi = \phi\theta_x(u)\in\phi V_{x\phi}\;,\]
	so $\phi$ is a homomorphism of local Moufang sets.
	
	Conversely, assume $\M_1$ and $\M_2$ are isomorphic. Then the bijective morphism $\phi$ and its induced $\theta$-maps satisfy the conditions.
\end{proof}

We can also check when two local Moufang sets given by \autoref{constr:MUtau} are isomorphic.

\begin{proposition}\label{prop:isomorphism2}
	Let $\M(U,\tau)$ and $\M(V,\tau')$ be two local Moufang sets acting on $X$ and $Y$, constructed using \autoref{constr:MUtau}. If there is a map $\phi\colon X\to Y$ and a group morphism $\theta\colon U\to V$ such that
	\begin{romenumerate}
		\item $x\sim x'\iff x\phi\sim x'\phi$ for all $x,x'\in X$;
		\item $u\phi = \phi\theta(u)$ for all $u\in U$;
		\item $\tau\phi = \phi\tau'$.
	\end{romenumerate}
	then $\phi\colon\M(U,\tau)\to\M(V,\tau')$ is a homomorphism of local Moufang sets. If $\phi$ is a bijection, then $\M(U,\tau)$ and $\M(V,\tau')$ are isomorphic.
\end{proposition}
\begin{proof}
	With $\theta$, we first construct a group morphism $\theta_x\colon U_x\to V_{x\phi}$ for every $x\in X$. We set $\theta_\infty = \theta$. Next, we set
	\[\theta_0\colon U_0\to V_0\colon u\mapsto \tau'^{-1}\theta_\infty(u^{\tau^{-1}})\tau'\;.\]
	As $\tau^{-1}\phi = \phi\tau'^{-1}$, we have $u\phi=\phi\theta_0(u)$ for all $u\in U_0$.
	Finally, we get two cases:
	\[\theta_x \colon U_x\to V_{x\phi}\colon u\mapsto \begin{cases}
	  	\alpha_{x\phi}^{-1}\theta_0(u^{\alpha_x^{-1}})\alpha_{x\phi} &\text{ for $x\nsim\infty$} \\[0.5ex]
	  	\gamma_{x\phi\tau'^{-1}}^{-1}\theta_\infty(u^{\gamma_{x\tau}^{-1}})\gamma_{x\phi\tau'^{-1}} &\text{ for $x\sim\infty$.}
	  \end{cases}\]
	By $u\phi=\phi\theta_\infty(u)$ for all $u\in U$ and $u\phi=\phi\theta_0(u)$ for all $u\in U_0$, we get $\alpha_x\phi = \phi\alpha_{x\phi}$ and $\gamma_{x\tau^{-1}}\phi = \phi\gamma_{x\phi\tau'^{-1}}^{-1}$. Hence for all $u\in U_x$, we get $u\phi=\phi\theta_x(u)$. As a consequence, $U_x\phi\subset \phi U_{x\phi}$, so $\phi$ is a homomorphism of local Moufang sets.
	
	If $\phi$ is also a bijection, \autoref{cor:isomorphism} shows that $\M(U,\tau)$ and $\M(V,\tau')$ are isomorphic.
\end{proof}

\subsection{Images and quotients}

After defining homomorphisms, we want to check that the image of a homomorphism induces a local Moufang subset.

\begin{proposition}
	Let $\phi$ be a homomorphism from $\M_1$ to $\M_2$. Then $\Im(\phi)$\notatlink{imphi} induces a local Moufang subset of $\M_2$ with root groups $W_y:=\Im(\theta_x)$ for all $y\in\Im(\phi)$ and $x\in X$ such that $x\phi = y$.
\end{proposition}
\begin{proof}
	We first prove that the proposed root groups are well-defined. Assume $y = x\phi = x'\phi$ for $x,x'\in X$. There is a point $z\in X$ not equivalent to $x$ and $x'$ and an element $g\in U_z$ with $xg=g'$. By \autoref{lem:diagram_thetas}, 
	\[\theta_x(u) = \theta_{x'}(u^g)\]
	for all $u\in U_x$. As $u\mapsto u^g$ is a bijection, this means $\Im(\theta_x) = \Im(\theta_{x'})$.
	
	Now we check what the induced local Moufang subset would be.
	\begin{align*}
		\{v\in V_\infty\mid 0v\in \Im(\phi)\} &= \{v\in V_\infty\mid\exists x\in X\colon 0v = x\phi\} \\
			&= \{v\in V_\infty\mid 0v = 0\alpha_x\phi\text{ for some $\alpha_x\in U_\infty$}\} \\
			&= \{\theta_\infty(u)\mid u\in U_\infty\} = \Im(\theta_\infty) = W_\infty
	\end{align*}
	As this is the image of a group homomorphism, it is a group, so \ref{axiom:S1} holds. Secondly, if $w\in W_\infty$, there is an element $u\in U_\infty$ such that
	\[x\phi w = x\phi\theta_\infty(u) = xu\phi\in\Im(\phi)\;,\]
	so $\Im(\phi)w\subset \Im(\phi)$ and \ref{axiom:S2} is satisfied.
	Finally, take any unit $x'\in X$, then $x'\phi$ is a unit in $\M_2$ and we get
	\begin{align*}
		&x\phi\mu_{x'\phi} = x\mu_{x'}\phi\in\Im(\phi) \\
		&x\phi\mu_{x'\phi}^{-1} = x\phi\mu_{-x'\phi} = x\mu_{-x'}\phi\in\Im(\phi)\;,
	\end{align*}
	so $\Im(\phi)\mu_{x'\phi} = \Im(\phi)$, proving \ref{axiom:S3}.
\end{proof}

It seems plausible that we could define the kernel of a homomorphism, and study what type of substructure would give rise to natural quotients of local Moufang sets. At this point, we avoid this and simply define quotients of local Moufang sets as the images of surjective homomorphisms:

\begin{definition}
	We call a local Moufang set $\M_2$ a \define{quotient} of another local Moufang set $\M_1$ if there is a surjective homomorphism from $\M_1$ to $\M_2$.
\end{definition}

When we first defined local Moufang sets, we also observed that there is natural Moufang set induced by every local Moufang set $\M$. We called this Moufang set the quotient Moufang set of $\M$, which seems to indicate that this is indeed a quotient.

\begin{proposition}\label{prop:quotientMS}
	Let $\M = (X,\{U_x\}_{x\in X})$ be a local Moufang set. Then the map $\pi\colon X\to \class{X}\colon x\mapsto \class{x}$ is a surjective homomorphism to the quotient Moufang set $\induced{\M}$. Hence $\induced{\M}$ is a quotient of $\M$. For all $x\in X$, the induced group homomorphism $\theta_x$ is given by $\theta_x(u) = \induced{u}$.
\end{proposition}
\begin{proof}
	As $\pi$ is clearly surjective, we only need to check if $\pi$ is a homomorphism of local Moufang sets. First observe that $x\sim x'$ if and only if $\class{x}=\class{x'}$, and this is equivalent to $x\pi\sim x'\pi$, as the equivalence on the quotient Moufang set is equality. Next take any $x\in X$ and any $u\in U_x$. Then for the induced action $\induced{u}$ of $u$ on $\class{X}$, we get $u\pi = \pi\induced{u}$, so $U_x\pi\subset \pi U_{\class{x}}$.
\end{proof}

The quotient Moufang set of a local Moufang set has more useful properties in connection to homomorphisms. For one, any homomorphism $\phi$ induces an injective homomorphism of the quotient Moufang sets:

\begin{proposition}\label{prop:induced_quotient_morphism}
	If $\phi$ is a homomorphism from $\M_1$ to $\M_2$, then $\induced{\phi}$ is an injective homomorphism from $\induced{\M_1}$ to $\induced{\M_2}$.
\end{proposition}
\begin{proof}
	By \autoref{rem:induced_inj_map}, $\induced{\phi}$ is an injective map from $\class{X}$ to $\class{Y}$. As the equivalence in this case is equality, we get $\class{x}=\class{x'}\iff\class{x}\,\induced{\phi}=\class{x'}\,\induced{\phi}$. Next, let $\induced{u}$ be in $U_{\class{x}}$ for some $\class{x}\in \class{X}$. Then for any $\class{x'}\in \class{X}$, we have
	\[\class{x'}\,\induced{u}\,\induced{\phi} = \class{x'u\phi} = \class{x'\phi\theta_x(u)} = \class{x'}\,\induced{\phi}\,\induced{\theta_x(u)}\;,\]
	so $\induced{u}\,\induced{\phi} = \induced{\phi}\,\induced{\theta_x(u)}\in\induced{\phi}V_{\class{x\phi}}$, hence $U_{\class{x}}\induced{\phi}\subset\induced{\phi}V_{\class{x\phi}}$ and $\induced{\phi}$ is a homomorphism.
\end{proof}

As a consequence of this, we find that the quotient Moufang set is the universal quotient of $\M$ in the following sense:

\begin{proposition}\label{prop:surjective_induces_iso}
	Let $\phi\colon\M_1\to\M_2$ be a surjective homomorphism. Then $\induced{\phi}\colon\induced{\M_1}\to\induced{\M_2}$ is an isomorphism such that
	\begin{center}
		\begin{tikzpicture}[regulararrow]
			\node(T1) {$\M_1$};
			\node[right of=T1, node distance=2.5cm](T2) {$\M_2$};
			\node[below of=T1, node distance=1.5cm](B1) {$\induced{\M_1}$};
			\node[right of=B1, node distance=2.5cm](B2) {$\induced{\M_2}$};
			\draw[->] (T1) -- node[above]() {$\phi$} (T2);
			\draw[->] (T1) -- node[left]() {$\pi_1$} (B1);
			\draw[->] (T2) -- node[left]() {$\pi_2$} (B2);
			\draw[->] (B1) -- node[below]() {$\induced{\phi}$} (B2);
		\end{tikzpicture}
	\end{center}
	commutes.
\end{proposition}
\begin{proof}
	By the previous proposition, $\induced{\phi}\colon \induced{\M_1}\to \induced{\M_2}$ is an injective homomorphism. But as $\phi$ is surjective, so is $\induced{\phi}$, so $\induced{\phi}$ is a bijective homomorphism. By definitions of $\pi_1$, $\pi_2$ and $\induced{\phi}$, we have
	\[(x\phi)\pi_2 = \class{x\phi} = \class{x}\,\induced{\phi} = x\pi_1\induced{\phi}\;,\]
	so the given diagram commutes.
\end{proof}

\subsection{The category \texorpdfstring{$\LMou$}{LMou}}

Now that we have a notion of homomorphisms of local Moufang sets, we can define a category:

\begin{definition}
	The \define{category of local Moufang sets} $\LMou$\notatlink{LMou} is the category with local Moufang sets as objects, and the homomorphisms as defined before as morphisms.
\end{definition}

It is of interesting to observe the relations between $\LMou$, $\Mou$\notatlink{Mou} and $\Set$\notatlink{Set}.

\begin{proposition}\preenum
	\begin{romenumerate}
		\item The mapping $I\colon\Mou\to\LMou\colon\M\mapsto\M$ adding the identity relation as equivalence relation is a full and faithful functor.
		\item The forgetful functor $F\colon\LMou\to\Set\colon \M\mapsto X$ is a faithful functor.
		\item The mapping $Q\colon\LMou\to\Mou\colon\M\mapsto\induced{\M}$ is a functor.
	\end{romenumerate}
\end{proposition}
\begin{proof}\preenum
	\begin{romenumerate}
		\item The induced action on morphisms is $I(\phi) = \phi$. Clearly, this means $I$ is faithful. Now assume $\phi\colon \M\to\M'$ is a homomorphism of local Moufang sets with $\M$ and $\M'$ Moufang sets. By \autoref{prop:hom_mouf_set}, $\phi$ is a homomorphism of Moufang sets, so $\phi$ also exists in $\Mou$. Hence $\phi$ is in the image of $I$ and $I$ is full.
		\item The induced action on morphisms is $F(\phi) = \phi$, hence $F$ is faithful.
		\item The induced action on morphisms is $Q(\phi) = \induced{\phi}$. Then $Q$ is a functor by \autoref{prop:induced_quotient_morphism}, \autoref{prop:hom_mouf_set} and \autoref{rem:induced_inj_map}.
		\qedhere
	\end{romenumerate}
\end{proof}


\section{Inverse limits}	
\subsection{Sets with equivalence relation}

Now that we have defined the category of local Moufang sets, we can take a closer look at inverse limits in this category. We consider an inverse system $(\M_i,\phi_{ij})$ of local Moufang sets. If we want to find an inverse limit of this inverse system, we first need a set with equivalence relation on which to act.

\begin{proposition}\label{prop:inverse_limit_set}
	Let $(\M_i,\phi_{ij})$\notatlink{Miphiij} be an inverse system over $I$, assuming $\M_i$ acts on the set with equivalence relation $(X_i,\sim)$. Then for $(x_i)_i,(y_i)_i\in \varprojlim X_i$\notatlink{limXi} the following are equivalent:
	\begin{romenumerate}
		\item $x_i\sim y_i$ for some $i\in I$;
		\item $x_i\sim y_i$ for all $i\in I$.
	\end{romenumerate}
	If we define $(x_i)_i\sim(y_i)_i$ if one of these conditions is satisfied, we get a set with equivalence relation $(\varprojlim X_i,\sim)$.
\end{proposition}
\begin{proof}
	Take any $(x_i)_i,(y_i)_i\in \varprojlim X_i$. Clearly, if $x_i\sim y_i$ for all $i\in I$, then $x_i\sim y_i$ for some $i\in I$. Now, assume $x_i\sim y_i$ for some $i\in I$. Take any $j\in I$. As $I$ is directed, there is an index $k\in I$ such that $k\succcurlyeq i$ and $k\succcurlyeq j$. Now $\phi_{ki}$ and $\phi_{kj}$ are homomorphisms of local Moufang sets, so
	\begin{align*}
		x_i\sim y_i 	&\iff x_k\phi_{ki}\sim y_k\phi_{ki}\iff x_k\sim y_k \\
				&\iff x_k\phi_{kj}\sim y_k\phi_{kj}\iff x_j\sim y_j\;.
	\end{align*}
	As $j$ was arbitrary, $x_j\sim y_j$ for all $j\in I$.
	
	The relation $\sim$ on $\varprojlim X_i$ is an equivalence relation since $(X_i,\sim)$ is a set with equivalence relation for all $i\in I$.
\end{proof}

Essentially, we have found a construction for the inverse limit in the category of sets with equivalence relation, where morphisms are mappings which map equivalent points to equivalent points and vice-versa.

\begin{definition}
	Let $(\M_i,\phi_{ij})$ be an inverse system in $\LMou$, where each $\M_i$ acts on $(X_i,\sim)$. We define
	\[\varprojlim (X_i,\sim) := (\varprojlim X_i,\sim)\;,\]
	with $(x_i)_i\sim(y_i)_i$ if $x_i\sim y_i$ for some $i\in I$.
\end{definition}

As was the case for sets, it is possible that $\varprojlim X_i$ is empty, so we cannot be sure that $\abs{\class{\varprojlim X_i}}>2$. This is a possible obstruction for the existence of an inverse limit of $(\M_i,\phi_{ij})$.

\subsection{Root groups and \texorpdfstring{$\mu$}{\textit{\textmugreek}}-maps}

The next thing we need for a local Moufang set is a root groups for every $x\in\varprojlim X_i$. Such $x$ is in fact a tuple $(x_i)_i$, and for each $i\succcurlyeq j$, the map $\phi_{ij}$ induces a map $\theta_{x_i,x_j}\colon U_{x_i}\to U_{x_j}$. With these maps, $(U_{x_i},\theta_{x_i,x_j})$ becomes an inverse system of groups. The inverse limit of these groups will be our candidate root group.

\begin{proposition}\label{prop:inverse_limit_Ux}
	Let $(\M_i,\phi_{ij})$ be an inverse system over $I$. For $x=(x_i)_i\in\varprojlim X_i$, we define
	\[U_x := \varprojlim U_{x_i}\;,\]
	the inverse limit of the inverse system $(U_{x_i},\theta_{x_i,x_j})$. Then the following hold:
	\begin{romenumerate}
		\item $U_x$ acts equivalence-preserving on $\varprojlim X_i$;
		\item $U_x$ acts sharply transitively on $\varprojlim X_i\setminus \class{x}$;\label{itm:inverse_limit_Ux2}
		\item the induced action of $U_x$ on $\class{\varprojlim X_i}$ is sharply transitive on $\class{\varprojlim X_i}\setminus\{\class{x}\}$.
	\end{romenumerate}
\end{proposition}
\begin{proof}\preenum
	\begin{romenumerate}
		\item The action of $(u_i)_i$ on $\prod_i X_i$ is given by $(y_i)_i(u_i)_i := (y_iu_i)_i$. Now if $(y_i)_i\in \varprojlim X_i$, we have
			\[y_iu_i\phi_{ij} = y_i\phi_{ij}\theta_{x_i,x_j}(u_i) = y_ju_j\]
			for all $i\succcurlyeq j$, so $(y_i)_i(u_i)_i\in \varprojlim X_i$. Next, assume $(y_i)_i\sim(z_i)_i$, then there is an index $i\in I$ for which $y_i\sim z_i$, hence $y_iu_i\sim z_iu_i$, so $(y_iu_i)_i\sim(z_iu_i)_i$, which means $(u_i)_i$ preserves the equivalence.
		\item Now, take $(y_i)_i,(z_i)_i\in\varprojlim X_i\setminus\class{x}$. For each $i\in I$, we know $y_i,z_i\in X_i\setminus\class{x_i}$, so there 		is a unique $u_i\in U_{x_i}$ such that $y_iu_i=z_i$. Hence, if there is an element mapping $(y_i)_i$ to $(z_i)_i$, it must be $(u_i)_i$. 		We only need to check that $(u_i)_i\in U_x$, in particular, we need to check $\theta_{x_i,x_j}(u_i)=u_j$ for all $i\succcurlyeq j$. But
			\[y_iu_i\phi_{ij} = z_i\phi_{ij} = z_j = y_ju_j = y_i\phi_{ij}u_j\;,\]
			so as $y_i\nsim x_i$, \autoref{lem:theta_maps}\ref{itm:unique_theta} shows that indeed $\theta_{x_i,x_j}(u_i)=u_j$. Hence we found a unique element of $U_x$ mapping $(y_i)_i$ to $(z_i)_i$, so $U_x$ acts sharply transitively on $\varprojlim X_i\setminus\class{x}$.
		\item By \ref{itm:inverse_limit_Ux2}, the induced action of $U_x$ on $\class{\varprojlim X_i}$ is transitive on
			$\class{\varprojlim X_i}\setminus\{\class{x}\}$. Hence, we only need to check the sharpness of the transitivity. By \autoref{prop:inverse_limit_set}, we have $\class{(y_i)_i} = (\class{y_i})_i$ for all $(y_i)_i\in \varprojlim X_i$, which implies that $\induced{(u_i)_i} = (\induced{u_i})_i$ for all $(u_i)_i\in U_x$.
			Now suppose we have two elements $\class{(y_i)_i},\class{(z_i)_i}\in \class{\varprojlim X_i}\setminus\{\class{x}\}$ and two distinct elements $\induced{(u_i)_i},\induced{(v_i)_i}\in \induced{U_x}$ mapping $\class{(y_i)_i}$ to $\class{(z_i)_i}$.
			As $\induced{(u_i)_i}\neq\induced{(v_i)_i}$, there is an $i\in I$ for which $\induced{u_i}\neq\induced{v_i}$, which means there are two distinct elements of $\induced{U_{x_i}}$ mapping $\class{y_i}$ to $\class{z_i}$, contradicting the fact that $\M_i$ is a local Moufang set. Hence the induced action of $U_x$ on $\class{\varprojlim X_i}$ on $\class{\varprojlim X_i}\setminus\{\class{x}\}$ is sharply transitive.
		\qedhere
	\end{romenumerate}
\end{proof}

By the definition of the $\theta$-maps, we could have the following alternative definition for $U_x$:
\begin{align}
	U_x &= \left\{(u_i)_i\in\prod_i U_{x_i}\;\middle\vert\; u_i\phi_{ij}=\phi_{ij}u_j\text{ for all $i\succcurlyeq j$} \right\}\;. \label{eq:inverse_limit_Ux_alt}
\end{align}

Now, we will need to assume that $\abs{\class{\varprojlim X_i}}>2$. We will fix non-equivalent points $0$ and $\infty$ in $\varprojlim X_i$, and construct what should be the $\mu$-maps. To simplify matters, we can choose $0_i$ and $\infty_i$ in each $X_i$ in such a way that $0 = (0_i)_i$ and $\infty = (\infty_i)_i$.

\begin{proposition}
	Let $(\M_i,\phi_{ij})$ be an inverse system over $I$ and assume $\abs{\class{\varprojlim X_i}}>2$. Fix a basis $(0,\infty)$ in $\varprojlim X_i$ and assume $0 = (0_i)_i$ and $\infty = (\infty_i)_i$.
	\begin{romenumerate}
		\item If $x=(x_i)_i$ is a unit, the element $\mu_x:=(\mu_{x_i})_i$ is in $U_0\alpha_xU_0$ and $U_\infty^{\mu_x} = U_0$, where $\alpha_x$ is the unique element of $U_\infty$ mapping $0$ to $x$.\label{itm:inverse_limit_mu}
	\end{romenumerate}
	Now fix a unit $e=(e_i)_i$ and set $\tau:=\mu_e$.
	\begin{romenumerate}[resume]
		\item $U_\infty$ and $\tau$ satisfy \hyperref[axiom:C1]{\textnormal{(C1-2)}}.
		\item For $x=(x_i)_i\nsim0,\infty$, the element $\mu_x:=(\mu_{x_i})_i$ is the $\mu$-map in \autoref{constr:MUtau} for $\M(U_\infty,\tau)$.
	\end{romenumerate}
\end{proposition}
\begin{proof}\preenum
	\begin{romenumerate}
		\item By \autoref{prop:mu_defin}, for each $i\in I$, we have $\mu_{x_i} = g_i\alpha_{x_i}h_i$ with $g_i\in U_{0_i}$ mapping $\infty_i$ to $-x_i$ and $h_i\in U_{0_i}$ mapping $x_i$ to $\infty_i$. Now set $g = (g_i)_i$ and $h = (h_i)_i$. Now for all $i\succcurlyeq j$, we get 
		\[\infty_ig_i\phi_{ij} = (-x_i)\phi_{ij} = -(x_i\phi_{ij}) = -x_j = \infty_jg_j = \infty_i\phi_{ij}g_j\;,\]
		so by \autoref{lem:theta_maps}\ref{itm:unique_theta} we find $g_i\phi_{ij}=\phi_{ij}g_j$. This shows $g\in U_0$. A similar argument shows $h\in U_0$ and $(\alpha_{x_i})_i \in U_\infty$. Finally $0(\alpha_{x_i})_i = (x_i)_i = x$, so we get $(\alpha_{x_i})_i = \alpha_x$.
		
		By \autoref{lem:theta_maps}\ref{itm:morphism_mu}, $\mu_{x_i}\phi_{ij} = \phi_{ij}\mu_{x_j}$ for all $i\succcurlyeq j$. Hence, for any $u = (u_i)_i\in U_\infty$ we get
		\[\mu_{x_i}^{-1}u_i\mu_{x_i}\phi_{ij} = \phi_{ij}\mu_{x_j}^{-1}u_j\mu_{x_j}\;,\]
		so since $u_i^{\mu_{x_i}}\in U_{0_i}$, we have $u^{\mu_x}\in U_0$. This means $U_\infty^{\mu_x} \subset U_0$. As we can repeat the argument with $\mu_{-x} = \mu_x^{-1}$, we also get the other inclusion.
		\item The conditions \ref{axiom:C1} and \ref{axiom:C1'} are part of \autoref{prop:inverse_limit_Ux}. The third condition \ref{axiom:C2} is satisfied as by \ref{itm:inverse_limit_mu}, $\infty\tau = 0\nsim\infty$ and $\infty\tau^2=\infty$.
		\item By \autoref{remark:mu_hua_construction}, a $\mu$-map constructed in \autoref{constr:MUtau} is still the unique element of $U_0\alpha_xU_0$ interchanging $0$ and $\infty$, and as $(\mu_{x_i})_i$ is such an element, it must coincide with $\mu_x$ from the construction.\qedhere
	\end{romenumerate}
\end{proof}

\subsection{Inverse limit of local Moufang sets}

We can now use \autoref{constr:MUtau} to define the inverse limit of an inverse system of local Moufang sets, when it exists.

\begin{theorem}\label{thm:existence_inverse_limit}
	Let $(\M_i,\phi_{ij})$\notatlink{Miphiij} be an inverse system over $I$ and assume $\abs{\class{\varprojlim X_i}}>2$. Fix a basis $(0,\infty)$ in $\varprojlim X_i$ and assume $0 = (0_i)_i$ and $\infty = (\infty_i)_i$. Set $\tau:=\mu_e$ for some $e\neq0,\infty$ and $U:=\varprojlim U_{\infty_i}$.
	Then $\M(U,\tau)$ is a local Moufang set with for each $x=(x_i)_i$ as root group
	\[U_x := \varprojlim U_{x_i}.\]
	Finally, $\M(U,\tau)$ is the inverse limit $\varprojlim \M_i$\notatlink{limMi} of $(\M_i,\phi_{ij})$ with projection maps 
	\[p_j\colon \varprojlim X_i\to X_j\colon(x_i)_i\mapsto x_j\;.\]
\end{theorem}
\begin{proof}
	By the previous proposition, we can construct $\M(U,\tau)$ and $U^{\mu_x} = U^\tau$ for all units $x$. Hence, by \autoref{cor:construction_equivalentconditions}, $\M(U,\tau)$ is a local Moufang set. This means that we only need to check that it is the inverse limit of $(\M_i,\phi_{ij})$.
	
	We first prove that the maps $p_j$ are homomorphisms of local Moufang sets. By \autoref{prop:inverse_limit_set}, $x\sim y\iff xp_j\sim yp_j$. Next, take $x=(x_i)_i\in\varprojlim X_i$, and take any $u=(u_i)_i\in U_x$. Then one can check that $up_j = p_ju_j$, so we have $U_xp_j\subset p_j U_{x_j}$. This means that $p_j$ is indeed a homomorphism. Finally, as the maps $p_i$ are the projection maps for the inverse limit in $\Set$, we have $\phi_{ij}\circ p_i = p_j$ in $\Set$. By the faithfulness of the forgetful functor $F\colon\LMou\to\Set$, this identity lifts to $\LMou$. This means \ref{axiom:IL1} is satisfied.
	
	We still need to check the universal property \ref{axiom:IL2}. Assume we have a local Moufang set $\M$ with morphisms $q_i\colon\M\to\M_i$ for all $i\in I$ such that $\phi_{ij}\circ q_i = q_j$ for all $i\succcurlyeq j$. We denote the set with equivalence relation of $\M$ by $X$. We need to define a map $\psi\colon X\to\varprojlim X_i$ such that $xq_i = x\psi p_i$. If $x\psi = (y_i)_i$, then $x\psi p_i = y_i$. This means $y_i$ is uniquely determined to be $xq_i$. Hence if there is a homomorphism $\psi$ such that $q_i =  p_i\circ\psi$, it must be given by
	\[\psi\colon\M\to\varprojlim\M_i\colon x\mapsto (xq_i)_i\;.\]
	We need to check that this map is indeed a homomorphism. Firstly, $xq_i\phi_{ij} = xq_j$ for all $i\succcurlyeq j$, so $\psi$ indeeds maps into $\varprojlim X_i$. Secondly, $\psi$ must preserve equivalence and non-equivalence. We can take any $i\in I$, and find for all $x,y\in X$:
	\[x\sim y \iff xq_i\sim yq_i \iff (xq_i)_i\sim (yq_i)_i\iff x\psi\sim y\psi\;.\]
	Finally, we need $\psi$ to send root groups to root groups. Take any $x\in X$ and look at the root group $V_x$ of $x$ in $\M$. Then as each $q_i$ is a homomorphism, there are maps $\theta_{x,i}\colon V_x\to U_{xq_i}$ such that $vq_i = q_i\theta_{x,i}(v)$ for all $v\in V_x$. We now claim that $v\psi = \psi\big(\theta_{x,i}(v)\big)_i$. Indeed, take any $y\in X$, then
	\[yv\psi = (yvq_i)_i = (yq_i\theta_{x,i}(v))_i = (yq_i)_i\big(\theta_{x,i}(v)\big)_i = y\psi\big(\theta_{x,i}(v)\big)_i\;.\]
	The last thing we need to check is if $\big(\theta_{x,i}(v)\big)_i$ is in $U_{x\psi}$. Take any $y\in X\setminus\class{x}$, then $yq_i\nsim xq_i$. Furthermore
	\[yq_i\theta_{x,i}(v)\phi_{ij} = yvq_i\phi_{ij} = yvq_j = yq_j\theta_{x,j}(v) = yq_i\phi_{ij}\theta_{x,j}(v)\;.\]
	By \autoref{lem:theta_maps}\ref{itm:unique_theta}, this implies $\theta_{x,i}(v)\phi_{ij}=\phi_{ij}\theta_{x,j}(v)$, and this holds for all $i\succcurlyeq j$. Hence we find $\big(\theta_{x,i}(v)\big)_i\in U_{x\psi}$.
\end{proof}

All in all, we found that an inverse system of local Moufang sets has an inverse limit whenever $\varprojlim X_i$ has more than $2$ equivalence classes. One useful sufficient condition for this is the following:

\begin{proposition}\label{prop:inverse_limit_surjective}
	Let $(\M_i,\phi_{ij})$ be an inverse system over $I$. If there is some $j\in I$ for which the map $p_j\colon\varprojlim X_i\to X_j\colon (x_i)_i\to x_j$ is surjective, then we have $\abs{\class{\varprojlim X_i}}>2$, and hence $\varprojlim \M_i$ exists.
\end{proposition}
\begin{proof}
	Denote $\pi_j$ for the natural surjection $\M_j\to\induced{\M_j}\colon x\to\class{x}$. If $p_j$ is surjective, then also $\pi_j\circ p_j$ is surjective. But as $\M_j$ is a local Moufang set, $\abs{\class{X_j}}>2$. By \autoref{prop:inverse_limit_set}, $p_j$ preserves equivalence, and by definition $\pi_j$ also preserves equivalence. Hence $\pi_j\circ p_j$ maps equivalence classes of $\varprojlim X_i$ to equivalence classes of $\class{X_j}$. But equivalence classes in $\class{X_j}$ are singletons. Hence, as $\pi_j\circ p_j$ is surjective, $\abs{\class{X_j}}>2$ implies that $\abs{\class{\varprojlim X_i}}>2$.
\end{proof}

\subsection{Some special cases}

As in the categories of sets, groups..., it is common for an inverse system to be surjective (see \autoref{def:surj_inv_system}). In this case, the quotient Moufang sets of all local Moufang sets are isomorphic.

\begin{proposition}\label{prop:surjective_inv_sys}
	Let $(\M_i,\phi_{ij})$ be a surjective inverse system over $I$. Then $\induced{\M_i}\cong\induced{\M_j}$ for all $i,j\in I$. If we fix some $i\in I$, then we can define surjective homomorphisms $\rho_j\colon\M_j\to\induced{\M_i}$ such that $\rho_k\circ\phi_{jk} = \rho_j$ for all $j\succcurlyeq k$.
\end{proposition}
\begin{proof}
	If $i\succcurlyeq j$, the surjective homomorphism $\phi_{ij}$ induces an isomorphism $\induced{\phi_{ij}}$ between $\induced{\M_i}$ and $\induced{\M_j}$ by \autoref{prop:surjective_induces_iso}, so if $i$ and $j$ are comparable, $\induced{\M_i}\cong\induced{\M_j}$. Now if $i$ and $j$ are not comparable, there is an index $k\in I$ such that $k\succcurlyeq i$ and $k\succcurlyeq j$. By the previous case, we get $\induced{\M_i}\cong\induced{\M_k}\cong\induced{\M_j}$.
	
	For any $i\in I$, we denote the natural projection of $\M_i$ onto the quotient Moufang set $\induced{\M_i}$ by $\pi_i$. Now, fix $i\in I$ and take any $j\in I$. As $I$ is directed, there is an index $k\in I$ such that $k\succcurlyeq i$ and $k\succcurlyeq j$. Now both $\induced{\phi_{ki}}$ and $\induced{\phi_{kj}}$ are isomorphisms, so we can define $\rho_j := \induced{\phi_{ki}}\circ\induced{\phi_{kj}}^{\,-1}\circ\pi_j$. A priori, the map $\rho_j$ can depend on the choice of $k$, but we claim this is not the case. If $k'\succcurlyeq i$ and $k'\succcurlyeq j$, we could find $\ell\in I$ such that $\ell\succcurlyeq k$ and $\ell\succcurlyeq k'$, and then we get
	\begin{align*}
		\induced{\phi_{ki}}\circ\induced{\phi_{kj}}^{\,-1} &= \induced{\phi_{ki}}\circ\induced{\phi_{\ell k}}\circ\induced{\phi_{\ell k}}^{\,-1} \circ\induced{\phi_{kj}}^{\,-1} \\
		&= \induced{\phi_{ki}\circ\phi_{\ell k}}\circ\induced{\phi_{kj}\circ\phi_{\ell k}}^{\,-1} = \induced{\phi_{\ell i}}\circ\induced{\phi_{\ell j}}^{\,-1}
	\end{align*}
	and similarly $\induced{\phi_{k'i}}\circ\induced{\phi_{k'j}}^{\,-1} = \induced{\phi_{\ell i}}\circ\induced{\phi_{\ell j}}^{\,-1}$, so we find
	\[\induced{\phi_{ki}}\circ\induced{\phi_{kj}}^{\,-1}\circ\pi_j = \induced{\phi_{k'i}}\circ\induced{\phi_{k'j}}^{\,-1}\circ\pi_j\;,\]
	and hence $\rho_j$ does not depend on the choice of $k$. As $\rho_j$ is the composition of isomorphisms and one surjection, it is surjective.
	
	Finally, assume $j\succcurlyeq k$. Now there is an index $\ell\in I$ such that $\ell\succcurlyeq j$ and $\ell\succcurlyeq i$. Hence also $\ell\succcurlyeq k$ and $\phi_{\ell k} = \phi_{jk}\circ\phi_{\ell j}$. We find
	\begin{align*}
		\rho_j  &= \induced{\phi_{\ell i}}\circ\induced{\phi_{\ell j}}^{\,-1}\circ\pi_j 
			= \induced{\phi_{\ell i}}\circ\induced{\phi_{\ell j}}^{\,-1}\circ\induced{\phi_{jk}}^{\,-1}\circ\induced{\phi_{jk}}\circ\pi_j \\
			&= \induced{\phi_{\ell i}}\circ\induced{\phi_{\ell k}}^{\,-1}\circ\pi_k\circ\phi_{jk}
			= \rho_k\circ\phi_{jk}\;,
	\end{align*}
	proving the final statement.
\end{proof}

This means we could extend $(\M_i,\phi_{ij})$ by adding a minimal element to the inverse system. Essentially, we can add a new symbol $0$ to $I$ such that $j\succcurlyeq 0$ for all $j\in I$, and defining $\M_0 := \induced{\M_i}$ and $\phi_{j0} := \rho_j$ for all $j\in I$.

Now, if we also assume $I$ to have a cofinal sequence, then $\varprojlim \M_i$ always exists:

\begin{theorem}\label{thm:cofinal_inverse_limit}
	Let $(\M_i,\phi_{ij})$ be a surjective inverse system over a directed set $I$ with a cofinal sequence $i_1\preccurlyeq i_2\preccurlyeq\cdots$. Then \[p_j\colon \varprojlim X_i\to X_j\]
	is surjective for all $j\in I$. In particular, $\abs{\class{\varprojlim X_i}}>2$ and hence $\varprojlim \M_i$ exists.
\end{theorem}
\begin{proof}
	The surjectivity of $p_j$ follows from \autoref{thm:inverse_limit_surjections}. By \autoref{prop:inverse_limit_surjective} we now get the existence of $\varprojlim \M_i$.
\end{proof}

In particular, if $I = \{1,2,3,\ldots\}$ with the usual order, surjective inverse systems always have inverse limits.
	\chapter[Special local Moufang sets]{Special local\\ Moufang sets}\label{chap:chap5_special}
	Our final theoretical chapter introduces special local Moufang sets, similar to special Moufang sets. We are able to determine conditions that ensure unique $k$-divisibility of elements of the root groups, as well as the entire root groups. We also study special local Moufang sets with abelian root groups, showing that $\mu$-maps are involutions if there are enough classes of units. Finally, we introduce a paired structure in local Moufang sets, which is a first step towards constructing a Jordan pair in \autoref{chap:chap7_jordan}.

\section{Definition and \texorpdfstring{$k$}{\emph{k}}-divisibility}	
\subsection{Definition and immediate consequences}

We recall that a Moufang set is called special if $\til x = -x$ for all $x \in U^*$.
This property was introduced in the context of abstract rank one groups by F.~Timmesfeld \cite[p.~2]{TimmesfeldAR1G},
and has been thoroughly investigated for Moufang sets \cite{DMSegevTentSpecialMS}.
It is a difficult open problem whether special Moufang sets always have abelian root groups.
The converse, namely that proper Moufang sets with abelian root groups are always special, has been shown by Y.~Segev \cite{SegevAbelianSpecial}. We now generalize this notion to the theory of local Moufang sets. In this chapter, we will always assume we have a local Moufang set $\M$ with fixed basis $(0,\infty)$ and a fixed $\tau$, along with all the notation we introduced in \autoref{chap:chap2_definitions}.

\begin{definition}
	A local Moufang set $\M$ is called \define{special} if for all units $x\in X$, $\til x=-x$.
\end{definition}

This is equivalent to
\[(-x)\tau=-(x\tau)\text{ for all units $x\in X$.}\]
Some basic properties follow immediately from \autoref{lem:mu}:

\begin{proposition}\label{prop:specialmu}
	Let $x\in X$ be a unit in a special local Moufang set. Then
	\begin{romenumerate}
		\item $(-y)\mu_x = -(y\mu_x)$ for all units $y\in X$;
		\item $\mu_x = \alpha_x\alpha_{-x\tau^{-1}}^\tau\alpha_x$;\label{prop:specialmu-ii}
		\item $-x = x\mu_x = x\mu_{-x}$;\label{prop:specialmu-iii}
		\item $\mu_x = \alpha_x\alpha_x^{\mu_{\pm x}}\alpha_x$;
		\item $\mu_{-x} = \alpha_x\mu_{-x}\alpha_x\mu_{-x}\alpha_x$\label{prob:specialmu-mu-x}.
	\end{romenumerate}
\end{proposition}
\begin{proof}\preenum
	\begin{romenumerate}
		\item Immediate from the fact that $\til y$ does not depend on the choice of $\tau$ and the definition of special.
		\item This follows immediately from \autoref{lem:mu}\ref{itm:muform2}.
		\item By \autoref{lem:mu}\ref{itm:tilmu}, we have $-x = \til x = -((-x)\mu_x)$, so $x=(-x)\mu_x$ and hence $-x = x\mu_{-x}$. Secondly, we have
		\[x = \til(-x) = (-(-x)\mu_{-x})\mu_x\;,\]
		so $-x = x\mu_{-x} = -(-x)\mu_{-x}$, and hence $-x = x\mu_x$.
		\item Take $\tau = \mu_{\pm x}$ in \ref{prop:specialmu-ii}, and use \ref{prop:specialmu-iii}.
		\item This is a consequence of \autoref{lem:mu}\ref{itm:muform3} and the definition of special.
		\qedhere
	\end{romenumerate}
\end{proof}

Many identities involving $\mu$-maps become simplified for special local Moufang sets. Another important identity for special local Moufang sets is the following.

\begin{lemma}\label{lem:specialsum}
	If $x,y\in X$ are units in a special local Moufang set, and $x\alpha_y$ is a unit. Then
	\[x\mu_{x\alpha_y} = (-y)\alpha_{-x}\alpha_{x\mu_y}\alpha_{-y}\;.\]
\end{lemma}
\begin{proof}
	Let $z=x'\alpha_{-y'}\mu_{y'}\alpha_{\til y'}$ as in \autoref{prop:sumform}. By that proposition, and since we are in a special local Moufang set, we have
	\begin{align*}
		z &= x'\alpha_{-y'}\mu_{y'}\alpha_{-y'} = -\til z = -(y'\alpha_{-x'}\mu_{x'}\alpha_{-x'}) \\
		 &= 0\cdot(\alpha_{'y\alpha_{-x'}\mu_{x'}}\alpha_{-x'})^{-1} = x'\alpha_{-(y'\alpha_{-x'}\mu_{x'})} \\
		 &= x'\alpha_{(-(y'\alpha_{-x'}))\mu_{x'}} = x'\alpha_{x'\alpha_{-y'}\mu_{x'}}
	\end{align*}
	Hence we have $x'\alpha_{-y'}\mu_{y'} =x'\alpha_{x'\alpha_{-y'}\mu_{x'}}\alpha_{y'}$ for units $x',y'\in X$ such that $x'\nsim y'$.
	
	If we now set $x' = x\alpha_y$ and $y' = y$, note that if $x'\sim y'$, we would have $x\sim 0$, so $x$ would not be a unit. We get
	\[x\mu_y = x\alpha_y\alpha_{x\mu_{x\alpha_y}}\alpha_y\quad\text{ so }\quad0\cdot\alpha_{x\mu_y} = 0\cdot\alpha_x\alpha_y\alpha_{x\mu_{x\alpha_y}}\alpha_y\;,\]
	hence $\alpha_{x\mu_y} = \alpha_x\alpha_y\alpha_{x\mu_{x\alpha_y}}\alpha_y$. By rearranging, we get 
	\[\alpha_{x\mu_{x\alpha_y}}=\alpha_{-y}\alpha_{-x}\alpha_{x\mu_y}\alpha_{-y}\;,\]
	which gives the desired formula after applying it to $0$.
\end{proof}

\subsection{Unique \texorpdfstring{$k$}{\emph{k}}-divisibility}

We study the $k$-divisibility of elements, similarly to what has been done before for (ordinary) Moufang sets in \cite[Proposition~4.6]{DMSegevIdentitiesMS}.

\begin{definition}
	For $x\in X\setminus\class{\infty}$ and $k\geq1$, we define $x\cdot k:=0\cdot\alpha_x^k$\notatlink{xn}. We call $y\in X\setminus\class{\infty}$ \define[k-divisible@$k$-divisible]{$k$-divisible} if there is an $x$ such that $y=x\cdot k$. If there is a unique such $x$, we call $y$ \define[k-divisible@$k$-divisible!uniquely k-divisible@uniquely $k$-divisible]{uniquely} $k$-divisible.
\end{definition}

Clearly, $(\alpha_x^k)^{-1} = \alpha_{-x}^k$, so $-(x\cdot k) = (-x)\cdot k$.

\begin{lemma}\label{lem:timesn_sim}
	Let $k\in\N$ and $x,y\in X\setminus\class{\infty}$. If $x\sim y$, then $x\cdot k\sim y\cdot k$.
\end{lemma}
\begin{proof}
	First observe that by \ref{axiom:LM2'} we have
	\begin{align*}
		x\sim y &\iff 0\alpha_x\alpha_y^{-1}\sim 0 \iff \class{0}\induced{\alpha}_x\induced{\alpha}_y^{-1} = \class{0} \\
			&\iff \induced{\alpha}_x\induced{\alpha}_y^{-1} = \id \iff \induced{\alpha_x} = \induced{\alpha_y}\;.
	\end{align*}
	Hence we get
	\begin{align*}
		x\sim y &\iff \induced{\alpha_x} = \induced{\alpha_y} \\
			&\implies \induced{\alpha_x}^n = \induced{\alpha_y}^n \\
			&\iff \induced{\alpha_{x\cdot n}} = \induced{\alpha_{y\cdot k}} \iff x\cdot k\sim y\cdot k\;,
	\end{align*}
	so indeed $x\sim y\implies x\cdot k\sim y\cdot k$.
\end{proof}

The following lemma shows a sufficient condition for an element to be (uniquely) $k$-divisible. The proof is by induction, and this induction requires the technical conditions to work.

\begin{lemma}\label{lem:ndiv}
	Let $n\in\N$ and $x$ be a unit in a special local Moufang set. If $x\cdot k$ is a unit for all $1\leq k\leq n$, then the following hold:
	\begin{romenumerate}
		\item $(x\cdot k)\mu_{-x}\cdot k = -x$ for all $1\leq k\leq n$.\label{lem:ndiv-i}
		\item $x\tau\cdot k$ is a unit for all $1\leq k\leq n$.\label{lem:ndiv-ii}
		\item $(x\cdot k)\tau\cdot k = x\tau$ for all $1\leq k\leq n$.\label{lem:ndiv-iii}
		\item For all $1\leq k\leq n$, $y_k := (-x\cdot k)\mu_{-x}$ is the unique element such that $y_k\cdot k = x$, and $y_k\cdot \ell$ is a unit for all $\ell\leq n$.\label{lem:ndiv-iv}
		\item $(x\cdot n)\cdot k$ is a unit for all $1\leq k\leq n$.\label{lem:ndiv-v}
	\end{romenumerate}
\end{lemma}
\begin{proof}
	We prove all these statements simultaneously by induction on $n$. Observe that they clearly hold for $n=1$, using \autoref{prop:specialmu}. We now assume the following:
	\[\text{the lemma holds for $n$ and all $x$ satisfying the conditions}\tag{IH}\label{IH}\]
	and prove the lemma for $n+1$ and all $x$ satisfying the conditions. Hence, we now assume we have some $x$ such that $x\cdot k$ is a unit for $1\leq k\leq n+1$.

	We first claim the following:
	\begin{align}\label{eq:bn(n+1)k-unit}
		y_n\cdot(n+1)\cdot k\text{ is a unit for all $1\leq k\leq n$.}
	\end{align}
	Suppose this were not the case; then there is some $1\leq k\leq n$ for which $y_n\cdot(n+1)\cdot k\sim 0$, so
	\[x\cdot(n+1)\cdot k = y_n\cdot(n+1)\cdot k\cdot n\sim 0\;.\]
	From this we get $x\cdot n\cdot k\sim-x\cdot k$, and hence 
	\[(x\cdot n\cdot k)\mu_x\cdot k\sim(-x\cdot k)\mu_x\cdot k\;.\]
	Using \ref{lem:ndiv-v} and \ref{lem:ndiv-iii} of the induction hypothesis, we get
	\[(x\cdot n)\mu_x \sim (-x)\mu_x \text{, so } x\cdot n\sim -x \text{ and hence } x\cdot(n+1)\sim 0\;,\]
	which contradicts the assumption.

	Now we prove \ref{lem:ndiv-i}. By induction, $(x\cdot k)\mu_{-x}\cdot k = -x$ for all $1\leq k\leq n$, so we only need to show this for $k=n+1$. We have
	\begin{align*}
		&\mathord{-}(x\cdot(n+1))\mu_{-x} \\
		&= (-x\cdot(n+1))\mu_{-x} \\
		&= (-x\cdot n)\alpha_{-x}\mu_{-x} \\
		&=(-x\cdot n)\mu_{-x}\alpha_x\mu_{-x}\alpha_x & \text{(by \autoref{prop:specialmu}\ref{prob:specialmu-mu-x})} \\
		&= y_n\alpha_x\mu_{-x}\alpha_x & \text{(by \eqref{IH})} \\ 
		&= y_n\alpha_{y_n}^n\mu_{-x}\alpha_x \\
		&= (y_n\cdot(n+1))\mu_{-x}\alpha_x \\
		&= \bigl((y_n\cdot(n+1)\cdot n)\mu_{-x}\cdot n\bigr)\alpha_x & \text{(by \eqref{IH} and \eqref{eq:bn(n+1)k-unit})} \\ 
		&= \bigl((x\cdot(n+1))\mu_{-x}\cdot n\bigr)\alpha_x\;.
	\end{align*}
	Hence $\alpha_{(x\cdot(n+1))\mu_{-x}}^{-1} = \alpha_{(x\cdot(n+1))\mu_{-x}}^n\alpha_x$, so indeed
    \[ -x = 0\cdot \alpha_{-x} = 0\cdot\alpha_{(x\cdot(n+1))\mu_{-x}}^{n+1} = (x\cdot(n+1))\mu_{-x}\cdot(n+1)\;. \]
	Next we prove \ref{lem:ndiv-ii}, where again, we only need to check that $x\tau\cdot(n+1)$ is a unit. By \autoref{lem:huaAut}, $(x\cdot n)\mu\tau = x\mu\tau\cdot n$ for any $\mu$-maps $\tau$ and $\mu$. Hence
	\[ x\tau\cdot(n+1) = (-x)\mu_{-x}\tau\cdot(n+1) = (-x)\cdot(n+1)\mu_{-x}\tau\nsim 0\;, \]
	as $x\cdot(n+1)\nsim 0$.

	Similarly we prove \ref{lem:ndiv-iii} using \ref{lem:ndiv-i} and \autoref{lem:huaAut}.
	\begin{align*}
		(x\cdot k)\tau\cdot k 	&= (x\cdot k)\mu_{-x}\mu_x\tau\cdot k = \bigl((x\cdot k)\mu_{-x}\cdot k\bigr)\mu_x\tau \\
					&= -x\mu_x\tau = x\tau\;.  
	\end{align*}

	To prove \ref{lem:ndiv-iv}, we first observe that $y_{n+1} = (-x\cdot(n+1))\mu_{-x}$ indeed satisfies $y_{n+1}\cdot(n+1) = x$, by \ref{lem:ndiv-iii}. We first show the following statement.
	\begin{align}\label{eq:ck-unit}
		\text{If $z\cdot(n+1)=x$, then $z\cdot k$ is a unit for all $1\leq k\leq n+1$.}
	\end{align}
	Indeed, if $z\cdot k\sim 0$, then also $z\cdot k\cdot(n+1)\sim 0$, so $x\cdot k\sim 0$, contradicting the fact that $x\cdot k$ is a unit for all $k\leq n+1$.

	Now we prove that $y_{n+1}$ is unique. Suppose $z\cdot(n+1)=x$, then
	\begin{align*}
		-x\cdot(n+1) 	&= x\mu_{-x}\cdot (n+1) = (z\cdot(n+1))\mu_{-x}\cdot (n+1) \\
				&= z\mu_{-x}\;,
	\end{align*}
	using \ref{lem:ndiv-iii} for $z$, which is allowed by \eqref{eq:ck-unit}. Hence, $z=y_{n+1}$, and indeed $y_{n+1}$ is unique. By~\eqref{eq:ck-unit}, $y_{n+1}\cdot k$ is a unit for $1\leq k\leq n+1$. We only need to show that $y_k\cdot(n+1)$ is a unit for $1\leq k\leq n$. Suppose $y_k\cdot(n+1)\sim 0$, then also $x\cdot(n+1) = y_k\cdot(n+1)\cdot k\sim 0$, which is a contradiction.

	It only remains to show \ref{lem:ndiv-v}, which we do in two steps. First, if $x\cdot(n+1)\cdot k\sim 0$ for some $1\leq k\leq n$, we would have
	\begin{align*}
		& x\cdot n\cdot k \sim -x\cdot k \\
		&\implies (x\cdot n\cdot k)\mu_x\cdot k \sim (-x\cdot k)\mu_x\cdot k \\
		&\implies (x\cdot n)\mu_x \sim (-x)\mu_x &\text{(by \eqref{IH} and \ref{lem:ndiv-iii})}\\
		&\implies x\cdot n \sim -x \\
		&\implies x\cdot(n+1)\sim 0,
	\end{align*}
	which is a contradiction; so $x\cdot(n+1)\cdot k$ is a unit for $1\leq k\leq n$.

	Now if $x\cdot(n+1)\cdot(n+1)\sim 0$, we would have
	\begin{align*}
		& \mathord{-}x\cdot(n+1)\cdot n \sim x\cdot(n+1) \\
		&\implies -(x\cdot(n+1)\cdot n)\mu_x\cdot n \sim (x\cdot(n+1))\mu_x\cdot n \\
		&\implies -(x\cdot(n+1))\mu_x \sim (x\cdot(n+1))\mu_x\cdot n \\
		&\implies (x\cdot(n+1))\mu_x\cdot(n+1)\sim 0 \\
		&\implies -x\sim 0 ,
	\end{align*}
	which is again a contradiction. This shows that $x\cdot(n+1)\cdot k$ is a unit for $1\leq k\leq n+1$.

	By induction, the lemma now holds for all $n$.
\end{proof}

The essence of the previous lemma is contained in the following corollary.

\begin{corollary}\label{cor:ndiv}
	Let $\M$ be a special local Moufang set and assume $x\cdot k$ is a unit for all $1\leq k\leq n$.
	\begin{romenumerate}
		\item there is a unique $y$ such that $y\cdot n = x$, which we denote by $x\cdot\frac{1}{n}$\notatlink{gdivk2}\notatlink{gdivk};
		\item $(x\cdot n)\tau = x\tau\cdot\frac{1}{n}$ and $\bigl(x\cdot\frac{1}{n}\bigr)\tau = x\tau\cdot n$;
		\item if $z\sim x$, then $z\cdot k$ is a unit for all $1\leq k\leq n$ and $x\cdot\frac{1}{n}\sim z\cdot\frac{1}{n}$.
	\end{romenumerate}
\end{corollary}
\begin{proof}\preenum
	\begin{romenumerate}
		\item By \autoref{lem:ndiv}\ref{lem:ndiv-iv}, $y:=(-x\cdot n)\mu_{-x}$ is the unique element satisfying $y\cdot n=x$.
		\item \autoref{lem:ndiv}\ref{lem:ndiv-iii} gives us $(x\cdot n)\tau\cdot n = x\tau$, so 
		\[(x\cdot n)\tau = x\tau\cdot\tfrac{1}{n}\;.\]
		Now by \autoref{lem:ndiv}\ref{lem:ndiv-iv} the element $x\cdot\frac{1}{n}$ also satisfies the conditions of \autoref{lem:ndiv}. Hence we have 
		\[\bigl(\bigl(x\cdot\tfrac{1}{n}\bigr)\cdot n\bigr)\tau\cdot n = \bigl(x\cdot\tfrac{1}{n}\bigr)\tau\;,\]
		so $x\tau\cdot n = \bigl(x\cdot\frac{1}{n}\bigr)\tau$.
		\item By \autoref{lem:timesn_sim}, $z\cdot k \sim x\cdot k\nsim 0$ for all $1\leq k\leq n$. Now we have
		\begin{align*}
			x\sim z &\implies x\tau^{-1}\sim z\tau^{-1}\implies x\tau^{-1}\cdot n\sim z\tau^{-1}\cdot n \\
			&\implies (x\tau^{-1}\cdot n)\tau\sim (z\tau^{-1}\cdot n)\tau\implies x\cdot\tfrac{1}{n}\sim z\cdot\tfrac{1}{n}\qedhere
		\end{align*}
	\end{romenumerate}
\end{proof}
	

\section[Abelian root groups]{Special local Moufang sets with\\ abelian root groups}	
\subsection{Unique \texorpdfstring{$k$}{\emph{k}}-divisibility again}

We have already considered the $k$-divisibility of elements, but one can also look at $k$-divisibility of an entire group.

\begin{definition}
	A group $U$ is \define[k-divisible@$k$-divisible!uniquely k-divisible@uniquely $k$-divisible]{(uniquely)} \define[k-divisible@$k$-divisible]{$k$-divisible} if for every $u\in U$, there is a unique $v\in U$ such that $v^k = u$.
\end{definition}

Now, if the root groups of a special local Moufang set are abelian, and the units are uniquely $k$-divisible, then so is the entire root group.

\begin{proposition}\label{prop:ndivglobal}
	Let $\M$ be a special local Moufang set with $U_\infty$ abelian and $n\in\N$. If for all units $x$ and all $1\leq k\leq n$, $x\cdot k$ is also a unit, then $U_\infty$ is uniquely $k$-divisible for all $1\leq k\leq n$.
\end{proposition}
\begin{proof}
	Let $1\leq k\leq n$. \autoref{cor:ndiv} already shows that, if $x$ is a unit, there is a unique $y$ such that $\alpha_y^k=\alpha_x$; therefore, it only remains to check the unique $k$-divisibility for non-units.
	So suppose that $x$ is not a unit. Take any unit $e$; then $\alpha_x = \alpha_{x\alpha_{-e}}\alpha_{e}$. Now $x\alpha_{-e}$ and $e$ are units, so both $\alpha_{x\alpha_{-e}}$ and $\alpha_e$ are uniquely $k$-divisible, say with $\alpha_y^k = \alpha_{x\alpha_{-e}}$ and $\alpha_z^k = \alpha_e$. Since $U_\infty$ is abelian, we get
	\[(\alpha_y\alpha_z)^k = \alpha_y^k\alpha_z^k = \alpha_{x\alpha_{-e}}\alpha_{e} = \alpha_x\,.\]
	To show uniqueness, suppose there are two elements $u,u' \in U_\infty$ with $u^k = u'^k = \alpha_x$. Then
	\[(\alpha_y^{-1}u)^k = \alpha_y^{-k}\alpha_x = \alpha_y^{-k}\alpha_y^k\alpha_z^k = \alpha_e\,,\]
	and similarly $(\alpha_y^{-1}u')^k = \alpha_e$. By the uniqueness for units, we get $\alpha_y^{-1}u' = \alpha_y^{-1}u$, so $u=u'$.
\end{proof}

\subsection{\textit{\textmugreek}-maps and Hua maps}

If the root groups are abelian, there are a lot of very nice identities for $\mu$-maps and Hua maps. To prove these, we need to know there are at least two classes of units. In the case where for all units $x$, $x\cdot 2$ is also a unit, we automatically get at least two classes of units.

\begin{lemma}
	Let $\M$ be a special local Moufang set with $U_\infty$ abelian and $\abs{\class{X}}>3$. For any unit $x$, we have $h_x = h_{-x}$.
\end{lemma}
\begin{proof}
	First assume $y$ is a unit with $y\nsim -x$, then
	\begin{align*}
		(-y)\mu_{-x} &= (-y)\alpha_{-x}\tau^{-1}\alpha_{x\tau^{-1}}\tau\alpha_{-x} \\
		&= (-y\alpha_x)\tau^{-1}\alpha_{x\tau^{-1}}\tau\alpha_{-x} \\
		&= (-y\alpha_x\tau^{-1})\alpha_{x\tau^{-1}}\tau\alpha_{-x} &\text{(as $-y\alpha_x$ is a unit)} \\
		&= (-y\alpha_x\tau^{-1}\alpha_{-x\tau^{-1}}\tau)\alpha_{-x} \\
		&= -y\alpha_x\tau^{-1}\alpha_{-x\tau^{-1}}\tau\alpha_x \\
		&= -y\mu_x\;,
	\end{align*}
	where we repeatedly use 
	\[(-x')\alpha_{-y'} = 0\alpha_{x'}^{-1}\alpha_{y'}^{-1} = 0(\alpha_{x'}\alpha_{y'})^{-1} = -(x'\alpha_{y'})\;.\]
	From this, we get $y\mu_x = -(-y\mu_x) = -(-y)\mu_{-x} = y\mu_{-x}$, so for any unit $y$ with $y\nsim -x$, we have $yh_x = y\tau\mu_x = y\tau\mu_{-x} = yh_{-x}$.
	
	Next, we look at the case where $y\sim -x$. Take a unit $e$ such that $e\nsim -x$, then $y\alpha_{-e}$ is a unit and $y\alpha_{-e}\nsim -x$, so we can use the previous case to get
	\begin{align*}
		yh_x &= y\alpha_{-e}\alpha_eh_x \\
		&= (y\alpha_{-e})h_x\alpha_{eh_x} &\text{(by \autoref{lem:huaAut})} \\
		&= (y\alpha_{-e})h_{-x}\alpha_{eh_{-x}} &\text{(by the previous case)} \\
		&= y\alpha_{-e}\alpha_eh_{-x} &\text{(by \autoref{lem:huaAut})} \\
		&= yh_{-x}\;.
	\end{align*}
	
	Thirdly, we look at the case where $y\sim 0$. Take any unit $e$, then $y\alpha_{-e}$ is a unit, so we can repeat the previous argument to get $yh_x = yh_{-x}$.
	
	Finally, we need to cover the case $y\sim\infty$.
	\begin{align*}
		yh_x &= y\tau h_x^{\tau} \tau^{-1} \\
		     &= y\tau h_{-x\tau} \tau^{-1} &\text{(by \autoref{lem:hua}\ref{itm:huatau})} \\
		     &= y\tau h_{x\tau} \tau^{-1} &\text{(by the previous case with $x\tau$ replacing $x$)} \\
		     &= y\tau h_{-x}^\tau \tau^{-1} &\text{(by \autoref{lem:hua}\ref{itm:huatau})} \\
		     &= y h_{-x}\;.
	\end{align*}
	Now we know $yh_x = yh_{-x}$ for all $y\in X$, so $h_x = h_{-x}$.
\end{proof}

This lemma has many implications.

\begin{proposition}\label{prop:mu-involution}
	Let $x$ and $y$ be units in a special local Moufang set with $U_\infty$ abelian and $\abs{\class{X}}>3$.
	\begin{romenumerate}
		\item $\mu_x = \mu_{-x}$, so $\mu_x^2 = \id$;\label{itm:mu-involution}
		\item $\mu_x^{\mu_y} = \mu_{x\mu_y}$;
		\item if $x\alpha_y$ is a unit, then $\mu_x\mu_{x\alpha_y}\mu_y = \mu_y\mu_{x\alpha_y}\mu_x = \mu_{(x\tau\alpha_{y\tau})\tau}$. \label{itm:mucommute}
		\item $h_{x\tau} = h_x^{-1}$;
		\item $h_xh_yh_x = h_{yh_x}$.
	\end{romenumerate}
\end{proposition}
\begin{proof}\preenum
	\begin{romenumerate}
		\item Since $h_x = h_{-x}$, we also have $\mu_x = \tau^{-1}h_x = \tau^{-1}h_{-x} = \mu_{-x}$.
		\item By \autoref{lem:mu}\ref{itm:mutau}, we have $\mu_{x\mu_y} = \mu_{-x}^{\mu_y}$, so this follows from \ref{itm:mu-involution}.
		\item Let $z = x\tau\alpha_{y\tau}\tau$. Then, by \autoref{prop:sumform}, we have 
		\[\mu_{-y}\mu_z\mu_{-x} = \mu_{(-x)\alpha_{-y}}\;,\]
		so by \ref{itm:mu-involution} we get $\mu_z = \mu_x\mu_{y\alpha_x}\mu_y$. If we interchange $x$ and $y$, $z$ remains the same by the commutativity of $U_{\infty}$, so we also get $\mu_z = \mu_y\mu_{x\alpha_y}\mu_x$. By the commutativity of $U_{\infty}$ again, $x\alpha_y = 0\alpha_x\alpha_y = y\alpha_x$.
		\item Using the fact that $\tau$ is a $\mu$-map and hence an involution and \autoref{lem:hua}\ref{itm:huatau}, we get
		\[h_{x\tau} = h_{-x}^{\tau} = \tau\tau\mu_{-x}\tau = (\tau\mu_x)^{-1} = h_x^{-1}\;.\]
		\item By \autoref{lem:hua}\ref{itm:huahua} and the previous property, we get
		\[h_{yh_x} = h_{-x}h_{y\tau}^{-1}h_x = h_xh_yh_x\;.\qedhere\]
	\end{romenumerate}
\end{proof}

We can now wonder what the relation is between $\mu_x$ and $\mu_{x\cdot\ell}$, and similarly between $h_x$ and $h_{x\cdot\ell}$

\begin{proposition}\label{prop:mu-hua-a.s-part}
	Let $\M$ be a special local Moufang set with $U_\infty$ abelian and take $n\in\N$. Assume that for all units $x$ and all $1\leq k\leq n$, $x\cdot k$ is also a unit. Then we have $y\mu_x\cdot \ell^2 = y\mu_{x\cdot \ell}$ and $yh_x\cdot \ell^2 = yh_{x\cdot \ell}$ for all units $y$ and for $\ell\in\{n,n^{-1}\}$.
\end{proposition}
\begin{proof}
	If $n=1$, there is nothing to prove, so we assume $n\geq 2$. In this case, as $x\nsim x\cdot 2$ for any unit $x$, we get $\abs{\class{X}}>3$.
	
	First assume $y$ is a unit with $y\nsim -x\cdot n$. Then
	\begin{align*}
		y\mu_{x\cdot n} &= y\alpha_{x\cdot n}\tau\alpha_{-(x\cdot n)\tau}\tau\alpha_{x\cdot n} \\
		&= \bigl(\bigl(y\cdot\tfrac{1}{n}\bigr)\alpha_x\cdot n\bigr)\tau\alpha_{-(x\cdot n)\tau}\tau\alpha_{x\cdot n} \\
		&= \bigl(\bigl(y\cdot\tfrac{1}{n}\bigr)\alpha_x\tau\cdot\tfrac{1}{n}\bigr)\alpha_{-x\tau\cdot\frac{1}{n}}\tau\alpha_{x\cdot n} \\
		&= \bigl(\bigl(y\cdot\tfrac{1}{n}\bigr)\alpha_x\tau\alpha_{-x\tau}\cdot\tfrac{1}{n}\bigr)\tau\alpha_{x\cdot n} \\
		&= \bigl(\bigl(y\cdot\tfrac{1}{n}\bigr)\alpha_x\tau\alpha_{-x\tau}\tau\cdot n\bigr)\alpha_{x\cdot n} \\
		&= \bigl(y\cdot\tfrac{1}{n}\bigr)\alpha_x\tau\alpha_{-x\tau}\tau\alpha_{x}\cdot n \\
		&= \bigl(y\cdot\tfrac{1}{n}\bigr)\mu_x\cdot n = y\mu_x\cdot n^2
	\end{align*}
	If $y\sim -x\cdot n$, we use the previous case to get $(-y)\mu_{x\cdot n} = (-y)\mu_x\cdot n^2$, so as $\M$ is special, we also get $y\mu_{x\cdot n}=y\mu_x\cdot n^2$, which now holds for all units $y$.
	
	By substituting $x$ by $x\cdot \frac{1}{n}$, we get
	\[y\mu_x = y\mu_{x\cdot\frac{1}{n}}\cdot n^2\text{, so }y\mu_{x\cdot\frac{1}{n}} = y\mu_x\cdot\tfrac{1}{n^2}\;,\]
	so we get $y\mu_{x\cdot \ell} = y\mu_x\cdot \ell^2$ for both values of $\ell$.
	
	Using the previous case, we get, for all units $y$.
	\[yh_x\cdot \ell^2 = y\tau\mu_x\cdot \ell^2 = y\tau\mu_{x\cdot \ell} = yh_{x\cdot \ell}\;.\qedhere\]
\end{proof}

\subsection{A paired structure}

We would now like to know what $y\mu_{x\cdot\ell}$ is when $y$ is not a unit. Of course, if $y\sim 0$, we get $y\mu_x \sim \infty$, so it would not make sense to compare $y\mu_{x\cdot\ell}$ to $y\mu_x \cdot\ell^2$, since the second expression is not even defined. To resolve this, we also use the `multiplication by $n$' for $U_0$. In short, we consider two opposite structures in the local Moufang set, one corresponding to $U_\infty$ and one corresponding to $U_0$.

\begin{definition}
	For $x\in X\setminus\overline{0}$ and $k\geq 1$, we define 
	\[x\cdottil k:=\infty\cdot\gamma_{x\tau^{-1}}^k\;.\notatlink{xntil}\]
\end{definition}

\begin{remark}
	The definition of $\cdottil$ corresponds to the definition of $\cdot$ when we interchange $0$ and $\infty$. As we could repeat the previous sections with $0$ and $\infty$ interchanged, we still get valid statements when we replace $\cdot$ with $\cdottil$ in \autoref{cor:ndiv} and \autoref{prop:mu-hua-a.s-part}.
	
	Even though $\tau$ appears in the definition of $\cdottil$, remember that $\gamma_{x\tau^{-1}}$ is the unique element of $U_0$ mapping $\infty$ to $x$ (independent of $\tau$), so $\cdottil$ does not depend on the choice of $\tau$.
\end{remark}

The first thing we can observe is that $\cdot$ and $\cdottil$ are closely related:

\begin{lemma}
	Let $\M$ be a special local Moufang set.
	\begin{romenumerate}
		\item If $x\nsim\infty$, then $(x\cdot n)\tau = x\tau\cdottil n$. \label{itm:xnt}
		\item If $x\nsim0$, then $(x\cdottil n)\tau = x\tau\cdot n$.
	\end{romenumerate}
\end{lemma}
\begin{proof}\preenum
	\begin{romenumerate}
		\item We have
		\[x\tau\cdottil n = \infty\gamma_x^n = \infty\alpha_x^{\tau n} = \infty\tau^{-1}\alpha_x^n\tau = (0\alpha_x^n)\tau = (x\cdot n)\tau\;.\]
		\item This follows from~\ref{itm:xnt}, using $x\tau^{-1}$ and replacing $\tau^{-1}$ by $\tau$.\qedhere
	\end{romenumerate}
\end{proof}

Combining this with \autoref{cor:ndiv}, we are able to express $\cdot\frac{1}{n}$ in terms of $\cdottil$ and we can extend \autoref{prop:mu-hua-a.s-part}:

\begin{proposition}\label{prop:mu-x.s}
	Let $\M$ be a special local Moufang set with abelian root groups, and take $n\in\N$. Assume that for all units $x$ and all $1\leq k\leq n$, $x\cdot k$ is also a unit. Let $\ell\in\{n,n^{-1}\}$. Then for all units $x$, we have
	\begin{romenumerate}
		\item $x\cdottil\ell = x\cdot\ell^{-1}$, hence $(x\cdot\ell)\tau = x\tau\cdottil\ell$ and $(x\cdottil\ell)\tau = x\tau\cdot\ell$;
		\item For $y\nsim\infty$, we have $yh_{x\cdot \ell} = yh_x\cdot \ell^2$.\label{itm:h-x.sfirst}
		\item For $y\nsim 0$, we have $yh_{x\cdot \ell^{-1}} = yh_{x\cdottil \ell} = yh_x\cdottil \ell^2$.
		\item For $y\nsim\infty$, we have $y\mu_{x\cdot\ell^{-1}} = y\mu_{x\cdottil\ell} = y\mu_x\cdottil\ell^{2}$;\label{itm:mu-x.sfirst}
		\item For $y\nsim 0$, we have $y\mu_{x\cdot\ell} = y\mu_x\cdot\ell^2$.\label{itm:mu-x.s}
	\end{romenumerate}
\end{proposition}
\begin{proof}\preenum
	\begin{romenumerate}
		\item By \autoref{cor:ndiv} and the previous lemma, we have
		\[(x\cdottil n)\tau = x\tau\cdot n = \bigl(x\cdot\tfrac{1}{n}\bigr)\tau\;.\]
		so $x\cdottil n = x\cdot\frac{1}{n}$. By switching the roles of $0$ and $\infty$, we also get $x\cdottil\frac{1}{n} = x\cdot n$. Furthermore, we get $x\tau\cdottil\ell = x\tau\cdot\ell^{-1} = (x\cdot\ell)\tau$ and $x\tau\cdot\ell = (x\cdot\ell^{-1})\tau = (x\cdottil\ell)\tau$.
		\item We already know the identity if $y$ is a unit, so we only need to treat the case where $y\sim 0$. In that case, $y = y\alpha_{-e}\alpha_e$, so
		\begin{align*}
			yh_{x\cdot\ell} &= y\alpha_{-e}\alpha_eh_{x\cdot\ell} \\
				&= (y\alpha_{-e})h_{x\cdot\ell}\alpha_e^{h_{x\cdot\ell}} \\
				&= (y\alpha_{-e}h_x\cdot\ell^2)\alpha_{eh_{x\cdot\ell}} \\
				&= (y\alpha_{-e}h_x\cdot\ell^2)\alpha_{eh_x\cdot\ell^2} \\
				&= (y\alpha_{-e}h_x\alpha_{eh_x})\cdot\ell^2 \\
				&= (y\alpha_{-e}\alpha_e)h_x\cdot\ell^2 \\
				&= yh_x\cdot\ell^2\;.
		\end{align*}
		\item Now assume $y\nsim 0$, then
		\begin{align*}
			yh_{x\cdot\ell^{-1}} &= y\tau h_{x\cdot\ell^{-1}}^\tau\tau^{-1} \\
				 &= y\tau h_{(-x\cdot\ell^{-1})\tau}\tau^{-1} \\
				 &= y\tau h_{(-x)\tau\cdot\ell}\tau^{-1} \\
				 &= (y\tau h_{(-x)\tau}\cdot\ell^2)\tau^{-1} \\
				 &= (y\tau h_x^\tau\cdot\ell^2)\tau^{-1} \\
				 &= y\tau h_x^\tau\tau^{-1}\cdottil\ell^2 \\
				 &= yh_x\cdottil\ell^2\;.
		\end{align*}
		\item If $y\nsim\infty$, we have $y\tau^{-1}\nsim 0$, so
		\[y\mu_{x\cdottil\ell} = y\tau^{-1}h_{x\cdottil\ell} = y\tau^{-1}h_x\cdottil\ell^2 = y\mu_x\cdottil\ell^2\;.\]
		\item If $y\nsim0$, we have $y\tau^{-1}\nsim\infty$, so
		\[y\mu_{x\cdot\ell} = y\tau^{-1}h_{x\cdot\ell} = y\tau^{-1}h_x\cdot\ell^2 = y\mu_x\cdot\ell^2\;.\qedhere\]
	\end{romenumerate}
\end{proof}

We can now make a paired structure by looking at $X\setminus\class{\infty}$ and $X\setminus\class{0}$, and consider $\mu$-maps as going from one side to the other (and Hua maps as maps from each side to itself). When we restrict ourselves to these sets, \autoref{prop:mu-x.s} indicates that $x\mapsto\mu_x$ is close to being quadratic in some sense. This will be important in defining a Jordan pair in \autoref{sec:LMStoJP}.


	\part{Examples of local Moufang sets and characterizations}\label{part:examples}

	\chapter{Projective local Moufang sets}\label{chap:chap6_projective}
	The first examples of local Moufang sets we consider are those with a projective line over a local ring $R$ as underlying set. We need to define a natural equivalence relation on this set, and define the root groups to finally get a local Moufang set $\M(R)$. Next, we consider a local Moufang set satisfying some extra conditions \hyperref[axiom:R1]{(R1-4)} and we construct a ring from it. We use this construction to characterize which local Moufang sets are isomorphic to $\M(R)$ for some local ring $R$ with characteristic of the residue field different from $2$. In the final section, we give some connections between the Bruhat-Tits tree of $\SL_2$ over a valued field, and these projective local Moufang sets.

\section{Defining projective local Moufang sets}	
\subsection{Projective spaces over local rings}

In this section, let $R$ be a unital, commutative local ring with maximal ideal $\m$. We start with a general way of constructing $n$-dimensional projective spaces over such rings. The geometric properties of these spaces have been studied by F.D.~Veldkamp in \cite{VeldkampGeometryRings}.

\begin{definition}
	Let $(a_0,a_1,\ldots,a_n)\in R^{n+1}$ such that 
	\[a_0R+\cdots+a_nR = R\;.\]
	Then we write 
	\[[a_0,\ldots,a_n] := \{(ua_0,ua_1,\ldots, ua_n)\mid u\in R^\times\}\;.\notatlink{[ab]}\]
	We define the \define{projective space} of dimension $n$ over $R$ as
	\[\P^n(R):= \left\{[a_0,\ldots,a_n]\;\middle|\; 
	\begin{array}{l} (a_0,a_1,\ldots,a_n)\in R^{n+1} \\ a_0R+a_1R+\cdots a_nR=R \end{array}\right\}\;.\notatlink{P1R}\]
\end{definition}

This definition works for general rings, though it gives more structure if we have a local ring. In this case, the condition $a_0R+a_1R+\cdots+a_nR=R$ implies that at least one $a_i$ is invertible. Furthermore, we can use this to find a unique representative for each point of $\P^n(R)$: if $[a_0,a_1,\ldots,a_n]\in\P^n(R)$, we take the minimal $i$ for which $a_i$ is invertible, and we rescale by $a_i^{-1}$. This means that on the positions $0,\ldots,i-1$, we have elements of $\m$, then we have a $1$, and on positions $i+1,\ldots,n$, we have elements of $R$. In short, we get a decomposition of $\P^n(R)$ according to these unique representatives:
\begin{align}
	\P^n(R) &= \{[1,r_1,\ldots,r_n]\mid r_1,\ldots,r_n\in R\} \notag\\
		&\quad\cup \{[m_0,1,r_2,\ldots,r_n]\mid m_0\in\m, r_2,\ldots,r_n\in R\} \notag\\
		&\quad\cup \{[m_0,m_1,1,r_3,\ldots,r_n]\mid m_0,m_1\in\m, r_3,\ldots,r_n\in R\} \notag\\
		&\quad\cup \cdots \notag\\
		&\quad\cup \{[m_0,m_1,\ldots,m_{n-1},1]\mid m_0,\ldots,m_{n-1}\in\m\} \label{eq:projspace_representatives}
\end{align}
Another way of defining the projective space, which does not require the choice of a basis for $R^{n+1}$, is
\[\P^n(R) = \{V\leq R^{n+1}\mid V\text{ is a free submodule of rank $1$}\}\;.\]
It is useful to observe what happens to the projective space when we take a local ring homomorphism $\phi$ from $R$ to another local ring $S$. As invertible elements are mapped to invertible elements, 
\[\P^n(\phi)\colon \P^n(R)\to \P^n(S)\colon [a_0,a_1,\ldots,a_n]\mapsto[\phi(a_0),\phi(a_1),\ldots,\phi(a_n)]\]
is a well-defined map. Note that this makes $\P^n$ into a functor from the category of local rings with only local ring homomorphisms to the category of sets.

\begin{lemma}\label{lem:projspace_representatives}
	Let $\phi\colon R\to S$ be a local ring homomorphism, take any $[a_0,a_1,\ldots,a_n]\in\P^n(R)$ and let
	\[\P^n(\phi)([a_0,\ldots,a_n]) = [b_0,\ldots,b_n]\;.\]
	If we choose the unique representatives for $[a_0,\ldots,a_n]$ and $[b_0,\ldots,b_n]$ as in \eqref{eq:projspace_representatives}, we have $\phi(a_i)=b_i$ for all $i\in\{0,\ldots,n\}$.
\end{lemma}
\begin{proof}
	If $\P^n(\phi)([a_0,\ldots,a_n]) = [b_0,\ldots,b_n]$, the $i$'s for which $a_i$ is invertible are the same as those for which $b_i$ is invertible. Suppose $j$ is the minimal $i$ for which $a_i$ (and hence also $b_i$) is invertible. If $[a_0,\ldots,a_n]$ and $[b_0,\ldots,b_n]$ are the representatives as in \eqref{eq:projspace_representatives}, we have $a_j=1$ and $b_j=1$. Now, by the definition of $\P^n(\phi)$, $\phi(a_i) = sb_i$ for some fixed invertible element $s\in S$ and all $i\in\{0,\ldots,n\}$. By looking at the $j$-th coordinate, $s=1$, so $\phi(a_i)=b_i$ for all $i$.
\end{proof}

Using all this lemma, we can prove the following:
\begin{theorem}\label{thm:inverse_limit_P1}
	Let $R$ be a local ring with completion $\complete{R}$. Then
	\[\P^n(\complete{R}) = \varprojlim \P^n(R/\m^i)\;,\]
	with projection maps $\P^n(p_j)$, with $p_j\colon \complete{R}\to R/\m^j\colon (r_i)_i\mapsto r_j$.
\end{theorem}
\begin{proof}
	We denote $\phi_{ij}$ for the natural projection $R/\m^i\to R/\m^j$ when $i\geq j$. As $\P^n$ is a functor, \ref{axiom:IL1} is satisfied.
	For the universal property, assume $X$ is another set with maps $q_j\colon X\to \P^n(R/\m^j)$ which commute with the maps $\phi_{ij}$. Now take $x\in X$ and write $q_j(x) = [a_{0,j},\ldots, a_{n,j}]$, using representatives from \eqref{eq:projspace_representatives}. As we have 
	\[\P^n(\phi_{ij})([a_{0,i},\ldots, a_{n,i}]) = [a_{0,j},\ldots, a_{n,j}]\]
	for all $i\geq j$, \autoref{lem:projspace_representatives} implies that $\phi_{ij}(a_{k,i}) = a_{k,j}$ for all $k$. This means that for all $k$, $(a_{k,i})_i\in\complete{R}$. As there is some $k$ for which all $a_{k,i}$ are invertible (there even is a number $k$ for which $(a_{k,i})_i=1$), we find $[(a_{0,i})_i,\ldots,(a_{n,i})_i]\in\P^n(\complete{R})$, giving us a map $\psi\colon X\to\P^n(\complete{R})$. This is also the unique $\psi$ satisfying $p_j\circ\psi=q_j$. Hence $\P^n(\complete{R})$ is indeed the inverse limit.
\end{proof}

\subsection{Construction of projective local Moufang sets}

We can now turn to defining projective local Moufang sets. These are generalizations of projective Moufang sets, and hence they should act on a projective line. In our context, we want to add an equivalence relation to this projective line.

\begin{definition}
	Let $(R,\m)$ be a local ring. The \define{projective line} $\P^1(R)$ is the union
	\[\P^1(R) := \{[1,r]\mid r\in R\}\cup\{[m,1]\mid m\in\m\}\;.\notatlink{P1R}\notatlink{[ab]}\]
	We set
	\begin{equation}
	\begin{aligned}
		&[1,r]\sim[1,r']\iff r-r'\in\m	&&\text{ for all $r,r'\in R$} \\
		&[m,1]\sim[m',1]		&&\text{ for all $m,m'\in\m$} \\
		&[m,1]\nsim[1,r]		&&\text{ for all $m\in\m,r\in R$.}
	\end{aligned} \label{eq:projline_equiv}
	\end{equation}
\end{definition}

We can let $\PSL_2(R)$ act on $\P^1(R)$ on the right, and we will construct a local Moufang set using a subgroup and an element of $\PSL_2(R)$.

\begin{proposition}\label{prop:projectiveLMS_preLMS}
	Consider $(\P^1(R),\sim)$ and define
	\begin{align*}
		U &= \left\{\begin{bmatrix}
			1 & r \\
			0 & 1
		\end{bmatrix}\in\PSL_2(R)\;\middle|\; r\in R\right\} &
		\tau &=\begin{bmatrix} 0 & -1 \\ 1 & 0 \end{bmatrix}\in\PSL_2(R)\;.
	\end{align*}
	Then $U$ and $\tau$ preserve the equivalence and satisfy \hyperref[axiom:C1]{\textnormal{(C1-2)}}, where $\infty=[0,1]$ and $0=[1,0]$.
\end{proposition}
\begin{proof}
	Let $u=\begin{bsmallmatrix} 1 & r \\ 0 & 1 \end{bsmallmatrix}\in U$, then we immediately get $[1,s]\sim[1,s']$ is equivalent to $[1,s]u\sim[1,s']u$. Furthermore, 
	\[[m,1]u = [m,mr+1] = [m(mr+1)^{-1},1]\;,\]
	which means that $u$ also preserves the set $\{[m,1]\mid m\in\m\}$. Hence $u$ preserves equivalence. For $\tau$ we have a few different cases:
	\begin{nrenumerate}
		\item If $r,r'$ are not invertible, then $[1,r]\sim[1,r']$ and also
			\[[-r,1]\sim[-r',1]\;,\]
			as $r,r'\in\m$. This means $[1,r]\tau\sim[1,r']\tau$.
		\item If $r,r'$ are both invertible, then
		\begin{align*}
			[1,r]\sim[1,r'] &\iff r-r'\in\m \iff r\equiv r'\modulo{\m} \\
					&\iff r^{-1}\equiv r'^{-1} \modulo{\m} \iff r'^{-1}-r^{-1}\in\m \\
					&\iff [1,-r^{-1}]\sim[1,-r'^{-1}] \iff [1,r]\tau\sim[1,r']\tau
		\end{align*}
		\item If $m,m'\in\m$, then $[1,m]\sim[1,m']$, and also $[-m,1]\sim[-m',1]$, meaning $[1,m]\tau\sim[1,m']\tau$.
	\end{nrenumerate}
	As $\tau^{-1}=\tau$, we know $\tau$ preserves $\sim$.
	
	Next, we verify the three conditions \ref{axiom:C1}, \ref{axiom:C1'} and \ref{axiom:C2}. For the first, we note that $U$ fixes $[0,1]=:\infty$, and for any $[1,r]$ and $[1,s]$ in $X\setminus\class{\infty}$, there is a unique element of $U$ mapping the first to the second: $\begin{bsmallmatrix} 1 & s-r \\ 0 & 1 \end{bsmallmatrix}$. To show \ref{axiom:C1'}, we use the projection $R\to R/\m \colon r\mapsto\class{r}$. We can see that $\class{[a,b]} = [\class{a},\class{b}]$ and $\induced{\begin{bsmallmatrix} 1 & r \\ 0 & 1 \end{bsmallmatrix}} = \begin{bsmallmatrix} \class{1} & \class{r} \\ \class{0} & \class{1} \end{bsmallmatrix}$. The argument to show \ref{axiom:C1'} is now identical to that for \ref{axiom:C1}, but using the residue field instead of the local ring. The last condition is straightforward: $\infty\tau = [1,0]\nsim \infty$, and $\infty\tau^2 = \infty$.
\end{proof}

We now want to use \autoref{thm:constrMouf} to show that we actually get a local Moufang set. In order to do this, we need to determine the units, $\alpha$-maps and the Hua maps.

\begin{proposition}
	Consider $(\P^1(R),\sim)$ and set
	\begin{align*}
		U &= \left\{\begin{bmatrix}
			1 & r \\
			0 & 1
		\end{bmatrix}\in\PSL_2(R)\;\middle|\; r\in R\right\} &
		\tau &=\begin{bmatrix} 0 & -1 \\ 1 & 0 \end{bmatrix}\in\PSL_2(R)\;.
	\end{align*}
	The units are the points $[1,r]$ with $r$ invertible, we have $\alpha_{[1,r]} = \begin{bsmallmatrix} 1 & r \\ 0 & 1 \end{bsmallmatrix}$ and $h_{[1,r]} = \begin{bsmallmatrix} r^{-1} & 0 \\ 0 & r \end{bsmallmatrix}$.
\end{proposition}
\begin{proof}
	A point of $\P^1(R)$ is a unit if it is not equivalent to $0$ or $\infty$, i.e.\ it is not equivalent to $[0,1]$ or $[1,0]$. The first means that a unit must be of the form $[1,r]$, the second means $r$ must be invertible. Clearly, $\alpha_{[1,r]}$ as defined in the statement maps $[1,0]$ to $[1,r]$. Finally, if $[1,r]$ is a unit, we find
	\begin{align*}
		h_{[1,r]} &= \tau\alpha_{[1,r]}\tau^{-1}\alpha_{-([1,r]\tau^{-1})}\tau\alpha_{-(-([1,r]\tau^{-1}))\tau} \\
		&= 	\begin{bmatrix} 0 & -1 \\ 1 & 0 \end{bmatrix}
			\begin{bmatrix} 1 & r \\ 0 & 1 \end{bmatrix}
			\begin{bmatrix} 0 & -1 \\ 1 & 0 \end{bmatrix}
			\begin{bmatrix} 1 & r^{-1} \\ 0 & 1 \end{bmatrix}
			\begin{bmatrix} 0 & -1 \\ 1 & 0 \end{bmatrix}
			\begin{bmatrix} 1 & r \\ 0 & 1 \end{bmatrix} \\
		&= \begin{bmatrix} r^{-1} & 0 \\ 0 & r \end{bmatrix}\;.\qedhere
	\end{align*}
\end{proof}

With an exact expression for the Hua maps, we can now prove that we get a local Moufang set.

\begin{theorem}
	Let $(R,\m)$ be a local ring and set
	\begin{align*}
		U &= \left\{\begin{bmatrix}
			1 & r \\
			0 & 1
		\end{bmatrix}\in\PSL_2(R)\;\middle|\; r\in R\right\} &
		\tau &=\begin{bmatrix} 0 & -1 \\ 1 & 0 \end{bmatrix}\in\PSL_2(R)
	\end{align*}
	acting on $(\P^1(R),\sim)$. Then $\M(U,\tau)$ is a local Moufang set.
\end{theorem}
\begin{proof}
	By \autoref{prop:projectiveLMS_preLMS} and \autoref{thm:constrMouf}, it is sufficient to show that $h_{[1,r]}$ normalizes $U$ for all units. Let $[1,r]$ be a unit and take any $u=\begin{bsmallmatrix} 1 & s \\ 0 & 1 \end{bsmallmatrix}\in U$. We can compute that $u^{h_{[1,r]}} = \begin{bsmallmatrix} 1 & sr^2 \\ 0 & 1 \end{bsmallmatrix}\in U$. Hence $U$ is normalized by $h_{[1,r]}$.
\end{proof}

\begin{definition}
	Let $(R,\m)$ be a local ring and set
	\begin{align*}
		U &= \left\{\begin{bmatrix}
			1 & r \\
			0 & 1
		\end{bmatrix}\in\PSL_2(R)\;\middle|\; r\in R\right\} &
		\tau &=\begin{bmatrix} 0 & -1 \\ 1 & 0 \end{bmatrix}\in\PSL_2(R)
	\end{align*}
	acting on $(\P^1(R),\sim)$. We call the local Moufang set $\M(U,\tau)$ a \define{projective local Moufang set}, which we denote by $\M(R)$\notatlink{M(R)}.
\end{definition}

\subsection{Some properties}

We are interested in seeing what some of the other notions of local Moufang sets become for projective local Moufang sets.

\begin{proposition}
	Let $\M(R)$ be a projective local Moufang set. Then
	\begin{romenumerate}
		\item $U_0 = \left\{\begin{bmatrix}
			1 & 0 \\
			r & 1
			\end{bmatrix}\in\PSL_2(R)\;\middle|\; r\in R\right\}$;
		\item for a unit $[1,r]$, $\mu_{[1,r]} = \begin{bmatrix} 0 & -r^{-1} \\ r & 0 \end{bmatrix}$;\label{prop:projLMS_mu}
		\item for a unit $[1,r]$, $\til[1,r] = [1,-r] = -[1,r]$;
		\item the Hua subgroup is $H = \left\{\begin{bmatrix}
			r & 0 \\
			0 & r^{-1}
			\end{bmatrix}\in\PSL_2(R)\;\middle|\; r\in R^\times\right\}$;
		\item the little projective group is $G = \PSL_2(R)$\notatlink{PSLn}.
	\end{romenumerate}
	Hence $\M(R)$ is special, has abelian root groups and an abelian Hua subgroup.
\end{proposition}
\begin{proof}\preenum
	\begin{romenumerate}
		\item If $u=\begin{bsmallmatrix} 1 & r \\ 0 & 1 \end{bsmallmatrix}\in U$, then we can compute $u^\tau = \begin{bsmallmatrix} 1 & 0 \\ -r & 1 \end{bsmallmatrix}$. As $r$ runs through $R$, we find $U_0$ as desired.
		\item As we have a local Moufang set, $\mu_{[1,r]} = \tau^{-1}h_{[1,r]}$.
		\item We compute $\til[1,r] = (-[1,r]\tau^{-1})\tau = [1,r^{-1}]\tau = [1,-r]$.
		\item This follows immediately from \ref{prop:projLMS_mu}.
		\item By \autoref{thm:PSL_is_E}, $\PSL_2(R)$ is the group generated by elementary matrices. But for $2\times2$-matrices, the set of elementary matrices is precisely the union of $U_\infty$ and $U_0$. As the little projective group is generated by $U_\infty$ and $U_0$, it coincides with $\PSL_2(R)$.
		\qedhere
	\end{romenumerate}
\end{proof}

Another interesting observation is that if $\characteristic(R/\m)\neq 2$, then $2$ is invertible in $R$. Hence, if $r$ is invertible, so is $2r$, and this implies that if $[1,r]$ is a unit, then so is $[1,r]\cdot 2$.

Finally, we want to check what happens to the local Moufang sets if we change our local ring with a local ring homomorphism.
\begin{proposition}
	Let $(R,\m)$ and $(S,\n)$ be local rings and $f\colon R\to S$ a local ring morphism. Then $\P^1(f)$ is a homomorphism of local Moufang sets from $\M(R)$ to $\M(S)$. If $R\subset S$, then $\M(R)$ can be embedded in $\M(S)$. If $S$ is a quotient of $R$, then $\M(S)$ is a quotient of $\M(R)$.
\end{proposition}
\begin{proof}
	We first show that $\P^1(f)$ preserves $\sim$. For all $r,r'\in R$ we have
	\begin{align*}
		[1,r]\sim[1,r'] &\iff r-r'\in\m \\
				&\iff f(r-r')\in\n &\text{(this uses $f(\m)\subset\n$)} \\
				&\iff [1,f(r)]\sim[1,f(r')]\;.
	\end{align*}
	Furthermore, $[f(m),1]\sim[f(m'),1]$ for $m,m'\in\m$, and $[f(m),1]\nsim[1,f(r)]$ for any $m\in\m$ and $r\in R$.
	
	Assume $\M(R) = \M(U,\tau)$ and $\M(S) = \M(U',\tau')$ as in the definition of projective local Moufang sets. We define
	\[\theta\colon U\to U'\colon \begin{bmatrix}
			1 & r \\
			0 & 1
			\end{bmatrix}\mapsto \begin{bmatrix}
			1 & f(r) \\
			0 & 1
			\end{bmatrix}\;.\]
	We can now check that $u\phi = \phi\theta(u)$ for all $u\in U$ and $\tau\phi = \phi\tau'$. By \autoref{prop:isomorphism2}, we see that $\P^1(f)$ is a homomorphism of local Moufang sets.
\end{proof}

By \autoref{cor:morphism_inj} this proposition means that if $R$ is a local subring of $S$, then $\M(R)$ is a local Moufang subset of $\M(S)$. Similarly \autoref{cor:morphism_surj} implies that if $S$ is a quotient of $R$, then $\M(S)$ is a quotient of $\M(R)$.


\section{Characterization of \texorpdfstring{$\M(R)$}{M(R)}}	
\subsection{Constructing a ring from a local Moufang set}

We have seen that $\M(R)$ is an example of a special local Moufang set.
It is natural to ask what conditions can be put on a local Moufang set to ensure that it is equal to $\M(R)$ for some local ring~$R$.
With some additional assumptions, it is possible to recover the ring structure from the local Moufang set, at least provided the characteristic of the residue field is different from $2$.
We will use a method similar to the related result for Moufang sets in \S 6 of \cite{DMWeissMSJordanAlgebras}.

\begin{importantconstruction}\label{constr:ring}
	Suppose that $\M$ is a local Moufang set satisfying the following conditions:
	\begin{manualenumerate}[label=\textnormal{(R\arabic*)},labelwidth=\widthof{(R1)}]
		\item $\M$ is special;\label{axiom:R1}
		\item $U_\infty$ is abelian;\label{axiom:R2}
		\item the Hua subgroup $H$ is abelian;\label{axiom:R3}
		\item if $x$ is a unit, then so is $x\cdot2$.\label{axiom:R4}
	\end{manualenumerate}
	We consider the set $R := X\setminus\class{\infty}$, and define an addition and a multiplication. Note that there is a bijection between $R$ and $U_\infty$ by $x\mapsto\alpha_x$. We define the addition on $R$ as
		\[x+y:=0\cdot\alpha_x\alpha_y \notatlink{xyplus}\]
	for all $x,y \in R$.
	This addition is simply the translation of the group composition in $U_\infty$ to the set $R$. Since $U_\infty$ is an abelian group, so is $(R,+)$. By \autoref{prop:ndivglobal}, $U_\infty$ is uniquely $2$-divisible, hence also $(R,+)$ is uniquely $2$-divisible.

	To define the multiplication, we first choose a fixed unit $e \in R$\notatlink{e}, which will be the identity element of the multiplication.
	We will use the Hua maps corresponding to $\tau=\mu_e$, i.e.\@~we use $h_x = h_{x,\mu_e} = \mu_e \mu_x$. Note that by \ref{axiom:R1}, \ref{axiom:R2}, \ref{axiom:R4} and \autoref{prop:mu-involution}\ref{itm:mu-involution}, $\mu_e$ is an involution, so we also have $\mu_eh_x = \mu_x$. For any $y\in R$, we now define a map $R_y \colon X\to X$ by
		\[ xR_y := \begin{cases} \begin{aligned}
			& x\cdot h_{e+y} - x\cdot h_y - x && \hspace{1.15em}\text{$y$, $y+e$ units,}\\
			& \mathord{-}x\cdot h_{-e+y}+x + x\cdot h_y && 
					\left\{\begin{aligned}
						&\text{$y$ a unit} \\
						&\text{$y+e$ not a unit}
					\end{aligned}\right. \\
			& x\cdot h_{2e+y}-x\cdot h_{e+y}-x\cdot h_{-2e}+x && \hspace{1.15em}\text{$y$ not a unit.}
		\end{aligned} \end{cases} \]
	(We will verify in the proof of \autoref{le:monoid} below
	that all Hua maps occurring in this definition can indeed be defined.)
		Combining this with the unique $2$-divisibility, we can now define the multiplication on~$R$ by
	\[ xy:=xR_y\cdot\tfrac{1}{2} \]
	for all $x,y \in R$.
\end{importantconstruction}

We will now prove that this structure is a local ring, so for the remainder of the subsection, our local Moufang set satisfies \hyperref[axiom:R1]{(R1-4)}. First, we observe that the action of any Hua map on~$R$ is a group automorphism of $(R,+)$:
\begin{lemma}
	Let $h$ be a Hua map. Then for any $x,y\in R$, we have $(x+y)h = xh+yh$. In particular, $(-x)h = -xh$ and $(x/2)h = xh/2$.
\end{lemma}
\begin{proof}
	The first identity is \autoref{lem:huaAut} rewritten in terms of the addition in $R$.
    The second and the third are immediate consequences, using the unique $2$-divisibility.
\end{proof}

Note that if $h$ and $h'$ are Hua maps, we can define $x(h+h'):= xh+xh'$, and the map $h+h'$ is a group endomorphism of $(R,+)$.
In particular, each $R_x$ is a group endomorphism of $(R,+)$.

\begin{lemma}\label{le:monoid}
	The structure $(R,\cdot)$ is a commutative monoid with identity element $e$.
\end{lemma}
\begin{proof}
	By the previous lemma and observation, we have 
	\[\bigl(x\cdot\tfrac{1}{2}\bigr)R_y = xR_y\cdot\tfrac{1}{2}\;.\]
	Furthermore, since $H$ is abelian, every two maps $R_x$ and $R_y$ commute. Hence
	\begin{align*}
		(xy)z &= \bigl(xR_y\cdot\tfrac{1}{2}\bigr)R_z\cdot\tfrac{1}{2} = xR_yR_z\cdot\tfrac{1}{4} = xR_zR_y\cdot\tfrac{1}{4}\\
			&= \bigl(xR_z\cdot\tfrac{1}{2}\bigr)R_y\cdot\tfrac{1}{2} = (xz)y
	\end{align*}
	for all $x,y,z \in R$.
		This proves the identity
	\begin{equation}\label{eq:xyz}
		(xy)z = (xz)y
	\end{equation}
	for all $x,y,z \in R$, which will be crucial for showing commutativity and associativity.

	Next, we show that $e$ is a left identity element for the multiplication, i.e.\@~that $ex=x$ for all $x\in R$, and we will show this in each of the three cases from the definition of $R_x$. We use \autoref{prop:specialmu} and \autoref{lem:specialsum}, which translates to
\[x\mu_{x+y} = -y-x+x\mu_y-y\]
	for all $x,y \in R$ such that $x$, $y$ and $x+y$ are units.
	\begin{nrenumerate}
		\item Assume $x$ and $x+e$ are units. We get
		\begin{align*}
			eR_x 	&= e\cdot h_{e+x} - e\cdot h_x - e = e\mu_e\mu_{e+x}-e\mu_e\mu_x-e \\
				&= (-e)\mu_{e+x}-(-e)\mu_x-e = -e\mu_{e+x}+e\mu_x-e \\
				&= -(-x-e+e\mu_x-x)+e\mu_x-e = 2x\,,
		\end{align*}
		so $ex = eR_x/2 = x$.
		\item Assume $x$ is a unit and $x+e$ is not a unit.
		Hence $x-e$ is a unit, or we would have $x-e\sim0\sim x+e$, which would imply $e\sim-e$, contradicting \ref{axiom:R4}. We get
		\begin{align*}
			eR_x 	& = -e\cdot h_{-e+x}+e + e\cdot h_x = -e\cdot \mu_e\mu_{-e+x}+e + e\cdot \mu_e\mu_x \\
				&= -(-e)\mu_{-e+x}+e + (-e)\mu_x \\
				&= -(-x+e+(-e)\mu_x-x) + e + (-e)\mu_x = 2x\,,
		\end{align*}
		so $ex = eR_x/2 = x$.
		\item Assume $x$ is not a unit, so $x+e$ and $x+2e$ are units (since $e$ is a unit and $2e\nsim0$). We get
		\begin{align*}
			eR_x & = e\cdot h_{2e+x}-e\cdot h_{e+x}-e\cdot h_{-2e}+e \\
				&= e\cdot \mu_e\mu_{2e+x}-e\cdot \mu_e\mu_{e+x}-e\cdot \mu_e\mu_{-2e}+e \\
				&= -e\mu_{2e+x}+e\mu_{e+x}-(-e)\mu_{-2e}+e \\
				&= -(-e-x-e+e\mu_{e+x}-e-x)+e\mu_{e+x} \\
				&\qquad-(e+e+(-e)\mu_e+e)+e \\
				&= 2x\,,
		\end{align*}
		so $ex = eR_x/2 = x$.
	\end{nrenumerate}
    Substituting $e$ for $x$ in~\eqref{eq:xyz}, we get $yz=zy$ for all $y,z \in R$, so the multiplication is commutative.
    In particular, $xe=ex=x$ for all $x\in R$, so $e$ is an identity element.
    Finally, we apply commutativity to~\eqref{eq:xyz}, and we get 
    \[(yx)z=(xy)z = (xz)y = y(xz)\]
    for all $x,y,z \in R$, so the multiplication is associative.
\end{proof}

The previous lemma contains most of the work in proving that $R$ is a ring. The final ingredient we need is distributivity.

\begin{theorem}
	The structure $(R,+,\cdot)$ is a unital, commutative ring.
\end{theorem}
\begin{proof}
	We know that $(R,+)$ is an abelian group with identity element $0$ (isomorphic to $U_\infty$), and that the multiplicative structure is a commutative monoid with identity element $e$. It only remains to show the distributivity. We have seen before that the maps $R_x$, as linear combinations of Hua maps, act linearly on $(R,+)$. This means that $(x+y)R_z = xR_z+yR_z$, and so
	\[ (x+y)z = (xR_z+yR_z)\cdot\tfrac{1}{2} = xR_z\cdot\tfrac{1}{2}+yR_z\cdot\tfrac{1}{2} = xz+yz \]
	for all $x,y,z \in R$.
	By commutativity, we also get 
	\[x(y+z)=(y+z)x=yx+zx = xy+xz\]
	for all $x,y,z \in R$. We conclude that $R$ is indeed a unital, commutative ring.
\end{proof}

Our next goal is to show that $R$ is a local ring. To do this, we will identify the invertible elements, and show that the non-invertible elements form an ideal.

\begin{proposition}\label{prop:inverse}
	If $x\in X$ is a unit, then $x\in R$ is invertible with inverse $x^{-1}:=(-x)\mu_e$.
\end{proposition}
\begin{proof}
	Note that we only need to show $x^{-1}x=e$, since by commutativity we will also get $xx^{-1}=e$.
	We will again use \autoref{prop:specialmu} and \autoref{lem:specialsum}.
	Again, we need to proceed case by case, but since $x$ is a unit, only the two first cases of the multiplication occur.
	\begin{nrenumerate}
		\item In this case, both $x$ and $x+e$ are units. We get
		\begin{align*}
			(-x)\mu_eR_x 	&= (-x)\mu_e\cdot h_{e+x} - (-x)\mu_e\cdot h_x - (-x)\mu_e \\
					&= (-x)\mu_{e+x} - (-x)\mu_x - (-x)\mu_e \\
					&= -x\mu_{x+e} - x +x\mu_e \\
					&= -(-e-x+x\mu_e-e)-x+x\mu_e \\
					&= 2e\,,
		\end{align*}
		so $x^{-1}x = x^{-1}R_x/2 = e$.
		\item In this case, $x$ a unit and $x+e$ not a unit (so $x-e$ a unit), hence
		\begin{align*}
			(-x)\mu_eR_x 	&= -(-x)\mu_e\cdot h_{-e+x}+(-x)\mu_e + (-x)\mu_e\cdot h_x \\
					&= -(-x)\mu_{-e+x}+(-x)\mu_e + (-x)\mu_x \\
					&= x\mu_{x-e}+(-x)\mu_e + x \\
					&= (e-x+x\mu_e+e)-x\mu_e+x \\
					&= 2e\,,
		\end{align*}
		so again $x^{-1}x = x^{-1}R_x/2 = e$.
        \qedhere
	\end{nrenumerate}
\end{proof}

Observe that we have shown that the elements we called `units' in the local Moufang set correspond to the units in the ring $R$.

\begin{proposition}
	The set $\m:=\class{0}\subset R$ is an ideal in $R$.
\end{proposition}
\begin{proof}
	Let $x,y \in \m$; then 
	\[x+y = x\alpha_y\sim0\alpha_y=y\sim0\;,\]
	so $x+y\sim0$ and $x+y\in\m$. Also 
	\[-x = 0\alpha_{-x}\sim x\alpha_{-x}=0\;,\]
	so $-x\in\m$. Next, we need to verify that $\m$ is closed under multiplication with $R$. Take $x\in\m$ and $r\in R$. We get $xr = xR_r/2$, but $R_r$ is a linear combination of Hua maps. Hence $xR_r$ is a sum of $x\cdot h_{r_i}$ for some $r_i$, and $x h_{r_i}\sim 0 h_{r_i}=0$. Hence $xR_r\sim 0$, and~\ref{axiom:R4} then implies that also $xr\sim 0$. Hence $\m$ is an ideal.
\end{proof}
We can now summarize our results of this section.
\begin{theorem}\label{thm:Rlocal}
	Suppose that $\M$ is a local Moufang set satisfying \hyperref[axiom:R1]{\textnormal{(R1)--(R4)}}.
    Then the ring $R$ obtained from \autoref{constr:ring} is a local ring, and $2$ is invertible in $R$.
\end{theorem}
\begin{proof}
	We have shown that $\m$ is an ideal in $R$.
	On the other hand, if $r\in R\setminus\m$, then $r\in X\setminus\class{\infty}$ is a unit, so $\m$ is exactly the set of non-invertible elements of $R$;
	it must therefore be the unique maximal ideal of $R$.
	Since $2=2e$ is a unit, it is invertible in $R$.
\end{proof}

\subsection{Characterizing \texorpdfstring{$\M(R)$}{M(R)}}

Our goal is to use \autoref{thm:Rlocal} to characterize $\M(R)$ as a special local Moufang set satisfying certain conditions.
As a first step, we will apply \autoref{constr:ring} on $\M(R)$;
we will see that the resulting local ring is indeed isomorphic to the ring $R$ we started with.

\begin{theorem}
	If $R$ is a local ring with residue field not of characteristic $2$, then the ring $R'$ we get from $\M(R)$ using \autoref{constr:ring} with unit $[1,e]$ is isomorphic to $R$.
\end{theorem}
\begin{proof}
	We define a bijection $\phi \colon R\rightarrow R' \colon r\mapsto[1,re]$. Then
		\begin{align*}
		\phi(1) &= [1,e] \,, \\
			\phi(r+s) &= [1,(r+s)e] = [1,re]+[1,se] = \phi(r)+\phi(s)\,, \\
		\phi(rs) &= [1,rse] = [1,ree^{-1}se] = [1,re][1,se] = \phi(r)\phi(s)\,,
		\end{align*}
	for all $r,s \in R$.
	We conclude that $\phi$ is a ring isomorphism.
\end{proof}
\begin{remark}
    The ring $R'$ is, in fact, an {\em isotope} of $R$ with new unit $e$,
    and we have simply illustrated the (well known) fact that an isotope of an associative ring is always isomorphic to the original ring.
\end{remark}
\begin{corollary}
	If $\M(R)$ is isomorphic to $\M(R')$ for local rings $R$ and $R'$ with residue field not of characteristic $2$, then $R\cong R'$.
\end{corollary}

We will now characterize $\M(R)$ purely based on data from the local Moufang set.
We will need one extra assumption on the local Moufang set, in addition to \hyperref[axiom:R1]{(R1-4)}.
\begin{theorem}\label{thm:PSL2equiv}
	Let $\M$ be a local Moufang set satisfying \hyperref[axiom:R1]{\textnormal{(R1-4)}}.
	Let $e$ and $R_x$ be as in \autoref{constr:ring}.
	Assume furthermore that
	\begin{align}
		x\mu_e\alpha_y = yR_x\alpha_{-2e}\mu_eR_x\mu_e \quad \text{for all $x\sim0$ and $y\nsim\infty$.}\label{eq:extraPSL}\tag{$\star$}
	\end{align}
	Then $\M$ is isomorphic to $\M(R)$, where $R$ is the local ring obtained from \autoref{constr:ring}.
\end{theorem}
\begin{proof}
	We adopt the notations from \autoref{constr:ring} for $\M$,
    and we will denote the root group $U_{[0,e]}$ of $\M(R)$ by $U'$.
    We will construct a bijection from $X$ to $\P^1(R)$ preserving the equivalence, a bijection from $U_\infty$ to $U'$,
    and an involution $\tau$ of $\M(R)$ 
    such that the action of $U_\infty$ and $\mu_e$ on~$X$ is permutationally equivalent with the action of $U'$ and $\tau$ on~$\P^1(R)$. By \autoref{prop:isomorphism2} this will show that $\M$ is indeed isomorphic to $\M(R)$.

    By construction, $R = X\setminus\class{\infty}$, and we have $\class{\infty} = \class{0}\mu_e$ by the definition of $\mu$-maps. So we define
	\[ \phi \colon X\to \P^1(R)  \colon  x\mapsto
        \begin{cases}
            [e,x] &\text{if $x\in R$,} \\
            [-x\mu_e, e] & \text{if $x\in \class{\infty}$.}
        \end{cases} \]
	Note that in the second case $x\mu_e\nsim\infty$, so this is indeed an element of the ring $R$.
    It is clear that $\phi$ is a bijection;
    we claim that $\phi$ preserves the equivalence. First, if $x,y\in R$, then
	\begin{align*}
		x\sim y	&\iff x\alpha_{-y}\sim 0\iff x-y\sim 0 \\
			&\iff x-y\in\m\iff [e,x]\sim [e,y]\,,
	\end{align*}
	where the last equivalence follows from~\eqref{eq:projline_equiv} on p.\@~\pageref{eq:projline_equiv}.
    Secondly, if $x\sim y\sim\infty$, then $x\mu_e\sim y\mu_e\sim 0$, so both are in $\m$, and hence $x\phi\sim y\phi$.
    Finally, if $x\sim\infty$ but $y\in R$, then again $x\mu_e\sim 0$, so $x\mu_e\in \m$, hence $x\phi\nsim y\phi$.

    Let $\tau = \begin{bsmallmatrix} 0 & e \\ -e & 0 \end{bsmallmatrix}$.
	It remains to show that the actions of $U_\infty$ and $\mu_e$ on $X$ are permutationally equivalent with the actions of $U'$ and $\tau$ on $\P^1(R)$ via $\phi$.
    For $\tau$ and $\mu_e$, we compute, using $\mu_e^2=\id$,
	\begin{align*}
	x\mu_e\phi &= \begin{cases}
	               	[-x,e] & \text{if $x\mu_e\sim\infty\iff x\sim0$}\,, \\
	               	[e,x\mu_e] = [e, -x^{-1}]\hspace{17.15ex} & \text{if $x\in R^\times$}\,, \\
	               	[e,x\mu_e] & \text{if $x\mu_e\sim 0\iff x\sim\infty$}\,;
	               \end{cases} \\
	x\tau\phi &= \begin{cases}
	               	[e,x]\tau = [-xe, e^2] = [-x,e] & \text{if $x\sim0$}\,, \\
	               	[-x,e] = [e, -x^{-1}] & \text{if $x\in R^\times$}\,, \\
	               	[-x\mu_e,e]\tau = [-e^2, -x\mu_ee] = [e, x\mu_e] & \text{if $x\sim\infty$}\,;
	               \end{cases}
	\end{align*}
	so $x\mu_e\phi = x\tau\phi$. To show that the actions of $U_\infty$ and $U'$ are the same, we first observe that the map
	\[ \theta \colon U_\infty\mapsto U' \colon  \alpha_x\mapsto \alpha_{[e,x]} \]
    for all $x \in R$, is a group isomorphism because
    \begin{align*}
    	\theta(\alpha_x)\theta(\alpha_y) &= \alpha_{[e,x]}\alpha_{[e,y]} = \begin{bmatrix} e & x \\ 0 & e \end{bmatrix} \begin{bmatrix} e & y \\ 0 & e \end{bmatrix}
        = \begin{bmatrix} e & x+y \\ 0 & e \end{bmatrix}\\
        &= \alpha_{[e,x+y]} = \theta(\alpha_{x+y}) = \theta(\alpha_x\alpha_y)
    \end{align*}
    for all $x,y \in R$.
    It only remains to show that
    \[ x\alpha_y\phi = x\phi\theta(\alpha_y) \quad \text{for all $x \in X$ and $y \in R$.} \]
    We distinguish two cases: if $x\in R$, then
	\[x\alpha_y\phi = (x+y)\phi = [e,x+y] = [e,x]\alpha_{[e,y]} = x\phi\theta(\alpha_y)\,.\]
	If $x\sim\infty$, we set $x' = x\mu_e^{-1}$, which is then equivalent to $0$, so by \eqref{eq:extraPSL}
	\begin{align*}
        x\alpha_y
            &= x'\mu_e\alpha_y = yR_{x'}\alpha_{-2e}\mu_eR_{x'}\mu_e = (2yx'-2e)\mu_eR_{x'}\mu_e \\
            &= -(2yx'-2e)^{-1}R_{x'}\mu_e = (-2^{-1}(yx'-e)^{-1})R_{x'}\mu_e \\
            &= (2(-2^{-1}(yx'-e)^{-1})x')\mu_e = (-(yx'-e)^{-1}x')\mu_e\,,
	\end{align*}
    where we have used the fact that $(2yx'-2e)$ is invertible because $2yx' \in \m$.
    Since $x\alpha_y \sim \infty$, this implies
	\[ x\alpha_y\phi = [-(-(yx'-e)^{-1}x')\mu_e\mu_e, e] = [(yx'-e)^{-1}x', e]\,, \]
    where we use $\mu_e^2=\id$. On the other hand,
	\begin{align*}
		x\phi\theta(\alpha_y) &= [-x\mu_e, e]\alpha_{[e,y]} = [-x', e]\begin{bmatrix} e & y \\ 0 & e \end{bmatrix} \\
			&= [-x'e, -x'y+e] = [(yx'-e)^{-1}x', e] \,,
	\end{align*}
	and we conclude that $x\alpha_y\phi = x\phi\theta(\alpha_y)$ also in this case.
\end{proof}


\section{Serre's tree for \texorpdfstring{$\PSL_2$}{PSL\textunderscore2}}\label{sec:Serre}
\subsection{Fields with a discrete valuation and lattices}

In \cite{SerreTrees}, J.-P.\ Serre describes the Bruhat-Tits building corresponding to $\SL_2$ over fields with discrete valuation. As such fields have a local ring as valuation ring, there are connections to local Moufang sets. We refer to \cite{ShalenSL2Tree} for a more expanded explanation on Serre's tree, from which some of the arguments here are taken.

\begin{definition}
	Let $K$\notatlink{K} be a field. A \define{discrete valuation} is a surjective map $v\colon K\to\Z\cup\{\infty\}$\notatlink{va} such that
	\begin{manualenumerate}[label=\textnormal{(DV\arabic*)},labelwidth=\widthof{(DV0)}]
		\item $v(a) = \infty\iff a=0$ for all $a\in K$;
		\item $v(ab) = v(a)+v(b)$ for all $a,b\in K$;
		\item $v(a+b) \geq \min(v(a),v(b))$ for all $a,b\in K$.
	\end{manualenumerate}
	We write $\rO:=\{a\in K\mid v(a)\geq0\}$\notatlink{O} for the \define{valuation ring} and $\pi$\notatlink{pi} for the \define{uniformizer}, a fixed element such that $v(\pi)=1$. The \define{residue field} is $\rO/\pi\rO$.
\end{definition}

It is well-known that $\rO$ is a local ring with maximal ideal $\pi\rO$. The vertices of Serre's tree are homothety classes of lattices in the vector space $K^2$.

\begin{definition}
	An $\rO$-\define{lattice} $L$\notatlink{L} is a free $\rO$-module embedded in a vector space $K^n$. Two lattices $L$ and $L'$ are \define[lattice!homothetic lattices]{homothetic} if there is an $a\in K$ such that $aL = L'$.
\end{definition}

We will be working with full lattices in $K^2$, meaning our lattices have rank $2$. Furthermore, if $L$ and $L'$ are homothetic, there is a number $k\in\Z$ for which $\pi^k L=L'$, as $\pi$ is a uniformizer.

\begin{definition}
	If $L$ is an $\rO$-lattice, we write $[L] := \{\pi^k L\mid k\in\Z\}$\notatlink{[L]} for the \define{homothety class} of $L$.
\end{definition}

If we have two lattice classes, we want to compare them somehow. One can show that it is always possible to find a lattice in each homothety class such that one is contained in the other. This will be the basis for defining a distance between lattices.

\begin{lemma}
	If $L_0$ and $L$ are lattices, then there is a number $k\in\Z$ such that $\pi^kL\subset L_0$. Hence there is a maximal lattice in $[L]$ contained in $L_0$.
\end{lemma}
\begin{proof}
	Let $L_0 = e_1\rO+e_2\rO$, then $\{e_1,e_2\}$ is a basis for $K^2$, so 
	\[L = (ae_1+be_2)\rO + (ce_1+de_2)\rO\]
	for some $a,b,c,d\in K$. Now set $k = \min\{v(a),v(b),v(c),v(d)\}$. Then 
	\[v(\pi^{-k}a) = v(a)-k\geq0\;,\]
	so $\pi^{-k}a\in\rO$, and similarly $\pi^{-k}b,\pi^{-k}c,\pi^{-k}d\in\rO$, so $\pi^{-k}L\subset L_0$.
\end{proof}

\begin{definition}
	Let $L_0$ and $L$ be lattices. Assume $L$ is the maximal lattice in $[L]$ contained in $L_0$. Then also there is a minimal $n\in\N$ such that $\pi^n L_0\subset L$. We define $d([L_0],[L]) = n$.
\end{definition}

In \citenobackref{SerreTrees}, this map is defined in a different way, from which it is obvious that it is a well-defined (in our definition, this distance can still depend on the choice of $L_0$). This distance corresponds to the distance in a graph of lattice classes.

\begin{theorem}
	The graph $T$\notatlink{T} defined by 
	\[V(T) = \{[L]\mid L\text{ an $\rO$-lattice in }K^2\}\]
	where $[L]$ is adjacent to $[L']$ if and only if $d([L],[L']) = 1$\notatlink{dLL} is a tree and the distance in the graph between two lattice classes $[L]$ and $[L']$ is precisely $d([L],[L'])$.
\end{theorem}
\begin{proof}
	This is \citenobackrefoptional{Theorem~1, p.~70}{SerreTrees}.
\end{proof}

We would like to find a connection between some actions on this tree and local Moufang sets. To do this, we will make precise a remark of Serre which relates the lattice classes at a fixed distance from one lattice class to points of a projective line (see \citenobackrefoptional{p.~72}{SerreTrees}). For this, we need more details on the connection between lattice classes at distance $n$.

\begin{lemma}\label{lem:latticeproperties}
	Let $L_0$ be a lattice not in $[L]$, and assume $L$ is the maximal lattice in $[L]$ contained in $L_0$. Denote $n=d([L_0],[L])$. Then 
	\begin{romenumerate}
		\item $L\not\subset \pi L_0$.
		\item If $L_0 = e_1\rO+e_2\rO$, then there are $a,b\in\rO$ such that 
			\[L = (ae_1+be_2)\rO + \pi^nL_0\]
			with at least one of $a,b$ invertible (in $\rO$).\label{lem:latticeproperties-coords}
		\item If $L_0 = e_1\rO+e_2\rO$ and \label{lem:latticeproperties-projcoords}
			\[L = (ae_1+be_2)\rO + \pi^nL_0 = (a'e_1+b'e_2)\rO + \pi^nL_0\]
			with at least one of $a,b$ invertible, then there is an invertible element $u\in\rO$ such that
			\[a\equiv ua'\modulo{\pi^n}\text{ and }b\equiv ub'\modulo{\pi^n}\;.\]
		\item Assume $L_0 = e_1\rO+e_2\rO$ and $L = (ae_1+be_2)\rO + \pi^nL_0$ with at least one of $a,b$ invertible. If $a',b'\in\rO$ are such that there is an invertible $u\in\rO$ such that
			\[a\equiv ua'\modulo{\pi^n}\text{ and }b\equiv ub'\modulo{\pi^n}\;,\]
			then $L = (a'e_1+b'e_2)\rO + \pi^nL_0$.\label{lem:latticeproperties-projcoordsbis}
	\end{romenumerate}
\end{lemma}
\begin{proof}\preenum
	\begin{romenumerate}
		\item If $L\subset\pi L_0$, then $L\subset\pi^{-1}L\subset L_0$, which contradicts the assumption that $L$ was the maximal lattice in $[L]$ contained in $L_0$.
		\item As $\{e_1,e_2\}$ is a basis for $K^2$, we can write 
			\[L = (ae_1+be_2)\rO+(ce_1+de_2)\rO\;.\]
			Since $L\subset L_0$, $a,b,c,d\in\rO$, and since $L\not\subset\pi L_0$, at least one of $a,b,c,d$ is invertible in $\rO$. Without loss of generality, we can assume $a$ is invertible, which means that we get 
			\begin{align*}
				L 	&= (ae_1+be_2)\rO + \bigl((ce_1+de_2)-ca^{-1}(ae_1+be_2)\bigr)\rO \\
					&= (ae_1+be_2)\rO + (d-ca^{-1}b)e_2\rO \\
					&= (ae_1+be_2)\rO + \pi^ke_2\rO\;,
			\end{align*}
			with $k = v(d-ca^{-1}b)$. As $\pi^n L_0\subset L$, we have $k\leq n$. Furthermore, we can see that $\pi^k L_0\subset L$, so $n\leq k$, which means $k=n$. Hence we have $L = (ae_1+be_2)\rO + \pi^n L_0$.
		\item As $(ae_1+be_2)\rO + \pi^nL_0 = (a'e_1+b'e_2)\rO + \pi^nL_0$, there are $u,r,s\in\rO$ such that
		\[ae_1+be_2 = u(a'e_1+b'e_2) + r\pi^n e_1 + s\pi^n e_2\;,\]
		hence $a \equiv ua'\modulo{\pi^n}$ and  $b \equiv ub'\modulo{\pi^n}$. As at least one of $a$ or $b$ are invertible, $u$ must be invertible.
		\item By the assumptions, there are $r,s\in\rO$ such that 
		\begin{align*}
			&(a'e_1+b'e_2)\rO + \pi^nL_0 \\
			&= \bigl((au^{-1}+r\pi^n)e_1+(bu^{-1}+s\pi^n)e_2\bigr)\rO+\pi^nL_0 
		\end{align*}
		and hence
		\begin{align*}
			(a'e_1+b'e_2)\rO + \pi^nL_0 &= (au^{-1}e_1+bu^{-1}e_2)\rO+\pi^nL_0 \\
						&= (ae_1+be_2)\rO+\pi^nL_0 = L\;.\qedhere
		\end{align*}
	\end{romenumerate}
\end{proof}

By \autoref{lem:latticeproperties}\ref{lem:latticeproperties-coords}, the lattice classes at distance $n$ from $[L_0]$ now correspond to the set
\[T_n = \{L\text{ lattice in $K^2$}\mid d([L_0],[L])=n, \pi^n L_0\subset L\subset L_0\}\;,\]
and by \autoref{lem:latticeproperties}\ref{lem:latticeproperties-projcoords} the map
\begin{align}
\begin{aligned}
	\chi_n 	&\colon T_n\to \P^1(\rO/\pi^n\rO) \\
		&\colon (ae_1+be_2)\rO+\pi^n L_0\mapsto[a\modulo{\pi^n},b\modulo{\pi^n}]\;.
\end{aligned}\label{eq:chi_n}
\end{align}
gives a well-defined map to the projective line $\P^1(\rO/\pi^n\rO)$. By \autoref{lem:latticeproperties}\ref{lem:latticeproperties-projcoordsbis}, the inverse map $[a,b]\mapsto (\tilde{a}e_1+\tilde{b}e_2)\rO+\pi^nL_0$ is also well-defined, where $\tilde{a}$ and $\tilde{b}$ are lifts of $a,b$ to $\rO$, so $\chi_n$ is a bijection between $T_n$ and $\P^1(\rO/\pi^n\rO)$.

We can also prove a converse to \autoref{lem:latticeproperties}\ref{lem:latticeproperties-coords}:

\begin{lemma}\label{lem:latticeproperties2}
	Let $L_0 = e_1\rO+e_2\rO$ be a lattice that is not in $[L]$. If 
	\[L = (ae_1+be_2)\rO + \pi^nL_0\]
	with $a,b\in\rO$ and at least one of $a,b$ invertible, then $d([L_0],[L])=n$. Furthermore, $L$ is the unique lattice in $[L]$ that can be written in such a way.
\end{lemma}
\begin{proof}
	We immediately see that $L\subset L_0$ and $\pi^n L_0\subset L$. If $n=1$, it is clearly minimal with $\pi^n L_0\subset L$, otherwise we would get $L_0\subset L\subset L_0$ and then $L_0\in[L]$, contradicting the assumptions. Hence, if $n=1$, we get $d([L_0],[L])=1$.
	
	We can now assume $n>1$ and without loss of generality, we assume $a$ is invertible. Suppose $n$ was not the minimal number for which $\pi^n L_0\subset L$, then $\pi^{n-1}L_0\subset L$. Hence $\pi^{n-1}e_i\in L$ for $i=1,2$. This means that there are $r,s,t\in\rO$ such that
	\begin{align*}
		\pi^{n-1}e_2 = (ar+\pi^ns)e_1 + (br+\pi^nt)e_2\\
	\intertext{and hence}
		ar = -\pi^ns \quad\text{ and }\quad br = \pi^{n-1}-\pi^ns\;.
	\end{align*}
	As $a$ is invertible, we get $v(r) \geq n$ from the first equality. The second equality gives $n-1 = v(br) = v(b) + v(r)\geq n$, a contradiction. Hence $n$ was minimal to begin with, so $d([L_0],[L])=n$.
	
	Assume we also have $L' = (a'e_1+b'e_2)\rO + \pi^{n'}L_0\in[L]$. By the previous, $n'=d([L_0],[L])=n$. Now $L$ is the unique element of $[L]$ such that $\pi^nL_0\subset L\subset L_0$, but $L'$ also satisfies this inclusion. Hence $L=L'$.
\end{proof}

As a corollary of these lemmas, we can change the definition of Serre's tree to use these specific lattices:
\begin{corollary}\label{cor:tree_lattices}
	Let $L_0 = e_1\rO+e_2\rO$ be a lattice. Define the graph $T'$ by
	\[V(T') = \left\{(ae_1+be_2)\rO + \pi^nL_0 \;\middle|\;
	\begin{array}{l}
		n\in\N, a,b\in\rO \\
		\text{at least one of $a,b$ invertible}
	\end{array}\right\}\;,\]
	and two vertices $L,L'$ are adjacent if they can be written as
	\[L = (ae_1+be_2)\rO + \pi^nL_0\text{ and }L' = (ae_1+be_2)\rO + \pi^{n+1}L_0\;,\]
	for $n\in\N$, and $a,b\in\rO$ with at least one of $a,b$ invertible.
	Then the graph $T'$ is isomorphic to Serre's tree $T$ by the map $L\mapsto[L]$.
\end{corollary}
\begin{proof}
	By \autoref{lem:latticeproperties}\ref{lem:latticeproperties-coords}, the map $L\mapsto[L]$ is surjective (observe that $L_0 = (ae_1+be_2)\rO + \pi^0L_0$ for any $a,b\in\rO$). By the final statement of \autoref{lem:latticeproperties2}, the map is injective.
	
	Take $L$ and $L'$ in $V(T')$. First assume that $L$ and $L'$ are adjacent in $T'$. Then clearly $\pi L \subset L'\subset L$ and $L\neq L'$, hence $d([L],[L'])=1$, so $[L]$ and $[L']$ are adjacent in $T$. Conversely, assume $[L]$ and $[L']$ are adjacent. Then without loss of generality,
	\[d([L_0],[L'])=d([L_0],[L])+1\;.\]
	Write $d([L_0],[L'])=n$. By \autoref{lem:latticeproperties2}, we can write 
	\[L' = (ae_1+be_2)\rO + \pi^{n+1}L_0\;,\]
	with at least one of $a,b\in\rO$ invertible. We now claim that 
	\[L=(ae_1+be_2)\rO+\pi^n L_0 = L'+\pi^n L_0\;.\]
	By \autoref{lem:latticeproperties2}, $d([L_0],[L'+\pi^n L_0])=n$. Furthermore,
	\[\pi(L'+\pi^n L_0)\subset L'\subsetneq L'+\pi^n L_0\;,\]
	so $d([L'],[L'+\pi^n L_0])=1$. As $T$ is a tree, this means that 
	\[[L'+\pi^n L_0] = [L]\;.\]
	Using \autoref{lem:latticeproperties2} one last time, we get $L = L'+\pi^n L_0$, which means $L$ and $L'$ are adjacent in $T$.
\end{proof}

Let $\{e_1,e_2\}$ by the standard basis of $K^2$. We now look at the action of $\SL_2(K)$. Clearly if $g\in\SL_2(K)$, $g$ maps a lattice onto a lattice. Also, $(\pi^k L)g = \pi^k Lg$, so $g$ acts on $V(T)$. Finally, $g$ preserves inclusion of lattices, so $g$ preserves $d(\cdot,\cdot)$, and hence it preserves the tree $T$. In short: $\SL_2(K)$ acts on the tree $T$. We are interested in the point stabilizer of this action.

\begin{proposition}
	Let $L_0 = e_1\rO+e_2\rO$, then 
	\[\SL_2(K)_{[L_0]} = \SL_2(K)_{L_0} = \SL_2(\rO)\;.\notatlink{SLn}\]
\end{proposition}
\begin{proof}
	Clearly $\SL_2(\rO)\subset\SL_2(K)_{L_0}\subset\SL_2(K)_{[L_0]}$. Next, we assume $g\in\SL_2(K)_{[L_0]}$, i.e.\ $L_0 g = \pi^k L_0$ for some $k\in\Z$. Let $g = \big(\begin{smallmatrix} a&b\\c&d \end{smallmatrix}\big)$. As $L_0 g = \pi^k L_0$, there is a matrix $\big(\begin{smallmatrix} a'&b'\\c'&d' \end{smallmatrix}\big)$ in $\GL_2(\rO)$ such that
	\[ae_1 +be_2 = a'\pi^k e_1 + b'\pi^ke_2\text{ and }ce_1+de_2 = c'\pi^ke_1+d'\pi^ke_2\;.\]
	Since $\det g = 1$, we get 
	\[0=v(\det g) = v(ad-bc) = 2k+v(a'd'-b'c') = 2k\;,\]
	so $k=0$ and $a,b,c,d\in\rO$. Hence $g\in\SL_2(\rO)$.
\end{proof}

The induced action of such a point-stabilizer is exactly what will turn out to be an action of a local Moufang set.

\subsection{Action on spheres}

We now fix $L_0 = e_1\rO+e_2\rO$ and look at the induced action of $\SL_2(\rO)$ on the vertices at some fixed distance $n$. This action will give rise to a local Moufang set which will be isomorphic to $\M(\rO/\pi^n\rO)$.

\begin{theorem}\label{thm:localSerreLMS}
	Assume $L_0 = e_1\rO+e_2\rO$. We set
	\[T_n = \{L\text{ lattice in $K^2$}\mid d([L_0],[L])=n, \pi^n L_0\subset L\subset L_0\}\]
	and define $L\sim L'\iff L+\pi L_0 = L'+\pi L_0$.
	
	Then the induced action of $\SL_2(\rO)$ on $(T_n,\sim)$ is isomorphic to the action of $\PSL_2(\rO/\pi^n\rO)$ on $(\P^1(\rO/\pi^n\rO),\sim)$, where  the correspondence between $T_n$ and $\P^1(\rO/\pi^n\rO)$ is given by $\chi_n$ as defined by \eqref{eq:chi_n} on p.~\pageref{eq:chi_n}.
\end{theorem}
\begin{proof}
	The first thing we need to check is if $\SL_2(\rO)$ actually acts on $T_n$ preserving $\sim$. Let $g\in \SL_2(\rO)$, then $L_0 g = L_0$ and $g$ preserves the distance in the tree $T$. Hence for any lattice $L\in T_n$ and any $L'\sim L$ we get
	\begin{align*}
		&d([L_0 g],[L g]) = n \text{, so } d([L_0],[L g]) = n\\
		&\pi^n L_0 g\subset Lg\subset L_0 g \text{, so } \pi^n L_0\subset Lg\subset L_0 \\
		&L g+\pi L_0 g = L' g+\pi L_0 g \text{, so } L g+\pi L_0 = L' g+\pi L_0\;.
	\end{align*}
	This means $Lg \in T_n$ and $L g\sim L'g$.
	
	Next, we want to know what the induced action of $\SL_2(\rO)$ on $T_n$ is. We want to determine the kernel of the group action, i.e.\ we want to find 
	\[N := \Ker(\SL_2(\rO)\to\Sym(T_n,\sim))\;.\]
	Let $g\in N$ and write $g = \big(\begin{smallmatrix} a&b\\c&d \end{smallmatrix}\big)$. We now have 
	\[e_1\rO+\pi^nL_0 =(e_1\rO+\pi^nL_0)g = (ae_1+be_2)\rO+\pi^nL_0\;.\]
	By \autoref{lem:latticeproperties}\ref{lem:latticeproperties-projcoords}, there is an invertible $u\in\rO$ such that 
	\[a \equiv u\modulo{\pi^n}\text{ and }b\in\pi^n\rO\;.\] Similarly
	\[e_2\rO+\pi^nL_0 =(e_2\rO+\pi^nL_0)g = (ce_1+de_2)\rO+\pi^nL_0\;,\]
	from which we get an invertible $u'\in\rO$ such that $c\in\pi^n\rO$ and $d\equiv u'\modulo{\pi^n}$. Finally we use
	\begin{align*}
		(e_1+e_2)\rO+\pi^nL_0 	&= ((e_1+e_2)\rO+\pi^nL_0)g \\
					&= ((a+c)e_1+(b+d)e_2)\rO+\pi^nL_0\;,
	\end{align*}
	and find an element $u''\in\rO$ such that $a+c\equiv u''\equiv b+d\modulo{\pi^n}$. Combining these, we get
	\[u\equiv a+c\equiv b+d\equiv u'\modulo{\pi^n}\;,\]
	so $g$ is a matrix which reduces to a scalar matrix modulo $\pi^n\rO$. Conversely, assume
	\[g = \begin{pmatrix} u+\pi^na'&\pi^nb'\\\pi^nc'&u+\pi^nd' \end{pmatrix}\]
	with $a',b',c',d'\in\rO$. If $L = (ae_1+be_2)\rO+\pi^nL_0$ is any lattice in $T_n$, we get
	\begin{align*}
		Lg &= \bigl((au+\pi^naa'+\pi^nbc')e_1+(bu+\pi^nab'+\pi^nbd')e_2\bigr)\rO+\pi^n L_0 \\
		   &= (aue_1+bue_2)\rO+\pi^n L_0 = (ae_1+be_2)\rO+\pi^n L_0 = L
	\end{align*}
	so $g$ fixes all of $T_n$. Hence the kernel of the action of $\SL_2(\rO)$ on $T_n$ are those matrices that reduce to scalar matrices modulo $\pi^n\rO$. This means we get an action of $\PSL_2(\rO/\pi^n\rO)$ on $T_n$ by
	\begin{align}
		\begin{aligned}
			&\bigl((ae_1+be_2)\rO+\pi^nL_0\bigr)\begin{bmatrix}
				a'&b'\\c'&d'
			\end{bmatrix} \\
			&= \bigl((aa'+bc')e_1+(ab'+bd')e_2\bigr)\rO+\pi^nL_0\;.
		\end{aligned}\label{eq:lattice_action}
	\end{align}
	Hence the action of $\begin{bsmallmatrix} a'&b'\\c'&d' \end{bsmallmatrix}$ commutes with $\chi_n$, so the action of $\PSL_2(\rO/\pi^n\rO)$ on $T_n$ and $\P^1(\rO/\pi^n\rO)$ is isomorphic.
	
	The last thing we need to check is that $\chi_n$ preserves the equivalence relation. Let $L,L'\in T_n$, then $L+\pi L_0$ and $L'+\pi L_0$ are two lattices at distance $1$ of $L_0$. This means, if $\chi_n(L)=[a,b]$ and $\chi_n(L')=[a',b']$, then by \autoref{lem:latticeproperties}\ref{lem:latticeproperties-projcoords}, for any lifts $\tilde{a}$, $\tilde{b}$, $\tilde{a}'$ and $\tilde{b}'$ of $a$, $b$, $a'$ and $b'$,
	\begin{align*}
		&L+\pi L_0 = L'+\pi L_0 \\
		&\iff (\tilde{a}e_1+\tilde{b}e_2)\rO+\pi L_0 = (\tilde{a}'e_1+\tilde{b}'e_2)\rO+\pi L_0 \\
		&\iff \exists u\in\rO^\times\colon a'\equiv ua\modulo{\pi}\text{ and }b'\equiv ub\modulo{\pi}\;.
	\intertext{There are now two cases: if $a$ is invertible, then so is $a'$ and we get $u\equiv a'a^{-1}\modulo{\pi}$. In this case}
		&\iff b'\equiv a'a^{-1}b\modulo{\pi} \iff ba^{-1}-b'a'^{-1}\in\pi\rO \\
		&\iff[1,ba^{-1}]\sim[1,b'a'^{-1}]
	\end{align*}
	In the second case $a\in\pi\rO/\pi^n\rO$, so $a'\equiv ua\equiv0\modulo{\pi}$ and we can always choose $u=b'b^{-1}$, so we always have $L+\pi L_0 = L'+\pi L_0$. But in this case, we also always have $[ab^{-1},1]\sim[a'b'^{-1},1]$. We conclude that in all cases, $L\sim L'\iff\chi_n(L)\sim\chi_n(L')$.
\end{proof}

\subsection{Action on the boundary}

In \autoref{thm:localSerreLMS}, we observed that the action of $\SL_2(\rO)$ on Serre's tree induces many local Moufang sets, which are also `local' in the tree, in the sense that these actions are limited to a part of the tree at a finite distance of a fixed point. We would like to find local Moufang sets which have a more global action. These will occur by looking at the action on the ends of the tree.

\begin{definition}
	Let $T$ be a tree. A \define{ray} is a path in $T$ that is infinite in one direction. In other words, it is a sequence $(x_i)_i = (x_0,x_1,\ldots)$ of vertices where $x_i$ is adjacent to $x_{i+1}$ for all $i\in\N$. Two rays are \define[ray!parallel rays]{parallel} if their intersection is also a ray. An \define{end} is a parallel class of rays. The set of ends of $T$ is the \define{boundary} of the tree, which we denote by $\partial T$\notatlink{dT}.
\end{definition}

To simplify matters, we will identify the boundary with rays starting in a fixed vertex.

\begin{proposition}\label{prop:ends_rays}
	Let $T$ be a tree with a vertex $x_0$. Then there is a correspondence between $\partial T$ and the rays starting in $x_0$.
\end{proposition}
\begin{proof}
	Clearly, rays starting in $x_0$ give rise to ends by taking the parallel classes. If two such rays $(x_i)_i$ and $(x'_i)_i$ are parallel, then the intersection is a ray $(x_i)_{i\geq N}$ for some $N$. As $x_0$ is in the intersection, $N=0$, so the rays coincide.
	
	Conversely, if we have an end represented by a ray $(y_i)_i$, then we can take $N$ to be such that $d(x_0,y_N)$ is minimal. As $T$ is a tree, this $N$ is unique. We can take the unique path $(x_0,x_1,\ldots,x_d=y_N)$ from $x_0$ to $y_N$. The ray 
	\[(x_0,x_1,\ldots,x_d,y_{N+1},y_{N+2},\ldots)\]
	is now parallel to $(y_i)_i$ and starts in $x_0$. Hence any end contains a unique ray starting in $x_0$.
\end{proof}

Now let's return to Serre's tree, and determine what the boundary is. We use the description $T'$ of the tree as in \autoref{cor:tree_lattices}, as well as the sets $T_n$ from \autoref{thm:localSerreLMS}. Remark that 
\[T_n = \{L\in V(T')\mid d([L],[L_0])=n\}\;,\]
and if $L\in T_n$, then $L+\pi^{n-1}L_0\in T_{n-1}$. This means that the rays of $T'$ are sequences $(L_i)_i\in\prod_i T_i$ such that $L_i = L_{i+1}+\pi^iL_0$ for all $i$. In other words, using the correspondence of \autoref{prop:ends_rays}, we get
\begin{align}
	\varprojlim T_i &\cong \partial T\quad\text{by }(L_i)_i\mapsto \text{the parallel class of }([L_i])_i\;.\label{eq:inverse_limit_boundary}
\end{align}
We will use this correspondence to identify the ends of the tree with a projective line.

\begin{proposition}
	The boundary of $T$ corresponds to $\P^1(\complete{\rO})$.
\end{proposition}
\begin{proof}
	We already know that the boundary of $T$ corresponds to $\varprojlim T_i$ where 
	\[\phi_{ij}\colon T_i\to T_j\colon L\mapsto L+\pi^jL_0\;.\]
	We will prove that the inverse system $(T_i,\phi_{ij})$ is isomorphic to the inverse system $(\P^1(\rO/\pi^i\rO),\P^1(\psi_{ij}))$, where $\psi_{ij}\colon\rO/\pi^i\rO\to\rO/\pi^j\rO$ is the natural projection map.
	
	In order to show the isomorphism of these inverse systems, we need bijections between $T_i$ and $\P^1(\rO/\pi^i\rO)$ for all $i$. These bijections are precisely the maps $\chi_i$ as defined in \autoref{thm:localSerreLMS}. What we need to show is the commutativity of the following squares for all $i\geq j$:
	\begin{center}
		\begin{tikzpicture}[regulararrow,node distance=2cm]
			\node (Ti) {$T_i$};
			\node[right of=Ti,xshift=2cm] (Tj) {$T_j$};
			\node[below of=Ti,yshift=0.5cm] (Pi) {$\P^1(\rO/\pi^i\rO)$};
			\node[right of=Pi,xshift=2cm] (Pj) {$\P^1(\rO/\pi^j\rO)$};
			\draw[->]	(Ti) edge node[above] {\scriptsize$\phi_{ij}$} (Tj)
					(Ti) edge node[left] {\scriptsize$\chi_i$} (Pi)
					(Tj) edge node[right] {\scriptsize$\chi_j$} (Pj)
					(Pi) edge node[above] {\scriptsize$\P^1(\psi_{ij})$} (Pj);
		\end{tikzpicture}
	\end{center}
	This can be checked by the definitions of the maps:
	\begin{align*}
		\bigl((ae_1+be_2)\rO+\pi^i L_0\bigr)\phi_{ij}\chi_j 
		&= \bigl((ae_1+be_2)\rO+\pi^j L_0\bigr)\chi_j \\
		&= [a\modulo{\pi^j},b\modulo{\pi^j}] \\[0.5ex]
		\bigl((ae_1+be_2)\rO+\pi^i L_0\bigr)\chi_i\P^1(\psi_{ij})
		&= [a\modulo{\pi^i},b\modulo{\pi^i}]\P^1(\psi_{ij}) \\
		&= [\psi_{ij}(a\modulo{\pi^i}),\psi_{ij}(b\modulo{\pi^i})] \\
		&= [a\modulo{\pi^j},b\modulo{\pi^j}]
	\end{align*}
	This means that $\varprojlim T_i\cong \varprojlim \P^1(\rO/\pi^i\rO)$. By \autoref{thm:inverse_limit_P1}, this is $\P^1(\complete{\rO})$.
\end{proof}

Remark that this correspondence is also in \cite[p.~72]{SerreTrees}, but it is not made explicit. We can make it explicit with the following maps:
\begin{center}
	\begin{tikzpicture}[regulararrow,node distance=0.6cm]
		\node (PO) {$\P^1(\complete{\rO})$};
		\node[right of=PO,xshift=3.8cm] (lim) {$\varprojlim T_i$};
		\node[right of=lim,xshift=3.8cm] (partT) {$\partial T$};
		\node[below of=PO,yshift=-0.2cm] (ab) {$[a,b]$};
		\node[below of=lim,yshift=-0.2cm] (lattices) {$\bigl((\tilde{a}_ie_1+\tilde{b}_ie_2)\rO+\pi^i L_0\bigr)_i$};
		\node[below of=partT,yshift=-0.2cm] (endmid) {};
		\node[below of=lattices] (ray) {$(L_i)_i$};
		\node[below of=endmid] (end) {parallel class of $([L_i])_i$};
		
		\draw[->] 	(PO) edge node[above] {\scriptsize$\chi$} (lim)
				(lim) edge (partT);
		\draw[|->] 	(ab) edge (lattices)
				(ray) edge (end);
	\end{tikzpicture}
\end{center}
Here $\tilde{a}_i\in\rO$ is such that $a\equiv\tilde{a}_i\modulo{\pi^i}$ and similarly for $\tilde{b}_i$. We can use this correspondence to look at the action of $\SL_2(\rO)$ on the boundary. Remember that $\SL_2(\rO)$ fixes $L_0$, so any $g\in\SL_2(\rO)$ actually sends a ray starting in $L_0$ to another such ray. This means that we can view the action of $\SL_2(\rO)$ on $\partial T$ by looking at the action of $\SL_2(\rO)$ on $\varprojlim T_i$.

\begin{proposition}\label{prop:isomorphic_action_boundary}
	Assume $\rO$ is complete, then the induced action of $\SL_2(\rO)$ on $\varprojlim T_i$ is isomorphic to that of $\PSL_2(\rO)$ on $\P^1(\rO)$, with the correspondence $\chi$.
\end{proposition}
\begin{proof}
	We first determine the induced action on $\varprojlim T_i$. Let 
	\[g\in\Ker(\SL_2(\rO)\to\Sym(\varprojlim T_i))\;.\]
	Write $g = \bigl(\begin{smallmatrix} a&b\\c&d \end{smallmatrix}\bigr)$. As $g$ acts trivially, we have
	\[e_1\rO+\pi^iL_0 = (e_1\rO+\pi^iL_0)g = (ae_1+be_2)\rO+\pi^iL_0\]
	for all $i$. By \autoref{lem:latticeproperties}\ref{lem:latticeproperties-projcoords}, $b\in\pi^i\rO$ for all $i\in\rO$. This means $v(b)\geq i$ for all $i\in\N$, so $v(b)=\infty$ and hence $b=0$. Similarly,\[e_2\rO+\pi^iL_0 = (e_2\rO+\pi^iL_0)g = (ce_1+de_2)\rO+\pi^iL_0\]
	for all $i$, so $c=0$. We now get
	\[(e_1+e_2)\rO+\pi^iL_0 = \bigl((e_1+e_2)\rO+\pi^iL_0\bigr)g = (ae_1+de_2)\rO+\pi^iL_0\]
	for all $i$, so there are units $u_i\in\rO$ such that $a\equiv u_i\equiv d\modulo{\pi^i}$ for all $i$. This means that $v(a-d) = 0$, so $a=d$. Hence $g$ is a scalar matrix. As all scalar matrices in $\SL_2(\rO)$ preserve all lattices, we see that the induced action of $\SL_2(\rO)$ on $\varprojlim T_i$ is that of $\PSL_2(\rO)$.
	
	Next, we show the correspondence between the action of $\PSL_2(\rO)$ on $\varprojlim T_i$ and on $\P^1(\rO)$. We know the bijection
	\[\chi\colon \P^1(\rO)\to\varprojlim T_i\colon [a,b]\mapsto\bigl((ae_1+be_2)\rO+\pi^i L_0\bigr)_i \;.\]
	If $g=\begin{bsmallmatrix} a' & b' \\ c' & d' \end{bsmallmatrix}$, we get
	\begin{align*}
		[a,b]g &= [aa'+bc',ab'+bd']
	\end{align*}
	and
	\begin{align*}
		&\bigl((ae_1+be_2)\rO+\pi^i L_0\bigr)_ig \\
		&= \bigl(((ae_1+be_2)\rO+\pi^i L_0)g\bigr)_i \\
		&= \bigl(((aa'+bc')e_1+(ab'+bd')e_2)\rO+\pi^i L_0\bigr)_i\;.
	\end{align*}
	Hence $\chi$ commutes with the action of $\PSL_2(\rO)$, so the actions are isomorphic.
\end{proof}

This means that the action of $\PSL_2(\rO)$ on $\varprojlim T_i$ can be given the structure of a local Moufang set. A precise description is given in the following theorem:
\begin{theorem}\label{thm:Serre_LMS_complete}
	Let $T$ be Serre's tree and assume $\rO$ is complete. Endow $\varprojlim T_i$ with an equivalence relation by 
	\[(L_i)_i\sim(L'_i)_i\iff L_1=L'_1\;.\]
	Then the action of $\PSL_2(\rO)$ on $(\varprojlim T_i,\sim)$ gives rise to a local Moufang set. When we set $\infty = (e_2\rO+\pi^iL_0)_i$ and $0 = (e_1\rO+\pi^iL_0)_i$, we can find the root groups by \autoref{constr:MUtau} using
	\begin{align*}
		U &= \left\{\begin{bmatrix}
			1 & r \\
			0 & 1
		\end{bmatrix}\in\PSL_2(\rO)\;\middle|\; r\in \rO\right\} &
		\tau &=\begin{bmatrix} 0 & -1 \\ 1 & 0 \end{bmatrix}\in\PSL_2(\rO)\;.
	\end{align*}
\end{theorem}
\begin{proof}
	We show that $\chi$ preserves $\sim$. Assume $[1,r]\sim[1,r']$ for $r,r'\in\rO$. Then $r - r'\in\pi\rO$, so $r\equiv r'\modulo{\pi}$. By \autoref{lem:latticeproperties}\ref{lem:latticeproperties-projcoordsbis}, this means
	\[(e_1+re_2)\rO+\pi L_0 = (e_1+r'e_2)\rO + \pi L_0\;,\]
	so $[1,r]\chi\sim[1,r']\chi$. Similarly, we prove that $[m,1]\sim[m',1]$ implies $[m,1]\chi\sim[m',1]\chi$ for $m,m'\in\pi\rO$. Finally, if $m\in\pi\rO$ and $r\in\rO$, we have
	\[(me_1+e_2)\rO+\pi L_0 \neq (e_1+re_2)\rO + \pi L_0\;,\]
	as equality would imply that there is an invertible $u\in\rO$ such that $u\equiv m\equiv 0\modulo{\pi}$, which is impossible. Hence $\chi$ preserves equivalence.
	
	The remaining statements now follow from \autoref{prop:isomorphic_action_boundary}, using the description of the isomorphism between the actions.
\end{proof}

Finally, we look at the case where $\rO$ is not complete. We still have an action of $\PSL_2(\rO)$ on $\varprojlim T_i$, but it is no longer transitive. This corresponds to the action of $\PSL_2(\rO)$ on $\P^(\complete{\rO})$ not being transitive. The inclusion $\P^1(\rO)\xrightarrowtail{}\P^1(\complete{\rO})$ gives us a local Moufang subset, of which the little projective group is $\PSL_2(\rO)$. This corresponds to a subset of the ends of $T$, on which $\PSL_2(\rO)$ acts as a local Moufang set.

\begin{theorem}
	Let $\M$ be the local Moufang set defined by the action of $\PSL_2(\complete{\rO})$ on $(\varprojlim T_i,\sim)$ and set
	\begin{align*}
		Y :=& \bigl\{\bigl((ae_1+be_2)\rO+\pi^iL_0\bigr)_i\in\varprojlim T_i \;\bigm|\; \\
		&\qquad\qquad a,b\in\rO\text{ with one of $a,b$ invertible}\bigr\}\;.
	\end{align*}
	Then $Y$ induces a local Moufang subset of $\M$ which is isomorphic to $\M(\rO)$.
\end{theorem}
\begin{proof}
	As in \autoref{thm:Serre_LMS_complete}, we set $\infty = (e_2\rO+\pi^iL_0)_i$ and $0 = (e_1\rO+\pi^iL_0)_i$. We get
	\begin{align*}
		V 	:=& \{u\in U_\infty\mid 0u\in Y\} \\
			 =& \left\{\begin{bmatrix}
					1 & r \\
					0 & 1
				\end{bmatrix}\in\PSL_2(\rO)\;\middle|\; r\in \rO\right\}\;,
	\end{align*}
	which clearly is a group, hence \ref{axiom:S1} holds. Furthermore, for $v=\begin{bsmallmatrix} 1 & r \\ 0 & 1\end{bsmallmatrix}\in V$ and $\bigl((ae_1+be_2)\rO+\pi^iL_0\bigr)_i\in Y$, we get
	\begin{align*}
		\bigl((ae_1+be_2)\rO+\pi^iL_0\bigr)_i v &= \bigl((ae_1+(ar+b)e_2)\rO+\pi^iL_0\bigr)_i\in Y\;,
	\end{align*}
	so \ref{axiom:S2} is satisfied. Remark that $\tau = \begin{bsmallmatrix} 0 & -1 \\ 1 & 0 \end{bsmallmatrix}$ is a $\mu$-map, as $\tau = \mu_{[1,1]\chi}$, and $[1,1]\chi\in Y$. Now 
	\begin{align*}
		\bigl((ae_1+be_2)\rO+\pi^iL_0\bigr)_i \tau &= \bigl((be_1-ae_2)\rO+\pi^iL_0\bigr)_i\in Y\;,
	\end{align*}
	which shows \ref{axiom:S3}. Hence $\M(V,\tau)$ is a local Moufang subset of $\M$, and by the descriptions of $V$ and $\tau$, it is isomorphic to $\M(\rO)$.
\end{proof}
	
	\chapter{Local Moufang sets \& Jordan pairs}\label{chap:chap7_jordan}
	We construct a local Moufang set $\M(V)$ from any local Jordan pair $V$. Next we take a local Moufang set satisfying assumptions \hyperref[itm:J1]{(J1-4)}, and we construct a local Jordan pair. Finally, we connect these two constructions and characterize those local Moufang sets that are isomorphic to $\M(V)$ for some local Jordan pair.

\section[Local Jordan pairs to local Moufang sets]{From local Jordan pairs to local Moufang sets}\label{sec:JP}
\subsection{Projective space}

If we want to construct a local Moufang set from a Jordan pair, we need a set to act on. The set we will use is the projective space of a Jordan pair $V$. This concept was introduced by O.~Loos in \cite{LoosHomAlgVar}. The description we use comes from Loos' more recent article \cite{LoosDecompProjSpace}.

\begin{definition}
	Two pairs $(x,y),(x',y')\in V$ are \define[Jordan pair!projective equivalence]{projectively equivalent} if
	\[(x,y-y')\text{ is quasi-invertible and }x' = x^{y-y'}\;.\]
	Using \autoref{prop:JPbasic}\ref{itm:JPbasicB}, this can be shown to be an equivalence relation, and we will denote the equivalence class of $(x,y)$ by $[x,y]$. The \define[Jordan pair!projective space]{projective space of $V$} is the set
	\[\P(V) := \{[x,y]\mid (x,y)\in V\}\;.\notatlink{P(V)}\notatlink{[xy]}\]
\end{definition}

The condition for two pairs to be projectively equivalent is not so easy to grasp. To simplify matters, we determine some nice representatives for points of $\P(V)$, when $V$ is a local Jordan pair.

\begin{proposition}\label{prop:PVrepr}
	Let $V$ be a local Jordan pair with $e$ invertible. Take any $(x,y)\in V$ then at least one of the following occurs:
	\begin{manualenumerate}[label*=\textnormal{(\Roman*)},labelwidth=\widthof{(II)}]
		\item There is a unique $t\in V^+$ such that $[x,y]=[t,0]$.
		\item There is a unique $t\in V^-$ such that $[x,y]=[e,e^{-1}+t]$.
	\end{manualenumerate}
	If in either of the cases $t$ is non-invertible, then the other case cannot occur. If $t$ is invertible, we can have
	\[[t,0] = [e,e^{-1}-t^{-1}].\]
\end{proposition}
\begin{proof}
	Let $(x,y)\in V$. Assume first that $(x,y)$ is quasi-invertible. Then we immediately have $[x,y] = [x^y,0]$, so we are in the first case.

	So assume now that $(x,y)$ is not quasi-invertible. Then $x$ is invertible by \autoref{prop:JPbasic}\ref{itm:JPbasicD}. In this case, set $t = y-x^{-1}$. Now, using \autoref{prop:JPbasic}\ref{itm:JPbasicA}, we have
	\begin{align*}
		&[x,y] = [e,e^{-1}+t] \\
		&\iff (e,e^{-1}-x^{-1})\text{ is quasi-invertible and }e^{e^{-1}-x^{-1}} = x \\
		&\iff \left\{
		\begin{aligned}
			&B_{e,e^{-1}-x^{-1}}\text{ is invertible} \\
			&e-(e^{-1}-x^{-1})Q_e= xB_{e,e^{-1}-x^{-1}}
		\end{aligned}\right. \\
		&\iff \left\{
		\begin{aligned}
			&Q_{e^{-1}-(e^{-1}-x^{-1})}Q_e\text{ is invertible}\\
			&x^{-1}Q_e = xQ_{e^{-1}-(e^{-1}-x^{-1})}Q_e
		\end{aligned}\right. \\
		&\iff Q_{x^{-1}}Q_e\text{ is invertible and }x^{-1}Q_e = xQ_{x^{-1}}Q_e
	\end{align*}
	Now $e$ and $x$ are invertible, so $Q_e$ and $Q_{x^{-1}}$ are invertible, and $x^{-1} = xQ_{x^{-1}}$. So indeed, we found a representative for $[x,y]$ of the second form.

	Now assume that $[t,0]=[e,e^{-1}+s]$ for some $s \in V^-$. Then $B_{e,e^{-1}-s}$ must be invertible, but $B_{e,e^{-1}-s} = Q_{s}Q_e$, since $e$ is invertible. Hence $Q_{s}$ must be invertible, so $s$ must be invertible. In this case
	\begin{align*}
		t 	&= e^{e^{-1}+s} = (e - (e^{-1}+s)Q_e)Q_e^{-1}Q_{-s}^{-1} \\
			&= (e - e-sQ_e)Q_e^{-1}Q_{s}^{-1} = -sQ_{s}^{-1} = -s^{-1}\;,
	\end{align*}
	so $t$ is also invertible. This proves the remaining statements.
\end{proof}

By \autoref{prop:PVrepr}, we now have a nice set of representatives for $\P(V)$ as follows:
\begin{equation}\label{eq:P(V)}
    \P(V) = \{[x,0]\mid x\in V^+\}\cup \{[e,e^{-1}+y]\mid y\in \Rad V^-\}\;.\notatlink{P(V)}
\end{equation}
The second subset consists of projective points that are ``close'' to each other, in the sense that they only differ by a non-invertible element. We can define a similar closeness relation on the first subset.

\begin{definition}\label{def:radequiv}
	We define a \define{radical equivalence} relation $\sim$ on $\P(V)$ by
	\begin{align*}
		[x,0]&\sim[x',0]\iff x-x'\in\Rad V^+ \\
		[e,e^{-1}+y]&\sim [e,e^{-1}+y']\iff y-y'\in\Rad V^- \\
		[x,0]&\not\sim [e,e^{-1}+y]\hspace{3.2em}\text{if $x\in\Rad V^+$ or $y\in\Rad V^-$.}
	\end{align*}
\end{definition}

Note that this equivalence is well-defined by \autoref{prop:JPbasic}\ref{itm:JPbasicInvertRadical}.
\begin{remark}
	We could have avoided the explicit choice of representatives for $\P(V)$ by defining the radical equivalence by
	\[ [x,y] \sim [x',y'] \iff
		\begin{aligned}
			&\text{ there are }(\hat{x},\hat{y})\in[x,y]\text{ and }(\hat{x}',\hat{y}')\in[x',y'] \\
			&\text{ such that } (\hat{x},\hat{y})\equiv(\hat{x}',\hat{y}')\modulo{\Rad V}\;.
		\end{aligned}
	\]
\end{remark}
\begin{remark}
	Observe that $[0,0]\nsim[e,e^{-1}]\nsim [e,0]\nsim[0,0]$, so the set of equivalence classes $\class{\P(V)}$ contains at least $3$ classes.
\end{remark}

\subsection{Root groups and \texorpdfstring{$\mu$}{\textit{\textmugreek}}-maps}

We will now construct a local Moufang set using \autoref{constr:MUtau}. To be more precise, we will use the approach described in \autoref{rem:constr_UU'} and define two root groups acting on $(\P(V),\sim)$.

\begin{definition}\label{def:alphazetaJP}For all $v\in V^+$:
	\[\alpha_v : \begin{cases}
		[x,0]\mapsto [x+v,0] & \text{for all $x\in V^+$} \\
		[e,e^{-1}+y]\mapsto [e,e^{-1}+y^v] & \text{for all $y\in \Rad V^-$}
	\end{cases}\]
	For all $w\in V^-$:
	\[\zeta_w : \begin{cases}
		[x,0]\mapsto [x^w,0] & \text{for all $x\in \Rad V^+$} \\
		[e,e^{-1}+y]\mapsto [e,e^{-1}+y+w] & \text{for all $y\in V^-$}
	\end{cases}\]
\end{definition}

Before we continue, we have to check that these maps preserve the radical equivalence.

\begin{proposition}
	The maps $\alpha_v$ and $\zeta_w$ preserve the radical equivalence on $\P(V)$.
\end{proposition}
\begin{proof}
	First, $[x,0]\sim[x',0]$ if and only if $x-x'\in\Rad V^+$, which is equivalent to $(x+v)-(x'+v)\in\Rad V^+$, so in this case $\alpha_v$ preserves equivalence. If furthermore $x\in\Rad V^+$ then $[x,0]\sim[x',0]$ if and only if $x'\in\Rad V^+$. Since $x'\in\Rad V^+\iff x'^w\in\Rad V^+$ and $x^w\in\Rad V^+$, we find that $\zeta_w$ also preserves equivalence in this case.

	Similarly, $[e,e^{-1}+y]\sim[e,e^{-1}+y']$ is equivalent to $[e,e^{-1}+y]\zeta_w\sim[e,e^{-1}+y']\zeta_w$ and if $y\in\Rad V^-$, \[[e,e^{-1}+y]\sim[e,e^{-1}+y']\iff[e,e^{-1}+y]\alpha_v\sim[e,e^{-1}+y']\alpha_v\;.\]
	Finally, assume $x\in\Rad V^+$ and $y\in\Rad V^-$, so $[x,0]\nsim[e,e^{-1}+y]$. Then $x^w\in\Rad V^+$, so also $[x,0]\zeta_w\nsim[e,e^{-1}+y]\zeta_w$ and $y^v\in\Rad V^-$, so also $[x,0]\alpha_v\nsim[e,e^{-1}+y]\alpha_v$. These cover all cases.
\end{proof}

We will use the set of all $\alpha_v$ to get $U_\infty$ and the set of $\zeta_w$ to get $U_0$. Whatever approach we take now, we will need the $\mu$-maps. More precisely, we need the action of the $\mu$-maps, which we determine in the next proposition. 
The bulk of the computational work of this section is contained in the proof of this proposition.

\begin{proposition}\label{prop:JordanMuaction}
	Let $\mu_{v} = \zeta_{v^{-1}}\alpha_v\zeta_{v^{-1}}$ for $v\in V^+$ invertible. Then
	\begin{alignat*}{2}
		&[e,e^{-1}+y]\mu_{v} = [yQ_v,0] && \text{for $y\in \Rad V^-$}\;, \\
		&[e,e^{-1}+y]\mu_{v} = [e,e^{-1}-y^{-1}Q_v^{-1}] &\quad& \text{for $y\in V^-\setminus\Rad V^-$}\;, \\
		&[x,0]\mu_{v} = [e,e^{-1}+xQ_v^{-1}] && \text{for $x\in \Rad V^+$}\;.
    \end{alignat*}
	As a consequence, $\mu_v^2 = \id$. Using the other representations, we get
    \begin{alignat*}{2}
        &[e,e^{-1}+y]\mu_{v} = [yQ_v,0] && \text{ for all $y\in V^-$,} \\
        &[x,0]\mu_{v} = [e,e^{-1}+xQ_v^{-1}] \quad && \text{ for all $x\in V^+$,} \\
        &[x,0]\mu_v = [-x^{-1}Q_v,0] && \text{ for all $x\in V^+\setminus\Rad V^+$.}
    \end{alignat*}
\end{proposition}
\begin{proof}
For simplicity, we will set $w = v^{-1}$ throughout this proof, so $\mu_v = \zeta_w\alpha_v\zeta_w$, $wQ_v = v$ and $vQ_w=w$.

In the first case, we start with $[e,e^{-1}+y]$ for $y\in \Rad V^-$ and we want to compute $[e,e^{-1}+y]\mu_v$. 

We first check that $y+w$ is invertible and $-(y+w)^{-1}+v\in\Rad V^+$. Since $y\in\Rad V^-$ and $w$ is invertible, clearly $y+w$ is invertible. Secondly, take $z\in V^-$ arbitrary, then
\[\bigl(-(y+w)^{-1}+v,z\bigr) \equiv \bigl(-w^{-1}+v,z\bigr)\equiv (0,z)\modulo{\Rad V}\;,\]
so $(-(y+w)^{-1}+v,z)\modulo{\Rad V}$ is quasi-invertible, and by \autoref{prop:JPbasic}\ref{itm:JPbasicQIlift} that means $(-(y+w)^{-1}+v,z)$ is quasi-invertible. As $z$ was arbitrary, that means $-(y+w)^{-1}+v\in \Rad V^+$.
We now get
\begin{align*}
	[e,e^{-1}+y]\mu_v &= [e,e^{-1}+y]\zeta_w\alpha_v\zeta_w = [e,e^{-1}+y+w]\alpha_v\zeta_w \\
			  &= [-(y+w)^{-1},0]\alpha_v\zeta_w = [-(y+w)^{-1}+v,0]\zeta_w \\
			  &= [(-(y+w)^{-1}+v)^w,0]\;,
\end{align*}
so we want to prove $(-(y+w)^{-1}+v)^w = yQ_v$. We get
\begin{align*}
&(v-(y+w)^{-1})^w = yQ_v \\
&\iff (w^{\bigl(v-(y+w)^{-1}\bigr)}-w)Q_w^{-1} = yQ_v \hspace{28.8pt}\text{(by \autoref{prop:JPbasic}\ref{itm:JPbasicswitch})} \\
&\iff w^{\bigl(v-(y+w)^{-1}\bigr)} = y + w \\
&\iff \bigl(w-(v-(y+w)^{-1})Q_w\bigr)B_{w,v-(y+w)^{-1}}^{-1} = y+w \\
&\iff \bigl(w-vQ_w+(y+w)^{-1}Q_w)\bigr)\bigl(Q_{v-(v-(y+w)^{-1})}Q_w\bigr)^{-1} = y+w \\
&	\hspace{203.75pt}\text{(by \autoref{prop:JPbasic}\ref{itm:JPbasicA})} \\
&\iff (y+w)^{-1}Q_wQ_w^{-1}Q_{(y+w)^{-1}}^{-1} = y+w \\
&\iff (y+w)^{-1}Q_{y+w} = y+w\;,
\end{align*}
which holds, so the identity holds.

In the second case, we start with $[e,e^{-1}+y]$ for $y\in V^-\setminus\Rad V^-$. The permutation $\zeta_w$ maps this element to $[e,e^{-1}+y+w]$.
We now distinguish two cases according to whether $y+w\in\Rad V^-$ or not.
Assume first that $y+w\in\Rad V^-$; then
\begin{align*}
	[e,e^{-1}+y]\mu_v &= [e,e^{-1}+y+w]\alpha_v\zeta_w = [e,e^{-1}+(y+w)^v]\zeta_w \\*
			&= [e,e^{-1}+(y+w)^v+w]\;.
\end{align*}
We need to check that $(y+w)^v+w = -y^{-1}Q_v^{-1}$:
\begin{align*}
&(y+w)^v+w = -y^{-1}Q_v^{-1} \\
&\iff	(v^{y+w}-v)Q_v^{-1}+w = -y^{-1}Q_v^{-1} \hspace{23.48pt}\text{(by \autoref{prop:JPbasic}\ref{itm:JPbasicswitch})} \\
&\iff	v^{y+w}-v+wQ_v = -y^{-1} \\
&\iff	\bigl(v-(y+w)Q_v\bigr)B_{v,y+w}^{-1} = -y^{-1}  \\
&\iff	\bigl(v-(y+w)Q_v\bigr)(Q_{w-(y+w)}Q_v)^{-1} = -y^{-1} \\
&	\hspace{203.75pt}\text{(by \autoref{prop:JPbasic}\ref{itm:JPbasicA})} \\
&\iff	\bigl(vQ_v^{-1}-(y+w)\bigr)Q_{-y}^{-1} = -y^{-1} \\
&\iff	-yQ_y^{-1} = -y^{-1}\;,
\end{align*}
which holds, so the identity holds.

Assume now that $y+w\not\in\Rad V^-$. If we set $x=v-(y+w)^{-1}$, then $y = (v-x)^{-1}-w$. If $x$ was not invertible, then we would have $y \equiv 0\modulo{\Rad V}$ (using \autoref{prop:JPbasic}\ref{itm:JPbasicInvertRadical}), which contradicts $y\in V^-\setminus\Rad V^-$. Hence $v-(y+w)^{-1}$ is invertible, so
\begin{align*}
	[e,e^{-1}+y]\mu_v &= [e,e^{-1}+y+w]\alpha_v\zeta_w = [-(y+w)^{-1},0]\alpha_v\zeta_w \\
			&= [v-(y+w)^{-1},0]\zeta_w \\
			&= [e,e^{-1}+((y+w)^{-1}-v)^{-1}]\zeta_w \\
			&= [e,e^{-1}+((y+w)^{-1}-v)^{-1}+w]
\end{align*}
We want to show $((y+w)^{-1}-v)^{-1}+w = -y^{-1}Q_v^{-1}$.
\begin{align*}
& ((y+w)^{-1}-v)^{-1}+w = -y^{-1}Q_w \\
&\iff v-(y+w)^{-1} = (y^{-1}Q_w+w)^{-1} \\
&\iff \bigl(v-(y+w)^{-1}\bigr)Q_{y^{-1}Q_w+w} = y^{-1}Q_w+w
\end{align*}
Now first observe that
\begin{align*}
	(y+w)Q_y^{-1} &= (y+w)^{-1}Q_{y+w}Q_y^{-1} \\
			&= (y+w)^{-1}(Q_{y,w}+Q_y+Q_w)Q_y^{-1} \\
			&= (y+w)^{-1}Q_{y,w}Q_y^{-1}+(y+w)^{-1} + (y+w)^{-1}Q_wQ_y^{-1}
\end{align*}
Using \ref{itm:JPbasicidDQ} and \ref{itm:JPbasicid} from \autoref{prop:JPbasic}, we get
\begin{align*}
	&\bigl(v-(y+w)^{-1}\bigr)Q_{y^{-1}Q_w+w} \\
	&= \bigl(v-(y+w)^{-1}\bigr)(Q_{y^{-1}Q_w,w}+Q_{y^{-1}Q_w}+Q_w) \\
	&= \bigl(v-(y+w)^{-1}\bigr)(Q_{y^{-1}Q_w,w}+Q_wQ_y^{-1}Q_w+Q_w) \\
	&= vQ_{y^{-1}Q_w,w}+vQ_wQ_y^{-1}Q_w+vQ_w \\*
	&\qquad-(y+w)^{-1}(Q_{y^{-1}Q_w,w}+Q_wQ_y^{-1}Q_w+Q_w) \\
	&= vQ_{y^{-1}Q_w,w}+wQ_y^{-1}Q_w+w \\*
	&\qquad-(y+w)^{-1}Q_{y^{-1}Q_w,w}-(y+w)^{-1}Q_wQ_y^{-1}Q_w-(y+w)^{-1}Q_w \\
	&= vQ_wD_{w,y^{-1}}+wQ_y^{-1}Q_w+w \\*
	&\qquad-(y+w)^{-1}D_{y^{-1},w}Q_w-(y+w)^{-1}Q_wQ_y^{-1}Q_w-(y+w)^{-1}Q_w \\
	&= w+wD_{w,y^{-1}}+wQ_y^{-1}Q_w \\*
	&\qquad-(y+w)^{-1}D_{y^{-1},w}Q_w-(y+w)^{-1}Q_wQ_y^{-1}Q_w-(y+w)^{-1}Q_w \\
	&= w+2y^{-1}Q_w+wQ_y^{-1}Q_w \\*
	&\qquad-(y+w)^{-1}D_{y^{-1},w}Q_w-(y+w)^{-1}Q_wQ_y^{-1}Q_w-(y+w)^{-1}Q_w \\
	&= w+y^{-1}Q_w+\bigl(y^{-1}+wQ_y^{-1}-(y+w)^{-1}D_{y^{-1},w} \\*
	&\hspace{9em}-(y+w)^{-1}Q_wQ_y^{-1}-(y+w)^{-1}\bigr)Q_w \\
	&= w+y^{-1}Q_w+\bigl((y+w)Q_y^{-1}-(y+w)^{-1}D_{y^{-1},w} \\*
	&\hspace{9em}-(y+w)^{-1}Q_wQ_y^{-1}-(y+w)^{-1}\bigr)Q_w \\
	&= w+y^{-1}Q_w+\bigl((y+w)^{-1}Q_{y,w}Q_y^{-1}-(y+w)^{-1}D_{y^{-1},w}\bigr)Q_w \\
	&= w+y^{-1}Q_w
\end{align*}
This finishes the second case.

In the third case, we start with $[x,0]$ for $x\in \Rad V^+$. Since $x\in \Rad V^+$, we also have $x^w\in\Rad V^+$, so as $v$ is invertible, so is $x^w+v$. We get
\begin{align*}
	[x,0]\mu_{v} &= [x,0]\zeta_w\alpha_v\zeta_w = [x^w,0]\alpha_v\zeta_w = [x^w+v,0]\zeta_w \\
		&= [e,e^{-1}-(x^w+v)^{-1}]\zeta_w = [e,e^{-1}-(x^w+v)^{-1}+w]\;.
\end{align*}
Hence, we want to prove $-(x^w+v)^{-1}+w = xQ_v^{-1}$.
\begin{align*}
	&	 -(x^w+v)^{-1}+w = xQ_v^{-1} \\
	&\iff	 (x^w+v)^{-1} = w-xQ_v^{-1} \\
	&\iff	 x^w = (w-xQ_v^{-1})^{-1}-v \\
	&\iff	 x = \bigl((w-xQ_v^{-1})^{-1}-v\bigr)^{-w} & \text{(using \autoref{prop:JPbasic}\ref{itm:JPbasicB})}\\
	&\iff	 yQ_v = \bigl((w-y)^{-1}-v\bigr)^{-w} & \text{(set $y=xQ_v^{-1}$)} \\
	&\iff	 yQ_{v'} = \bigl(v'-(w'+y)^{-1}\bigr)^{w'} & \text{(set $w = -w'$ and $v = -v'$)}
\end{align*}
This is precisely the identity we have proven in the first case.
\end{proof}

\subsection{The local Moufang set \texorpdfstring{$\M(V)$}{M(V)}}

Now we would like to use the permutations we have to construct a local Moufang set with \autoref{constr:MUtau} from p.~\pageref{constr:MUtau}. Of course, that requires the conditions for the construction to be satisfied.

\begin{proposition}
	Let $V$ be a local Jordan pair with invertible element $e\in V^+$. Then $U = \{\alpha_v\mid v\in V^+\}$ is a group, and together with the permutation $\tau = \mu_e$, the conditions \ref{axiom:C1}, \ref{axiom:C1'} and \ref{axiom:C2} are satisfied, with $0=[0,0]$, $\infty=[e,e^{-1}]$.
\end{proposition}
\begin{proof}
	We claim first that $\alpha_0 = \id$. We have $[x,0]\alpha_0 = [x,0]$ for all $x\in V^+$, and $[e,e^{-1}+y]\alpha_0 = [e,e^{-1}+y^0]$ for all $y\in\Rad V^-$. Now $y^0 = (y-0Q_y)B_{y,0}^{-1} = y\id$, so indeed $\alpha_0 = \id$. Next, we claim $\alpha_z\alpha_v = \alpha_{z+v}$. By definition $[x,0]\alpha_z\alpha_v = [x,0]\alpha_{z+v}$ for all $x\in V^+$. By \autoref{prop:JPbasic}\ref{itm:JPbasicB}, we have
	\[[e,e^{-1}+y]\alpha_z\alpha_v = [e,e^{-1}+(y^z)^v] = [e,e^{-1}+y^{z+v}] = [e,e^{-1}+y]\alpha_{z+v}\]
	for all $y\in \Rad V^-$. Using these two, we get $\alpha_v^{-1} = \alpha_{-v}$, so $U$ is a group.
	
	The group $U$ fixes $[e,e^{-1}]$, as 
	\[0^v = (0-vQ_0)B_{0,v}^{-1} = (0-0)\id = 0\;,\]
	hence we choose $\infty:=[e,e^{-1}]$. Furthermore, by \autoref{def:radequiv} and \eqref{eq:P(V)}
	\[\P(V)\setminus \class{[e,e^{-1}]} = \{[x,0]\mid x\in V^+\}\;,\]
	and $\alpha_v$ acts on this set by $x\mapsto x+v$. This action of $U$ on $\P(V)\setminus \class{[e,e^{-1}]}$ is the regular representation of $(V^+,+)$, and hence a regular action. This proves \ref{axiom:C1}.

	For $x\in V^\sigma$, denote $\class{x}$ for the image of $x$ in the quotient $V^\sigma/\Rad V^\sigma$. Now $\class{\P(V)}\setminus \{\class{[e,e^{-1}]}\}$ has a natural correspondence to 
	\[\{[\class{x},\class{0}]\mid \class{x}\in V^+/\Rad V^+\}\;.\]
	The induced action of $\alpha_v$ on this set is given by $\class{x}\mapsto\class{x+v}$, which only depends on $\class{x}$. The action of $\induced{U}$ on 
	\[\{[\class{x},\class{0}]\mid \class{x}\in V^+/\Rad V^+\}\]
	is the regular representation of $(V^+/\Rad V^+,+)$, and hence a regular action. This shows \ref{axiom:C1'}.

	By \autoref{prop:JordanMuaction}, we have $[e,e^{-1}]\tau = [0Q_e,0] = [0,0]$, which is not radically equivalent to $[e,e^{-1}]$. This means we can take $0:=[0,0]$. By the same proposition, $[0,0]\tau = [e,e^{-1}+0Q_e^{-1}] = [e,e^{-1}]$, which proves \ref{axiom:C2}.
\end{proof}

As we can now use \autoref{constr:MUtau} to create $\M(U,\tau)$, we would like to use one of the equivalent conditions of \autoref{cor:construction_equivalentconditions} to prove we have a local Moufang set. In order to do this, we need to know how the maps of type $\alpha_x$, $\gamma_x$ and $\mu_x$ correspond to the maps we have already defined.
Our notation in \autoref{def:alphazetaJP} and \autoref{prop:JordanMuaction} suggests what this correspondence will be, and we make this precise in \autoref{prop:gammaJP} below.

\begin{proposition}\label{prop:gammaJP}
	Let $V$ be a local Jordan pair with invertible element $e\in V^+$ and let $U = \{\alpha_v\mid v\in V^+\}$ and $\tau = \mu_e$. For all $v,t\in V^+$ with $t$ invertible, we have $\alpha_v^{\mu_t} = \zeta_{vQ_t^{-1}}$. Using this, we get $\alpha_{[v,0]} = \alpha_v$, $\gamma_{[v,0]} = \zeta_{vQ_e^{-1}}$ and $\mu_{[t,0]} = \mu_t$. Moreover, $(-[t,0])\tau = -([t,0]\tau)$.
\end{proposition}
\begin{proof}
	We compute the action of $\alpha_v^{\mu_t}$ on the points of $\P(V)$ using \autoref{prop:JordanMuaction}. First, take $[e,e^{-1}+y]$ with $y\in V^-$. We get
	\begin{align*}
		[e,e^{-1}+y]\mu_t^{-1}\alpha_v\mu_t 
			&= [yQ_t,0]\alpha_v\mu_t = [yQ_t+v,0]\mu_t \\
			&= [e,e^{-1}+(yQ_t+v)Q_t^{-1}] \\
			&= [e,e^{-1}+y+vQ_t^{-1}] \\
			&= [e,e^{-1}+y]\zeta_{vQ_t^{-1}}\;.
	\end{align*}
	Next, take $x\in\Rad V^+$, then $xQ_t^{-1}\in\Rad V^-$, so
	\begin{align*}
		[x,0]\mu_t^{-1}\alpha_v\mu_t &= [e,e^{-1}+xQ_t^{-1}]\alpha_v\mu_t = [e,e^{-1}+(xQ_t^{-1})^v]\mu_t \\
			& = [(xQ_t^{-1})^vQ_t,0] = [x^{vQ_t^{-1}},0] \\
			& = [x,0]\zeta_{vQ_t^{-1}}\;,
	\end{align*}
	where we used \autoref{prop:JPbasic}\ref{itm:JPbasicQy}. Hence for all points of $\P(V)$, the image of $\zeta_{vQ_y}$ is equal to that of $\alpha_v^{\mu_t}$, so these permutations are equal.

	For the other statements, observe first that $\alpha_{[v,0]} = \alpha_v$ since $\alpha_v$ is the unique element of $U$ mapping $[0,0]$ to $[v,0]$, and by definition $\gamma_{[v,0]} = \alpha_{[v,0]}^\tau = \alpha_v^{\mu_e} = \zeta_{vQ_e^{-1}}$. Finally, if $t$ is invertible, we have
	\[(-[t,0])\tau = [-t,0]\tau = [e,e^{-1}+t^{-1}]\tau = [t^{-1}Q_e,0]\]
	and similarly $-([t,0]\tau) = [t^{-1}Q_e,0]$, which shows the last statement. Using the definition of $\mu_{[t,0]}$ in \autoref{constr:MUtau}, we get
	\begin{align*}
		\mu_{[t,0]} :=& \gamma_{(-[t,0])\tau^{-1}}\alpha_{[t,0]}\gamma_{-([t,0]\tau^{-1})} \\
			=& \gamma_{[t^{-1}Q_e,0]}\alpha_{[t,0]}\gamma_{[t^{-1}Q_e,0]} \\
			=& \zeta_{t^{-1}}\alpha_t\zeta_{t^{-1}} = \mu_t\;.\qedhere
	\end{align*}
\end{proof}

As we now know what all the maps of \autoref{constr:MUtau} are, we can use them to show that we have a local Moufang set.

\begin{theorem}\label{thm:MV}
	Let $V$ be a local Jordan pair with invertible element $e$. Set $U:=\{\alpha_v\mid v\in V^+\}$ and $\tau = \mu_e$,
	where $\alpha_v$ and $\mu_e$ are as in \autoref{def:alphazetaJP} and \autoref{prop:JordanMuaction}, respectively.
	Then $\M(U,\tau)$ is a local Moufang set.
\end{theorem}
\begin{proof}
	In \autoref{constr:MUtau}, we have 
	\[U_0 := U^\tau = \{\zeta_{vQ_e^{-1}}\mid v\in V^+\} = \{\zeta_w \mid w\in V^-\}\;,\]
	where the final equality follows from the fact that $Q_e$ is invertible. Now let $[t,0]$ be an arbitrary unit in $\P(V)$, then $\mu_{[t,0]} = \mu_t$, so
	\begin{align*}
		U^{\mu_{[t,0]}} &= \{\alpha_v^{\mu_t}\mid v\in V^+\} = \{\zeta_{vQ_t^{-1}}\mid v\in V^+\} \\
				&= \{\zeta_w\mid w\in V^-\} = U_0\;,
	\end{align*}
	since $t$, and hence $Q_t$ is invertible. Hence $U_0 = U_\infty^{\mu_{[t,0]}}$ for all units $[t,0]$, and \autoref{constr:MUtau} gives a local Moufang set by \autoref{thm:constrMouf}.
\end{proof}

\begin{definition}
	Let $V$ be a local Jordan pair with invertible element $e$. Set $U = \{\alpha_v\mid v\in V^+\}$ and $\tau = \mu_e$. We define $\M(V) := \M(U,\tau)$.\notatlink{M(V)}
\end{definition}


\section[Local Moufang sets to local Jordan pairs]{From local Moufang sets to local Jordan pairs}\label{sec:LMStoJP}
\subsection{The construction and basic properties}

We now investigate the reverse construction: we try to make a local Jordan pair starting from a local Moufang set satisfying some additional assumptions.
One obvious necessary assumption is that the root groups have to be abelian, and by \autoref{prop:gammaJP}, we also know that the local Moufang set has to be special.
We will also impose a restriction to avoid the cases where $V/\Rad V$ has characteristic $2$ or $3$.
Finally, we will need a linearity assumption.

Notice that for a given Jordan pair $V$, the Moufang set $\M(V)$ cannot detect the base ring~$k$ over which the Jordan pair was initially defined.
For this reason, the Jordan pair that we will (try to) construct will be defined over the base ring $\Z$, i.e., it will consist of a pair of $\Z$-modules.

\begin{importantconstruction}\label{constr:Jpair}
	Suppose $\M$ is a local Moufang set satisfying the following properties:
	\begin{manualenumerate}[label=\textnormal{(J\arabic*)},labelwidth=\widthof{(J1)}]
		\item $\M$ is special;\label{itm:J1}
		\item $U_\infty$ is abelian;\label{itm:J2}
		\item if $x$ is a unit, then so is $x\cdot2$ and $x\cdot 3$.\label{itm:J3}
	\end{manualenumerate}
	Then we define two $\Z$-modules as follows:
	\begin{itemize}
		\item $V^+ := X\setminus\class{\infty}$ with $x+z:=0\alpha_x\alpha_z$;\notatlink{xyplus}\notatlink{xymin}
		\item $V^- := X\setminus\class{0}$ with $y\plustil w:=\infty\gamma_{y\tau}\gamma_{w\tau}$.\notatlink{xyplustil}\notatlink{xymintil}
	\end{itemize}
	Now we have, for all $x,z\in V^+$ and units $t$, $(x+z)\mu_t = x\mu_t\plustil z\mu_t$, and similarly for all $y,w\in V^-$, $(y\plustil w)\mu_t = y\mu_t+ w\mu_t$. Hence $\mu$-maps induce group isomorphisms between $V^+$ and $V^-$. We now define the following maps:
	\begin{align}
		&\begin{aligned}
			&\mu_{x,z} := \mu_{x+z}-\mu_x-\mu_z:V^-\to V^+\notatlink{muxy}\\
			&\text{for units $x,z\in V^+$ such that $x+z$ is a unit;}
		\end{aligned} \label{map:muxz} \\
		&\begin{aligned}
			&\mutil_{y,w} := \mu_{y\plustil w}\mintil\mu_y\mintil\mu_w:V^+\to V^-\notatlink{muxytil}\\
			&\text{for units $y,w\in V^+$ such that $y\plustil w$ is a unit.}
		\end{aligned}\label{map:muyw}
	\end{align}
	The final assumption we make is the following:
	\begin{manualenumerate}[label=\textnormal{(J\arabic*)},labelwidth=\widthof{(J1)},resume]
		\item There are bilinear maps
		\begin{align*}
			&\mu_{\cdot,\cdot}\colon V^+ \times V^+ \to \Hom(V^-, V^+) \\
			&\mutil_{\cdot,\cdot}\colon V^- \times V^- \to \Hom(V^+, V^-)
		\end{align*}
		that extend \eqref{map:muxz} and \eqref{map:muyw}. \label{itm:J4}
	\end{manualenumerate}
	We now have a pair of $\Z$-modules $(V^+,V^-)$ and bilinear maps $\mu_{\cdot,\cdot}$ and $\mutil_{\cdot,\cdot}$ which will define a local Jordan pair, as will be shown in \autoref{thm:localJP}.
\end{importantconstruction}
\begin{remark}
	By \hyperref[itm:J1]{(J1-3)}, $\tau$ is an involution, so $\gamma_{y\tau}$ is the unique element of $U_0$ mapping $\infty$ to $y$, and hence it does not depend on $\tau$. In particular, $y\plustil w$ does not depend on the choice of $\tau$.
\end{remark}

\begin{remark}\label{rem:expression_muxz}
	If we have a local Moufang set satisfying \hyperref[itm:J1]{(J1-4)}, we can express $\mu_{x,z}$ (and similarly $\mutil_{y,w}$) in terms of $\mu$-maps for any pair of $x,z\in V^+$ by the linearity:
	\begin{align*}
		\mu_{x,z} &= \mu_{x+z}-\mu_x-\mu_z &&\text{$x$, $z$ and $x+z$ units;} \\
		\mu_{x,z} &:= -\mu_{-x,z} &&\text{$x$, $z$ units but $x+z$ not a unit;} \\
		\mu_{x,z} &:= \mu_{x,x+z}-\mu_{x,x} &&\text{$x$ a unit but $z$ not a unit;} \\
		\mu_{x,z} &:= \mu_{z,x} &&\text{$z$ a unit but $x$ not a unit;} \\
		\mu_{x,z} &:= \mu_{x+e,z}-\mu_{e,z} &&\text{$x$, $z$ not units and $e$ an arbitrary unit.}
	\end{align*}
\end{remark}

For the remainder of this section, we will always assume that we have a local Moufang set satisfying \hyperref[itm:J1]{(J1-4)}.
We start by showing some basic identities, which will help to show that the construction gives us a Jordan pair.

\begin{lemma}\label{lem:JMS_identities}
	In a local Moufang set where \hyperref[itm:J1]{\normalfont{(J1-4)}} holds, we have the following identities:
	\begin{romenumerate}
		\item $t\mu_{t,x} = -x\cdot 2$ \quad for all units $t$ and all $x\in V^+$;\label{itm:xxy}
		\item $y\mu_{t,t} = y\mu_t\cdot 2$ \quad for all units $t$ and all $y\in V^-$\label{itm:xyx};
		\item $\mu_s\mu_{s,t}\mu_t = \mu_t\mu_{s,t}\mu_s = \mutil_{s,t}$ \quad for all units $s,t$;\label{itm:reverse_order}
		\item $s\mu_t\mu_{s,t} = -t\mu_s\cdot 2$ \quad for all units $s,t$;\label{itm:xxmuyy}
		\item $y\mu_{x,z}\tau = y\tau\mutil_{x\tau,z\tau}$ \quad for all $x,z\in V^+$ and $y\in V^-$.
	\end{romenumerate}
\end{lemma}
\begin{proof}\preenum
	\begin{romenumerate}
		\item First assume $t,x$ are units such that $t+x$ is also a unit. Then by \autoref{lem:specialsum}, we have
		\[t\mu_{t+x} = -x\cdot2-t+t\mu_x\;,\]
		hence
		\[t\mu_{t,x} = -x\cdot2-t+t\mu_x - t\mu_t - t\mu_x = -x\cdot2\;,\]
		since $t\mu_t = -t$. If $t+x$ is not a unit, we can replace $x$ by $-x$ (as $x\cdot2$ is a unit by \ref{itm:J3}, $t-x=t+x-(x\cdot 2)$ is a unit) and get
		\[t\mu_{t,x} = -t\mu_{t,-x} = -(x\cdot 2) = -x\cdot 2\;.\]
		Finally, if $x$ is not a unit, we can take any unit $e$ and get
		\[t\mu_{t,x} = t\mu_{t,x-e}+t\mu_{t,e} = -(x-e)\cdot 2-e\cdot 2 = -x\cdot2\;.\]
		\item We have $\mu_{t,t} = \mu_{t\cdot 2}-\mu_t\cdot 2$. By \autoref{prop:mu-x.s}\ref{itm:mu-x.s}, we get $y\mu_{t\cdot 2} = y\mu_t\cdot 4$, so
		\[y\mu_{t,t} = y\mu_t\cdot 4 - y\mu_t\cdot 2 = y\mu_t\cdot 2\;.\]
		\item Assume first that $s+t$ is also a unit. By \autoref{prop:mu-involution}\ref{itm:mucommute}, we then have $\mu_s\mu_{s+t}\mu_t = \mu_t\mu_{s+t}\mu_s = \mu_{(s\tau+t\tau)\tau}$. By the linearity of $\tau$, we have $\mu_{(s\tau+t\tau)\tau} = \mu_{s\plustil t}$, so we can now use the definitions of $\mu_{s,t}$ and $\mutil_{s,t}$ to get
		\begin{align*}
			\mu_s(\mu_{s,t}+\mu_s+\mu_t)\mu_t &= \mu_t(\mu_{s,t}+\mu_s+\mu_t)\mu_s = \mutil_{s,t}\plustil\mu_s\plustil\mu_t\;,
			\intertext{so since all $\mu$-maps are involutions,}
			\mu_s\mu_{s,t}\mu_t \plustil \mu_t \plustil \mu_s &= \mu_t\mu_{s,t}\mu_s \plustil \mu_t \plustil \mu_s = \mutil_{s,t}\plustil\mu_s\plustil\mu_t\;,
			\intertext{hence}
			\mu_s\mu_{s,t}\mu_t &= \mu_t\mu_{s,t}\mu_s = \mutil_{s,t}\;.
		\end{align*}
		If $s+t$ is not a unit, we can replace $t$ by $-t$ and use linearity the get the result.
		\item From \ref{itm:xxy} we know $-s\mu_{s,t} = t\cdot 2$. Hence, using \ref{itm:reverse_order}, we get
		\[t\cdot 2 = s\mu_s\mu_{s,t} = s\mu_t\mu_{s,t}\mu_s\mu_t\;.\]
		Applying $\mu_t\mu_s$ to both sides, we get
		\[t\mu_t\mu_s\cdot 2 = s\mu_t\mu_{s,t}\quad\text{ and hence }\quad s\mu_t\mu_{s,t} = -t\mu_s\cdot 2\;.\]
		\item Assume first that $x$, $z$ and $x+z$ are units. Then we have
		\begin{align*}
			y\mu_{x,z}\tau &= y(\mu_{x+z}-\mu_x-\mu_z)\tau \\
				&= y\tau(\mu_{(x+z)\tau}\mintil\mu_{x\tau}\mintil\mu_{z\tau}) \\
				&= y\tau(\mu_{x\tau\plustil z\tau}\mintil\mu_{x\tau}\mintil\mu_{z\tau}) \\
				&= y\tau\mutil_{x\tau,z\tau}\;.
		\end{align*}
		By linearity, this identity now holds for all $x$ and $z$ in $V^+$.
		\qedhere
	\end{romenumerate}
\end{proof}

\begin{remark}
	Since $V^+$ and $V^-$ play the same role in the construction (our choice of $0$ and $\infty$ could have been reversed), any identity we have proven will also hold with $+$ and $-$ interchanged. For example, the identity $t\mutil_{t,y} = \mintil y\cdottil 2$ also holds for all units $t$ and all $y\in V^-$.
\end{remark}

In the next two subsections, we will prove the axioms \ref{axiom:JP1} and \ref{axiom:JP2} of a Jordan pair. In the process, we will also show the linearizations of those axioms, so by \autoref{prop:JPsufficientaxioms}\ref{itm:sufficient_JP}, we will then have shown that we have a Jordan pair, since assumptions \hyperref[itm:J1]{\normalfont{(J1-3)}} imply in particular that there is no $2$-torsion.
To prove \ref{axiom:JP1} and \ref{axiom:JP2}, we will first restrict everything to units.

\subsection{Proving the axioms for units}

We first prove \ref{axiom:JP1} by linearizing some of the basic identities we have shown earlier. The proof is based on ideas from the proof of Theorem 5.11 in \cite{DMSegevIdentitiesMS}. Remark that $\mu$-maps will correspond to the quadratic maps of the Jordan pair, and $\mu_{\cdot,\cdot}$ (and $\mutil_{\cdot,\cdot}$) will correspond to the bilinearizations $Q_{\cdot,\cdot}$. In these terms, \ref{axiom:JP1} translates to $yQ_{x,zQ_x} = xQ_{y,z}Q_x$, or in the local Moufang set: $y\mutil_{x,z\mu_x} = x\mu_{y,z}\mu_x$. Up to renaming, this is the identity we will prove for units:

\begin{proposition}\label{prop:JMS_JP1_units}
	In a local Moufang set where \hyperref[itm:J1]{\normalfont{(J1-4)}} holds, we have the following identities:
	\begin{romenumerate}
		\item $\mu_{r,s}\mu_s\mu_{r,t}+\mu_{t,s}\mu_s\mu_r = \mu_{r,t}\mu_s\mu_{r,s}+\mu_r\mu_s\mu_{t,s}$ \quad for all units $r,s,t$;\label{itm:linearize_commuting}
		\item $r\mu_s\mu_{t,s}+t\mu_s\mu_{r,s} = -s\mu_{r,t}\cdot 2$ \quad for all units $r,s,t$;\label{itm:linearize_xxmuyy}
		\item $r\mu_s\mu_{t,s} = t\mutil_{r,s}\mu_r$ \quad for all units $r,s,t$;\label{itm:JMS_JP1_units_iii}
		\item $x\mutil_{z\mu_y,y} = y\mu_{x,z}\mu_y = z\mutil_{x\mu_y,y}$ \quad for all units $x,z\in V^+$ and all units $y\in V^-$.\label{itm:JP1_units}
	\end{romenumerate}
\end{proposition}
\begin{proof}\preenum
	\begin{romenumerate}
		\item We start with the first identity of Lemma~\ref{lem:JMS_identities}\ref{itm:reverse_order}, conjugating both sides by $\mu_t$ and then replacing $t$ by $r+t\cdot \ell$, for those $\ell\in\{1,2,3\}$ for which $r+t\cdot \ell$ is a unit. Since $\mu_{r+t\cdot \ell} = \mu_{r,t}\cdot \ell+\mu_r+\mu_t\cdot \ell^2$, we get
		\begin{align*}
			&(\mu_{r,s}+\mu_{t,s}\cdot \ell)\mu_s(\mu_{r,t}\cdot \ell+\mu_r+\mu_t\cdot \ell^2) \\
			&= (\mu_{r,t}\cdot \ell+\mu_r+\mu_t\cdot \ell^2)\mu_s(\mu_{r,s}+\mu_{t,s}\cdot \ell)
		\end{align*}
		and after expanding,
		\begin{align*}
			&\mu_{r,s}\mu_s\mu_r + (\mu_{r,s}\mu_s\mu_{r,t}+\mu_{t,s}\mu_s\mu_r)\cdot \ell \\
			&\qquad+ (\mu_{t,s}\mu_s\mu_{r,t}+\mu_{r,s}\mu_s\mu_t)\cdot \ell^2 + \mu_{t,s}\mu_s\mu_t\cdot \ell^3 \\
			&= \mu_r\mu_s\mu_{r,s} + (\mu_{r,t}\mu_s\mu_{r,s}+\mu_r\mu_s\mu_{t,s})\cdot \ell \\
			&\qquad+ (\mu_{r,t}\mu_s\mu_{t,s}+\mu_t\mu_s\mu_{r,s})\cdot \ell^2 + \mu_t\mu_s\mu_{t,s}\cdot \ell^3
		\end{align*}
		Observe that the constant terms and terms with $\ell^3$ cancel due to Lemma~\ref{lem:JMS_identities}\ref{itm:reverse_order}, so we have
		\begin{align*}
			&(\mu_{r,s}\mu_s\mu_{r,t}+\mu_{t,s}\mu_s\mu_r)\cdot \ell + (\mu_{t,s}\mu_s\mu_{r,t}+\mu_{r,s}\mu_s\mu_t)\cdot \ell^2 \\
			&= (\mu_{r,t}\mu_s\mu_{r,s}+\mu_r\mu_s\mu_{t,s})\cdot \ell + (\mu_{r,t}\mu_s\mu_{t,s}+\mu_t\mu_s\mu_{r,s})\cdot \ell^2\;.
		\end{align*}
		Observe now that there are at least two values of $\ell$ for which $r+t\cdot \ell$ is a unit, since $t$ and $t\cdot2$ are units and adding a unit to a non-unit gives a unit. Using those two values, we can deduce that the coefficients on the left and right hand side of both $\ell$ and $\ell^2$ are equal. This means that
		\[\mu_{r,s}\mu_s\mu_{r,t}+\mu_{t,s}\mu_s\mu_r = \mu_{r,t}\mu_s\mu_{r,s}+\mu_r\mu_s\mu_{t,s}\;.\]
		\item We similarly linearize \autoref{lem:JMS_identities}\ref{itm:xxmuyy}, replacing $r$ by $r+t$ if $r+t$ is a unit (if not, replace $r$ by $r-t$). We get
		\begin{align*}
			(r+t)\mu_s\mu_{r+t,s} &= -s\mu_{r+t}\cdot 2 \;,
		\end{align*}
		so using \ref{itm:J4} and the definition of $\mu_{r,t}$,
		\begin{align*}
			&r\mu_s\mu_{r,s}+r\mu_s\mu_{t,s}+t\mu_s\mu_{r,s}+t\mu_s\mu_{t,s} \\
			&= -s\mu_{r,t}\cdot 2 - s\mu_r\cdot 2 - s\mu_t\cdot 2\;.
		\end{align*}
		We can now use \autoref{lem:JMS_identities}\ref{itm:xxmuyy} twice to get
		\begin{align*}
			r\mu_s\mu_{t,s}+t\mu_s\mu_{r,s} &= -s\mu_{r,t}\cdot 2\;.
		\end{align*}
		\item We apply identity \ref{itm:linearize_commuting} to the element $r$, and use \autoref{lem:JMS_identities} to get
		\begin{align*}
			(-s\cdot 2)\mu_s\mu_{r,t} + r\mu_{t,s}\mu_s\mu_r &= (-t\cdot 2)\mu_s\mu_{r,s} + (-r)\mu_s\mu_{t,s} \;. \\
		\intertext{Using linearity, this yields}
			s\mu_{r,t}\cdot 2 + r\mu_{t,s}\mu_s\mu_r &= -t\mu_s\mu_{r,s}\cdot 2 -r\mu_s\mu_{t,s} \;, \\
		\intertext{and by \ref{itm:linearize_xxmuyy},}
			r\mu_{t,s}\mu_s\mu_r &= -t\mu_s\mu_{r,s} \,. \\
		\intertext{We now replace $r$ by $r\mu_s$ and apply $\mu_s\mu_r$:}
			r\mu_s\mu_{t,s}\mu_r\mu_s\mu_s\mu_r &= -t\mu_s\mu_{r\mu_s,s}\mu_s\mu_r \\
			\implies\quad r\mu_s\mu_{t,s} &= -t\mutil_{r,s\mu_s}\mu_r \\
			\implies\quad r\mu_s\mu_{t,s} &= t\mutil_{r,s}\mu_r\;.
		\end{align*}
		\item By \autoref{lem:JMS_identities}\ref{itm:reverse_order}, we have $\mutil_{r,s}\mu_r = \mu_s\mu_{r,s}$, so \ref{itm:JMS_JP1_units_iii} becomes
		\[ r\mu_s\mu_{t,s} = t\mu_s\mu_{r,s}\;.\]
		We can now plug this in \ref{itm:linearize_xxmuyy} and use the unique $2$-divisibility to get
		\[r\mu_s\mu_{t,s} = -s\mu_{r,t}\;.\]
		We now apply $\mu_s$ to both sides and use linearity to get
		\[s\mu_{r,t}\mu_s = \mathord{\mintil}r\mu_s\mu_{t,s}\mu_s = \mathord{\mintil}r\mutil_{t\mu_s,\mathord{\mintil}s} = r\mutil_{t\mu_s,s}\;.\]
		After renaming variables, we get the first identity we wanted to prove. For the second identity, remark that $y\mu_{x,z}\mu_y$ is symmetric in $x$ and $z$, hence
		\[ x\mutil_{z\mu_y,y} = y\mu_{x,z}\mu_y = z\mutil_{x\mu_y,y}\;.\qedhere\]
	\end{romenumerate}
\end{proof}

\begin{remark}
	The technique used in the previous lemma will be used extensively to linearize many identities which hold when all unknowns are units. We describe it here in generality: replace an unknown $x$ by $x+\hat{x}\cdot \ell$ for some unused variable name $\hat{x}$, and $\ell\in\{1,2,3,4\}$. Next, we can combine several facts to expand the resulting identity as a polynomial in powers of $\ell$:
	\begin{itemize}
		\item the linearity of $\mu_{\cdot,\cdot}$ and $\mutil_{\cdot,\cdot}$;
		\item the definition of $\mu_{\cdot,\cdot}$ to expand $\mu_{x+\hat{x}\cdot\ell} = \mu_{x,\hat{x}}\cdot\ell + \mu_x + \mu_{\hat{x}\cdot\ell}$, which requires $x+\hat{x}\cdot\ell$ to be a unit;
		\item the identity $\mu_{\hat{x}\cdot\ell} = \mu_{\hat{x}}\cdot\ell^2$, which requires $\hat{x}$ to be a unit.
	\end{itemize}
	We assume the highest power of $\ell$ occurring is $\ell^4$. By the identity we started with, the coefficients of $\ell^0$ and $\ell^4$ will always be equal. If we now find $3$ values for which $x+\hat{x}\cdot \ell$ is a unit, we can solve the Vandermonde system of equations and then we know that the coefficients of $\ell^1$, $\ell^2$ and $\ell^3$ are also equal.

	This technique will be used in many proofs to come, but it does not necessarily work for any identity in $\mu_\cdot$, $\mu_{\cdot,\cdot}$ and $\mutil_{\cdot,\cdot}$. It can be checked that it does work whenever it is used.
\end{remark}

Next, we will prove \ref{axiom:JP2} for units. This axiom for Jordan pairs corresponds to the Triple Shift Formula for Jordan algebras (see \cite[p.~202]{McCrimmonTaste}), which can be deduced from the axioms of Jordan algebras. We will use ideas from \cite{JacobsonJordanAlgebrasNotes} where such a deduction is made, and adapt them to the context of local Moufang sets. This will require many intermediate identities and will also require the choice of a fixed invertible element of the Jordan-pair-to-be. Hence we fix a unit $e$\notatlink{e} of our local Moufang set, which we will use throughout the following few lemmas. We begin with two basic consequences of \autoref{prop:JMS_JP1_units}.
\begin{lemma}\label{lem:JMS_basic_identities}
	In a local Moufang set where \hyperref[itm:J1]{\normalfont{(J1-4)}} holds, we have the following identities:
	\begin{romenumerate}
		\item for all units $x\in V^+$ and $y\in V^-$\label{itm:muxe}
			\[y\mu_{x,e} = e\mutil_{y,x\mu_e}\mu_e = -e\mu_{y\mu_e,x}\;;\]
		\item $x\mu_e\mu_{x,e} = -e\mu_x\cdot 2$ \quad for all units $x\in V^+$. \label{itm:xmuemuxe}
	\end{romenumerate}
\end{lemma}
\begin{proof}\preenum
	\begin{romenumerate}
		\item We take \autoref{prop:JMS_JP1_units}\ref{itm:JP1_units}, interchange the roles of $V^+$ and $V^-$ and replace $y$ by $e$, $x$ by $y$ and $z$ by $x\mu_e$. This gives the first equality. For the second, we use
		\[e\mutil_{y,x\mu_e}\mu_e = e\mu_e\mu_{y\mu_e,x\mu_e\mu_e} = (\mathord{\mintil}e)\mu_{y\mu_e,x} = -e\mu_{y\mu_e,x}\;.\]
		\item Setting $y = x\mu_e$ in \ref{itm:muxe}, we get $x\mu_e\mu_{x,e} = -e\mu_{x,x} = -e\mu_x\cdot 2$.\qedhere
	\end{romenumerate}
\end{proof}

We start building up some identities that we will use to prove \ref{axiom:JP2} for units.

\begin{lemma}\label{lem:JMS_QJ-identities}
	Let $\M$ be a local Moufang set satisfying \hyperref[itm:J1]{\normalfont{(J1-4)}}.
    Then for all units $x,z,v \in V^+$ and all units $y,w \in V^-$, the following identities hold:
	\begin{romenumerate}
		\item $\mu_x\mu_y\mu_z+\mu_z\mu_y\mu_x+\mu_{x,z}\mu_y\mu_{x,z} = \mu_{y\mu_{x,z}} + \mu_{y\mu_x,y\mu_z}$; \label{QJ7}
		\item $\mu_{x,e}\mu_e\mu_{x,e}+\mu_{e\mu_x,e} = \mu_x\cdot 2$; \label{QJ20}
		\item $y\mu_{x,w\mu_{x,z}}+y\mu_{z,w\mu_x} = x\mutil_{y,w}\mu_{x,z} + z\mutil_{y,w}\mu_x$; \label{QJ9}
		\item $e\mu_{z,e\mu_x} = x\mu_e\mu_{z,x}$; \label{QJ26}
		\item $e\mu_{v,e\mu_{x,z}} = x\mu_e\mu_{v,z}+z\mu_e\mu_{v,x}$. \label{QJ27}
	\end{romenumerate}
\end{lemma}
\begin{proof}\preenum
	\begin{romenumerate}
		\item We start from the identity $\mu_x\mu_y\mu_x = \mu_{y\mu_x}$, and linearize $y$ to $y\plustil y'\cdottil\ell$.
        Equating the coefficients of $\ell$ on both sides yields
		\[\mu_x\mu_{y,y'}\mu_x = \mu_{y\mu_x,y'\mu_x}\;.\]
		Next, we linearize $x$ to $x+z\cdot\ell$, and equate the coefficients of $\ell^2$ on both sides of the equality; this gives
		\begin{align*}
			&\mu_x\mu_{y,y'}\mu_z+\mu_z\mu_{y,y'}\mu_x+\mu_{x,z}\mu_{y,y'}\mu_{x,z} \\
			&= \mu_{y\mu_x,y'\mu_z}+\mu_{y\mu_z,y'\mu_x}+\mu_{y\mu_{x,z},y'\mu_{x,z}}\;.
		\end{align*}
		We can now set $y'=y$ and use \autoref{lem:JMS_identities}\ref{itm:xyx} to get
		\[\mu_x\mu_y\mu_z\cdot 2+\mu_z\mu_y\mu_x\cdot 2+\mu_{x,z}\mu_y\mu_{x,z}\cdot 2 = \mu_{y\mu_x,y\mu_z}\cdot 2+\mu_{y\mu_{x,z}}\cdot 2\;.\]
		The unique $2$-divisibility now gives us the desired identity.
		\item We set $y=z=e$ in \ref{QJ7} and get
		\[\mu_x\mu_e\mu_e + \mu_e\mu_e\mu_x + \mu_{x,e}\mu_e\mu_{x,e} = \mu_{e\mu_{x,e}} + \mu_{e\mu_x,e\mu_e}\;,\]
		which reduces to
		\[\mu_x\cdot 2 + \mu_{x,e}\mu_e\mu_{x,e} = \mu_{-x\cdot 2} + \mu_{e\mu_x,-e}\]
		and by $\mu_{-x\cdot 2} = \mu_x\cdot 4$ and linearity, we get
		\[\mu_x\cdot 2 + \mu_{x,e}\mu_e\mu_{x,e} = \mu_{x}\cdot 4 - \mu_{e\mu_x,e}\;,\]
		so after rearranging we get the identity we wanted to prove.
		\item Starting from the first equality of \autoref{prop:JMS_JP1_units}\ref{itm:JP1_units}, we interchange the roles of $V^+$ and $V^-$ and rename some variables to get $y\mu_{w\mu_x,x} = x\mutil_{y,w}\mu_x$. Next, we linearize $x$ to $x+z\cdot\ell$; the coefficients of $\ell^1$ give the desired equality.
		\item Set $x=e$ and $w = z\mu_e$ in \ref{QJ9} to get
		\[y\mu_{e,z\mu_e\mu_{e,z}}+y\mu_{z,z} = e\mutil_{y,z\mu_e}\mu_{e,z} + z\mutil_{y,z\mu_e}\mu_e\;.\]
		By \autoref{lem:JMS_basic_identities}\ref{itm:xmuemuxe}, $z\mu_e\mu_{e,z} = -e\mu_z\cdot 2$, so the previous identity becomes
		\begin{equation}
			-y\mu_{e,e\mu_z}\cdot 2+y\mu_z\cdot 2 = e\mutil_{y,z\mu_e}\mu_{e,z} + z\mu_e\mu_{y\mu_e,z}\;. \label{eq:QJ26_stepi}
		\end{equation}
		Next, we take identity \ref{QJ20}, replace $x$ by $z$, and apply it to $y$. This gives
		\[y\mu_{z,e}\mu_e\mu_{z,e}+y\mu_{e\mu_z,e} = y\mu_z\cdot 2\;,\]
		which we can combine with \eqref{eq:QJ26_stepi} to
		\[-y\mu_{e,e\mu_z}+y\mu_{z,e}\mu_e\mu_{z,e} = e\mutil_{y,z\mu_e}\mu_{e,z} + z\mu_e\mu_{y\mu_e,z}\;.\]
		By \autoref{lem:JMS_basic_identities}\ref{itm:muxe}, we have
		\[y\mu_{z,e} = e\mutil_{y,z\mu_e}\mu_e\;,\]
		so
		\[y\mu_{z,e}\mu_e\mu_{z,e} = e\mutil_{y,z\mu_e}\mu_{e,z}\;.\]
		From this, we get
		\[-y\mu_{e,e\mu_z} = z\mu_e\mu_{y\mu_e,z}\;.\]
		Again by \autoref{lem:JMS_basic_identities}\ref{itm:muxe}, we have
		\[-y\mu_{e,e\mu_z} = -(-e\mu_{y\mu_e,e\mu_z}) = e\mu_{y\mu_e,e\mu_z}\;.\]
		Combining these last two identities, replacing $z$ by $x$ and $y$ by $z\mu_e$ gives the desired identity.
		\item We linearize $x$ to $x+v\cdot\ell$ in \ref{QJ26}, take the coefficients in $\ell^1$ and interchange $z$ and $v$ to get the desired identity.\qedhere
	\end{romenumerate}
\end{proof}

We are now ready to prove \ref{axiom:JP2} for units. Our starting point is an identity which is symmetric in two unknowns on one side of the equality sign, and hence must also be symmetric in those unknowns on the other side.
For clarity in the notation, we will occasionally replace $\mu_{\cdot,\cdot}$ by $\mu(\cdot,\cdot)$.

\begin{proposition}\label{prop:JMS_JP2_units}
	In a local Moufang set where \hyperref[itm:J1]{\normalfont{(J1-4)}} holds, we have the following identities:
	\begin{romenumerate}
		\item $e\mu(z\mu_e\mu_x,z) = e\mu(x\mu_e\mu_z,x)$ \quad for all units $x,z\in V^+$;\label{QJ29}
		\item for all units $x,z\in V^+$ and any $v\in V^+$\label{QJ34}
		\begin{align*}
			e\mu(z\mu_e\mu_{x,v},z) = e\mu(x\mu_e\mu_z,v)+e\mu(v\mu_e\mu_z,x)\;;
		\end{align*}
		\item for all units $x,z,v\in V^+$ \label{QJ33}
		\begin{align*}
			&e\mu(x\mu_e\mu_z,x\mu_e\mu_{v,e}) + e\mu(x\mu_e\mu_{v,e}\mu_e\mu_z,x) \\
			&= e\mu_{z,v}\mu_e\mu_x\mu_e\mu_{z,e} + e\mu(z\mu_e\mu_x,v)\mu_e\mu_{z,e} \;;
		\end{align*}
		\item for all units $x,z,v\in V^+$\label{QJ32}
		\begin{align*}
			&v\mu_e\mu(z\mu_e\mu_x,z) - v\mu_e\mu(x\mu_e\mu_z,x) \\
			&= e\mu(x\mu_e\mu_z,x\mu_e\mu_{v,e}) - e\mu(z\mu_e\mu_x,v)\mu_e\mu_{z,e} \;;
		\end{align*}
		\item $x\mu_y\mu_{x,z} = y\mu_{y\mu_x,z}$ \quad for all units $x,z\in V^+$ and all units $y\in V^-$.\label{itm:JP2_units}
	\end{romenumerate}
\end{proposition}
\begin{proof}\preenum
	\begin{romenumerate}
		\item We start with \autoref{prop:JMS_JP1_units}\ref{itm:JP1_units}, where we interchange the roles of $V^+$ and $V^-$, set $x=e$ and rename the other variables variables:
		\[e\mu_{y\mu_z,z} = z\mutil_{e,y}\mu_z\;.\]
		Linearizing $z$ to $z+x\cdot\ell$ and taking the coefficients of $\ell^2$ gives us
		\[e\mu_{y\mu_{x,z},x}+e\mu_{y\mu_x,z} = x\mutil_{e,y}\mu_{x,z}+z\mutil_{e,y}\mu_x\;.\]
		Substituting $z\mu_e$ for $y$, we get
		\[e\mu(z\mu_e\mu_{x,z},x)+e\mu(z\mu_e\mu_x,z) = x\mutil_{e,z\mu_e}\mu_{x,z}+z\mutil_{e,z\mu_e}\mu_x\;.\]
		We want to show that $e\mu(z\mu_e\mu_x,z)$ is symmetric in $x$ and $z$, i.e.\ we need to show that the remaining terms are symmetric in $x$ and $z$. By \autoref{prop:JMS_JP1_units}\ref{itm:JP1_units},
		\[x\mutil_{e,z\mu_e}\mu_{x,z} = e\mu_{x,z}\mu_e\mu_{x,z}\;,\]
		so this term is symmetric. Hence it remains to show that $z\mutil_{e,z\mu_e}\mu_x-e\mu(z\mu_e\mu_{x,z},x)$ is symmetric.
		We have
		\begin{align*}
			&z\mutil_{e,z\mu_e}\mu_x-e\mu(z\mu_e\mu_{x,z},x) \\
			&= e\mu_{z,z}\mu_e\mu_x - e\mu(e\mu_{x,e\mu_z},x) \\
			&= e\mu_z\mu_e\mu_x\cdot 2 - e\mu(e\mu_{x,e\mu_z},x) \\
		\intertext{by \autoref{prop:JMS_JP1_units}\ref{itm:JP1_units} and \autoref{lem:JMS_QJ-identities}\ref{QJ26}. By \autoref{lem:JMS_QJ-identities}\ref{QJ20} applied to $e\mu_z\mu_e$, this is}
			&= e\mu_z\mu_e\mu_{x,e}\mu_e\mu_{x,e} + e\mu_z\mu_e\mu_{e\mu_x,e} - e\mu(e\mu_{x,e\mu_z},x) \\
			&= e\mu_z\mu_e\mu_{x,e}\mu_e\mu_{x,e} - e\mu_{e\mu_z,e\mu_x} - e\mu(e\mu_{x,e\mu_z},x)\;,
		\end{align*}
		by \autoref{lem:JMS_basic_identities}\ref{itm:muxe}. The term $e\mu_{e\mu_z,e\mu_x}$ is again symmetric in $x$ and $z$, so it is sufficient to prove that the remaining difference is symmetric. We will, in fact, show that this expression is always $0$ by using \autoref{lem:JMS_basic_identities}\ref{itm:muxe} twice:
		\begin{align*}
			e\mu_z\mu_e\mu_{x,e}\mu_e\mu_{x,e} &= -e\mu(e\mu_z\mu_e\mu_{x,e},x) \\
				&= -e\mu(-e\mu_{e\mu_z,x},x) = e\mu(e\mu_{x,e\mu_z},x)\;.
		\end{align*}
		Putting everything together, we get
		\[e\mu(z\mu_e\mu_x,z) = e\mu_{x,z}\mu_e\mu_{x,z} - e\mu_{e\mu_z,e\mu_x}\;,\]
		which is symmetric in $x$ and $z$, hence we must have 
		\[e\mu(z\mu_e\mu_x,z) = e\mu(x\mu_e\mu_z,x)\;.\]
		\item Linearize $x$ to $x+v\cdot\ell$ in \ref{QJ29} and take the coefficients of $\ell^1$. This shows the desired identity for any unit $v\in V^+$. If $v$ is not a unit, add the identity for $v-e$ and $e$ (both units) and use the linearity in $v$ to get the identity for $v$.
		\item Start with \autoref{lem:JMS_QJ-identities}\ref{QJ9} and set $z=e$, $v = v\mu_e$ and $y = z\mu_e$:
		\begin{align*}
			&z\mu_e\mu(x,v\mu_e\mu_{x,e})+z\mu_e\mu(e,v\mu_e\mu_x) \\
			&= x\mutil_{z\mu_e,v\mu_e}\mu_{x,e} + e\mutil_{z\mu_e,v\mu_e}\mu_x\;.
		\end{align*}
		Using \autoref{lem:JMS_basic_identities}\ref{itm:muxe} twice, we get
		\begin{align*}
			&z\mu_e\mu(x,v\mu_e\mu_{x,e}) \\
			&= x\mutil_{z\mu_e,v\mu_e}\mu_{x,e} + e\mutil_{z\mu_e,v\mu_e}\mu_x - z\mu_e\mu(e,v\mu_e\mu_x) \\
			&= x\mu_e\mu_{z,v}\mu_e\mu_{x,e} - e\mu_{z,v}\mu_e\mu_x - z\mu_e\mu(v\mu_e\mu_x,e) \\
			&= -e\mu(x\mu_e\mu_{z,v},x) - e\mu_{z,v}\mu_e\mu_x + e\mu(z,v\mu_e\mu_x)\;,
		\intertext{and by \ref{QJ34} with $x$ and $z$ interchanged, this becomes}
			&= - e\mu_{z,v}\mu_e\mu_x - e\mu(z\mu_e\mu_x,v)\;.
		\end{align*}
		Next, we apply $\mu_e\mu_{z,e}$ to this identity, and we use $v\mu_e\mu_{x,e} = x\mu_e\mu_{v,e}$, a consequence of \autoref{prop:JMS_JP1_units}\ref{itm:JP1_units}:
		\begin{align*}
		&e\mu_{z,v}\mu_e\mu_x\mu_e\mu_{z,e} + e\mu(z\mu_e\mu_x,v)\mu_e\mu_{z,e} \\
		&= -z\mu_e\mu(x,x\mu_e\mu_{v,e})\mu_e\mu_{z,e} \\
		&= e\mu(z,z\mu_e\mu(x,x\mu_e\mu_{v,e}))\;.
		\end{align*}
		Finally, take \ref{QJ34} and set $v = x\mu_e\mu_{v,e}$ (this need not be a unit, but we have shown this identity for non-units as well). This gives
		\begin{align*}
			&e\mu(z\mu_e\mu(x,x\mu_e\mu_{v,e}),z) \\
			&= e\mu(x\mu_e\mu_z,x\mu_e\mu_{v,e}) + e\mu(x\mu_e\mu_{v,e}\mu_e\mu_z,x)\;,
		\end{align*}
		hence
		\begin{align*}
			&e\mu(x\mu_e\mu_z,x\mu_e\mu_{v,e}) + e\mu(x\mu_e\mu_{v,e}\mu_e\mu_z,x) \\
			&= e\mu_{z,v}\mu_e\mu_x\mu_e\mu_{z,e} + e\mu(z\mu_e\mu_x,v)\mu_e\mu_{z,e}\;.
		\end{align*}
		\item Set $z = e$ and $y = z\mu_e$ in \autoref{lem:JMS_QJ-identities}\ref{QJ7}. This gives
		\begin{align*}
			&\mu_e\mu_x\mu_e\mu_z\mu_e+\mu_e\mu_z\mu_e\mu_x\mu_e+\mu_{x,e}\mu_e\mu_z\mu_e\mu_{x,e} \\
			&= \mu(z\mu_e\mu_{x,e}) + \mu(z\mu_e\mu_x,z)\;.
		\end{align*}
		We take the difference of this identity with the same identity, but interchanging $x$ and $z$. Using $z\mu_e\mu_{x,e} = x\mu_e\mu_{z,e}$, we get
		\[\mu_{x,e}\mu_e\mu_z\mu_e\mu_{x,e}-\mu_{z,e}\mu_e\mu_x\mu_e\mu_{z,e} = \mu(z\mu_e\mu_x,z) - \mu(x\mu_e\mu_z,x)\;,\]
		which we apply to $v\mu_e$:
		\begin{align*}
			&v\mu_e\mu_{x,e}\mu_e\mu_z\mu_e\mu_{x,e}-v\mu_e\mu_{z,e}\mu_e\mu_x\mu_e\mu_{z,e}  \\
			&= v\mu_e\mu(z\mu_e\mu_x,z) - v\mu_e\mu(x\mu_e\mu_z,x)\;.
		\end{align*}
		We repeatedly use \autoref{lem:JMS_basic_identities}\ref{itm:muxe} to get
		\begin{align*}
			&v\mu_e\mu(z\mu_e\mu_x,z) - v\mu_e\mu(x\mu_e\mu_z,x) \\
			&= v\mu_e\mu_{x,e}\mu_e\mu_z\mu_e\mu_{x,e} - v\mu_e\mu_{z,e}\mu_e\mu_x\mu_e\mu_{z,e} \\
			&= x\mu_e\mu_{v,e}\mu_e\mu_z\mu_e\mu_{x,e} + e\mu_{v,z}\mu_e\mu_x\mu_e\mu_{z,e} \\
			&= -e\mu(x,x\mu_e\mu_{v,e}\mu_e\mu_z) + e\mu_{v,z}\mu_e\mu_x\mu_e\mu_{z,e} \\
			&= e\mu(x\mu_e\mu_z,x\mu_e\mu_{v,e}) - e\mu(z\mu_e\mu_x,v)\mu_e\mu_{z,e} \\
			&= -e\mu(x\mu_e\mu_z,e\mu_{x,v}) + e\mu(z,e\mu(z\mu_e\mu_x,v))
		\end{align*}
		in which the second last step follows from \ref{QJ33}.
		\item Set $z = v$ and $v = x\mu_e\mu_z$ in \autoref{lem:JMS_QJ-identities}\ref{QJ27} to get
		\[e\mu(x\mu_e\mu_z,e\mu_{x,v}) - v\mu_e\mu(x\mu_e\mu_z,x) = x\mu_e\mu(x\mu_e\mu_z,v)\;.\]
		Combining this with \ref{QJ32} gives
		\[x\mu_e\mu(x\mu_e\mu_z,v) = e\mu(z,e\mu(z\mu_e\mu_x,v)) - v\mu_e\mu(z\mu_e\mu_x,z)\;.\]
		Next, we substitute $z\mu_e\mu_x$ for $z$, $z$ for $v$ and $v$ for $x$ in identity \ref{QJ27} of \autoref{lem:JMS_QJ-identities} to get
		\[e\mu(z,e\mu(z\mu_e\mu_x,v)) - v\mu_e\mu(z\mu_e\mu_x,z) = z\mu_e\mu_x\mu_e\mu_{z,v}\;.\]
		Combining these last two identities gives us
		\[z\mu_e\mu_x\mu_e\mu_{z,v} = x\mu_e\mu(x\mu_e\mu_z,v)\;.\]
		Substituting $x\mu_e$ for $y$, $x$ for $z$ and $z$ for $v$ turns this into
		\[x\mu_y\mu_{x,z} = y\mu_{y\mu_x,z}\;,\]
		which is the identity we wanted.\qedhere
	\end{romenumerate}
\end{proof}

\subsection{The construction gives a Jordan pair}

We can now prove \ref{axiom:JP1} and \ref{axiom:JP2} by linearizing their counterparts for units. We will henceforth use the notation $\{\cdot\cdot\cdot\}$ from Jordan pairs to denote the triple product, i.e.\ $\{x\,y\,z\}:= y\mu_{x,z}$ and $\{y\,x\,w\}:= x\mutil_{y,w}$. We remind the reader that these triple products are linear in all their arguments.

\begin{proposition}\label{prop:JP1}
	In a local Moufang set where \hyperref[itm:J1]{\normalfont{(J1-4)}} holds, we have the following identities:
	\begin{romenumerate}
		\item for all units $y,w\in V^-$ and all units $x,z\in V^+$\label{itm:JP1almostlinear}
			\begin{align*}
				\{x\,y\,w\mu_{x,z}\} + \{z\,y\,w\mu_x\} &= \{y\,x\,w\}\mu_{x,z} + \{y\,z\,w\}\mu_x\;;
			\end{align*}
		\item for all $y,w\in V^-$ and all $x,z,v\in V^+$ \label{itm:JP1linear}
			\begin{align*}
				&\{x\,y\,w\mu_{v,z}\} + \{v\,y\,w\mu_{x,z}\} + \{z\,y\,w\mu_{x,v}\} \\
				&= \{y\,x\,w\}\mu_{v,z} + \{y\,v\,w\}\mu_{x,z} + \{y\,z\,w\}\mu_{x,v}
			\end{align*}
		\item $\{x\,y\,w\mu_{x,x}\} = \{y\,x\,w\}\mu_{x,x}$ \quad for all $y,w\in V^-$ and $x\in V^+$.
	\end{romenumerate}
\end{proposition}
\begin{proof}\preenum
	\begin{romenumerate}
		\item After renaming, $\{x\,y\,w\mu_x\} = \{y\,x\,w\}\mu_x$ follows from \autoref{prop:JMS_JP1_units}\ref{itm:JP1_units}. We linearize $x$ to $x+z\cdot\ell$ in this identity, and the equality of the coefficients of $\ell^1$ is then the desired identity.
		\item We linearize $x$ to $x+v\cdot\ell$ in \ref{itm:JP1almostlinear} and take the coefficients of $\ell^1$ to get the identity we want for units, i.e.
		\begin{align*}
			&\{x\,y\,w\mu_{v,z}\} + \{v\,y\,w\mu_{x,z}\} + \{z\,y\,w\mu_{x,v}\} \\
			&= \{y\,x\,w\}\mu_{v,z} + \{y\,v\,w\}\mu_{x,z} + \{y\,z\,w\}\mu_{x,v}
		\end{align*}
		holds for all units $x,z,v\in V^-$ and all units $y,w\in V^+$. We now claim that the variables do not need to be units. If any of the variables $X$ is not a unit, take any unit $e$ and write $X = (X-e)+e$ (or $X = (X\mintil e)\plustil e$ if the variable is in $V^-$). The required identity then follows, using the linearity in $X$, and the fact that the identity holds for the units $X-e$ and $e$. Hence the identity holds for any $x,z,v\in V^-$ and $y,w\in V^+$.
		\item We take $x=z=v$ in \ref{itm:JP1linear}, and hence get
		\[\{x\,y\,w\mu_{x,x}\}\cdot 3 = \{y\,x\,w\}\mu_{x,x}\cdot 3\]
		for any $x\in V^-$ and $y,w\in V^+$. By the unique $3$-divisibility, we get the desired identity.\qedhere
	\end{romenumerate}
\end{proof}

\begin{proposition}\label{prop:JP2}
	In a local Moufang set where \hyperref[itm:J1]{\normalfont{(J1-4)}} holds, we have the following identities:
	\begin{romenumerate}
		\item for all units $x,z,v\in V^+$ and any unit $y\in V^-$
			\begin{align*}
				\{v\,x\mu_y\,z\} + \{x\,v\mu_y\,z\} = \{y\mu_{x,v}\,y\,z\}\;;
			\end{align*}
		\item for all $x,z,v\in V^+$ and all $y,w\in V^-$ \label{itm:JP2linear}
			\begin{align*}
				&\{v\,x\mu_{y,w}\,z\} + \{x\,v\mu_{y,w}\,z\} \\
				&= \{y\mu_{x,v}\,w\,z\} + \{w\mu_{x,v}\,y\,z\}\;;
			\end{align*}
		\item $\{x\,x\mu_{y,y}\,z\} = \{y\mu_{x,x}\,y\,z\}$ \quad for all $x,z\in V^+$ and $y\in V^+$.
	\end{romenumerate}
\end{proposition}
\begin{proof}\preenum
	\begin{romenumerate}
		\item Using the definition of the triple product, we can rewrite \autoref{prop:JMS_JP2_units}\ref{itm:JP2_units} as $\{x\,x\mu_y\,z\} = \{y\mu_x\,y\,z\}$ for all units $x,z\in V^+$ and all units $y\in V^-$. We linearize $x$ to $x+v\cdot\ell$ and take coefficients of $\ell^1$ to get the desired identity.
		\item We linearize $y$ to $y\plustil w\cdottil\ell$ and take the coefficients of $\ell^1$ to get the required identity for units, i.e.
		\[
		\{v\,x\mu_{y,w}\,z\} + \{x\,v\mu_{y,w}\,z\} = \{y\mu_{x,v}\,w\,z\} + \{w\mu_{x,v}\,y\,z\}\]
		holds for all units $x,z,v\in V^+$ and all units $y,w\in V^-$. As in the proof of \autoref{prop:JP1}\ref{itm:JP1linear}, we can use linearity to prove this identity for all $x,z,v\in V^+$ and $y,w\in V^-$.
		\item In \ref{itm:JP2linear}, set $v=x$ and $w=y$ to get
		\[\{x\,x\mu_{y,y}\,z\}\cdot 2 = \{y\mu_{x,x}\,y\,z\}\cdot 2\]
		for any $x,z\in V^+$ and any $y\in V^-$. By the unique $2$-divisibility, we get the desired identity.\qedhere
	\end{romenumerate}
\end{proof}

Using these linearizations, we can immediately show that we have a Jordan pair.

\begin{theorem}\label{thm:localJP}
	Let $\M$ be a local Moufang set satisfying \hyperref[itm:J1]{\normalfont{(J1-4)}}.
    Then \autoref{constr:Jpair} gives a Jordan pair $(V^+,V^-)$ with
	\[Q_x^+ = \mu_{x,x}\cdot\frac{1}{2}\text{ for all $x\in V^+$}\quad\text{and}\quad Q_y^- = \mutil_{y,y}\cdottil\frac{1}{2}\text{ for all $y\in V^-$.}\]
	Furthermore, the non-invertible elements form a proper ideal 
	\[I = (I^+,I^-) = (\class{0},\class{\infty})\;,\]
	so $V$ is a local Jordan pair with $\Rad V = I$. Moreover, $V^+$ is uniquely $2$- and $3$-divisible.
\end{theorem}
\begin{proof}
	By \ref{itm:J2}, both $V^+$ and $V^-$ are $\Z$-modules, and by the fact that $\mu$-maps are morphisms between the (abelian) groups $V^+$ and $V^-$, the maps $Q_x^+$ and $Q_y^-$ are homomorphisms. By \ref{itm:J4}, the map $x\mapsto \mu_{x,x}$ is quadratic in $x$ and $y\mapsto\mutil_{y,y}$ is quadratic in $y$. By the \autoref{prop:JP1} and \autoref{prop:JP2} (which also hold when interchanging $+$ and $-$), \ref{axiom:JP1} and \ref{axiom:JP2} hold, along with their linearizations, so by \autoref{prop:JPsufficientaxioms}\ref{itm:sufficient_JP}, $(V^+,V^-)$ is a Jordan pair.

	Next, we want to prove the Jordan pair is local. We first claim that if $x\in V^\sigma$ is a unit, then it is invertible in the Jordan pair. For such $x$, we have $Q_x^\sigma = \mu_x$, which is an involution and hence invertible with $x^{-1} = x\mu_x = -x$. Next, we show that $I$ is an ideal. If $x\in I^\sigma$, we have $x\mu_z\in I^{-\sigma}$ for any unit $z$. For any $y\in V^{-\sigma}$ the element $xQ_y$ is a linear combination of such $x\mu_z$ by \autoref{rem:expression_muxz}, so $xQ_y\in I^{-\sigma}$. Next, if $x\in I^\sigma$ and $y\in V^{-\sigma}\setminus I^{-\sigma}$, we have \[yQ_x = \{x y x\} = \{y\,x\,x\mu_y\}\mu_y\in I^{-\sigma}\;,\]
	again because this is a linear combination of $x\mu_z\in I^{-\sigma}$. Here we used the fact that $y$ is a unit (so $\mu_y$ is invertible) together with \ref{axiom:JP1}. Finally, if $x\in I^\sigma$, $y\in V^{-\sigma}\setminus I^{-\sigma}$ and $z\in V^\sigma$, we have $\{xyz\} = \{y\,x\,z\mu_y\}\mu_y\in I^\sigma$. Hence $I$ is a (proper) ideal. In particular, $I$ does not contain any invertible elements. As all elements of $V\setminus I$ are invertible, $I$ is precisely the set of non-invertible elements; we conclude that $V$ is a local Jordan pair and $I = \Rad V$. The fact that $V^+$ is uniquely $2$- and $3$-divisible is a consequence of \ref{itm:J3}, using \hyperref[itm:J1]{(J1-2)}.
\end{proof}


\section{There and back again}\label{sec:connect}

\subsection{The Jordan pair from \texorpdfstring{$\M(V)$}{M(V)}}

In \autoref{sec:JP} we described a way to create a local Moufang set $\M(V)$ from a local Jordan pair $V$, while \autoref{sec:LMStoJP} contains a way to construct local Jordan pairs from certain local Moufang sets. We want to investigate how these two constructions interact.

Suppose first that we start with a local Jordan pair $V$ and apply \autoref{thm:MV} to obtain a local Moufang set $\M(V)$.
It is natural to ask whether we can apply \autoref{thm:localJP} to $\M(V)$ in order to retrieve the local Jordan pair $V$.
We begin by verifying that the conditions required to apply this theorem are indeed satisfied.
\begin{proposition}
	Let $V = (V^+,V^-)$ be a local Jordan pair such that $V^+$ is uniquely $2$- and $3$-divisible. Then $\M(V)$ satisfies the conditions \hyperref[itm:J1]{\normalfont{(J1-4)}}
    from \autoref{constr:Jpair}.
\end{proposition}
\begin{proof}
	By the definition of $\M(V)$ we have $[x,0]\alpha_{[v,0]} = [x+v,0]$, so 
	\[U_\infty = \{\alpha_x\mid x\in \P(V)\setminus\class{\infty}\}\cong V^+\;.\]
	Hence $U_\infty$ is abelian and by \autoref{prop:gammaJP}, $\M(V)$ is special. Next, if $[x,0]$ is a unit, then $x$ is invertible, which means $Q_x$ is invertible. As $Q_{2x} = Q_x\cdot 4$ and $Q_{3x} = Q_x\cdot 9$, we also have $2x$ and $3x$ invertible, so $[2x,0]$ and $[3x,0]$ are also units. This means \hyperref[itm:J1]{(J1-3)} are satisfied.

	To show \ref{itm:J4}, we compute $\mu_{x,x'}$ for units $x=[v,0]$ and $x'=[v',0]$ such that $x+x'$ is also a unit. By \autoref{prop:gammaJP}, we have $\mu_{x,x'} = \mu_{v+v'}-\mu_v-\mu_{v'}$. If we apply this to any $[e,e^{-1}+y]$, we get
	\begin{align*}
		[e,e^{-1}+y](\mu_{v+v'}-\mu_v-\mu_{v'}) &= [yQ_{v+v'},0] - [yQ_v,0]-[yQ_{v'},0] \\
			&= [yQ_{v,v'},0]\;.
	\end{align*}
	Therefore, we can define $\mu_{x,x'}$ for arbitrary $x$ and $x'$ by
	\[ [e,e^{-1}+y]\mu_{x,x'} := [yQ_{v,v'},0]\;. \]
	As $(x,x')\mapsto Q_{v,v'}$ is bilinear, so is the map $(x,x')\mapsto\mu_{x,x'}$. A similar argument shows that we can define
	\[[x,0]\mutil_{[e,e^{-1}+w],[e,e^{-1}+w']} := [e,e^{-1}+xQ_{w,w'}]\]
	for arbitrary $y=[e,e^{-1}+w]$ and $y'=[e,e^{-1}+w']$, and that the map $(y,y')\mapsto\mutil_{y,y'}$ is bilinear. Hence \ref{itm:J4} holds.
\end{proof}
We now know that we can apply \autoref{thm:localJP} on $\M(V)$, so we can compare the resulting local Jordan pair to the original local Jordan pair $V$.

\begin{theorem}
	Let $V = (V^+,V^-)$ be a local Jordan pair with quadratic maps $Q$ such that $V^+$ is uniquely $2$- and $3$-divisible. Denote the local Jordan pair we get from applying \autoref{thm:localJP} to $\M(V)$ by $W = (W^+,W^-)$. Then $V\cong W$.
\end{theorem}
\begin{proof}
    Denote the quadratic maps of the Jordan pair $W$ by $U$.
	By construction, $W^+ = \{[x,0]\mid x\in V^+\}$ and $W^- = \{[e,e^{-1}+y]\mid y\in V^-\}$. We compute the addition on $W$:
	\begin{align*}
		[x,0]+[x',0] &= [0,0]\alpha_{v}\alpha_{v'} = [v+v',0] \\
		[e,e^{-1}+y]\plustil[e,e^{-1}+y'] &= [e,e^{-1}]\gamma_{[e,e^{-1}+y]\tau}\gamma_{[e,e^{-1}+y']\tau} \\
			&= [e,e^{-1}]\gamma_{[yQ_e,0]}\gamma_{[y'Q_e,0]} \\
			&= [e,e^{-1}]\zeta_{y}\zeta_{y'} = [e,e^{-1}+y+y']\;,
	\end{align*}
	where we used \autoref{def:alphazetaJP} and \autoref{prop:gammaJP}. A second ingredient we will need, is the actions of the $\mu$-maps. By \autoref{prop:gammaJP} we have $\mu_{[x,0]} = \mu_{x}$ for all invertible $x\in V^+$. By \autoref{prop:JordanMuaction} this means $[e,e^{-1}+y]\mu_{[x,0]} = [yQ_x,0]$. Similarly, for all invertible $y\in V^-$, $\mu_{[e,e^{-1}+y]} = \mu_{[-y^{-1},0]} = \mu_{-y^{-1}}$, so
	\[[x,0]\mu_{[e,e^{-1}+y]} = [e,e^{-1}+xQ_{-y^{-1}}^{-1}] = [e,e^{-1}+xQ_y]\;.\]
	We are now ready to define an isomorphism between $V$ and $W$:
	\[h_+\colon W^+\to V^+\colon [x,0]\mapsto x\qquad h_-\colon W^-\to V^-\colon [e,e^{-1}+y]\mapsto y\;.\]
	What remains to be proven is the linearity of these maps, and the fact that they preserve the quadratic maps of the Jordan pairs. Linearity is immediate from our computation of the addition on $W$. Next, take any $[x,0]\in W^+$ and $[e,e^{-1}+y]\in W^-$. If $[x,0]$ is a unit, we have
	\begin{align*}
		h_+([e,e^{-1}+y]U^+_{[x,0]}) 	&= h_+([e,e^{-1}+y]\mu_{[x,0]}) = yQ_x \\
						&= h_-([e,e^{-1}+y])Q^+_{h_+([x,0])}\;.
	\end{align*}
	If $[x,0]$ is not a unit, we get
	\begin{align*}
		&h_+([e,e^{-1}+y]U^+_{[x,0]}) \\
			&= h_+([e,e^{-1}+y]\mu_{[x,0],[x,0]}\cdot\tfrac{1}{2}) \\
			&= h_+([e,e^{-1}+y](\mu_{[e+x,0]}\cdot2-\mu_{[2e+x,0]}+\mu_{[e,0]}\cdot2)) \\
			&= h_+([e,e^{-1}+y](\mu_{e+x}\cdot2-\mu_{2e+x}+\mu_{e}\cdot2)) \\
			&= h_+([yQ^+_{e+x}\cdot2-yQ^+_{2e+x}+yQ^+_{e}\cdot2,0]) \\
			&= y(Q^+_{e,x}+Q^+_e+Q^+_x)\cdot2 \\*
			&\qquad-y(Q^+_{e,x}\cdot 2+Q^+_x+Q^+_{e}\cdot 4)+yQ^+_{e}\cdot2 \\
			&= yQ^+_x = h_-([e,e^{-1}+y])Q^+_{h_+([x,0])}\;.
	\end{align*}
	Similarly, we get
	\begin{align*}
		h_-([x,0]U^-_{[e,e^{-1}+y]}) &= h_-([x,0]\mu_{[e,e^{-1}+y]}) = xQ^-_y \\
			&= h_+([x,0])Q^-_{h_-([e,e^{-1}+y])}
	\end{align*}
	for $[e,e^{-1}+y]$ a unit, and otherwise
	\begin{align*}
		h_-([x,0]U^-_{[e,e^{-1}+y]}) &= h_-([x,0]\mutil_{[e,e^{-1}+y],[e,e^{-1}+y]}\cdottil\tfrac{1}{2}) \\
			&= h_-([x,0](\mu_{[e,e^{-1}+e^{-1}+y]}\cdottil2 \\
			&\qquad\mintil\mu_{[e,e^{-1}+2e^{-1}+y]}\plustil\mu_{[e,e^{-1}+e^{-1}]}\cdottil2)) \\
			&= h_-([xQ^-_{e^{-1}+y}\cdot 2 - xQ^-_{2e^{-1}+y} + xQ^-_{e^{-1}}\cdot 2]) \\
			&= xQ^-_{e^{-1}+y}\cdot 2 - xQ^-_{2e^{-1}+y} + xQ^-_{e^{-1}}\cdot 2 \\
			&= xQ^-_y = h_+([x,0])Q^-_{h_-([e,e^{-1}+y])}\;.
	\end{align*}
	Hence $(h_+,h_-)$ is a homomorphism from $W$ to $V$, and since it is a bijection, it is also an isomorphism.
\end{proof}
\begin{corollary}
	If $V$ and $W$ are local Jordan pairs and $\M(V)\cong \M(W)$, then $V\cong W$.
\end{corollary}

\subsection{Characterizing \texorpdfstring{$\M(V)$}{M(V)}}

Conversely, suppose now that we start with a local Moufang set $\M$ to which we apply \autoref{thm:localJP} to get a local Jordan pair $V$, and consider the local Moufang set $\M(V)$ obtained from $V$ by \autoref{thm:MV}; it is now natural to ask whether $\M \cong \M(V)$.
We will be able to give a positive answer to this question provided that we impose an additional assumption determining the action of $U$ on $\class{\infty}$.

\begin{theorem}\label{thm:extra}
	Let $\M$ be a local Moufang set satisfying \hyperref[itm:J1]{\normalfont{(J1-4)}}, and let $V$ be the local Jordan pair obtained from $\M$ by applying \autoref{thm:localJP}. Assume that
	\begin{align}\label{eq:extra}
		\begin{aligned}
			t\alpha_x \mintil x\mutil_{t,t\alpha_x} \plustil t\alpha_x\mu_{x,x}\mutil_{t,t}\cdottil\tfrac{1}{4}
				= t \mintil x\mutil_{t,t}\cdottil\tfrac{1}{2} \\
			\text{for all $t\sim\infty$ and $x\nsim \infty$\;.} 
		\end{aligned}
		\tag{$*$}
	\end{align}
	Then $\M\cong\M(V)$.
\end{theorem}
\begin{proof}
	To avoid confusion, we will denote the set with equivalence of the local Moufang set $\M$ by $(X,{\sim})$, and the corresponding root group $U_\infty$ by $U$. Recall that $\M(V)$ acts on the set $\P(V)$; we will denote the root group $U_{[0,0]}$ by $U'$. To prove that $\M \cong \M(V)$, we need an equivalence-preserving bijection $\phi\colon X\to\P(V)$, an isomorphism $\theta\colon U\to U'$ and a $\mu$-map in each local Moufang set, which we will denote by $\tau$ and $\tau'$, respectively, such that the action of $U$ and $\tau$ on $X$ are permutationally equivalent with the action of $U'$ and $\tau'$ on $\P(V)$.

	Let $e$ be a unit in $X$; then by \eqref{eq:P(V)} and \autoref{thm:localJP} we can describe $\P(V)$ as
    \[ \P(V) = \{[t,0]\mid t\nsim\infty\}\cup\{[e,e^{-1}+t]\mid t\sim\infty\}\;. \]
    We define
	\[ \phi \colon X\to \P(V)  \colon  t\mapsto
        \begin{cases}
            [t,0] &\text{if $t\in X\setminus\class{\infty}$,} \\
            [e,e^{-1}+t] & \text{if $t\in \class{\infty}$.}
        \end{cases} \]
        We check that this bijection preserves the equivalence, using \autoref{def:radequiv}. First, if $x,x'\nsim\infty$, we have
        \begin{align*}
        	x\sim x'&\iff x-x'\sim 0\iff x-x'\in\Rad V^+ \\
			&\iff [x,0]\sim[x',0] \iff x\phi\sim x'\phi\;.
        \end{align*}
        Second, if $y\sim y'\sim\infty$, then $y,y'\in\Rad V^-$, so 
        \[y\phi=[e,e^{-1}+y]\sim[e,e^{-1}+y']=y'\phi\;.\]
        Finally, if $x\nsim\infty$ and $y\sim\infty$ (or vice-versa) then $[x,0]\nsim[e,e^{-1}+y]$.

        Next, we set $\tau = \mu_e$ and $\tau' = \mu_{[e,0]}$. Then
        \begin{align*}
        	t\tau\phi = t\mu_e\phi &= \begin{cases}
			[t\mu_e,0] &\text{if $t\nsim 0$}, \\
			[e,e^{-1}+t\mu_e] & \text{if $t\sim 0$};
			\end{cases}\\
        	t\phi\tau' &= \begin{cases}
            [t,0]\mu_{[e,0]} = [e,e^{-1}+tQ_e^{-1}] &\text{if $t\sim 0$}, \\
            [t,0]\mu_{[e,0]} = [-t^{-1}Q_e,0] & \text{if $t$ is a unit}, \\
            [e,e^{-1}+t]\mu_{[e,0]} = [tQ_e,0] & \text{if $t\sim\infty$}.
        \end{cases}
        \end{align*}
	Observe that $Q_e = \mu_{e,e}\cdot\frac{1}{2} = \mu_e$, so $Q_e^{-1} = \mu_e^{-1} = \mu_e$, and that $t^{-1} = t\mu_t = -t$ if $t$ is a unit. Hence, in all cases, $t\tau\phi = t\phi\tau'$.

	We now define
	\[\theta \colon U\to U' \colon \alpha_x\mapsto\alpha_{[x,0]}\qquad\text{for all $x\nsim\infty$.}\]
	This is clearly a group isomorphism. It only remains to verify that $t\alpha_x\phi = x\phi\theta(\alpha_x)$. If $t\nsim\infty$, we have
	\[t\alpha_{x}\phi = (t+x)\phi = [t+x,0] = [t,0]\alpha_{[x,0]} = t\phi\theta(\alpha_{x})\;,\]
	so the only case left to consider is when $t\sim\infty$. By~\eqref{eq:extra}, we have
	\begin{alignat*}{2}
		&&t\alpha_x \mintil x\mutil_{t,t\alpha_x} \plustil t\alpha_x\mu_{x,x}\mutil_{t,t}\cdottil\tfrac{1}{4} &= t \mintil x\mutil_{t,t}\cdottil\tfrac{1}{2} \\
		&\implies& t\alpha_x \mintil t\alpha_xD_{t,x} \plustil t\alpha_xQ_xQ_t &= t \mintil xQ_t\\
		&\implies& t\alpha_x(\id \mintil D_{t,x} \plustil Q_xQ_t) &= t \mintil xQ_t\;.
	\end{alignat*}
	Now observe that $(t,x)$ is quasi-invertible because $V$ is a local Jordan pair and $t\in \Rad V^-$; hence $\id \mintil D_{t,x} \plustil Q_xQ_t$ is invertible and
	\begin{align*}
		t\alpha_x &= (t \mintil xQ_t)(\id \mintil D_{t,x} \plustil Q_xQ_t)^{-1} = t^x\;.
	\end{align*}
	We conclude that also in this case,
	\[t\alpha_{x}\phi = t^{x}\phi = [t^{x},0] = [t,0]\alpha_{[x,0]} = t\phi\theta(\alpha_x)\;.\]
	Hence we have shown that $\M$ and $\M(V)$ are isomorphic.
\end{proof}

\begin{remark}
	As can be observed in the proof, the extra condition~\eqref{eq:extra} is a translation of the original definition of $\alpha_x$ in $\M(V)$: $[e,e^{-1}+t]\alpha_x = [e,e^{-1}+t^x]$. It is at this point unclear to us whether this assumption is strictly necessary.
	It seems likely that there is a connection with the extra assumption needed in \autoref{thm:PSL2equiv}, but we have not been able to verify this.
\end{remark}
	\chapter{Hermitian local Moufang sets}\label{chap:chap8_hermitian}
	Our final examples are the Hermitian local Moufang sets, which are also our first examples which do not have abelian root groups. We first introduce orthogonal local Moufang sets, as these are a specific class of Hermitian local Moufang sets with a simpler description. Next, we define the underlying set and root groups of Hermitian local Moufang sets, and we prove that we indeed get local Moufang sets.

\section{Orthogonal local Moufang sets}	
\subsection{A set to act on}

In this section, we assume $R$ is a local ring with maximal ideal $\m$.

\begin{definition}
	Let $W$\notatlink{W} be a right $R$-module. A map $q\colon W\to R$\notatlink{q} is a \define{quadratic form} if
	\begin{manualenumerate}[label=\textnormal{(Q\arabic*)},labelwidth=\widthof{(Q1)}]
		\item $q(xr) = q(x)r^2$ for all $x\in W$ and $r\in R$;\label{itm:Q1}
		\item $f(x,y):=q(x+y)-q(x)-q(y)$ is bilinear.\notatlink{f}\label{itm:Q2}
	\end{manualenumerate}
	A quadratic form is \define[quadratic form!anisotropic ---]{anisotropic} if
	\begin{manualenumerate}[label=\textnormal{(Q\arabic*)},labelwidth=\widthof{(Q1)},resume]
		\item $q(x)\in\m\implies x\in W\m$ for all $x\in W$.\label{itm:Q3}
	\end{manualenumerate}
\end{definition}

Assume we have an anisotropic quadratic form on a right $R$-module $W$, then we can define
\[\tilde{W}:=R\times W\times R\qquad \tilde{q}\colon \tilde{W}\to R: (r,x,s)\mapsto q(x)-rs\;.\notatlink{qtil}\]
This is also a quadratic form (but not an anisotropic one).

We can now take the triples $(r,x,s)$ up to invertible scalar multiple
\[[r,x,s]:=\{(rt,xt,st)\in \tilde{W}\mid t \in R\setminus\m\}\;,\notatlink{[rxs]}\]
which, if we restrict to nice triples, gives a projective space:
\[\P(W):=\{[r,x,s]\mid (r,x,s)\in W, rR+sR=R\}\;.\notatlink{P(W)}\]
In this projective space, we can now look at the isotropic points of the form $\tilde{q}$:
\[\mathcal{Q}(W,q):=\{[r,x,s]\in\P(W)\mid \tilde{q}(r,x,s)=0\}\;.\notatlink{QWq}\]
Observe that the assumption $rR+sR=R$ implies that at least one of $r$ and $s$ is invertible, hence we get
\[\mathcal{Q}(W,q)=\{[1,x,q(x)]\mid x\in W\}\cup\{[q(x),x,1]\mid x\in W\m\}\;.\]
Lastly, we can define an equivalence relation on $\mathcal{Q}(W,q)$ by saying two points are equivalent if they reduce to the same point if we quotient out $\m$:
\begin{align*}
	&[1,x,r]\sim[1,y,s]\iff x-y\in W\m \\
	&[r,x,1]\sim[s,y,1]\iff x-y\in W\m \\
	&[r,x,s]\nsim[r',y,s']\qquad\text{ otherwise.}
\end{align*}

\subsection{A local Moufang set}

We can now define equivalence-preserving permutations of $\mathcal{Q}(W,q)$ to construct our local Moufang set. We define $\alpha_{[1,x,r]}$, $\zeta_{[r,x,1]}$ and $\tau$ by
\begin{align*}
	[1,y,s]\alpha_{[1,x,r]} &:= [1,y+x,s+r+f(x,y)]\\
	[s,y,1]\alpha_{[1,x,r]} &:= \bigg[\frac{s}{1+rs+f(x,y)},\frac{y+xs}{1+rs+f(x,y)},1\bigg]&\text{for $y\in W\m$}\\[1ex]
	[1,y,s]\zeta_{[r,x,1]} &:= \bigg[1,\frac{y+xs}{1+rs+f(x,y)},\frac{s}{1+rs+f(x,y)}\bigg]&\text{for $y\in W\m$}\\
	[s,y,1]\zeta_{[r,x,1]} &:= [s+r+f(x,y),y+x,1] \\[1ex]
	[r,x,s]\tau &:= [s,x,r]
\end{align*}
One can verify that all these permutations are in $\Sym\big(\mathcal{Q}(W,q),\sim\big)$.

We now set $U = \{\alpha_{[1,x,r]}\mid x\in W\}$, and we consider the construction $\M(U,\tau)$. It is easy to verify \ref{axiom:C1} and \ref{axiom:C1'}. Furthermore, $U$ fixes $[0,0,1]=:\infty$, $[0,0,1]\tau = [1,0,0]=:0$, and $0\tau = \infty$ so \ref{axiom:C2} also holds.

It is now possible to compute the following:
\begin{align*}
	\gamma_{[1,x,r]} &= \alpha_{[1,x,r]}^\tau = \zeta_{[r,x,1]} \\
	\mu_{[1,x,r]} &= \zeta_{[r^{-1},-xr^{-1},1]}\alpha_{[1,x,r]}\zeta_{[r^{-1},-xr^{-1},1]}&\text{ for $x\in W\setminus W\m$} \\
	[1,y,s]\mu_{[1,x,r]} &= \bigg[\frac{s}{r^2}, \Big(y-x\frac{f(x,y)}{r}\Big)\frac{1}{r},1\bigg]&\text{ for $x\in W\setminus W\m$} \\
	[s,y,1]\mu_{[1,x,r]} &= \bigg[1, \Big(y-x\frac{f(x,y)}{r}\Big)r,r^2s\bigg]&\text{ for $x\in W\setminus W\m$} \\
	\alpha_{[1,y,s]}^{\mu_{[1,x,r]}} &= \zeta_{[1,y,s]\mu_{[1,x,r]}}&\text{ for $x\in W\setminus W\m$}
\end{align*}

Now, by the construction, 
\[U_0 := \{\gamma_{[1,x,r]}\mid x\in W\} = \{\zeta_{[r,x,1]}\mid x\in W\}\;,\]
and hence
\begin{align*}
	U^{\mu_{[1,x,r]}} &= \{\alpha_{[1,y,s]}^{\mu_{[1,x,r]}} \mid y\in W\} = \{\zeta_{[1,y,s]\mu_{[1,x,r]}}\mid y\in W\} \\
			&= \{\zeta_{[1,y,s]}\mid y\in W\} = U_0
\end{align*}
for all $x\in W\setminus W\m$. Hence this construction gives a local Moufang set by \autoref{thm:constrMouf} and \autoref{lem:gleq}.

\begin{definition}
	An \define{orthogonal local Moufang set} is a local Moufang set originating from an anisotropic quadratic form $q$ on a right $R$-module $W$ in the preceding manner. We denote this local Moufang set by $\M(W,q)$\notatlink{M(W,q)}.
\end{definition}

\begin{remark}
	We have skipped computations here, as orthogonal local Moufang sets can be acquired from local Jordan pairs. Indeed, if $W$ is a right $R$-module with quadratic form $q$, then we can define a Jordan pair by
	\[V^+=V^-=W\qquad Q_x\colon V^\sigma\to V^{-\sigma}\colon y\mapsto yq(x) - xf(x,y)\;.\]
	In this case, one can check that $x$ is invertible if and only if $x\in W\setminus W\m$, and that $(W\m,W\m)$ is an ideal in the Jordan pair $(V^+,V^-)$. Hence we can construct a local Moufang set $\M(V^+,V^-)$. The descriptions of the set, maps\ldots , one would get using the Jordan pair construction will be different from the descriptions given here, but the local Moufang sets will be isomorphic.
	
	Alternatively, orthogonal local Moufang sets are also specific instances of Hermitian local Moufang sets, which we will describe in the next section.
\end{remark}


\section{Hermitian local Moufang sets}	
\subsection{A set to act on}

In this section, $R$ is a unital local ring (not necessarily commutative) with maximal ideal $\m$, and an involution $\ast$\notatlink{ast}. Remark that we automatically have $\m^\ast=\m$. Furthermore, $\epsilon$\notatlink{e0} will denote an element of $Z(R)$ such that $\epsilon\epsilon^\ast=1$. The following definitions and properties are from Chapter~5 in \cite{HahnOMearaClassicalGroups}.

\begin{definition}
	A \define{$\ast$-form} on a right $R$-module $W$\notatlink{W} is a biadditive map $f\colon W\times W\to R$ such that
	\[f(xr,ys) = r^\ast f(x,y)s\text{ for all $x,y\in W$ and $r,s\in R$.}\]
	We call such $f$ \define[e-Hermitian@$\epsilon$-Hermitian]{$\epsilon$-Hermitian} if $f(x,y) = f(y,x)^\ast\epsilon$  for all $x,y\in W$.
\end{definition}
Set 
\[\Lambda_{\mathrm{min}} := \{r-r^\ast\epsilon\mid r\in R\}\text{ and } \Lambda_{\mathrm{max}} := \{r\in R\mid r^\ast\epsilon = -r\}\;.\notatlink{Lambdamin}\notatlink{Lambdamax}\]
\begin{definition}
	A \define{form parameter} is an additive subgroup $\Lambda$\notatlink{Lambda} of $R$ such that $\Lambda_{\mathrm{min}}\subset\Lambda\subset\Lambda_{\mathrm{max}}$ and $r^\ast\Lambda r\subset\Lambda$ for all $r\in R$. We call $(R,\Lambda)$\notatlink{RLambda} a \define{form ring}.
\end{definition}

Suppose $s = s' + \lambda$ for some $\lambda\in\Lambda$ as in the previous definition, and take $r\in R$. Then $r^\ast s r = r^\ast s' r + r^\ast\lambda r$. As $r^\ast \lambda r$ is in $\Lambda$, we have $r^\ast sr+\Lambda = r^\ast s' r + \Lambda$. Hence, we can define $r^\ast(s+\Lambda)r := r^\ast s r + \Lambda$. This gives a right action of $R$ on the additive group $R/\Lambda$.

\begin{proposition}
	If $(R,\Lambda)$ is a form ring, and there is an $r\in Z(R)$ such that $r+r^\ast$ is invertible, then $\Lambda = \Lambda_{\mathrm{min}} = \Lambda_{\mathrm{max}}$. This is the case if $2\not\in\m$.
\end{proposition}
\begin{proof}
	\cite[Example 1 of Section 5.1C]{HahnOMearaClassicalGroups}
\end{proof}

\begin{definition}
	Let $(R,\Lambda)$ be a form ring and $W$ a right $R$-module. A pair of maps $(f,q)$\notatlink{fq} with $f\colon W\times W\to R$ and $q\colon W\to R/\Lambda$ is a \define[L-quadratic form@$\Lambda$-quadratic form]{$\Lambda$-quadratic form} on $W$ if there is a $\ast$-form $h$ such that
	\begin{align*}
		f(x,y) &:= h(x,y)+h(y,x)^\ast\epsilon \\
		q(x) &:= h(x,x) + \Lambda
	\end{align*}
	We say $(f,q)$ is \emph{defined by} $h$.
	We call a $\Lambda$-quadratic form \define[L-quadratic form@$\Lambda$-quadratic form!anisotropic ---]{anisotropic} if $q(x)\in\m+\Lambda\implies x\in W\m$ for all $x\in W$.
\end{definition}

\begin{proposition}
	If $(f,q)$ is a $\Lambda$-quadratic form, we have
	\begin{manualenumerate}[label=\textnormal{($\Lambda$Q\arabic*)},labelwidth=\widthof{($\Lambda$Q1)}]
		\item $q(x+y) = q(x)+q(y)+f(x,y)+\Lambda$ for all $x,y\in W$;\label{itm:LambdaQ1}
		\item $q(xr) = r^\ast q(x)r$ for all $x\in W$ and $r\in R$;\label{itm:LambdaQ2}
		\item $f(x,x) = r + r^\ast\epsilon$ for all $x\in W$ and $r\in R$\\
			such that $q(x)=r+\Lambda$.\label{itm:LambdaQ3}
	\end{manualenumerate}
	For any projective module $W$, any $\epsilon$-Hermitian form $f$ and any map $q\colon W\to R/\Lambda$ satisfying these three properties, the pair $(f,q)$ is a $\Lambda$-quadratic form defined by some $\ast$-form $h$.
\end{proposition}
\begin{proof}
	\citenobackrefoptional{5.1.15}{HahnOMearaClassicalGroups}
\end{proof}

Now assume we have an anisotropic $\Lambda$-quadratic form $(f,q)$ on $W$ defined by some $\ast$-form $h$. We set
\begin{align*}
	&\tilde{W} := R\times W\times R\\
	&\tilde{h}\colon\tilde{W}\times \tilde{W}\to R\colon\big((r,x,s),(r',y,s')\big)\mapsto h(x,y)-r^\ast s'\;.
\end{align*}
I.e.\ we have a $\ast$-form $\tilde{h}$ on $\tilde{W}$, and the corresponding $\Lambda$-quadratic form $(\tilde{f},\tilde{q})$\notatlink{fqtil} is
\begin{align*}
	\tilde{f}\big((r,x,s),(r',y,s')\big) &:= f(x,y)-r^\ast s'-s^\ast r'\epsilon \\
	\tilde{q}(r,x,s) &:= q(x)-r^\ast s + \Lambda
\end{align*}
We now take the triples $(r,x,s)$ in $\tilde{W}$ up to invertible scalar multiple
\[[r,x,s]:= \{(rt,xt,st)\in \tilde{W}\mid t\in R\setminus\m\}\;,\notatlink{[rxs]}\]
which, if we restrict to nice triples, gives a projective space:
\[\P(W) := \{[r,x,s]\mid (r,x,s)\in \tilde{W}, rR+sR=R\}\;.\notatlink{P(W)}\]
In this projective space, we now look at the isotropic points of $\tilde{q}$:
\[\mathcal{H}(W,q) := \{[r,x,s]\in\P(W)\mid q(x) = r^\ast s+\Lambda\}\;.\]
In a local ring, $rR+sR=R$ implies that at least one of $r$ and $s$ is invertible, so we get
\begin{align*}
	\mathcal{H}(W,q) &= \{[1,x,r]\in\P(W)\mid q(x) = r+\Lambda\}\notatlink{HWq}\\*
				&\quad\cup\{[r,x,1]\in\P(W)\mid q(x) = r^\ast+\Lambda, r\in\m\} \\
			&= \{[1,x,r]\in\P(W)\mid q(x) = r+\Lambda, r\in\m\}\\*
			&\quad\cup\{[r,x,1]\in\P(W)\mid q(x) = r^\ast+\Lambda\}.
\end{align*}
If the space and form are clear from context, we will simply write $\mathcal{H}$\notatlink{HW}. Finally, we can define an equivalence relation on $\mathcal{H}$ by saying two points are equivalent if they reduce to the same point when we quotient out $\m$. Concretely, this means
\begin{align*}
	&[1,x,r]\sim[1,y,s] \iff x-y\in W\m\text{ and }r-s \in\m \\
	&[r,x,1]\sim[s,y,1] \iff x-y\in W\m\text{ and }r-s \in\m \\
	&[r,x,s]\nsim[r',y,s']\qquad \text{ when $r,s'\in\m$ or $r',s\in\m$.}
\end{align*}
\begin{proposition}
	The relation $\sim$ is well-defined, and is an equivalence relation on $\mathcal{H}$.
\end{proposition}
\begin{proof}
	To check if $\sim$ is well-defined, we need to verify that whenever two points can both be written as $[1,x,r]$ and $[s,y,1]$, the equivalence of these two points is independent on the choice of representation. First assume $[1,x,r]\sim[1,y,s]$, then we want to show that $[r^{-1},xr^{-1},1]\sim[s^{-1},ys^{-1},1]$.
	We have
	\begin{align*}
		xr^{-1}-ys^{-1}\in W\m &\iff x-ys^{-1}r\in W\m \\
		&\iff x-ys^{-1}(r-s+s)\in W\m \\
		&\iff x-y-ys^{-1}(r-s)\in W\m\;,
	\end{align*}
	which is the case. For the second condition, we get
	\begin{align*}
		r^{-1}-s^{-1}\in\m
		&\iff 1-s^{-1}r\in\m \\
		&\iff 1-s^{-1}(r-s+s)\in\m \\
		&\iff s^{-1}(r-s)\in\m\;,
	\end{align*}
	which holds. The proof that $[r,x,1]\sim[s,y,1]$ implies $[1,xr^{-1},r^{-1}]\sim[1,ys^{-1},s^{-1}]$ is identical.
	
	It is clear that this relation is symmetric and reflexive. Proving transitivity requires a short case distinction, but is also easy when we observe that if $[1,x,r]\sim[1,y,s]$ and $r$ is invertible, then $s$ is also invertible (as it differs from $r$ by a non-invertible element).
\end{proof}

\subsection{A local Moufang set}

We can now define the permutations of $(\mathcal{H},\sim)$ that will be part of the local Moufang set.

\begin{definition}\label{def:alphatau_hermLMS}
	For $[1,x,r]\in \mathcal{H}$, we define $\alpha_{[1,x,r]}$ as
	\begin{align*}
	\alpha_{[1,x,r]} \colon \begin{cases*}
					[1,y,s]\mapsto[1,y+x,s+r+f(x,y)] \\
					\begin{aligned}
						[s,y,1]\mapsto&[s(1+rs+f(x,y))^{-1}, \\
								&\quad(y+xs)(1+rs+f(x,y))^{-1},1] 
					\end{aligned} & \text{for $s\in\m$.}
				\end{cases*}
	\end{align*}
	For $[r,x,1]\in \mathcal{H}$, we define $\zeta_{[r,x,1]}$ as
	\begin{align*}
	\zeta_{[r,x,1]} \colon \begin{cases*}
				[s,y,1]\mapsto[s+r+\epsilon^\ast f(x,y),y+x,1] \\
					\begin{aligned}
						[1,y,s]\mapsto& [1,(y+xs)(1+rs+\epsilon^\ast f(x,y))^{-1}, \\
								&\quad s(1+rs+\epsilon^\ast f(x,y))^{-1}] 
					\end{aligned} & \text{for $s\in\m$.}
				\end{cases*}
	\end{align*}
	Finally, we define $\tau$ by
	\begin{align*}
		[r,x,s]\tau := [s,x,r\epsilon] = [s\epsilon^\ast,x\epsilon^\ast,r]\;,
	\intertext{so }
		[r,x,s]\tau^{-1} := [s\epsilon^\ast,x,r] = [s,x\epsilon,r\epsilon]\;.
	\end{align*}
\end{definition}

We now want define a local Moufang set by the construction $\M(U,\tau)$ with $U:=\{\alpha_{[1,x,r]}\mid[1,x,r]\in\mathcal{H}\}$. For this, we first need to check if all $\alpha_{[1,x,r]}$ and $\tau$ preserve $\sim$, and if the conditions \hyperref[axiom:C1]{(C1-2)} are satisfied.

\begin{proposition}
	For all $[1,x,r]\in\mathcal{H}$, $\alpha_{[1,x,r]}^{-1} = \alpha_{[1,-x,f(x,x)-r]}$ and $\alpha_{[1,x,r]}$ preserves $\sim$. Secondly, $\tau$ preserves $\sim$. Finally, \ref{axiom:C1}, \ref{axiom:C1'} and \ref{axiom:C2} hold.
\end{proposition}
\begin{proof}
	The first claim follows by computing $\alpha_{[1,x,r]}\alpha_{[1,-x,f(x,x)-r]}$ and $\alpha_{[1,-x,f(x,x)-r]}\alpha_{[1,x,r]}$.

	Assume $[1,y,s]\sim[1,u,t]$, i.e.\ $y-u\in W\m$ and $s-t\in\m$. Then
	\[[1,y+x,s+r+f(x,y)]\sim [1,u+x,t+r+f(x,u)]\;,\]
	as $(y+x)-(u+x) = y-u\in W\m$ and 
	\[(s+r+f(x,y)) - (t+r+f(x,u)) = s-t+f(x,y-u)\in\m\;.\]
	Next, assume $y\in W\m$ and $s\in\m$, so $[s,y,1]\sim[0,0,1]$. Then
	\[[s(1+rs+f(x,y))^{-1},(y+xs)(1+rs+f(x,y))^{-1},1]\sim [0,0,1]\;,\]
	as $s(1+rs+f(x,y))^{-1}\in \m$ and $(y+xs)(1+rs+f(x,y))^{-1}\in W\m$. By transitivity, this means 
	\[[s,y,1]\sim[t,u,1]\implies [s,y,1]\alpha_{[1,x,r]}\sim[t,u,1]\alpha_{[1,x,r]}\]
	if $x,y\in W\m$ and $s,t\in\m$. This means that in all cases, we know $[s,y,t]\sim[s',y',t']$ implies $[s,y,t]\alpha_{[1,x,r]}\sim[s',y',t']\alpha_{[1,x,r]}$. Since this also holds for $\alpha_{[1,x,r]}^{-1}$, we get the opposite implication as well, so $\alpha_{[1,x,r]}$ preserves $\sim$.
	
	For $\tau$, we see
	\begin{align*}
		[1,x,r]\sim[1,y,s]&\iff x-y\in W\m\wedge r-s\in\m\\
		&\iff (x-y)\epsilon^\ast\in W\m\wedge(r-s)\epsilon^\ast\in\m \\
		&\iff [r\epsilon^\ast,x\epsilon^\ast,1]\sim[s\epsilon^\ast,y\epsilon^\ast,1]\;,
	\end{align*}
	and similarly for the other cases.
	
	Now $U$ clearly fixes $[0,0,1]=:\infty$. Now if we have $[r,x,1]$ with $r\in\m$, then clearly, $q(x) = r^\ast+\Lambda\in\m+\Lambda$, so $x\in W\m$, and hence $[r,x,1]\sim\infty$. This means that $\mathcal{H}\setminus\class{\infty} = \{[1,x,r]\in\mathcal{H}\}$. The group $U$ acts transitively on $\{[1,x,r]\in\mathcal{H}\}$, as for every $[1,x,r]$, there is an element mapping $[1,0,0]$ to $[1,x,r]$. In order to prove regularity, we observe that no element of $U$ except for $\alpha_{[1,0,0]}$ fixes $[1,0,0]$. This proves \ref{axiom:C1}.
	
	For \ref{axiom:C1'}, we need the find the induced action of $U$ on $\class{\mathcal{H}}$. This corresponds to the same construction, but with \[\class{W} := R/\m\times W/W\m \times R/\m\]
	and the induced quadratic form $(\induced{f},\induced{q})$. The argument for \ref{axiom:C1'} is the analogous to the above argument.
	
	Finally, we clearly have $[0,0,1]\tau = [1,0,0]\nsim[0,0,1]=:0$ and $[0,0,1]\tau^2 = [0,0,1]$, so we have \ref{axiom:C2}.
\end{proof}

Next, we need to compute the different maps we need to show $\M(U,\tau)$ is a local Moufang set.

\begin{proposition}
	For all $[1,x,r]\in\mathcal{H}$, we have 
	\[\gamma_{[1,x,r]} := \alpha_{[1,x,r]}^\tau = \zeta_{[1,x,r]\tau}\]
	and 
	\[-[1,x,r] = [1,-x,f(x,x)-r] = [1,-x,r^\ast\epsilon]\;.\]
	If $r\in R\setminus\m$, we have 
	\[\mu_{[1,x,r]} = \zeta_{[\epsilon^\ast r^{-\ast},-x\epsilon^\ast r^{-\ast},1]}\alpha_{[1,x,r]}\zeta_{[\epsilon^\ast r^{-\ast},-xr^{-1},1]}\;.\]
\end{proposition}
\begin{proof}
	First, take any $[s,y,1]\in\mathcal{H}$. Then
	\begin{align*}
		[s,y,1]\alpha_{[1,x,r]}^\tau &= [s,y,1]\tau^{-1}\alpha_{[1,x,r]}\tau = [1,y\epsilon,s\epsilon]\alpha_{[1,x,r]}\tau \\
		 &= [1,y\epsilon+x,s\epsilon+r+f(x,y)\epsilon]\tau \\
		 &= [s\epsilon\epsilon^\ast+r\epsilon^\ast+f(x,y)\epsilon\epsilon^\ast,y\epsilon\epsilon^\ast+x\epsilon^\ast,1] \\
		 &= [s+r\epsilon^\ast+\epsilon^\ast f(x\epsilon^\ast,y),y+x\epsilon^\ast,1] \\
		 &= [s,y,1]\zeta_{[r\epsilon^\ast,x\epsilon^\ast,1]} = [s,y,1]\zeta_{[1,x,r]\tau}\;.
	\end{align*}
	We still need to check that both actions are the same for $[1,y,s]\in\mathcal{H}$ for $s\in\m$.
	\begin{align*}
		&[1,y,s]\alpha_{[1,x,r]}^\tau \\
		&= [s,y,1]\tau^{-1}\alpha_{[1,x,r]}\tau = [s\epsilon^\ast,y,1]\alpha_{[1,x,r]}\tau \\
		 &= [s\epsilon^\ast(1+rs\epsilon^\ast+f(x,y))^{-1},(y+xs\epsilon^\ast)(1+rs\epsilon^\ast+f(x,y))^{-1},1]\tau \\
		 &= [1,(y+xs\epsilon^\ast)(1+rs\epsilon^\ast+f(x,y))^{-1},s\epsilon^\ast(1+rs\epsilon^\ast+f(x,y))^{-1}\epsilon] \\
		 &= [1,(y+x\epsilon^\ast s)(1+r\epsilon^\ast s+\epsilon^\ast f(x\epsilon^\ast,y))^{-1},s(1+r\epsilon^\ast s+f(x,y))^{-1}] \\
		 &= [1,y,s]\zeta_{[r\epsilon^\ast,x\epsilon^\ast,1]} = [1,y,s]\zeta_{[1,x,r]\tau}\;.
	\end{align*}
	For the second claim, we use $\alpha_{[1,x,r]}^{-1} = \alpha_{[1,-x,f(x,x)-r]}$, and we know that if $q(x) = r+\Lambda$, then $f(x,x) = r+r^\ast\epsilon$, so $f(x,x)-r = r^\ast\epsilon$, from which we get 
	\[[1,-x,f(x,x)-r] = [1,-x,r^\ast\epsilon]\;.\]
	
	The identity for $\mu_{[1,x,r]}$ is now immediate from the definition 
	\[\mu_{[1,x,r]} = \gamma_{(-[1,x,r])\tau^{-1}}\alpha_{[1,x,r]}\gamma_{-([1,x,r]\tau^{-1})}\;,\]
	and previous statements.
\end{proof}

The last big computation we need, is the precise action of the $\mu$-maps.

\begin{proposition}
	Let $[1,x,r]\in\mathcal{H}$ with $r\in R\setminus\m$. Then
	\begin{align*}
		[s,y,1]\mu_{[1,x,r]} &= [1 ,\bigl(y-xr^{-1}f(x,y)\bigr)r^{\ast}\epsilon,rsr^{\ast}\epsilon] \\
		&\hspace{9em}\text{ for all $[s,y,1]\in\mathcal{H}$;} \\
		[1,y,s]\mu_{[1,x,r]} &= [\epsilon^\ast r^{-\ast}sr^{-1},\bigl(y-xr^{-1}f(x,y)\bigr)r^{-1},1] \\
		&\hspace{9em}\text{ for all $[1,y,s]\in\mathcal{H}$.}
	\end{align*}
	From this one can get $\mu_{[1,x,r]}^{-1} = \mu_{[1,-x,r^\ast\epsilon]}$.
\end{proposition}
\begin{proof}
	We first compute the image of $[s,y,1]\in\mathcal{H}$ under $\mu_{[1,x,r]}$.
	\begin{align*}
		&[s,y,1]\mu_{[1,x,r]} \\
		&= [s,y,1]\zeta_{[\epsilon^\ast r^{-\ast},-x\epsilon^\ast r^{-\ast},1]}\alpha_{[1,x,r]}\zeta_{[\epsilon^\ast r^{-\ast},-xr^{-1},1]} \\
		&= [s+\epsilon^\ast r^{-\ast}-\epsilon^\ast f(x\epsilon^\ast r^{-\ast},y),y-x\epsilon^\ast r^{-\ast},1]\alpha_{[1,x,r]}\zeta_{[\epsilon^\ast r^{-\ast},-xr^{-1},1]} \\
		&= [s+\epsilon^\ast r^{-\ast}-r^{-1}f(x,y),y-x\epsilon^\ast r^{-\ast},1]\alpha_{[1,x,r]}\zeta_{[\epsilon^\ast r^{-\ast},-xr^{-1},1]}
	\end{align*}
	We now have two cases depending on whether or not 
	\[s+\epsilon^\ast r^{-\ast}-r^{-1}f(x,y)\in\m\;.\]
	We first assume this is the case, and then we also know $y-x\epsilon^\ast r^{-\ast}\in W\m$. In the computation, we then get the inverse of the following expression:
	\begin{align*}
		& 1+r(s+\epsilon^\ast r^{-\ast}-r^{-1}f(x,y))+f(x,y-x\epsilon^\ast r^{-\ast}) \\
		&= 1+rs+r\epsilon^\ast r^{-\ast}-f(x,y)+f(x,y)-f(x,x)\epsilon^\ast r^{-\ast} \\
		&= 1+rs+r\epsilon^\ast r^{-\ast}-(r+r^\ast\epsilon)\epsilon^\ast r^{-\ast} = rs
	\end{align*}
	this immediately means $s$ is invertible, and we get
	\begin{align*}
		& [s+\epsilon^\ast r^{-\ast}-r^{-1}f(x,y),y-x\epsilon^\ast r^{-\ast},1]\alpha_{[1,x,r]} \\
		&= [(s+\epsilon^\ast r^{-\ast}-r^{-1}f(x,y))(rs)^{-1}, \\*
		&\qquad\quad (y-x\epsilon^\ast r^{-\ast}+x(s+\epsilon^\ast r^{-\ast}-r^{-1}f(x,y)))(rs)^{-1}	,1] \\
		&= [(s+\epsilon^\ast r^{-\ast}-r^{-1}f(x,y))(rs)^{-1}
		,(y+xs-xr^{-1}f(x,y))(rs)^{-1}	,1]
	\end{align*}
	Hence we have
	\begin{align*}
		&[s,y,1]\mu_{[1,x,r]} \\
		&= [(s+\epsilon^\ast r^{-\ast}-r^{-1}f(x,y))(rs)^{-1} \\*
		&\qquad,(y+xs-xr^{-1}f(x,y))(rs)^{-1}	,1]\zeta_{[\epsilon^\ast r^{-\ast},-xr^{-1},1]} \\
		&= [\tilde{s},\tilde{y},1]\;,
	\end{align*}
	where we need to compute $\tilde{s}$ and $\tilde{y}$:
	\begin{align*}
	\tilde{y}&= (y+xs-xr^{-1}f(x,y))(rs)^{-1}-xr^{-1} \\
		&= \bigl(y+xs-xr^{-1}f(x,y)-xr^{-1}rs\bigr)(rs)^{-1} \\
		&= \bigl(y-xr^{-1}f(x,y)\bigr)(rs)^{-1}	\\[1ex]
	\tilde{s}&= (s+\epsilon^\ast r^{-\ast}-r^{-1}f(x,y))(rs)^{-1} + \epsilon^\ast r^{-\ast} \\*
		&\qquad+ \epsilon^\ast f(-xr^{-1},y+xs-xr^{-1}f(x,y))(rs)^{-1} \\
		&= \bigl(s+\epsilon^\ast r^{-\ast}-r^{-1}f(x,y) + \epsilon^\ast r^{-\ast}rs\\*
		&\qquad+ \epsilon^\ast f(-xr^{-1},y+xs-xr^{-1}f(x,y))\bigr)(rs)^{-1} \\
		&= \bigl(s+\epsilon^\ast r^{-\ast}-r^{-1}f(x,y) + \epsilon^\ast r^{-\ast}rs - \epsilon^\ast r^{-\ast}f(x,y) \\*
		&\qquad- \epsilon^\ast r^{-\ast}f(x,x)s + \epsilon^\ast r^{-\ast}f(x,x)r^{-1}f(x,y)\bigr)(rs)^{-1} \\
		&= \bigl(s+\epsilon^\ast r^{-\ast}-r^{-1}f(x,y) + \epsilon^\ast r^{-\ast}rs - \epsilon^\ast r^{-\ast}f(x,y) \\*
		&\qquad- \epsilon^\ast r^{-\ast}(r+r^\ast\epsilon)s + \epsilon^\ast r^{-\ast}(r+r^\ast\epsilon)r^{-1}f(x,y)\bigr)(rs)^{-1} \\
		&= \epsilon^\ast r^{-\ast}s^{-1}r^{-1} 
	\end{align*}
	This means, if $s+\epsilon^\ast r^{-\ast}-r^{-1}f(x,y)\in\m$, we have
	\begin{align*}
		[s,y,1]\mu_{[1,x,r]} &= [\epsilon^\ast r^{-\ast}(rs)^{-1},\bigl(y-xr^{-1}f(x,y)\bigr)(rs)^{-1},1] \\
				&= [1 ,\bigl(y-xr^{-1}f(x,y)\bigr)r^{\ast}\epsilon,rsr^{\ast}\epsilon]
	\end{align*}
	Next, we assume $s+\epsilon^\ast r^{-\ast}-r^{-1}f(x,y)\not\in\m$, so we need to compute 
	\begin{align*}
		&[1 , (y-x\epsilon^\ast r^{-\ast})(s+\epsilon^\ast r^{-\ast}-r^{-1}f(x,y))^{-1} , \\*
		&\qquad(s+\epsilon^\ast r^{-\ast}-r^{-1}f(x,y))^{-1}]\alpha_{[1,x,r]} = [1,\tilde{y},\tilde{s}]\;,
	\end{align*}
	with
	\begin{align*}
		\tilde{y} &= (y-x\epsilon^\ast r^{-\ast})(s+\epsilon^\ast r^{-\ast}-r^{-1}f(x,y))^{-1} + x \\
			&= \bigl(y-x\epsilon^\ast r^{-\ast} + x(s+\epsilon^\ast r^{-\ast}-r^{-1}f(x,y))\bigr) \\*
			&\qquad\cdot(s+\epsilon^\ast r^{-\ast}-r^{-1}f(x,y))^{-1} \\
			&= \bigl(y + xs-xr^{-1}f(x,y))\bigr)(s+\epsilon^\ast r^{-\ast}-r^{-1}f(x,y))^{-1} \\[0.5ex]
		\tilde{s} &= \bigl(1 + r(s+\epsilon^\ast r^{-\ast}-r^{-1}f(x,y))+f(x,y-x\epsilon^\ast r^{-\ast})\bigr) \\*
			&\qquad\cdot(s+\epsilon^\ast r^{-\ast}-r^{-1}f(x,y))^{-1} \\
			&= \bigl(1 + rs+r\epsilon^\ast r^{-\ast}-f(x,y)+f(x,y)-f(x,x)\epsilon^\ast r^{-\ast}\bigr) \\* &\qquad\cdot(s+\epsilon^\ast r^{-\ast}-r^{-1}f(x,y))^{-1} \\
			&= \bigl(1 + rs+r\epsilon^\ast r^{-\ast}-(r+r^\ast\epsilon)\epsilon^\ast r^{-\ast}\bigr)(s+\epsilon^\ast r^{-\ast}-r^{-1}f(x,y))^{-1} \\
			&= (rs)(s+\epsilon^\ast r^{-\ast}-r^{-1}f(x,y))^{-1}
	\end{align*}
	Now, if $s\in R\setminus\m$, we get
	\begin{align*}
		[s,y,1]\mu_{[1,x,r]} &= [(s+\epsilon^\ast r^{-\ast}-r^{-1}f(x,y))(rs)^{-1},\\*
					&\qquad(y + xs-xr^{-1}f(x,y)))(rs)^{-1},1]\zeta_{[\epsilon^\ast r^{-\ast},-xr^{-1},1]}
	\end{align*}
	and this is the same as in the previous case, so again we get 
	\begin{align*}
		[s,y,1]\mu_{[1,x,r]} &= [\epsilon^\ast r^{-\ast}(rs)^{-1},\bigl(y-xr^{-1}f(x,y)\bigr)(rs)^{-1},1] \\*
				&= [1 ,\bigl(y-xr^{-1}f(x,y)\bigr)r^{\ast}\epsilon,rsr^{\ast}\epsilon]\;.
	\end{align*}
	If $s\in\m$, we temporarily write $t = s+\epsilon^\ast r^{-\ast}-r^{-1}f(x,y)$, so we have
	\begin{align*}
		[s,y,1]\mu_{[1,x,r]} &= [1,(y + xs-xr^{-1}f(x,y))t^{-1},rst^{-1}]\zeta_{[\epsilon^\ast r^{-\ast},-xr^{-1},1]} \\*
		&= [1, \tilde{y}, \tilde{s}]
	\end{align*}
	where
	\begin{align*}
		\tilde{y} &= \bigl((y + xs-xr^{-1}f(x,y))t^{-1} - xr^{-1}rst^{-1}\bigr) \\*
		&\qquad\cdot\bigl(1+\epsilon^\ast r^{-\ast}rst^{-1}+\epsilon^\ast f(-xr^{-1},y + xs-xr^{-1}f(x,y))t^{-1}\bigr)^{-1} \\
		&= \bigl(y + xs-xr^{-1}f(x,y) - xs\bigr) \\*
		&\qquad\cdot\bigl(t+\epsilon^\ast r^{-\ast}rs+\epsilon^\ast f(-xr^{-1},y + xs-xr^{-1}f(x,y))\bigr)^{-1}
	\end{align*}
	We now look at the second factor separately:
	\begin{align*}
		&t+\epsilon^\ast r^{-\ast}rs+\epsilon^\ast f(-xr^{-1},y + xs-xr^{-1}f(x,y)) \\
		&= s+\epsilon^\ast r^{-\ast}-r^{-1}f(x,y)+\epsilon^\ast r^{-\ast}rs-\epsilon^\ast r^{-\ast}f(x,y) \\*
		&\quad- \epsilon^\ast r^{-\ast}f(x,x)s +\epsilon^\ast r^{-\ast}f(x,x)r^{-1}f(x,y) \\
		&= s+\epsilon^\ast r^{-\ast}-r^{-1}f(x,y)+\epsilon^\ast r^{-\ast}rs-\epsilon^\ast r^{-\ast}f(x,y) \\*
		&\quad- \epsilon^\ast r^{-\ast}(r+r^\ast\epsilon)s +\epsilon^\ast r^{-\ast}(r+r^\ast\epsilon)r^{-1}f(x,y) \\
		&= \epsilon^\ast r^{-\ast}\;.
	\end{align*}
	Hence, we get
	\begin{align*}
		\tilde{y} &= \bigl(y -xr^{-1}f(x,y) \bigr)r^{\ast}\epsilon \\
	\intertext{and secondly}
		\tilde{s} &= rst^{-1} \\*
		&\quad\cdot\bigl(1+\epsilon^\ast r^{-\ast}rst^{-1}+\epsilon^\ast f(-xr^{-1},y + xs-xr^{-1}f(x,y))t^{-1}\bigr)^{-1} \\
		&= rs\cdot\bigl(t+\epsilon^\ast r^{-\ast}rs+\epsilon^\ast f(-xr^{-1},y + xs-xr^{-1}f(x,y))\bigr)^{-1} \\
		&= rsr^{\ast}\epsilon
	\end{align*}
	All in all, we get
	\begin{align}
		[s,y,1]\mu_{[1,x,r]} &= [1 ,\bigl(y-xr^{-1}f(x,y)\bigr)r^{\ast}\epsilon,rsr^{\ast}\epsilon]\;. \label{action:muinft}
	\end{align}
	
	Now, we still need the image of $[1,y,s]$ under $\mu_{[1,x,r]}$, where $s\in\m$ and hence $y\in W\m$. We first compute
	\begin{align*}
		&[1,y,s]\zeta_{[\epsilon^\ast r^{-\ast},-x\epsilon^\ast r^{-\ast},1]}\alpha_{[1,x,r]} \\*
		&= [1,(y-x\epsilon^\ast r^{-\ast}s)(1+\epsilon^\ast r^{-\ast}s-f(xr^{-\ast},y))^{-1},\\*
		&\qquad\quad s(1+\epsilon^\ast r^{-\ast}s-f(xr^{-\ast},y))^{-1}]\alpha_{[1,x,r]} \\*
		&= [1,\tilde{y},\tilde{s}]\;,
	\end{align*}
	with
	\begin{align*}
		\tilde{y} &= (y-x\epsilon^\ast r^{-\ast}s)(1+\epsilon^\ast r^{-\ast}s-f(xr^{-\ast},y))^{-1} + x \\
			&= (y-x\epsilon^\ast r^{-\ast}s+x+x\epsilon^\ast r^{-\ast}s-xf(xr^{-\ast},y)) \\*
			&\qquad\cdot(1+\epsilon^\ast r^{-\ast}s-f(xr^{-\ast},y))^{-1}\\
			&= (y+x-xr^{-1}f(x,y))(1+\epsilon^\ast r^{-\ast}s-r^{-1}f(x,y))^{-1} \\[1ex]
		\tilde{s} &= \bigl(s + r(1+\epsilon^\ast r^{-\ast}s-r^{-1}f(x,y)) + f(x,y-x\epsilon^\ast r^{-\ast}s)\bigr) \\*
		&\qquad\cdot(1+\epsilon^\ast r^{-\ast}s-r^{-1}f(x,y))^{-1} \\
		&= \bigl(s + r+r\epsilon^\ast r^{-\ast}s-f(x,y) + f(x,y)-f(x,x)\epsilon^\ast r^{-\ast}s\bigr) \\*
		&\qquad\cdot(1+\epsilon^\ast r^{-\ast}s-r^{-1}f(x,y))^{-1} \\
		&= \bigl(s + r+r\epsilon^\ast r^{-\ast}s-(r+r^\ast\epsilon)\epsilon^\ast r^{-\ast}s\bigr)(1+\epsilon^\ast r^{-\ast}s-r^{-1}f(x,y))^{-1} \\
		&= r(1+\epsilon^\ast r^{-\ast}s-r^{-1}f(x,y))^{-1}
	\end{align*}
	Hence we get
	\begin{align*}
		&[1,y,s]\mu_{[1,x,r]} \\
		&= [(1+\epsilon^\ast r^{-\ast}s-r^{-1}f(x,y))r^{-1},  \\*
		&\qquad(y+x-xr^{-1}f(x,y))r^{-1},1]\zeta_{[\epsilon^\ast r^{-\ast},-xr^{-1},1]} \\
		&= [\tilde{s},(y-xr^{-1}f(x,y))r^{-1},1]\;,
	\end{align*}
	where we still need to compute $\tilde{s}$:
	\begin{align*}
		\tilde{s} &= (1+\epsilon^\ast r^{-\ast}s-r^{-1}f(x,y))r^{-1} + \epsilon^\ast r^{-\ast} \\*
		&\qquad+ \epsilon^\ast f(-xr^{-1},y+x-xr^{-1}f(x,y))r^{-1} \\
		&= r^{-1} + \epsilon^\ast r^{-\ast}sr^{-1} - r^{-1}f(x,y)r^{-1}+\epsilon^\ast r^{-\ast}-\epsilon^\ast r^{-\ast}f(x,y)r^{-1} \\
		&\quad	- \epsilon^\ast r^{-\ast} f(x,x)r^{-1} + \epsilon^\ast r^{-\ast} f(x,x)r^{-1}f(x,y)r^{-1} \\
		&= r^{-1} + \epsilon^\ast r^{-\ast}sr^{-1} - r^{-1}f(x,y)r^{-1}+\epsilon^\ast r^{-\ast}-\epsilon^\ast r^{-\ast}f(x,y)r^{-1} \\
		&\quad	- \epsilon^\ast r^{-\ast} (r+r^\ast\epsilon)r^{-1} + \epsilon^\ast r^{-\ast} (r+r^\ast\epsilon)r^{-1}f(x,y)r^{-1} \\
		&= \epsilon^\ast r^{-\ast}sr^{-1}
	\end{align*}
	Hence, for $[1,y,s]\in\mathcal{H}$ with $s\in\m$, we have
	\begin{align*}
		[1,y,s]\mu_{[1,x,r]} &= [\epsilon^\ast r^{-\ast}sr^{-1},(y-xr^{-1}f(x,y))r^{-1},1]\;.
	\end{align*}
	If $s$ is invertible, we have $[1,y,s] = [s^{-1},ys^{-1},1]$, and hence we can use identity \ref{action:muinft} to compute the action:
	\begin{align*}
		[1,y,s]\mu_{[1,x,r]} &= [s^{-1},ys^{-1},1]\mu_{[1,x,r]} \\
			&= [1,(ys^{-1}-xr^{-1}f(x,y)s^{-1})r^\ast\epsilon,rs^{-1}r^\ast\epsilon] \\
			&= [\epsilon^\ast r^{-\ast}sr^{-1},(y-xr^{-1}f(x,y))r^{-1},1]\;,
	\end{align*}
	so we get the same formula for all $[1,y,s]\in\mathcal{H}$.
\end{proof}

Using the previous computation, we can now verify that the conjugate of an $\alpha_{[1,y,s]}$ by a $\mu$-map corresponds to some $\zeta$:

\begin{proposition}
	For $[1,x,r],[1,y,s]\in\mathcal{H}$ with $r\in R\setminus\m$, we have $\alpha_{[1,y,s]}^{\mu_{[1,x,r]}} = \zeta_{[1,y,s]\mu_{[1,x,r]}}$.
\end{proposition}
\begin{proof}
	We will verify that the action of the two permutations is the same. Note that the right hand side is 
	\begin{align*}
		\zeta_{[1,y,s]\mu_{[1,x,r]}} = \zeta_{[\epsilon^\ast r^{-\ast}sr^{-1},(y-xr^{-1}f(x,y))r^{-1},1]}\;.
	\end{align*}

	First, take $[t,u,1]\in\mathcal{H}$, so
	\begin{align*}
		&[t,u,1]\zeta_{[1,y,s]\mu_{[1,x,r]}} \\
		&= [t+\epsilon^\ast r^{-\ast}sr^{-1} +\epsilon^\ast f\bigl((y-xr^{-1}f(x,y))r^{-1},u\bigr), \\*
		&\qquad u+\bigl(y-xr^{-1}f(x,y)\bigr)r^{-1},1]\;.
	\end{align*}
	Now we compute the image of the left hand side:
	\begin{align*}
		&[t,u,1]\mu_{[1,-x,r^\ast\epsilon]}\alpha_{[1,y,s]}\mu_{[1,x,r]} \\
		&= [1, (u-x\epsilon^\ast r^{-\ast}f(x,u))r,r^\ast t r\epsilon]\alpha_{[1,y,s]} \mu_{[1,x,r]} \\
		&= [1, (u-x\epsilon^\ast r^{-\ast}f(x,u))r + y , \\*
		&\qquad r^\ast t r\epsilon + s + f(y,u)r-f(y,x)\epsilon^\ast r^{-\ast}f(x,u)r] \mu_{[1,x,r]} \\
		&= [\tilde{t}, \tilde{u}, 1]\;,
	\end{align*}
	where
	\begin{align*}
		\tilde{t} &= \epsilon^\ast r^{-\ast}(r^\ast t r\epsilon + s + f(y,u)r-f(y,x)\epsilon^\ast r^{-\ast}f(x,u)r)r^{-1} \\
			&= t + \epsilon^\ast r^{-\ast}sr^{-\ast}+\epsilon^\ast\bigl(r^{-\ast}f(y,u)-r^{-\ast}f(y,x)\epsilon^\ast r^{-\ast}f(x,u)\bigr)\\
			&= t + \epsilon^\ast r^{-\ast}sr^{-\ast}+\epsilon^\ast\bigl(f(yr^{-1},u)-f(x(r^{-\ast}f(y,x)\epsilon^\ast r^{-\ast})^\ast,u)\bigr)\\
			&= t + \epsilon^\ast r^{-\ast}sr^{-\ast}+\epsilon^\ast\bigl(f(yr^{-1},u)-f(xr^{-1}f(x,y) r^{-1},u)\bigr)\\
			&= t + \epsilon^\ast r^{-\ast}sr^{-\ast}+\epsilon^\ast f\bigl(y(-xr^{-1}f(x,y))r^{-1},u\bigr)\\[1ex]
		\tilde{u} &= u-x\epsilon^\ast r^{-\ast}f(x,u) + yr^{-1} \\*
			&\qquad - xr^{-1}f(x,ur-x\epsilon^\ast r^{-\ast}f(x,u)r + y)r^{-1} \\
			&= u-x\epsilon^\ast r^{-\ast}f(x,u) + yr^{-1} - xr^{-1}f(x,u) \\*
			&\qquad+ xr^{-1}f(x,x)\epsilon^\ast r^{-\ast}f(x,u)- xr^{-1}f(x,y)r^{-1} \\
			&= u + \bigl(y- xr^{-1}f(x,y)\bigr)r^{-1}-x\epsilon^\ast r^{-\ast}f(x,u) - xr^{-1}f(x,u) \\*
			&\qquad+ xr^{-1}(r+r^\ast\epsilon)\epsilon^\ast r^{-\ast}f(x,u) \\
			&= u + \bigl(y- xr^{-1}f(x,y)\bigr)r^{-1}
	\end{align*}
	hence we have
	\begin{align*}
		[t,u,1]\mu_{[1,-x,r^\ast\epsilon]}\alpha_{[1,y,s]}\mu_{[1,x,r]} &= [t,u,1]\zeta_{[1,y,s]\mu_{[1,x,r]}}\;.
	\end{align*}
	Next, we look at the image of $[1,u,t]\in\mathcal{H}$ with $t\in\m$. The image of the right hand side is
	\begin{align*}
		&[1,u,t]\zeta_{[1,y,s]\mu_{[1,x,r]}} = \zeta_{[\epsilon^\ast r^{-\ast}sr^{-1},(y-xr^{-1}f(x,y))r^{-1},1]} \\
		&= [1, (u+(y-xr^{-1}f(x,y))r^{-1}t)d^{-1}, td^{-1}]\;,
	\end{align*}
	with $d = 1+\epsilon^\ast r^{-\ast}sr^{-1}t+\epsilon^\ast f((y-xr^{-1}f(x,y))r^{-1},u)$. We now compute the image of the left hand side:
	\begin{align*}
		&[1,u,t]\mu_{[1,-x,r^\ast\epsilon]}\alpha_{[1,y,s]}\mu_{[1,x,r]} \\
		&= [\epsilon^\ast r^{-1}t r^{-\ast}, (u-x\epsilon^\ast r^{-\ast}f(x,u))\epsilon^\ast r^{-\ast},1]\alpha_{[1,y,s]}\mu_{[1,x,r]} \\
		&= [\epsilon^\ast r^{-1}t r^{-\ast}\tilde{d}^{-1} , \\*
		&\qquad	\bigl((u-x\epsilon^\ast r^{-\ast}f(x,u))\epsilon^\ast r^{-\ast}+y\epsilon^\ast r^{-1}t r^{-\ast}\bigr)\tilde{d}^{-1}
			,1]\mu_{[1,x,r]}
	\end{align*}
	with
	\begin{align*}
		\tilde{d} &= 1+s\epsilon^\ast r^{-1}t r^{-\ast} + f(y,(u-x\epsilon^\ast r^{-\ast}f(x,u))\epsilon^\ast r^{-\ast}) \\
		&= 1+s\epsilon^\ast r^{-1}t r^{-\ast} + f(y,u)\epsilon^\ast r^{-\ast}-f(y,x)\epsilon^\ast r^{-\ast}f(x,u)\epsilon^\ast r^{-\ast} \\
		&= r^\ast\bigl(1+r^{-\ast}s\epsilon^\ast r^{-1}t + r^{-\ast}f(y,u)\epsilon^\ast \\*
		&\qquad - r^{-\ast}f(y,x)\epsilon^\ast r^{-\ast}f(x,u)\epsilon^\ast \bigr)r^{-\ast} \\
		&= r^\ast\bigl(1+\epsilon^\ast r^{-\ast}s r^{-1}t + \epsilon^\ast f(yr^{-1},u)  \\*
		&\qquad - \epsilon^\ast f(x(r^{-\ast}f(y,x)\epsilon^\ast r^{-\ast})^\ast,u) \bigr)r^{-\ast} \\
		&= r^\ast\bigl(1+\epsilon^\ast r^{-\ast}s r^{-1}t + \epsilon^\ast f((y - xr^{-1} f(x,y))r^{-1},u) \bigr)r^{-\ast} \\
		&= r^\ast d r^{-\ast}\;.
	\end{align*}
	Hence we get
	\begin{align*}
		&[1,u,t]\mu_{[1,-x,r^\ast\epsilon]}\alpha_{[1,y,s]}\mu_{[1,x,r]} \\
		&= [\epsilon^\ast r^{-1}t d^{-1}r^{-\ast} , \\*
		&\qquad	\bigl(u\epsilon^\ast-x\epsilon^\ast r^{-\ast}f(x,u)\epsilon^\ast +y\epsilon^\ast r^{-1}t \bigr)d^{-1} r^{-\ast}
			,1]\mu_{[1,x,r]} \\
		&= [1, \tilde{u}, \tilde{t}]
	\end{align*}
	with
	\begin{align*}
		\tilde{u} &= \bigl(u\epsilon^\ast-x\epsilon^\ast r^{-\ast}f(x,u)\epsilon^\ast +y\epsilon^\ast r^{-1}t \\*
		&\qquad- xr^{-1}f(x,u\epsilon^\ast-x\epsilon^\ast r^{-\ast}f(x,u)\epsilon^\ast +y\epsilon^\ast r^{-1}t)\bigr)d^{-1} \epsilon \\
		&= \bigl(u+y r^{-1}t -x\epsilon^\ast r^{-\ast}f(x,u)  \\*
		&\qquad- xr^{-1}f(x,u - x\epsilon^\ast r^{-\ast}f(x,u) +y r^{-1}t)\bigr)d^{-1} \\
		&= \bigl(u+y r^{-1}t- xr^{-1}f(x,y )r^{-1}t -x\epsilon^\ast r^{-\ast}f(x,u) \\*
		&\qquad - xr^{-1}f(x,u) + xr^{-1}f(x,x)\epsilon^\ast r^{-\ast}f(x,u)\bigr)d^{-1} \\
		&= \bigl(u+(y - xr^{-1}f(x,y))r^{-1}t -x\epsilon^\ast r^{-\ast}f(x,u) \\*
		&\qquad- xr^{-1}f(x,u) + xr^{-1}(r+r^\ast\epsilon)\epsilon^\ast r^{-\ast}f(x,u)\bigr)d^{-1} \\
		&= \bigl(u+(y - xr^{-1}f(x,y))r^{-1}t\bigr)d^{-1} \\[1ex]
		\tilde{t} &= r\epsilon^\ast r^{-1}t d^{-1}r^{-\ast}r^\ast\epsilon \\
			&=t d^{-1}\;,
	\end{align*}
	so indeed, we have, for $[1,u,t]\in\mathcal{H}$ with $t\in\m$
	\begin{align*}
		[1,u,t]\mu_{[1,-x,r^\ast\epsilon]}\alpha_{[1,y,s]}\mu_{[1,x,r]} &= [1,u,t]\zeta_{[1,y,s]\mu_{[1,x,r]}}\;.\qedhere
	\end{align*}
\end{proof}

After this extensive amount of computation, we can prove that the object we constructed is a local Moufang set.

\begin{theorem}
	Let $(f,q)$ be an isotropic $\Lambda$-quadratic form on $W$ and take $\tau$ and $U=\{\alpha_{[1,x,r]}\mid[1,x,r]\in\mathcal{H}\}$ as defined in \autoref{def:alphatau_hermLMS}. Then the structure $\M(U,\tau)$ created by \autoref{constr:MUtau} is a local Moufang set.
\end{theorem}
\begin{proof}
	By the construction, we have 
	\[U_0 = U^\tau = \{\alpha_{[1,x,r]}^\tau\mid[1,x,r]\in\mathcal{H}\} = \{\zeta_{[r,x,1]}\mid [r,x,1]\in\mathcal{H}\}\;.\]
	By \autoref{lem:gleq} and \autoref{thm:constrMouf}, it is sufficient to show $U^{\mu_{[1,x,r]}}=U_0$ for all units $[1,x,r]$. By the previous proposition, we have
	\begin{align*}
		U^{\mu_{[1,x,r]}} &= \{\alpha_{[1,y,s]}^{\mu_{[1,x,r]}}\mid[1,y,s]\in\mathcal{H}\} = \{\zeta_{[1,y,s]\mu_{[1,x,r]}}\mid [1,y,s]\in\mathcal{H}\} \\
			&= \{\zeta_{[s,y,1]}\mid [s,y,1]\in\mathcal{H}\} = U_0\;.
	\end{align*}
	Hence $\M(U,\tau)$ is a local Moufang set.
\end{proof}

\begin{definition}
	A \define{Hermitian local Moufang set} is a local Moufang set originating from an anisotropic $\Lambda$-quadratic form $(f,q)$ on a right $R$-module $W$ in the preceding manner. We denote this local Moufang set by $\M(W,f,q)$\notatlink{M(W,f,q)}.
\end{definition}

\part*{Appendices}
\appendix
	
	\chapter{Ideas for further investigation}
	During my research, I encountered many other interesting problems that could increase our understanding of local Moufang sets and how they are connected to other parts of mathematics. This appendix collects some of these problems, and in some cases, ideas on how to approach them.

\section{Twisting projective local Moufang sets}

In \autoref{chap:chap6_projective} we defined projective local Moufang sets, and we were able to characterize them in \autoref{thm:PSL2equiv}. The assumptions of this theorem are the natural conditions \hyperref[axiom:R1]{\textnormal{(R1-4)}}, but also the identity
\begin{align}
	x\mu_e\alpha_y = yR_x\alpha_{-2e}\mu_eR_x\mu_e \quad \text{for all $x\sim0$ and $y\nsim\infty$.}\label{eq:extraPSLbis}\tag{$\star$}
\end{align}
It is unclear whether or not \eqref{eq:extraPSLbis} can be removed from the conditions. In other words, we can wonder if $\M(R)$ is the only special local Moufang set structure we can get acting on the set with equivalence relation $\P^1(R)$, with root groups isomorphic to $(R,+)$ and abelian root groups. One possible way of trying to construct counterexamples to this, would be to try to \emph{twist} $\M(R)$ somehow, by changing the action of all $\alpha_{[1,r]}$ on $\class{[0,1]}$: for each $r\in R$, pick a $\tilde{\alpha}_r\in\Sym(\m)$ fixing $0$, and set
\begin{align*}
	\alpha'_{[1,r]} &\colon \begin{cases}
	                    	[1,s]\mapsto [1,s+r] \\
	                    	[m,1]\mapsto [m\tilde{\alpha}_r,1]
	                    \end{cases}
\end{align*}
One clear assumption we need to make, is the compatibility of the $\tilde{\alpha}_r$, i.e.\ we need $\tilde{\alpha}_r\tilde{\alpha}_{r'} = \tilde{\alpha}_{r+r'}$. We keep $\tau$ as defined for $\M(R)$ and set $U'=\{\alpha'_{[1,r]}\mid r\in R\}$. Then \hyperref[axiom:C1]{\textnormal{(C1-2)}} are satisfied, so we get a local pre-Moufang set $\M(U',\tau)$. If we took $m\tilde{\alpha}_r = m(mr+1)^{-1}$, we would find $\M(R)$. If we would be able to find other choices for $\tilde{\alpha}_r$ giving rise to a local Moufang set, we would have a candidate counterexample to \autoref{thm:PSL2equiv} without condition \eqref{eq:extraPSLbis}.
\begin{openproblem*}
	Can we find a local ring $R$, along with $\tilde{\alpha}_r\in\Sym(\m)$ fixing $0$ for all $r\in R$, such that $\M(U',\tau)$ as defined above is a local Moufang set?
\end{openproblem*}
Remark that if $R$ is generated by $\{r_1,\ldots,r_n\}$, then it is sufficient to define $\tilde{\alpha}_{r_i}$ for $i=1\ldots n$.

\section{Local Moufang sets from structurable algebras}

Structurable algebras have been introduced in by Allison in \cite{AllisonStructurable}, and division structurable algebras have been shown to give rise to Moufang sets.
\begin{definition*}
	Let $K$ be a field with $\characteristic(K)\not\in\{2,3\}$ and $A$ a $K$-algebra with involution $\ast$. For all $x,y,z\in A$, we set
	\[V_{x,y}(z) := \{x,y,z\} := (xy^\ast)z+(zy^\ast)x - (zx^\ast)y\;.\]
	We say $(A,\ast)$ is a \define{structurable algebra} if
	\[[V_{x,y},V_{z,w}] = V_{V_{x,y}(z),w} - V_{z,V_{y,x}(w)}\text{ for all $x,y,z,w\in A$.}\]
	We write $W_{x,y}z := V_{x,z}y$ and $W_xy:= W_{x,x}y$, and define the set of \emph{skew elements}
	\[S := \{x\in A\mid x^\ast = -x\}\;.\]
	An element $x\in A$ is called \emph{invertible} if there is a $y\in A$ such that $V_{x,y} = \id_A$, and in this case we call $\hat{x}:=y$ the \emph{inverse} of $x$. We call $A$ a \emph{structurable division algebra} if all nonzero elements are invertible.
\end{definition*}
The following theorem by L.~Boelaert now shows how one can construct a Moufang set from a structurable division algebra:
\begin{theorem*}
	Let $A$ be a structurable division algebra over a field of characteristic different from $2$, $3$ and $5$.
	Define
	\begin{align*}
		\psi&\colon A\times A\to S\colon (x,y)\mapsto xy^\ast-yx^\ast \\
		q_x&\colon S\to S\colon s\mapsto \tfrac{1}{6}\psi(x,W_x(sx))\;.
	\end{align*}
	We set $U = A\times S$ with group operation
	\[(x,s)+(y,t) = (x+y,s+t+\psi(x,y))\]
	and define a permutation $\tau$ of $U\setminus\{(0,0)\}$ by
	\[\tau \colon \begin{cases}
		            	(x,0)\mapsto (-\hat{x},0) \\
		            	(0,s)\mapsto (0,-\hat{s}) \\
		            	(x,s)\mapsto \bigl(\hat{s}(q_{\hat{x}}(s)+\hat{s})\!\hat{\phantom{x}}\hat{x}
							+\hat{s}(s+q_x(\hat{s}))\!\hat{\phantom{x}}x, -(s+q_x(\hat{s}))\!\hat{\phantom{x}}\bigr)
		      \end{cases}\]
	where all $x\neq 0$ and $s\neq 0$. Then $\M(U,\tau)$ is a Moufang set.
\end{theorem*}
\begin{proof}This is \cite[Theorem~6.25]{BoelaertL}\end{proof}

As structurable algebras generalize Jordan algebras (if $\ast$ is trivial, then a structurable algebra is a Jordan algebra), it seems likely that there is some way to define \emph{local structurable algebras}, and to construct local Moufang sets using these. An ideal in a structurable algebra is a two-sided ideal in the algebra that is invariant under $\ast$, so we could say a structurable algebra is \emph{local} if the set of non-invertible elements is a proper ideal.
\begin{openproblem*}
	Does this notion of local structurable algebras give rise to local Moufang sets?
\end{openproblem*}

\section{Suzuki-Tits local Moufang sets}

In our examples, we have omitted a few more exceptional Moufang sets. The easiest of those are the Suzuki-Tits Moufang sets (discovered in \cite{SuzukiSimplegroups}). These require a field $K$ of characteristic $2$ with a Tits endomorphism $\theta$ (an endomorphism such that $a^{\theta^2} = a^2$). Set $K^\theta = \{a^\theta\mid a\in K\}$, then $K$ is a $K^\theta$-vector space. Now take a subspace $L$ of $K$ which contains $K^\theta$.

Now set $U = L\times L$ with group operation
\[(a,b)+(c,d) := (a+c,b+d+ac^\theta)\;\]
and define a permutation $\tau$ of $U\setminus\{(0,0)\}$ by
\[(a,b)\tau := \left(\frac{b}{a^{\theta+2}+ab+b^\theta}, \frac{a}{a^{\theta+2}+ab+b^\theta}\right)\;.\]
Then $\M(U,\tau)$ is a Moufang set, called the \define{Suzuki-Tits Moufang set} of $(K,\theta,L)$. The set of points of this Moufang set can be embedded in a projective space (see for example \cite[p.~1880]{HVMSuzuki}), so we could try to generalize the Suzuki-Tits Moufang sets by trying to generalize the construction in a projective space over a local ring.

\begin{openproblem*}
	Can we generalize Suzuki-Tits Moufang sets to Suzuki-Tits local Moufang sets?
\end{openproblem*}

\section{Improper local Moufang sets}

All our examples of Moufang sets so far have been proper Moufang sets, meaning that the Hua subgroup is non-trivial. There are also improper Moufang sets, which are just sharply two-transitive permutation groups where the root groups are the point stabilizers. The only examples known originate from \emph{near-fields}
\begin{definition*}
	A triple $(F,+,\cdot)$ is called a \define{near-field} if $(F,+)$ is an abelian group, $(F^\ast,\cdot)$ is a group, $\cdot$ is left distributive over $+$, and $0\cdot a = 0$ for all $a\in F$.
\end{definition*}
If $F$ is a near-field, we can define a sharply two-transitive group
\[\mathsf{AG}(1,F) := \{\phi_{a,b}\colon F\to F\colon x\mapsto a\cdot x+b\mid a\in F^\ast, b\in F\}\;,\]
and hence we get a Moufang set.

We can generalize this to local Moufang sets using the following structures:
\begin{definition*}
	A \define{local near-ring} is an abelian group $(R,+)$ with a subgroup $M\leq R$ and an associative multiplication $\cdot$ such that
	\begin{manualenumerate}[label=\textnormal{(LNR\arabic*)},labelwidth=\widthof{(LNR1)}]
		\item $(R\setminus M,\cdot)$ is a group with identity element $1\in R$;
		\item none of the $m\in M$ have a multiplicative inverse;
		\item $\cdot$ is left distributive over $+$;
		\item $0\cdot r=0$ for all $r\in R$.
	\end{manualenumerate}
\end{definition*}
We get a group
\[\mathsf{AG}(1,R) := \{\phi_{r,s}\colon F\to F\colon x\mapsto r\cdot x+s\mid r\in R\setminus M, s\in F\}\;,\]
acting on $R$ with $r\sim s\iff r-s\in M$, with subgroup
\[U = \{\phi_{r,0}\mid r\in R\setminus M\}\]
and $\tau = \phi_{1,-1}$. Now, if we assume $\abs{R/M}>2$, we can show that \hyperref[axiom:C1]{\textnormal{(C1-2)}} are satisfied. Furthermore, we can compute that $\mu_r = \tau$ for all $r\in R\setminus M$, hence we get a local Moufang set $\M(U,\tau)$ with trivial Hua subgroup.

There are many results on the connection of sharply two-transitive groups and near-fields, for example every finite sharply two-transitive group arises from a near-field (see \cite{ZassenhausNearFields}).
\begin{openproblem*}
	To what extent can we generalize results on sharply two-transitive groups to improper local Moufang sets?
\end{openproblem*}

\section{Trees}

In \autoref{sec:Serre}, we studied the Bruhat-Tits tree corresponding to $\PSL_2$ over a local field. We have been able to connect this Bruhat-Tits tree to local Moufang sets using the nice description of the tree using lattices.

We hope that general Bruhat-Tits trees also have strong ties to local Moufang sets, though we have been unable to make a connection in full generality.
\begin{openproblem*}
	Which Bruhat-Tits trees induce actions corresponding to local Moufang sets?
\end{openproblem*}

Another notion from building theory is that of Moufang trees.
\begin{definition*}
	A \define{Moufang tree} is a tree $T$ with a bi-infinite path $(\ldots,x_{-2},x_{-1},x_0,x_1,x_2,\ldots)$ and for each $i\in\Z$, two groups $U_{(i,+)}$ and $U_{(i,-)}$ acting on $T$, satisfying the following properties:
	\begin{manualenumerate}[label=\textnormal{(MT\arabic*)},labelwidth=\widthof{(MT1)}]
		\item For each $i\in\Z$, $U_{(i,\pm)}$ fixes $(x_i, x_{i\pm1}, x_{i\pm2},\ldots)$ and acts sharply transitively on
			\[\{x\in V(T)\mid d(x,x_i) = 1\}\setminus\{x_{i\pm1}\}\;.\]
		\item For each $i<j$ we have
			\begin{align*}
				[U_{(i,+)},U_{(j,+)}]&\leq\langle U_{(i+1,+)},\ldots,U_{(j-1,+)} \rangle \\
				[U_{(i,-)},U_{(j,-)}]&\leq\langle U_{(i+1,-)},\ldots,U_{(j-1,-)} \rangle\;.
			\end{align*}
		\item For all $i\in\Z$ and all $u\in U_{(i,\pm)}$ with $u\neq\id$, there is a 
			\[m_u\in U_{(i,\mp)}u U_{(i,\mp)}\]
			such that $x_j m_u = x_{2i-j}$ for all $j\in\Z$. This means $m_u$ interchanges $(x_i, x_{i+1}, x_{i+2},\ldots)$ and $(x_i, x_{i-1}, x_{i-2},\ldots)$.\label{axiom:MT3}
		\item For all $n = m_u$ as defined in \ref{axiom:MT3}, $U_{(j,\pm)}^n = U_{(2i-j,\mp)}$.
	\end{manualenumerate}
\end{definition*}
A few properties that follow from this definition are the following (for proofs, see for example \cite{RonanBldgs}):
\begin{proposition*}
	Let $T$ be a Moufang tree. Then
	\begin{romenumerate}
		\item For every $i\in\Z$ and $k\in\N$, the set $U_{(i,\pm)}\cdots U_{(i-k,\pm)}$ is a group.
		\item For each $k\geq1$, the group $U_{(0,+)}\cdots U_{(-k,+)}$ acts sharply transitively on
		\[\{x\in V(T)\mid d(x,x_0) = k+1\text{ and }d(x,x_1)=k+2\}\;.\]
		\item The map 
			\[U_{(0,+)}\cdots U_{(-k,+)}\to U_{(0,+)}\colon u_0u_{-1}\cdots u_{-k}\mapsto u_0\]
			is a well-defined group homomorphism.
	\end{romenumerate}
\end{proposition*}
We can now set $U = U_{(0,+)}\cdots U_{(-k,+)}$ and $X = \{x\in V(T)\mid d(x,x_0) = k+1\}$. Define an equivalence on $X$ by saying $x\sim y$ if and only if the path from $x$ to $y$ does not contain $x_0$. Finally, take any $u\in U_{(0,+)}\setminus\{\id\}$, and set $\tau = m_u$. Then we can show that \hyperref[axiom:C1]{\textnormal{(C1-2)}} are satisfied, so $\M(U,\tau)$ is a local pre-Moufang set.
\begin{openproblem*}
	Which Moufang trees induce actions corresponding to local Moufang sets?
\end{openproblem*}

	\begin{otherlanguage}{dutch}
	\chapter{Nederlandstalige samenvatting}
	\section{Historische context}%

Matrixgroepen, of lineaire groepen, worden al bestudeerd sinds de 19de eeuw, met toepassingen in vele gebieden van de wiskunde en fysica.
Oorspronkelijk werden deze groepen vooral beschouwd over velden, maar vanaf het midden van de 20ste eeuw begon men ook lineaire groepen over algemene ringen te onderzoeken.
Lineaire groepen over algemene ringen zijn lastig om te bestuderen, maar voor enkele types ringen is het wel mogelijk om interessante resultaten te bewijzen over de corresponderende lineaire groepen.
In de jaren '60 bestudeerde Klingenberg een aantal lineaire groepen over een \emph{lokale ring}, een ring $R$ met een uniek maximaal ideaal $\m$.
In \cite{KlingenbergGL} bepaalde hij de normaaldelers van $\GL_n(R)$, en hij bestudeerde ook orthogonale en symplectische groepen in \cite{KlingenbergO,KlingenbergS}.

In de jaren '90 introduceerde J.~Tits \emph{Moufangverzamelingen} in \cite{TitsTwinBldgs} om enkelvoudige algebra\"ische groepen van relatieve rang $1$ axiomatisch te benaderen.
Een Moufangverzameling is een verzameling met een klasse groepen die op deze verzameling werken, en die voldoen aan een aantal axioma's.
Voor enkelvoudige algebra\"ische groepen van relatieve rang $1$ is de verzameling die van de parabolische deelgroepen, en is de klasse groepen de corresponderende worteldeelgroepen.
Er zijn een aantal equivalente manieren om naar Moufangverzamelingen te kijken: ze komen overeen met \emph{gespleten BN-paren van rang \'e\'en} (nog een notie ingevoerd door Tits), zijn sterk gerelateerd met \emph{abstracte rang \'e\'en groepen} (ingevoerd door Timmesfeld in \cite{TimmesfeldAR1G}), en zijn equivalent met \emph{delingsparen} (recent gedefinieerd door Loos in \cite{LoosRogdiv})

T.~De Medts en R.~Weiss begonnen in 2006 met de studie van willekeurige Moufangverzamelingen. Sindsdien is de theorie van Moufangverzamelingen uitgediept en zijn heel wat voorbeelden van Moufangverzamelingen beschreven.
Alle Moufangverzamelingen die we kennen zijn op \'e\'en of andere manier van algebra\"ische oorsprong, maar een classificatie van Moufangverzamelingen is nog lang niet in zicht.

De gekende Moufangverzamelingen zijn niet alleen algebra\"isch van oorprong, maar de onderliggende algebra\"ische structuren zijn allemaal `\emph{delingsstructuren}', waarin alle elementen verschillend van nul inverteerbaar zijn.
Bijvoorbeeld: projectieve Moufangverzamelingen worden gedefinieerd over \emph{alternatieve delingsalgebras}, elke \emph{Jordan delingsalgebra} geeft aanleiding tot een Moufangverzameling, en recenter bewees L.~Boelaert dat elke \emph{structureerbare delingsalgebra} een Moufangverzameling bepaalt (zie \cite{BoelaertL}).
Elke Moufangverzameling afkomstig van een enkelvoudige lineaire algebra\"ische groep van relatieve rang $1$ (over een veld met karakteristiek verschillend van $2$ and $3$), is afkomstig van zo een structureerbare delingsalgebra (dit is aangetoond in \cite{BoelaertDMStavrova}).
We zien de `delingsvoorwaarde' in de constructies van de gekende Moufangverzamelingen: er staan overal inversen!
Er zijn nog Moufangverzamelingen die niet rechtstreeks afkomstig zijn van algebra\"ische groepen, maar deze worden nog steeds gedefinieerd over velden (ook bekend als \emph{delingsringen}).

Een gerelateerd gevolg van de definitie van Moufangverzamelingen is dat morfismen van Moufangverzamelingen altijd injectief zijn. Dit betekent dat er relatief weinig morfismen zijn, en dat er bijvoorbeeld geen nuttige manier is om quoti\"enten van Moufangverzamelingen in te voeren.

We kunnen ons afvragen wat er gebeurt als we Moufangverzamelingen proberen te maken met algemenere algebra\"ische structuren, en dat is precies waar dit doctoraat over gaat.
Ter vervanging van de delingsstructuren beschouw ik \emph{lokale} structuren.
Dat betekent dat er niet-inverteerbare elementen kunnen zijn in de structuur, maar dat deze elementen in zekere zin beheersbaar zijn.
Ik probeerde de gekende constructies uit te voeren over \emph{lokale ringen} en \emph{lokale Jordanalgebras}, en destilleerde daaruit een aantal axioma's die Moufangverzamelingen veralgememen tot \emph{lokale Moufangverzamelingen}.

\section{Schets}%

Mijn doctoraat bestaat uit twee delen. In \autoref{part:theory} defini\"eren we lokale Moufangverzamelingen, en ontwikkelen we de algemene theorie van lokale Moufangverzamelingen. \autoref{part:examples} bevat een aantal voorbeelden van lokale Moufangverzamelingen, karakteriseert er enkele van, en verkent een aantal verbanden met andere gebieden van de wiskunde.

Het eerste deel begint in \autoref{chap:chap2_definitions} met de definitie van lokale Moufangverzamelingen en enkele basiseigenschappen.
Het grote verschil met Moufangverzamelingen is dat er meer structuur op de verzameling is, namelijk een equivalentierelatie. We hebben nog steeds een klasse groepen, genaamd \emph{wortelgroepen}, die op de verzameling werken. De axioma's moeten aangepast worden, zodat ze compatibel zijn met deze extra structuur.
Net zoals in de theorie van Moufangverzamelingen is het handig om twee (niet-equivalente) punten van de verzameling te kiezen (dit noemen we een \emph{basis}).
Een belangrijk nieuw concept voor lokale Moufangverzamelingen zijn de \emph{eenheden}, die in de voorbeelden verband houden met de inverteerbare elementen van de onderliggende algebra\"ische structuur.
Wanneer we gebruik maken van deze eenheden, kunnen we een groot deel van de theorie van Moufangverzamelingen veralgemenen naar lokale Moufangverzamelingen. Het bestaan van de $\mu$-afbeeldingen is daar een voorbeeld van. Dit zijn afbeeldingen die de twee punten van de basis verwisselen.

Een eerste doel bij het opbouwen van de theorie van lokale Moufangverzamelingen, is het bepalen van de stabilisator van de basis.
De \emph{Hua deelgroep}, de groep voortgebracht door producten van een even aantal $\mu$-afbeeldingen, is een goede kandidaat voor deze stabilisator.
We hebben aangetoond dat deze twee groepen samenvallen in \autoref{thm:hua2pt}. Om dit aan te tonen hebben we het concept van \emph{quasi-inverteerbaarheid}, dat afkomstig is uit Jordantheorie, ingevoerd voor lokale Moufangverzamelingen.

In \autoref{chap:chap3_constr} beschrijven we een algemene manier om lokale Moufangverzamelingen te construeren met heel wat minder informatie: we hebben slechts \'e\'en wortelgroep nodig, en \'e\'en permutatie die de basispunten verwisselt. Deze constructie is gelijkaardig aan de gekende constructie voor Moufangverzamelingen.
Enkel gebruik makend van deze twee objecten, kunnen we nu alle wortelgroepen maken, en zo alle data voor de lokale Moufangverzameling.
Het is echter niet altijd zo dat deze data voldoet aan de axioma's.
In \autoref{cor:construction_equivalentconditions} bepalen we enkele nodige en voldoende voorwaarden opdat de constructie een lokale Moufangverzameling geeft.
Deze constructie en voorwaarden zullen we zeer vaak gebruiken om lokale Moufangverzamelingen te defini\"eren, en om te bepalen wanneer een structuur een lokale Moufangverzameling is.

Homomorfismen van lokale Moufangverzamelingen worden in \autoref{chap:chap4_category} ingevoerd. Deze definitie vormt \'e\'en geheel met de definities van lokale Moufang deelverzamelingen, quoti\"enten, en uiteindelijk de categorie van lokale Moufangverzamelingen.
Het is interessant om de verbanden te zien tussen homomorfismen van lokale Moufangverzamelingen, afbeeldingen van de onderliggende verzamelingen, alsook groepsmorfismen tussen de wortelgroepen.
Deze concepten worden voornamelijk gebruikt voor het invoeren van de inverse limiet van lokale Moufangverzamelingen, en om te bepalen wanneer zo'n inverse limiet bestaat.

In het laatste theoretische hoofdstuk, \autoref{chap:chap5_special}, bestuderen we \emph{speciale} lokale Moufangverzamelingen, i.e.\ lokale Moufangverzamelingen die aan een specifieke extra voorwaarde voldoen.
Men verwacht dat deze eigenschap sterk gerelateerd is aan het abels zijn van de wortelgroepen.
Ondanks enkele noodzakelijke aanpassingen, kunnen we veel van de theorie van speciale Moufangverzamelingen veralgemenen naar speciale lokale Moufangverzamelingen.
De voornaamste resultaten in dit hoofdstuk veronderstellen dat de lokale Moufangverzameling speciaal is, en abelse wortelgroepen heeft.
In dat geval kunnen we aantonen dat de $\mu$-afbeeldingen involuties zijn (\autoref{prop:mu-involution}), en een voorwaarde geven waaruit volgt dat de wortelgroepen uniek $k$-deelbaar zijn (\autoref{prop:ndivglobal}).

Vervolgens bekijken we wat voorbeelden van lokale Moufangverzamelingen.
De eerste voorbeelden zijn de \emph{projectieve lokale Moufangverzamelingen} in \autoref{chap:chap6_projective}.
Deze worden gedefinieerd over een lokale ring $R$, waarbij de onderliggende verzameling de projectieve rechte is, en de wortelgroepen deelgroepen van $\PSL_2(R)$ zijn.
Omgekeerd, als we een lokale Moufangverzameling hebben die aan enkele voorwaarden voldoet, kunnen we er een lokale ring mee construeren.
Deze ring kunnen we gebruiken om projectieve lokale Moufangverzamelingen te karakteriseren (\autoref{thm:PSL2equiv}).
De derde sectie van dit hoofdstuk maakt de connectie met de \emph{Bruhat-Tits boom} van $\PSL_2$ over een veld $K$ met een \emph{discrete valuatie}.
Zo'n veld bevat een lokale ring, en de actie van $\PSL_2(K)$ op de boom induceert verschillende projectieve lokale Moufangverzamelingen. In het bijzonder gebruiken we de inverse limiet van lokale Moufangverzamelingen om de actie op de rand van de boom te beschrijven.

In \autoref{chap:chap7_jordan} onderzoeken we de link tussen lokale Moufangverzamelingen en Jordantheorie.
De meest natuurlijke aanpak maakt gebruik van \emph{Jordanparen}, en we geven een manier om een lokale Moufangverzameling te maken uit een lokaal Jordanpaar.
Omgekeerd kunnen we een lokaal Jordanpaar maken uit een lokale Moufangverzameling die aan enkele voorwaarden voldoet. Met deze constructie kunnen we de lokale Moufangverzamelingen die afkomstig zijn van lokale Jordanparen karakteriseren in \autoref{thm:extra}.

Zowel de projectieve lokale Moufangverzamelingen als de lokale Moufangverzamelingen afkomstig van Jordanparen zijn speciaal en hebben abelse wortelgroepen.
In \autoref{chap:chap8_hermitian} is ons doel om Hermitische lokale Moufangverzamelingen te defini\"eren.
We defini\"eren eerst orthogonale lokale Moufangverzamelingen. Deze zijn een speciaal geval van Hermitische lokale Moufangverzamelingen (en kunnen ook worden geconstrueerd m.b.v.\ Jordanparen).
Dit voorbeeld illustreert de aanpak die we gebruiken om Hermitische Moufangverzameling te veralgemenen naar de lokale versie.
De definitie van Hermitische lokale Moufangverzamelingen is nogal technisch, en om aan te tonen dat aan de axioma's van lokale Moufangverzamelingen is voldaan, hebben we heel wat algebra\"ische manipulaties nodig.

Een groot deel van deze thesis is vervat in de twee artikels \cite{DMRijckenPSL,DMRijckenLMS_JP}. \autoref{chap:chap4_category}, \autoref{sec:Serre} en \autoref{chap:chap8_hermitian} zijn de voornaamste delen van deze thesis die niet in de artikels te vinden zijn.
	\end{otherlanguage}

\backmatter

	\bookmarksetup{startatroot}
	\prebiblio
	\bibliography{universal.bib}
	\notocchapterspace
	\printindex
	\printnotation
	
	\newpage\null
	\pagestyle{empty}
	\cleardoublepage
	\vspace*{6em}
	\fquote{We were amazed and overwhelmed by \\
		the strange beauty of these sounds.\\
		The rest, to us, is until this day unexplainable.\\
		The transmission ended.}
	
\end{document}